\documentclass{amsart}
\usepackage{mathrsfs}
\usepackage[all]{xy}

\usepackage{mathtools}
\usepackage{mathbbol}
\usepackage{wasysym}
\usepackage{amssymb}

\usepackage{MnSymbol}
\usepackage{comment}
\usepackage{enumerate}
\usepackage{xcolor}

\usepackage{ushort}

\usepackage{changepage}
\usepackage{algorithm}
\usepackage[noend]{algpseudocode}
\makeatletter
\def\BState{\State\hskip-\ALG@thistlm}
\makeatother

\usepackage{ifpdf}
\ifpdf %if using pdfLaTeX in PDF mode
  \usepackage[pdftex]{graphicx}
  \DeclareGraphicsExtensions{.pdf,.png,.jpg,.jpeg,.mps}
  \usepackage{pgf}
\else %if using LaTeX or pdfLaTeX in DVI mode
  \usepackage{graphicx}
  \DeclareGraphicsExtensions{.eps,.bmp}
  \usepackage{pgf}
  \usepackage{pstricks}
\fi
\usepackage{epic,bez123}
\usepackage{wrapfig}% package for wrapfigure environment

\usepackage{soul}
\usepackage{url}
\usepackage{tikz}

\newtheorem{thm}{Theorem}[section]
\newtheorem{prop}[thm]{Proposition}
\newtheorem{cor}[thm]{Corollary}
\newtheorem{lem}[thm]{Lemma}

\newtheorem*{claim}{Claim}

\newtheorem{assumpD}{}

\theoremstyle{definition}
\newtheorem{defn}[thm]{Definition}

\theoremstyle{remark}
\newtheorem{quest}[thm]{Question}
\newtheorem{rem}[thm]{Remark}
\newtheorem{example}[thm]{Example}
\newtheorem{examples}[thm]{Examples}

\newcommand{\isom}{\mathrm{Isom}}

\newcommand{\U}{\mathrm X}
\newcommand{\bU}{\overline{\mathrm X}}
\newcommand{\hU}{\partial_h{\mathrm X}}
\newcommand{\pU}{\partial{\mathrm X}}

\newcommand{\act}{\curvearrowright}

\newcommand{\pG}{\Lambda (\Gamma o)}
\newcommand{\ccG}{\Lambda^{\textrm{con}}{(\Gamma o)}} 
\newcommand{\cG}{\Lambda^{\textrm{con}}_{\textrm{r,F}}{(\Gamma o)}}
\newcommand{\rcG}{\Lambda^{\textrm{reg}}_{\textrm{r,F}}{(\Gamma o)}}
\newcommand{\hG}{\Lambda^{\mathrm{hor}}{(Ho)}}
\newcommand{\HG}{\Lambda^{\mathrm{Hor}}{(Ho)}}

\newcommand{\mG}{\Lambda^{\mathrm{Myr}}{(\Gamma o)}}

\newcommand{\bX}{\partial_h{\mathrm X}}

\newcommand{\p}{\mathcal P}
\newcommand{\f}{\mathscr F}

\newcommand{\dirac}[1]{{\mbox{Dirac}}{(#1)}}

\newcommand{\e}[1]{\omega_{#1}}

\newcommand{\ax}{\mathrm{Ax}}
\newcommand{\diam }[1]{{\textbf{diam}\big(#1\big)}}
\newcommand{\proj}{\textbf{d}}

\newcommand{\pmf}{\mathscr {PMF}}
\newcommand{\mf}{\mathscr {MF}}

\newcommand{\ue}{\mathscr {UE}}

\newcommand{\T}{\mathcal {T}}

\newcommand{\len }{\ell}

\usepackage[bookmarks=true, pdfauthor={YANG Wenyuan}]{hyperref}

\usepackage[margin=1.1in]{geometry}
\usepackage[normalem]{ulem}

\newtheoremstyle{query}%
{}{}%space above/below
{\color{red}}%body style
{}%heading indent
{\sffamily\bfseries}{:}{12pt}%heading style/punctuation/space after
{}% head spec
\theoremstyle{query}

\newcommand{\ywy}[1]{{\color{red}{#1}}}

\begin{document}

\title[Confined subgroups in groups with contracting elements]{Confined subgroups in groups with contracting elements}

\author{Inhyeok Choi}
\address{June E Huh Center for Mathematical Challenges\\
Korea Institute for Advanced Study\\
Seoul, South Korea}
\email{inhyeokchoi48@gmail.com}

\thanks{(I.C.) Supported by Samsung Science \& Technology Foundation (SSTF-BA1301-51), Mid-Career Researcher Program (RS-2023-00278510) through the NRF funded by the government of Korea, KIAS Individual Grant (SG091901) via the June E Huh Center for Mathematical Challenges at KIAS, and by the Fields Institute for Research in Mathematical Sciences.}
\author{Ilya Gekhtman}

\author{Wenyuan Yang}

\address{Beijing International Center for Mathematical Research\\
Peking University\\
 Beijing 100871, China
P.R.}
\email{wyang@math.pku.edu.cn}
%\curraddr{Le Département de Mathématiques de la Faculté des
%Sciences d'Orsay, Université Paris-Sud 11, France}
%\email{yabziz@gmail.com}

\thanks{(W.Y.) Partially supported by National Key R \& D Program of China (SQ2020YFA070059) and  National Natural Science Foundation of China (No. 12131009 and No.12326601)}

\author{Tianyi Zheng}
% \author{Author Two}
% \address{Mathematical Research Section, School of Mathematical Sciences,
% Australian National University, Canberra ACT 2601, Australia}
% \email{two@maths.univ.edu.au}
% \thanks{Support information for the second author.}
\address{Department of Mathematics, University of California San Diego, 9500 Gilman Dr. La Jolla, CA 92093, USA.} \email{tzheng2@math.ucsd.edu}
%    General info

\date{May 4th 2024}

\dedicatory{}


\begin{abstract} 
In this paper, we study the growth of confined subgroups through boundary actions of groups with contracting elements. We establish that the growth rate of a confined subgroup is strictly greater than half of the ambient growth rate in groups with purely exponential growth. Along the way, several results are obtained on the Hopf decomposition for boundary actions of subgroups with respect to conformal measures. In particular, we prove that confined subgroups are conservative, and examples of subgroups with nontrivial Hopf decomposition are constructed. We show a connection between Hopf decomposition and quotient growth and provide a dichotomy on quotient growth of Schierer graphs for subgroups in hyperbolic groups.

\end{abstract}

\maketitle
\setcounter{tocdepth}{2} \tableofcontents

\section{Introduction}

Let  $G$ be a locally compact, second countable topological group. The space $\mathrm{Sub}(G)$ of all closed subgroups in $G$, equipped with the Chabauty topology, is a compact metrizable space on which $G$ acts continuously by conjugation. Measurable and topological dynamics of the action $ G\act \mathrm{Sub}( G)$ have been instrumental in recent advances in the study of lattices in semi-simple Lie groups, see  \cite{ABBGNRS,FG23}. In particular, on the measurable dynamics side, invariant random subgroups (IRS), which are invariant measures on $\mathrm{Sub}(G)$, have attracted a lot of interest since the terminology was coined in \cite{AGV14}. Their topological counterpart, namely $G$--minimal systems in $\mathrm{Sub}(G)$, were introduced as uniformly recurrent subgroups (URS) in \cite{GW15}. URS play an important role in the connections between reduced $C^{\ast}$--algebras and topological dynamics developed in \cite{KK17,BKKO17}. 

The notion of confined subgroups was introduced in  \cite{HZ97} in the study of the ideal structure of group rings of simple locally finite groups. More recently, it has been further investigated in \cite{LBMB18,LBMB22} for a variety of countable groups of dynamical origin. Let $\Gamma, H$ be two subgroups of $G$, where $H$ is not necessarily contained in $\Gamma$. We say that  $H$ is  \textit{confined} by $\Gamma$ in $G$ if its $\Gamma$--orbit (\emph{i.e.} all conjugates of $H$ under $\Gamma$) does not accumulate in $\mathrm{Sub}(G)$ to the trivial subgroup. Equivalently, there exists a compact \textit{confining} subset $P\subseteq G$ such that $g^{-1}Hg\cap P\setminus\{1\}\ne \emptyset$ for all $g\in \Gamma$. When $\Gamma=G$, we simply say that $H$ is a confined subgroup of $G$. Every subgroup $H$ in a non-trivial URS of $\Gamma$ is confined by $\Gamma$. Conversely, if $H$ is confined by $\Gamma$, then the $\Gamma$--orbit closure of $H$ in $\mathrm{Sub}(G)$ contains a non-trivial $\Gamma$--minimal system. 

In a semi-simple Lie group, discrete confined torsion-free subgroups can be described by the geometric condition that the action on the associated symmetric space has bounded injectivity radius from above. Fraczyk-Gelander \cite{FG23} showed that in a higher-rank simple Lie group, discrete confined subgroups are exactly lattices, confirming a conjecture of Margulis. This is obviously false in the rank-1 case: indeed, normal subgroups of uniform lattices are confined in both the lattice and the ambient Lie group. Instead, Gelander posed a conjecture on the growth rates of discrete confined subgroups. This conjecture was recently proved by Gekhtman-Levit \cite{GL23}. The methods of \cite{FG23, GL23} have probabilistic ingredients, using stationary random subgroups to derive geometric results.  

The main goal of the present article is to investigate confined subgroups in a general class of discrete groups $\Gamma$, which  acts on a geodesic metric space $(\U,d)$ with \textit{contracting elements} defined as follows.
\begin{defn}\footnote{There are several notions of contracting elements in the literature. Our notion here is usually called strongly contracting by other authors. As there are no other ones used in this paper, we'd keep use of this terms   to be consistent with \cite{YANG10,YANG22}} 
An isometry $g\in \isom(\U)$ is called  \textit{contracting} if the orbital map
$$
n\in \mathbb Z\longmapsto g^no\in \U
$$
is a quasi-isometric embedding, so that the image $A:=\{g^no: n\in \mathbb Z\}$ has \textit{contracting property}: any ball $B$ disjoint with $A$ has $C$-bounded projection to $A$ for a constant $C\ge 0$ independent of $B$. 
\end{defn}
Contracting elements are usually thought of as hyperbolic directions, exhibiting negatively curved feature of the ambient space. The main   examples we keep in mind with application  include the following:
\begin{examples}\label{SCCexamples}
    \begin{itemize}

        \item Fundamental groups of rank-1 Riemannian manifolds with a finite geodesic flow invariant measure of maximal entropy (the BMS measure), where contracting elements are exactly loxodromic elements.
        \item Hyperbolic groups; more generally, relative hyperbolic groups, where loxodromic elements are contracting  elements.
        \item CAT(0)-groups with rank-1 elements, including right-angled Coxeter/Artin groups, where  contracting elements coincide with rank-1 elements.
        \item 
        Mapping class groups on Teichm\"{u}ller space endowed with Teichm\"{u}ller metric, where  contracting elements are exactly pseudo-Anosov elements.
       
    \end{itemize}
\end{examples}

On a heuristic level, one might ask to what extent confined subgroups are large, or resemble normal subgroups. To be more precise, we consider two closely related aspects: the dynamics of a discrete subgroup $H<\mathrm{Isom}(\U,d)$ acting on a suitable boundary of the space $\U$; and the growth rates of $H$ and its quotient $\U/H$. Our investigation focuses mainly on the questions related to these aspects for a discrete confined subgroup $H$.  

Fix a basepoint $o\in \U$. We define the \textit{growth rate} (or \textit{critical exponent}) of the orbit $Ho$ as
\begin{equation}\label{GrowthRateDefn}
\e{H}=\limsup_{n\to\infty}\frac{\log\sharp N_H(o,n)}{n},     
\end{equation}
where $N_H(o,n)=\{go: g\in H, d(o,go)\le n\}$. When $\Gamma$ is discrete, we also consider the growth rate $\e{\Gamma/H}$, defined below in (\ref{QuotientGrowthRateDefn}) analogously to (\ref{GrowthRateDefn}), of the image $\Gamma o$ under the projection $\U \to \U/H$ equipped with the quotient metric.  \footnote{In the literature, if $H$ is a subgroup of $\Gamma$, $\e H$  is usually called relative growth; the relative growth of a normal subgroup $H$  is referred as \textit{cogrowth} relative to the \textit{growth} of $\Gamma/H$ (particularly on $\Gamma=\mathbf F_d$, see \cite{GH}). Occasionally, the terminology is reversed: $\e{\Gamma/H}$ is called cogrowth relative to the growth $\e \Gamma$ of $\Gamma$ (see \cite{Ol17,BO10}). We find that the former is more convenient in this paper.} 

These growth rates have been intertwined with other quantities in the study of the geometric and dynamic aspects of discrete groups for a long history. In rank-1 symmetric spaces, the famous Elstrodt-Patterson-Sullivan-Corlette formula relates growth rate to the bottom $\lambda_0$ of Laplace-Beltrami spectrum on $\U/H$ as follows:
$$
\begin{aligned}
\lambda_0(\U/H)=\begin{cases}
\e H(h(\U)-\e H), &\e H\ge h(\U)/2\\
\frac{h^2(\U)}{4},&\e H \le h(\U)/2
\end{cases}    
\end{aligned}
$$
where $h(\U)$ is the Hausdorff dimension of visual boundary. If $\Gamma$ is a lattice, then $h(\U)=\e \Gamma$.  

In the discrete setting, by the cogrowth formula due to Grigorchuk, for $H<\mathbf{F}_d$ a subgroup in the free group on $d$ generators, the spectral radius $\rho(\U/H)$ of simple random walks on  $\U/H$ is given by
$$
\begin{aligned}
\rho(\U/H)=\begin{cases}
\frac{\sqrt{2d-1}}{2d}(\frac{\sqrt{2d-1}}{\mathrm{e}^{\e H}}+\frac{\mathrm{e}^{\e H}}{\sqrt{2d-1}}), &\mathrm{e}^{\e H} \ge \sqrt{2d-1}\\
\frac{\sqrt{2d-1}}{d},&\mathrm{e}^{\e H} \le \sqrt{2d-1}
\end{cases}    
\end{aligned}
$$
where $\U$ is the $2d$--regular tree (the standard Cayley graph of  $\mathbf{F}_d$), and $\log (2d-1)$ is the Hausdorff dimension of the space of ends of $\U$.

 The seminal works of Kesten \cite{Kes59} and Brooks \cite{Br85} characterize amenability in terms of spectral radius of random walks and Laplace spectrum on Riemannian manifolds, respectively. Using the above formulae, these results can be interpreted in a common geometric setting: for $\Gamma$ of some natural classes of groups, a normal subgroup $H<\Gamma$ attains the maximal relative growth rate $\e H=\e \Gamma$ if and only if the quotient $\Gamma/H$ is amenable. This perspective is fruitful, with the latest results  in  \cite{CDS, CDST} for strongly positively recurrent actions (SPR) or statistically convex-cocompact actions (SCC) on hyperbolic spaces. In a different direction, Kesten's theorem is generalized beyond normal subgroups: for IRS in \cite{AGV16} and URS in \cite{F20}. We refer the reader to these papers and reference therein for a more comprehensive introduction. To put into perspective, we draw in Fig. \ref{fig:actionssurvol}  the relation between various actions including SPR and SCC actions considered in this paper. 

The question of whether the inequality $\e{\Gamma/H}<\e \Gamma$ holds for normal subgroups $H\neq\Gamma$ has been actively investigated \cite{AL, Sam2,DPPS} since the introduction of the notion of {\it growth tightness} in \cite{GH}. The most general results currently available for normal subgroups of divergence type actions are given by \cite{ACTao, YANG10}, while the situation for general subgroups $H$ remains largely unexplored in the literature. Furthermore, the relations between $ \e H, \e \Gamma$ and $\e{\Gamma/H}$ still remain mysterious, despite recent works \cite{JM20,Coulon22}.

In this paper, for a confined subgroup $H$ of $\Gamma$, we establish the conservativity of the boundary action of $H$, a strict lower bound $\e H>\e \Gamma/2$ (such an inequality is sometimes referred to as {\it cogrowth tightness}), and an inequality relating $\e H$ to $\e \Gamma$ and $\e{\Gamma/H}$. Our approach studies boundary actions equipped with conformal measures and relies on geometric arguments with contracting elements. Notably, we do not rely on any input from random walks or considerations of probability measures on the Chabauty spaces. We first present some applications before stating our general results. 
 
%Assume that $\e H<\infty$. We say the group $H$ is \textit{elementary} if it contains a finite index cyclic group. If $H$ is non-elementary and contains a contracting element, then $H$ contains a non-abelian free subgroup, thus $\e H>0$.

%Our methods use geometric and elementary arguments. One of our motivations is to give a more direct proof of the results in \cite{GL23}.  
\subsection{Main applications}\label{MainAppls}
Before presenting them in full generality, we state some of our main results in several well-studied geometric settings. 

First of all, let us consider a proper geodesic Gromov hyperbolic space or a CAT$(0)$ space $\U$, equipped with    the Gromov or visual boundary $\pU$ in the first and second cases respectively. Let $\Gamma<\isom(\U)$ be a non-elementary discrete subgroup containing a loxodromic or rank-1 element accordingly.  If the action $\Gamma\act \U$ has purely exponential growth (or more generally,  of divergence type), there exists a unique $\e \Gamma$--dimensional Patterson-Sullivan measure class $\mu_{\mathrm{PS}}$ on the Gromov or visual boundary $\pU$ constructed from the action $\Gamma\act \U$. Denote by $E(\Gamma)$ the set of isometries in $\isom(\U)$ which fix pointwise the limit set $\pG$ (Definition \ref{EllipticRadicalDefn}).  
\begin{thm}\label{mainthmhyperbolicandCAT(0)}
Assume the action $\Gamma\act \pU$ has purely exponential growth. Let $H<\isom(\U)$ be a nontrivial torsion-free discrete subgroup confined by $\Gamma$ with a compact confining subset. Then $\e H\ge \e \Gamma/2$. Furthermore,
\begin{enumerate}
    \item 
    If $H$ preserves the measure class of $\mu_{\mathrm{PS}}$, then the action $H\act (\partial X,\mu_{\mathrm{PS}})$ is conservative.
    \item 
    If $H$ admits a finite confining subset intersecting trivially $E(\Gamma)$, then $\e H> \e \Gamma/2$ and $\e H+\e {\Gamma/H}/2\ge \e \Gamma$.
    %\item 
    %If $H$ is normal in $\Gamma$, then the volume of the covering $\widetilde M/H$ associated to $H$ grows slower than  $\widetilde M$. 
\end{enumerate}   
\end{thm}
In many circumstances,  the group of isometries fixing boundary pointwise $\pU$ is trivial (or finite),  \emph{e.g.} if $\U$ is a CAT(-1) space or a geodesically complete CAT(0) space with a \textit{rank-1} geodesic (\emph{i.e.}, bounding no half flat). In case of $\pG = \pU$, $E(\Gamma)$ is finite and the assumption in the item (2) is always fulfilled.

%Let $(\widetilde M,\tilde x)\to (N,x)$ be a covering map of Riemannian manifolds. We say the volume of $N$ grows slower than $\widetilde M$ if $\frac{vol_N B(x,r)}{vol_{\widetilde M} B(\tilde x,r)}\to 0$ as $r\to \infty$ (which does not depend on the choice of $x\in N,\tilde x\in \widetilde M$).  

An important subclass of (uniquely) geodesically complete CAT$(0)$ spaces are provided by  Hadamard manifolds (\emph{i.e.}  a complete, simply connected $n$--dimensional Riemannian manifold  with non-positive sectional curvature). Its visual boundary $\partial  \U$, which
is defined as the set of equivalence classes of geodesic rays, is homeomorphic to $\mathbb S^{n-1}$.
Let  $\Gamma<\isom(\U)$ be a torsion-free discrete  group.
The quotient manifold $M=\U/\Gamma$ is called rank-1 if it admits a closed geodesic without a perpendicular parallel Jacobi field. In other words, $M$ is rank-1 if and only if $\Gamma$ contains a contracting element \cite[Theorem 5.4]{BF2}.  

\begin{thm}\label{mainthmRANK1}
Let $M=\U/\Gamma$ be a rank-1 manifold, and assume that the geodesic flow on the unit tangent bundle $T^{1}M$ has finite measure of maximal entropy. Let $\mu_{\mathrm{PS}}$ be the unique $\e \Gamma$--dimensional Patterson-Sullivan measure class on $\partial \U$ constructed from the action $\Gamma\act \U$. Let $H<\isom(\U)$ be a nontrivial torsion-free discrete subgroup confined by $\Gamma$ with a compact confining subset. Then $\e H\ge \e \Gamma/2$. Furthermore,
\begin{enumerate}
    \item 
    If $H$ preserves the measure class of $\mu_{\mathrm{PS}}$, then the action $H\act (\partial \U,\mu_{\mathrm{PS}})$ is conservative.
    \item 
    If $H$ admits a finite confining subset intersecting trivially $E(\Gamma)$, then $\e H> \e \Gamma/2$ and $\e H+\e {\Gamma/H}/2\ge \e \Gamma$.
    %\item 
    %If $H$ is normal in $\Gamma$, then the volume of the covering $\widetilde M/H$ associated to $H$ grows slower than  $\widetilde M$. 
\end{enumerate}   
\end{thm}
As noted above, if $\Gamma\act\U$ has full limit set $\Lambda(\Gamma o)=\pU$, then $E(\Gamma)$ is trivial.

A few remarks are in order on the statements and their background. We refer the reader to the subsequent subsections for more detailed discussions.  
\begin{rem}
\begin{itemize}
\item 
The fundamental group $\Gamma$ is not necessarily a lattice in $\isom(\U)$. The maximal entropy measure lies in the measure class  $\mu_{\mathrm{PS}}\times \mu_{\mathrm{PS}} \times \mathbf{Leb}$ modulo $\Gamma$--action on $\partial \U \times \partial \U \times \mathbb R$ (usually called the Bowen-Margulis-Sullivan measure in the pinched negatively curved case \cite{DOP} or Knieper measure in the rank-1 case \cite{Kneiper1}).  The finiteness of this measure can be characterized in several ways (\emph{e.g.} \cite{DOP} and \cite{Roblin}). 

If $\Gamma$  is a uniform lattice and $\U$ is a rank-1 symmetric space of noncompact type, then $\mu_{\mathrm{PS}}$ is the Lebesgue measure on $\partial\U$ invariant under $\isom(\U)$ (see \cite{Quint}). In this restricted setting, if $H$ is confined by $\isom(\U)$ (equivalently by any uniform lattice, see Lemma \ref{CharConfinedSubgroups}), the strict inequality $\e H> \e \Gamma/2$ was proved by Gekhtman-Levit \cite{GL23} (without the finite confining subset assumption). 

\item The item (1) is known in confined Kleinian groups for the Lebesgue measure on the sphere (\cite[Theorem 2.11]{McMBook}). However, its proof does not generalize to Gromov hyperbolic spaces. Our proof   is different and works in general metric spaces; see Theorem  \ref{ConfinedConsThm}.

\item
If $H$ is contained in $\Gamma$, the assumptions after ``Furthermore" are redundant, and the corresponding conclusions hold without them.  A confined subgroup in a non-uniform lattice $\Gamma$ is not necessarily confined in $\isom(\U)$, so the inequality in item (2) does not follow from \cite{GL23}.

%\item In the item (3), if $\Gamma$ has parabolic gap condition, then growth tightness implies $\e \Gamma>\e {\Gamma/H}$ so the volume of $\widetilde M/H$ actually grows exponentially slower than $\widetilde M$.  In general, it is possible that $\e \Gamma=\e {\Gamma/H}$, so the speed cannot be improved.

\end{itemize}
\end{rem}

%Note that the confinedness of a subgroup is a commensurable invariant. If $G$ is residually finite, then the torsion-free assumption on $H$ could be removed.
Another application is given for confined subgroups in the mapping class group $\mathrm{Mod}(\Sigma_g)$ of a closed orientable surface $\Sigma_g$ ($g\ge 2$). The finitely generated group $\mathrm{Mod}(\Sigma_g)$ is actually the orientation preserving isometry group of the Teichm\"{u}ller space $\T_g$ with Teichm\"{u}ller metric, on which it acts   properly with growth rate $(6g-6)$. Investigating the similarities between $\mathrm{Mod}(\Sigma_g)$ and  lattices in semi-simple Lie groups has been an active research theme.  The following result fits into the rank-1 phenomenon of $\mathrm{Mod}(\Sigma_g)$.

\begin{thm}\label{mainthmMCG}
Consider  the measure class preserving action of $G=\mathrm{Mod}(\Sigma_g)$  on the Thurston boundary $\pmf$ of the  Teichm\"{u}ller space $\T_g$ with Thurston measure $\mu_{\mathrm{Th}}$ (cf. \cite{ABEM} recalled in Example \ref{ThurstonMeasure}). 

Let $H<G$ be a nontrivial  confined subgroup. If $g=2$, assume $H$ is infinite.  
Then the following holds 
\begin{enumerate}
    \item 
    $H\act (\pmf,\mu_{\mathrm{Th}})$ is conservative. 
    \item 
    $\e H> \e G/2=3g-3$.
    \item 
    $\e H+{\e{G/H}}/{2}\ge 6g-6$.
\end{enumerate} 
\end{thm}

The items (2) and (3) for normal subgroups were proved by Arzhantseva-Cashen \cite{AC20}, and Coulon \cite{Coulon22} respectively. The item (1) is new even for the normal subgroup case.
\begin{rem}
The Mumford compactness theorem implies that $G$ acts cocompactly on the $\epsilon$--thick part of $\T_g$ for any fixed $\epsilon>0$.
It follows that in the setting of Theorem \ref{mainthmMCG}, $H$ is a confined subgroup of $G$ if and only if $\T_g/H$ has bounded injectivity radius from above over any fixed thick part of $\T_g$. 
\end{rem}

%We say that a subgroup $H$ is \textit{commensurated} in $G$ if any conjugate of $H$ is commensurable with $H$: $gHg^{-1}\cap H$ for given $g\in G$ is of finite index. If $G\act \U$ is proper action with contracting elements, then $H$ is a non-elementary subgroup with contracting elements (\cite{Coulon22}). Fix any set of three independent of elements $F$ in $H$. Then for any $g\in G$, there exist an integer $n>0$ and $f\in F$ so that $gf^ng^{-1}\in H$. We then could prove that $H$ is conservative, and $\e H>\e G/2$ (Coulon) and cogrowth inequality.

\subsection{Ergodic properties of boundary actions}\label{SSecErgodicity}
It is classical that any action of a countable group equipped with a quasi-invariant measure admits the Hopf decomposition into the conservative and dissipative parts. In 1939, Hopf found the dichotomy in the geodesic flow of Riemannian surface with constant curvature that the flow invariant measure is either ergodic (thus conservative) or completely dissipative. Hopf's result was extended to higher dimensional hyperbolic manifolds by Sullivan \cite{Sul}, forming a still-growing collection of the now-called Hopf-Tsuji-Sullivan dichotomy results in a number of settings by many authors. For the convenience of the reader, we state the HTS dichotomy   for conformal measures on the Gromov boundary for a discrete group on CAT(-1) space (\cite{Roblin}):
\begin{thm}[HTS dichotomy for CAT(-1) spaces]\label{HTSCAT(-1)}
Assume that $H\act \U$ is a proper action on CAT\emph{(-1)} space. Let $\{\mu_x:x\in\U \}$ be a $\omega$-dimensional $H$-conformal density on the Gromov boundary $\pU$ for some $\omega>0$. Then we have the dichotomy: either the following equivalent statements hold:
\begin{enumerate}
    \item[(I.1)] 
    $\p_H(s, x,y)=\sum_{h\in H}\mathrm{e}^{-sd(x,hy)}$  is divergent at $s=\omega$.
    \item [(I.2)] 
    $\mu_x$ is supported on the set of conical points $\Lambda^{\mathrm{con}}(Ho)$.
    \item [(I.3)] 
    $\mu_x\times \mu_x$ is ergodic. In particular, $\mu_x$ is ergodic.
\end{enumerate}
or the following equivalent statements hold:
\begin{enumerate}
    \item [(II.1)]  
    $\p_H(s, x,y)=\sum_{h\in H}\mathrm{e}^{-sd(x,hy)}$  is convergent at $s=\omega$.
    \item [(II.2)]
    $\mu_x$ is null on the set of conical points $\Lambda^{\mathrm{con}}(Ho)$.
    \item [(II.3)]
    $\mu_x\times \mu_x$ is completely dissipative. 
\end{enumerate}  
Under the set (I) of conditions, we must have $\omega=\e H$ and  $\mu_x$ is unique up to scaling. 
\end{thm}
This dichotomy and its variants have then been established for conformal densities on the visual boundary of CAT(0) spaces \cite{L18}; and in our context, groups with contracting elements acting on the convergence boundary $\pU$ in a sense in \cite{YANG22}. See Coulon's recent works \cite{Coulon22,Coulon24} for a more complete statement on horofunction boundary.   
 
The notion of {\it convergence boundary} $\pU$ (Definition \ref{ConvBdryDefn}) provides a unified framework for the following boundaries equipped with a conformal density (cf. Definition \ref{ConformalDensityDefn} and Examples \ref{ConvbdryExamples}): 
\begin{examples}\label{ConformalDensityExamples}
\begin{itemize}
    \item 
The prototype example is the rank-$1$ symmetric space $\U$ of non-compact type, with  a family of $\isom(\U)$--equivariant conformal measures $\{\mu_x:x\in\U\}$ in the Lebesgue measure class  on the visual boundary of $\U$. This class of conformal measures can be realized as Patterson-Sullivan measures for any uniform lattice $\Gamma<\isom(\U)$. See \cite{Quint}.

\item 
The Gromov boundary of a hyperbolic group, with quasi-conformal density constructed first by Patterson \cite{Patt} from a proper action of $G$ on  a hyperbolic space $\U$, and then studied extensively by Sullivan \cite{Sul,Sul2,Sul81} in late 1970s, with a long array of results \cite{Coor,Roblin} by many authors, to name a few.
\item 
Visual boundary of a CAT(0) space admitting a proper action with rank-1 elements. The Patterson's construction equally applies to produce a conformal density on visual boundary. 
\item 
Thurston boundary of Teichm\"{u}ller space with a conformal density constructed from Thurston measures (in the Lebesgue measure class) on the space of measured foliations in \cite{ABEM}. The conformal density could be realized as Patterson-Sullivan measures. See Example \ref{ThurstonMeasure} for more details. 
\item 
In recent works \cite{Coulon22,YANG22}, the horofunction boundary $\hU$ has been proven to be a convergence boundary for any proper action of a group with a contracting element; furthermore, on the horofunction boundary, a good theory of conformal density can be developed: in particular, the Shadow Lemma and HTS dichotomy are available. 
\end{itemize}
\end{examples}

Suppose that the space $\U$ admits a convergence boundary $\pU$ (or keep one of these examples in mind). Motivated by rank-1 symmetric space, we assume that there is an auxiliary proper action $\Gamma\act \U$, in order to endow a $\e\Gamma$--dimensional $\Gamma$--equivariant conformal density $\{\mu_x:x\in\U\}$ on $\pU$. Sometimes, we assume that $\Gamma\act \U$ is \textit{statistically convex-cocompact}     in the sense of \cite{YANG10} (SCC action; see Definition \ref{SCCDefn}); this is the case in Theorem \ref{HalfgrowthimplyDiss}. Among the groups in Example \ref{SCCexamples}, the first three classes are known to admit SCC actions; and all of these examples have \textit{purely exponential  growth} (PEG action):
\begin{equation}\label{PEGDefn}
\forall n\ge 0: \; \sharp N_\Gamma(o,n) \asymp \mathrm{e}^{n\e \Gamma}.\quad
\footnote{The symbol  $\asymp$ denotes the two sides equal up to a bounded multiplicative constant.} 
\end{equation}
Most of results actually only assume a PEG action $\Gamma\act \U$, or even a proper action of divergence type (DIV action) in a greater generality.  The class of DIV actions is naturally featured, as one of the two alternatives, in the HTS dichotomy \ref{HTSCAT(-1)}. To be clear, we shall make precise the actions $\Gamma\act \U$ assumed in our results stated in what follows, and refer to Fig. \ref{fig:actionssurvol} for the implication between these various actions. 

Our first main object is to analyze the Hopf decomposition for a measure-class-preserving action of a group $H$ on $(\pU, \mu_o)$.

In contrast to the HTS dichotomy of geodesic flow invariant measures, which can be applicable to the $H$--action on $(\pU \times \pU, \mu_o\times \mu_o)$, the ergodic behavior of actions on one copy of the boundary, which is intimately linked with the horocylic flow \cite{Kai00}, exhibits more complicated situations. For Kleinian groups, Sullivan \cite[Theorem IV]{Sul81} characterizes the conservative component with respect to the Lebesgue measure on the boundary in terms of the (small) horospheric limit set. This leads to several important applications in Ahlfors-Bers quasi-conformal deformation theory and Mostow-type rigidity theorems (\cite[Sect. V \&VI]{Sul81}).   
In a recent work \cite{GKN}, Grigorchuk-Kaimanovich-Nagnibeda carried out a detailed analysis of the boundary action of a subgroup $H$ in a free group $\mathbf F_d$  with respect to the uniform measure on $\partial \mathbf F_d$. These works serve as a source of inspiration in our considerations in a more general framework.

The limit set  $\Lambda(Ho)$  denotes the set of accumulation points of $Ho$ in a convergence boundary $\pU$. In the last two cases in Example \ref{ConformalDensityExamples}, $\Lambda(Ho)$  may depend on $o\in \U$. To avoid this difficulty, a nontrivial partition $[\cdot]$ on $\pU$ was introduced so that the set $[\Lambda(Ho)]$ of $[\cdot]$--classes on $\Lambda(Ho)$ is independent of the choice of base points.  Analogously to the notions in the actions of the convergence group, we can define {\it conical points} and the {\it big/small horospheric limit points}. More details and precise definitions are provided in Section \ref{prelim}.  

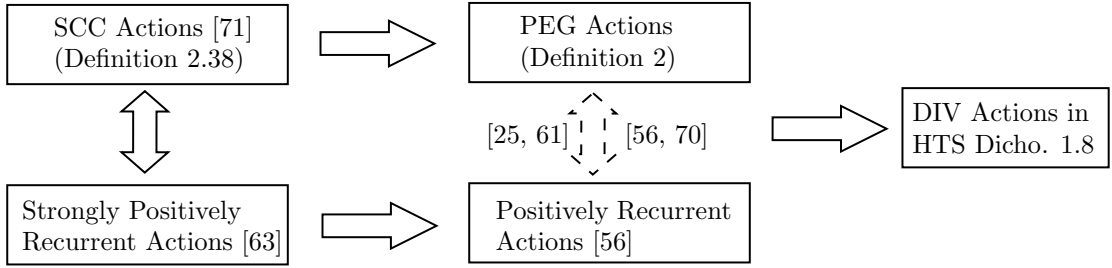
\begin{figure}
    \centering

\tikzset{every picture/.style={line width=0.75pt}} %set default line width to 0.75pt        

\begin{tikzpicture}[x=0.75pt,y=0.75pt,yscale=-1,xscale=1]

%Shape: Rectangle [id:dp3738853128576649] 
\draw   (15,58) -- (156.5,58) -- (156.5,98) -- (15,98) -- cycle ;
%Shape: Rectangle [id:dp38507113404414484] 
\draw   (248,57) -- (389.5,57) -- (389.5,97) -- (248,97) -- cycle ;
%Shape: Rectangle [id:dp4810713302343306] 
\draw   (15,149) -- (156.5,149) -- (156.5,189) -- (15,189) -- cycle ;
%Shape: Rectangle [id:dp3079012128311225] 
\draw   (248,148) -- (389.5,148) -- (389.5,188) -- (248,188) -- cycle ;
%Right Arrow [id:dp5915364384596116] 
\draw   (173,71.75) -- (207.5,71.75) -- (207.5,66) -- (230.5,77.5) -- (207.5,89) -- (207.5,83.25) -- (173,83.25) -- cycle ;
%Right Arrow [id:dp51590502107146] 
\draw   (174,163.75) -- (208.5,163.75) -- (208.5,158) -- (231.5,169.5) -- (208.5,181) -- (208.5,175.25) -- (174,175.25) -- cycle ;
%Left Right Arrow [id:dp4057041363362248] 
\draw   (84.42,101.53) -- (96.17,111.83) -- (90.29,111.82) -- (90.26,132.37) -- (96.14,132.38) -- (84.36,142.64) -- (72.61,132.35) -- (78.49,132.35) -- (78.52,111.8) -- (72.64,111.79) -- cycle ;
%Left Right Arrow [id:dp796245728660737] 
\draw  [dash pattern={on 4.5pt off 4.5pt}] (310.45,102.53) -- (322.19,112.85) -- (316.3,112.83) -- (316.24,133.38) -- (322.12,133.4) -- (310.33,143.64) -- (298.59,133.33) -- (304.48,133.34) -- (304.54,112.79) -- (298.66,112.77) -- cycle ;
%Shape: Rectangle [id:dp979462394375876] 
\draw   (466.5,100) -- (570.5,100) -- (570.5,140) -- (466.5,140) -- cycle ;
%Right Arrow [id:dp09965463406907804] 
\draw   (401,114.75) -- (435.5,114.75) -- (435.5,109) -- (458.5,120.5) -- (435.5,132) -- (435.5,126.25) -- (401,126.25) -- cycle ;

% Text Node
\draw (37,62) node [anchor=north west][inner sep=0.75pt]   [align=left] {SCC Actions \cite{YANG10}\\ (Definition \ref{SCCDefn})};
% Text Node
\draw (272,62) node [anchor=north west][inner sep=0.75pt]   [align=left] {PEG Actions \\(Definition \ref{PEGDefn})};
% Text Node
\draw (21,155) node [anchor=north west][inner sep=0.75pt]   [align=left] {Strongly Positively \\Recurrent Actions \cite{ST21}};
% Text Node
\draw (260,155) node [anchor=north west][inner sep=0.75pt]   [align=left] {Positively Recurrent\\Actions \cite{PS18}};
% Text Node
\draw (255,116) node [anchor=north west][inner sep=0.75pt]   [align=left] {\cite{DOP,Roblin}};
% Text Node
\draw (325,116) node [anchor=north west][inner sep=0.75pt]   [align=left] {\cite{PS18,YANG8}};
% Text Node
\draw (470,106) node [anchor=north west][inner sep=0.75pt]   [align=left] { DIV Actions in \\ HTS Dicho. \ref{HTSCAT(-1)}};

\end{tikzpicture}
    \caption{Relations between assumptions on actions under consideration. The equivalence in the second column is expected to hold in groups with contracting elements (some known cases as cited).}
    \label{fig:actionssurvol}
\end{figure}

First, we present a criterion when the action of a subgroup is completely dissipative. This is a vast generalization of \cite[Theorem 4.2]{GKN} for free groups, which can be traced back to the case of Fuchsian groups \cite{P77} (see \cite{Mat05}). 
\begin{thm}\label{HalfGrowthDisThm}
Let a group $H<\isom(\U)$ act properly with contracting elements, and $\mu_0$ be the PS measure constructed from a proper SCC action $\Gamma\act \U$. Assume that $\e H<\e \Gamma/2$. Then $\mu_o(\HG)=0$, where $\HG$ denotes the big horospheric limit set of $Ho$ defined in (\ref{HoroLimitPtsDef}).

In addition, assume that $H$ is torsion-free and preserves the measure class of $\mu_o$ on $\pU$ (in particular, if $H$ is a subgroup of $\Gamma$). Then the action of $H$ on $(\pU,\mu_o)$ is completely dissipative. %$Cons=\HG$ (mod 0) (\emph{i.e.} $\mu_o(\mathcal D_{>1}\setminus \mathcal D_{\infty})=0$). If  
\end{thm} 
Note that, if $\U$ is hyperbolic and $\Gamma\act \U$ is co-compact,
then any $H<\isom(\U)$ preserves the measure class of $\mu_o$ (see Lemma \ref{PreserveMeasureClass}). We expect Theorem \ref{HalfGrowthDisThm} to hold under the weaker assumption that $\Gamma\act\U$ has purely exponential growth (see Remark \ref{ConjPEGRem} on particular situations where this is true).

Next, we prove that the action of a confined subgroup is conservative.  Denote by $E(\Gamma)$ the set of isometries in $\isom(\U)$ which fix every $[\cdot]$--class in the limit set $[\pG]$ (Definition \ref{EllipticRadicalDefn}).  
\begin{thm}\label{ConfinedConsThm}
Suppose that a discrete subgroup $H<\isom(\U)$ is confined by $\Gamma$  with a compact confining subset $P$. Assume that $\Gamma\act \U$ is of divergence type, and furthermore: \begin{enumerate}
    
    \item [(i)]
    Either $H$ is torsion-free and $\U$ is a Gromov hyperbolic or geodesically complete CAT(0) space, or
    \item [(ii)]
   the confining subset $P$ is finite, and intersects trivially $E(\Gamma)$.

\end{enumerate} Then $\mu_o(\HG)=1$. In particular, $H\act (\pU,\mu_o)$ is conservative.
\end{thm}

%By the result on Hopf decomposition for general actions in \cite[Theorem 14]{Kai10}, Theorem \ref{ConfinedConsThm} implies:
%\begin{cor}
%In the setting of Theorem \ref{ConfinedConsThm}, the action of a torsion-free $H$ on $(\pU,\mu_o)$ is conservative.
%\end{cor}
Note that the action of $H$ on $(\pU,\mu_o)$ is not necessarily ergodic, with examples found in free groups \cite{GKN} and Kleinian groups \cite{McMBook}. 
As an immediate corollary, we obtain the following non-strict inequality. With a further inequality in Theorem \ref{CoulonInequality}, it turns out that the assumption of the SCC action can be relaxed to be the DIV action.
\begin{cor}\label{Nonstrict}
Assume that $\Gamma\act\U$ is SCC. In the setting of Theorem \ref{ConfinedConsThm},  we have $\e H\ge \e \Gamma/2$.  
\end{cor}

\subsection{Growth  inequalities for confined subgroups}
Via a different approach, we can upgrade the result in Corollary \ref{Nonstrict} to a strict inequality. We work with PEG actions $\Gamma\act \U$ strictly larger than the SCC action considered in \textsection\ref{SSecErgodicity} (particularly in Theorem \ref{HalfgrowthimplyDiss}).  

\begin{thm}\label{ConvTightThm}
Assume that the proper action $\Gamma\act \U$ has purely exponential growth with a contracting element.
If $H<\isom(\U)$ is a discrete subgroup confined by $\Gamma$ with a finite confining subset that intersects trivially with $E(\Gamma)$, then $\e H>\e \Gamma/2$.  
\end{thm}
If $H<\Gamma$ is normal, Coulon \cite{Coulon22} and independently the third named author \cite{YANG22} proved this inequality using boundary measures in greater generality (only assuming $\Gamma\act \U$ is of divergence type). In addition, if $H$ is a normal subgroup of divergence type, then $\e H=\e \Gamma$. Whether this holds for general confined subgroups remains open.

Together with Corollary \ref{Nonstrict}, Theorem \ref{ConvTightThm} has the following application to a finite tower of confined subgroups of $\Gamma$.  
\begin{cor}
Assume that $\Gamma\act\U$ is a PEG action with contracting elements. Let $H:=H_0< H_1\cdots < H_s=:\Gamma$ be a sequence of subgroups so that $H_i$ is confined in $H_{i+1}$ for each $0\le i< s$. Then $\e H> \e \Gamma/2^s$.  %If $H$ is of divergence type, then $\e H=\e G$.
    %\item  

\end{cor}
This in particular applies to an \textit{$s$--subnormal} subgroup $H$ for some integer $s\ge 1$, for which there exists a strictly increasing subnormal series, $H_0:=H\unlhd H_1\cdots \unlhd H_s=:\Gamma$, of length $s$. This statement was known in free groups by Olshanskii \cite{Ol17}, who also proved that the lower bound is sharp: for any $\epsilon>0$ and $s\ge 1$, there exists some $s$--subnormal subgroup for which $\e H<\epsilon+ \e \Gamma/2^s$.

Next, consider the space $\U/H:=\{Hx: x\in \U\}$ equipped with the quotient metric. That is, given $Hx, Hy\in \U/H$, define $\bar d(Hx, Hy)=\inf\{d(x,hy):h\in H\}$. Denote by $\pi: \U\to \U/H$ the natural projection.

Fixing a point $o\in \U$, consider the ball-like set in the image of $\Gamma o$ in $\U/H$:
$$N_{\Gamma/H}(o, n) =  \{Hg o: \bar d(Ho, Hgo)\le n\}$$ and define its growth rate $\e{\Gamma/H}$ as follows:
\begin{equation}\label{QuotientGrowthRateDefn}
\e {\Gamma/H} =\limsup_{n\to \infty} \frac{\log\sharp N_{\Gamma/H}(o, n)}{n}.    
\end{equation}
As $\Gamma o$ and $\Gamma o'$ have a finite Hausdorff distance, the growth rate $\e {\Gamma/H}$ does not depend on $o\in \U$. If $\U$ is the Cayley graph of $\Gamma$ and $H$ is a subgroup of $\Gamma$, then $\U/H$ is the Schreier graph associated with $H$. If in addition, $H$ is a normal subgroup of $\Gamma$, then $\e {\Gamma/H}$ is the usual growth rate of $\Gamma/H$ with respect to the projected generating set, which explains the notation. 

We obtain the following inequality relating the growth and co-growth of confined subgroups. 
\begin{thm}\label{CoulonInequality}
Assume that the proper action $\Gamma\act \U$ is of divergence type with a contracting element. If a discrete group $H<\isom(\U)$ is confined by $\Gamma$ with a finite confining subset that intersects trivially with $E(\Gamma)$, then  $$\e H+\frac{\e{\Gamma/H}}{2}\ge \e \Gamma$$    
\end{thm}

The inequality for normal subgroups was obtained earlier by Jaerisch and Matsuzaki in \cite{JM20} for free groups, and by Coulon \cite{Coulon22} in a proper action of divergence type with contracting elements. Our proof uses a shadow principle for confined subgroups, from which we deduce the inequality following closely Coulon's proof. Compared to their works, it is worth noting that the confined subgroup $H$ is not required to be contained in $\Gamma$.

As a corollary, we can relax the SCC action in Corollary \ref{Nonstrict} to be DIV actions.
\begin{cor}\label{NonstrictDIV}
In the setup of Theorem \ref{CoulonInequality},  we have $\e H\ge \e \Gamma/2$.    
\end{cor}

Note that it follows from Theorem \ref{CoulonInequality} that if $\Gamma/H$ has sub-exponential growth (\emph{i.e.} $\e{\Gamma/H}=0$), then $\e H= \e \Gamma$. By \cite{Coulon22}[Cor. 4.27], for a normal subgroup $H$, we have $\e H= \e \Gamma$ if $\Gamma/H$ is amenable. We do not know what condition on $\Gamma/H$ might characterize the equality $\e H= \e \Gamma$ for more general confined subgroups.

We are now in a position  to give the the proofs of Main Applications in \textsection\ref{MainAppls}. 
\begin{proof}[Proofs of Theorems \ref{mainthmhyperbolicandCAT(0)}, \ref{mainthmRANK1} and \ref{mainthmMCG}] Theorems \ref{mainthmhyperbolicandCAT(0)} and \ref{mainthmRANK1} follow  immediately from Theorems \ref{ConfinedConsThm}, \ref{ConvTightThm} and \ref{CoulonInequality} together. We note here that assumption in  (2) are redundant:  the elliptic radical $E(\Gamma)<\isom(\U)$ (\emph{i.e.} fixing $\Lambda(\Gamma o)$ pointwise) intersects $\Gamma$ in a finite subgroup. If $H<\Gamma$ is a torsion-free subgroup, then the confining subset of $H$ intersects trivially $E(\Gamma)$, so the assumption for (2) holds automatically.     

For mapping class groups, if $g>2$, then $E(G)$ is trivial, so  Theorem \ref{mainthmMCG} follows exactly as above. If $g=2$, $E(G)$ contains hyperelliptic involution, then an additional argument is needed to conclude the proof. Indeed, since $G$ is residually finite, we may pass to  finite index torsion-free subgroups $\hat G$ of $G$, where $\hat G\cap H$ is infinite and torsion-free, so Theorem \ref{mainthmMCG} follows. The growth rate remains the same after taking finite index subgroups, so the proof in the general case follows. 
\end{proof}

\subsection{Maximal quotient growth} We now relate the ergodic properties of the boundary actions studied in \textsection\ref{SSecErgodicity} to the growth of a $\Gamma$ orbit $\Gamma o$ in the quotient $\U/H$. 
A proper action $\Gamma\act \U$ is called {\it growth tight} if $\e{\Gamma/H}<\e \Gamma$ holds for any infinite normal subgroup $H<\Gamma$. In the setting of Theorem \ref{ConfinedConsThm}, the action of a torsion-free $H$ on the boundary is conservative. The inequality $\e{\Gamma/H}<\e \Gamma$ in particular implies that the quotient growth of $H\act \U$ is slower than $\Gamma\act \U$ in the following sense: $$\frac{\sharp N_{\Gamma/H}(o, n)}{\sharp N_\Gamma(o,n)}\to 0.$$
Following Bahturin-Olshanski \cite{BO10}, we say that $H\act \U$ has \textit{maximal quotient growth} \footnote{In terms of \cite{BO10}, $N_{\Gamma/H}(o, n)$ is the growth function of the right (usually non-isometric) action of $\Gamma$ on $H\backslash \Gamma$. The right action is not relevant here, so we take the term of quotient growth instead.} if the following holds: 
$$\sharp N_{\Gamma/H}(o, n)\asymp \sharp N_\Gamma(o,n).$$  

We shall establish an equivalence of conservative action and slower quotient growth, and a characterization of maximal right coset growth in terms of boundary actions. These are only obtained here for co-compact actions on hyperbolic spaces. We expect the characterizations to hold in a much more general setup. See Question \ref{questHorlimitset} and Remark \ref{ExtensionMCG} for conjectures for an extension to mapping class groups.

\begin{thm}\label{CharConsActioninHyp}
Assume that $\Gamma$ acts properly and co-compactly on a proper hyperbolic space $\U$ and $\mu_o$ is the unique Patterson-Sullivan measure class on the Gromov boundary $\pU$. Let $H<\Gamma$ be a subgroup. Then  the small/big horospheric limit set is $\mu_o$--full, if and only if  $\U/H$ grows slower than $\U$. Moreover,  $\mu_o(\hG)<1$ is equivalent to the following
$$
\sharp \{Hgo: \bar d(Ho,Hgo)\le n, g\in \Gamma\} \asymp \mathrm{e}^{\e \Gamma n}.
$$
\end{thm}

Before giving some corollaries, let us point out the following question which plays a key role in the proof of Theorem \ref{CharConsActioninHyp}:
\begin{quest}\label{questHorlimitset}
Let $\mu_o$ be a conformal measure on a boundary $\pU$ \textbf{without} atoms (cf. Example \ref{requirenoatoms}).
Under which conditions, do we have $\mu_o(\HG)=\mu_o(\hG)$?    
\end{quest}
This question has been raised in Sullivan's work \cite{Sul81}, where he showed for the real hyperbolic space $\mathbf{H}^n$ that a discrete subgroup $H<\isom(\mathbf{H}^n)$ is conservative with respect to $(\mathbf{S}^{n-1},\mathbf{Leb})$ if and only if the volume of $\mathbf{H}^n/H$ grows slower than $\mathbf{H}^n$.  
The key part of the proof is that the big and small horospheric limit sets differ in a $\mathbf{Leb}$--null set. We elaborate this proof to answer the question positively for any subgroup $H$ in Theorem \ref{CharConsActioninHyp}. See Theorem \ref{BigSmallEqualInHyp} for details.

For normal subgroups, adapting an argument of \cite{Kat02,FM20}, the question can be answered positively in the following specific situation.  
\begin{thm}\label{NormalsubgroupsHor}
Assume that  $\Gamma\act \U$ has purely exponential growth with a contracting element. Let $H$ be an infinite normal subgroup of $\Gamma$. Then $\mu_o(\hG)=\mu_o(\HG)=1$.  
\end{thm}

As another application of Theorem \ref{CharConsActioninHyp}, we obtain a characterization of conservative actions.
\begin{cor}\label{CorSlower}
Under the assumption of Theorem \ref{CharConsActioninHyp}, assume that $H$ is torsion-free. Then    the action $H\act (\pU,\mu_o)$ is conservative if and only if the quotient growth of $H\act \U$ is slower than $\Gamma\act \U$.   
\end{cor}

Another corollary is characterizing the purely exponential growth of right cosets as in the ``moreover" statement. 

Let $H, K<\Gamma$ be two subgroups with proper limit sets $[\Lambda (Ho)], [\Lambda (Ko)]\subsetneq [\Lambda(\Gamma o)]$ (we say these are  of \textit{second kind} following the terminology in Kleinian groups). In \cite{HYZ}, it is shown that the double cosets $HgK$ over  $g\in \Gamma$  have purely exponential growth:
$$
\sharp \{HgKo:  d(o,HgKo)\le n, g\in \Gamma\} \asymp \mathrm{e}^{\e \Gamma n}.
$$
If $K=\{1\}$, this amounts to counting the right coset $Hg$ as above. In hyperbolic groups, this case is thus covered in Theorem \ref{CharConsActioninHyp}, which furthermore gives a characterization of purely exponential growth of $\{Hg: g\in \Gamma\}$ by $\mu_o(\hG)<1$. If $H=\{1\}$, a characterization remains unknown to us for purely exponential growth of left cosets $gK$.

To conclude our study, we present some examples of subgroups with nontrivial  Hopf decomposition.
First such examples  are constructed in free groups (\cite{GKN}). Here we are able to construct  these subgroups for any SCC actions on hyperbolic spaces.

\begin{thm}\label{MixedHopfThm}
Suppose that $\Gamma$ admits a proper SCC action on a proper hyperbolic space $\U$. Let $f\in \Gamma$ be a loxodromic element. Then there exist a large integer $n$ so that the normal closure $\llangle f^n\rrangle$ contains subgroups of  second kind with nontrivial conservative part and dissipative part.    
\end{thm}

We expect this result extends to any SCC action. Indeed, denoting $G=\llangle f^n\rrangle$,  if   $\e G< \e \Gamma$ is proved, $\U$ could be replaced with  a general metric space.  Amenability criterion in \cite{CDST} for hyperbolic spaces says that $\e G=\e\Gamma$ is equivalent to the amenability of $\Gamma/G$. If $n$ is sufficiently large, $\Gamma/G$ is non-amenable, so $\e G<\e \Gamma$. If the amenability criterion holds for any SCC action, then  $\e G< \e \Gamma$ follows and so does Theorem \ref{MixedHopfThm}.

%This is actually also used in the proof of Theorem \ref{MixedHopfThm}.

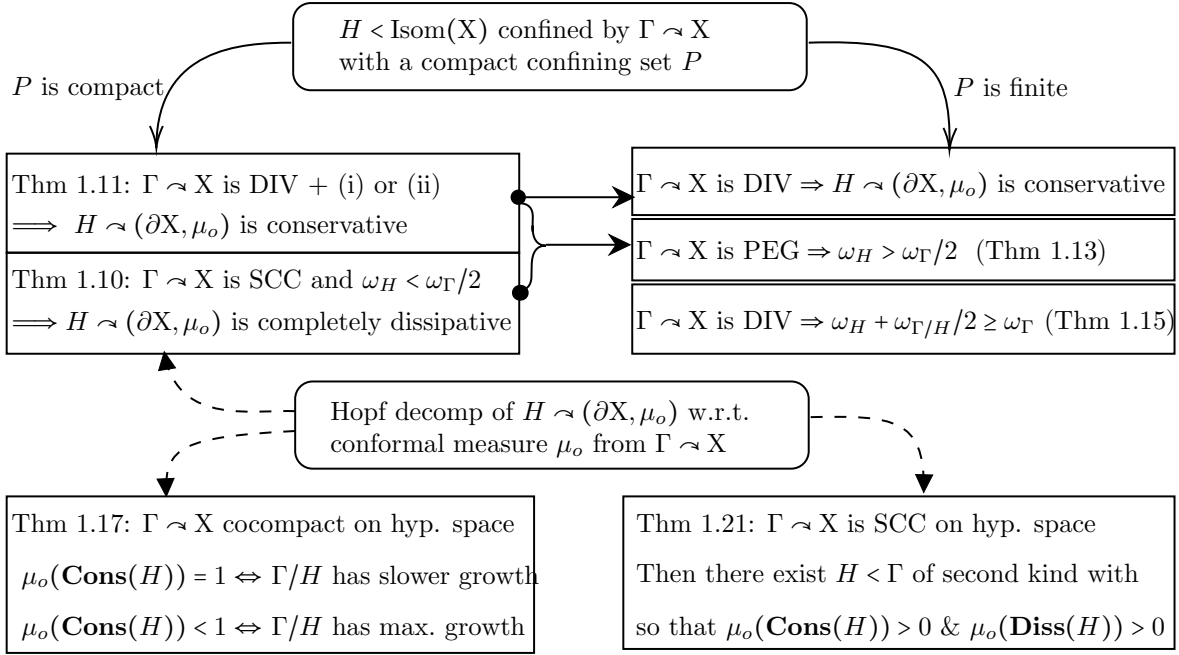
\begin{figure}
    \centering
 
\tikzset{every picture/.style={line width=0.75pt}} %set default line width to 0.75pt        

\begin{tikzpicture}[x=0.75pt,y=0.75pt,yscale=-1,xscale=1]
%uncomment if require: \path (0,352); %set diagram left start at 0, and has height of 352

%Shape: Rectangle [id:dp10580622685389174] 
\draw   (39,132) -- (297.5,132) -- (297.5,183) -- (39,183) -- cycle ;
%Shape: Rectangle [id:dp6182518588704533] 
\draw   (39,82) -- (297.5,82) -- (297.5,132) -- (39,132) -- cycle ;
%Shape: Rectangle [id:dp7931486612545136] 
\draw   (354.5,79) -- (628,79) -- (628,113) -- (354.5,113) -- cycle ;
%Shape: Rectangle [id:dp6412044693707206] 
\draw   (39,255) -- (305,255) -- (305,334) -- (39,334) -- cycle ;
%Shape: Rectangle [id:dp5261096101196323] 
\draw   (350,255) -- (628,255) -- (628,332) -- (350,332) -- cycle ;
%Shape: Rectangle [id:dp2578963650975228] 
\draw   (354.5,115) -- (628,115) -- (628,146) -- (354.5,146) -- cycle ;
%Shape: Rectangle [id:dp45980884473914485] 
\draw   (354.5,148) -- (628,148) -- (628,183) -- (354.5,183) -- cycle ;
%Rounded Rect [id:dp47822377572426067] 
\draw   (183.5,14.8) .. controls (183.5,9.94) and (187.44,6) .. (192.3,6) -- (433.7,6) .. controls (438.56,6) and (442.5,9.94) .. (442.5,14.8) -- (442.5,41.2) .. controls (442.5,46.06) and (438.56,50) .. (433.7,50) -- (192.3,50) .. controls (187.44,50) and (183.5,46.06) .. (183.5,41.2) -- cycle ;
%Curve Lines [id:da45079325348146204] 
\draw    (443.5,26) .. controls (502.6,24.03) and (515.13,36.61) .. (514.54,78.09) ;
\draw [shift={(514.5,80)}, rotate = 271.33] [color={rgb, 255:red, 0; green, 0; blue, 0 }  ][line width=0.75]    (10.93,-3.29) .. controls (6.95,-1.4) and (3.31,-0.3) .. (0,0) .. controls (3.31,0.3) and (6.95,1.4) .. (10.93,3.29)   ;
%Curve Lines [id:da5375877246869614] 
\draw    (182.5,26) .. controls (156.76,26) and (115.34,28.94) .. (114.51,77.51) ;
\draw [shift={(114.5,79)}, rotate = 270] [color={rgb, 255:red, 0; green, 0; blue, 0 }  ][line width=0.75]    (10.93,-3.29) .. controls (6.95,-1.4) and (3.31,-0.3) .. (0,0) .. controls (3.31,0.3) and (6.95,1.4) .. (10.93,3.29)   ;
%Straight Lines [id:da802305510463645] 
\draw    (297,104) -- (352.5,104) ;
\draw [shift={(355.5,104)}, rotate = 180] [fill={rgb, 255:red, 0; green, 0; blue, 0 }  ][line width=0.08]  [draw opacity=0] (10.72,-5.15) -- (0,0) -- (10.72,5.15) -- (7.12,0) -- cycle    ;
\draw [shift={(297,104)}, rotate = 0] [color={rgb, 255:red, 0; green, 0; blue, 0 }  ][fill={rgb, 255:red, 0; green, 0; blue, 0 }  ][line width=0.75]      (0, 0) circle [x radius= 3.35, y radius= 3.35]   ;

%Rounded Rect [id:dp13686983352786153] 
\draw   (184.5,205.8) .. controls (184.5,200.94) and (188.44,197) .. (193.3,197) -- (434.7,197) .. controls (439.56,197) and (443.5,200.94) .. (443.5,205.8) -- (443.5,232.2) .. controls (443.5,237.06) and (439.56,241) .. (434.7,241) -- (193.3,241) .. controls (188.44,241) and (184.5,237.06) .. (184.5,232.2) -- cycle ;
%Curve Lines [id:da07820366301641335] 
\draw  [dash pattern={on 4.5pt off 4.5pt}]  (184.5,212) .. controls (164.02,211.03) and (126.44,219.56) .. (119.01,184.76) ;
\draw [shift={(118.5,182)}, rotate = 81.03] [fill={rgb, 255:red, 0; green, 0; blue, 0 }  ][line width=0.08]  [draw opacity=0] (8.93,-4.29) -- (0,0) -- (8.93,4.29) -- cycle    ;
%Curve Lines [id:da43418587278152443] 
\draw  [dash pattern={on 4.5pt off 4.5pt}]  (184.5,222) .. controls (126.9,224.88) and (119.07,230.52) .. (119.42,252.21) ;
\draw [shift={(119.5,255)}, rotate = 267.61] [fill={rgb, 255:red, 0; green, 0; blue, 0 }  ][line width=0.08]  [draw opacity=0] (8.93,-4.29) -- (0,0) -- (8.93,4.29) -- cycle    ;
%Curve Lines [id:da30703618217056583] 
\draw  [dash pattern={on 4.5pt off 4.5pt}]  (444.5,215) .. controls (486.43,214.03) and (501.73,215.9) .. (501.54,250.29) ;
\draw [shift={(501.5,253)}, rotate = 271.55] [fill={rgb, 255:red, 0; green, 0; blue, 0 }  ][line width=0.08]  [draw opacity=0] (8.93,-4.29) -- (0,0) -- (8.93,4.29) -- cycle    ;
%Shape: Brace [id:dp5502745876326993] 
\draw  [line width=0.75]  (296.5,151) .. controls (301.17,151.11) and (303.55,148.83) .. (303.66,144.16) -- (303.79,138.29) .. controls (303.94,131.62) and (306.35,128.34) .. (311.02,128.45) .. controls (306.35,128.34) and (304.09,124.96) .. (304.24,118.3)(304.18,121.3) -- (304.34,114.16) .. controls (304.45,109.49) and (302.17,107.11) .. (297.5,107) ;
%Shape: Free Drawing [id:dp9135137751320419] 
\draw  [line width=6] [line join = round][line cap = round] (297.8,152.22) .. controls (297.8,152.22) and (297.8,152.22) .. (297.8,152.22) ;
%Straight Lines [id:da2341615273931963] 
\draw    (311,128) -- (350.5,128) ;
\draw [shift={(353.5,128)}, rotate = 180] [fill={rgb, 255:red, 0; green, 0; blue, 0 }  ][line width=0.08]  [draw opacity=0] (10.72,-5.15) -- (0,0) -- (10.72,5.15) -- (7.12,0) -- cycle    ;

% Text Node
\draw (355,89.84) node [anchor=north west][inner sep=0.75pt]    {$\Gamma \act \U \textrm{ is }\mathrm{DIV} \Rightarrow H\act(\pU,\mu_{o}) \textrm{ is conservative}$};
% Text Node
\draw (355,123.4) node [anchor=north west][inner sep=0.75pt]    {$\Gamma\act \U \textrm{ is } \mathrm{PEG} \Rightarrow \omega_{H}  >\omega_{\Gamma }/2$ \ (Thm \ref{ConvTightThm})};
% Text Node
\draw (355,158.84) node [anchor=north west][inner sep=0.75pt]    {$\Gamma\act \U \textrm{ is } \mathrm{DIV} \Rightarrow \omega_{H} +\omega_{\Gamma /H} /2\geq \omega_{\Gamma}$\ (Thm \ref{CoulonInequality})};

% Text Node
\draw (355,262) node [anchor=north west][inner sep=0.75pt]   [align=left] {Thm \ref{MixedHopfThm}: $\Gamma\act \U$ is SCC on hyp. space};
% Text Node
\draw (355,287.4) node [anchor=north west][inner sep=0.75pt]   [align=left] {Then there exist  $H<\Gamma$ of  second kind with}; 
% Text Node
\draw (355,313.4) node [anchor=north west][inner sep=0.75pt]    {so that $\mu_{o}(\mathbf{Cons}(H))  >0\ \& \ \mu_{o}(\mathbf{Diss}(H))  >0$};

% Text Node
\draw (40,90.4) node [anchor=north west][inner sep=0.75pt]    {Thm \ref{ConfinedConsThm}:   $\Gamma \act \U \textrm{ is } \mathrm{DIV}$ + (i) or (ii)};
% Text Node
\draw (40,110.4) node [anchor=north west][inner sep=0.75pt]    {  
$ \Longrightarrow \ H\act (\pU,\mu_o)$ is conservative};
% Text Node
\draw (40,138) node [anchor=north west][inner sep=0.75pt]   [align=left] {Thm \ref{HalfGrowthDisThm}:  $\Gamma\act \U$ is SCC and $\omega_{H} < \omega_{\Gamma}/2$};
% Text Node
\draw (40,159.4) node [anchor=north west][inner sep=0.75pt]    {$\Longrightarrow H\act(\pU,\mu_o) \textrm{ is completely dissipative}$};
% Text Node
\draw (40,262) node [anchor=north west][inner sep=0.75pt]   [align=left] {Thm \ref{CharConsActioninHyp}: $\Gamma\act \U$ cocompact on hyp. space};
% Text Node
\draw (45,287.4) node [anchor=north west][inner sep=0.75pt]    {$\mu_{o}(\mathbf{Cons}(H)) =1\Leftrightarrow$  $\Gamma/H$ has slower\ growth};
% Text Node
\draw (45,313.4) node [anchor=north west][inner sep=0.75pt]    {$\mu_{o}(\mathbf{Cons}(H)) < 1\Leftrightarrow$ $\Gamma/H$ has max.\ growth};

% Text Node
\draw (205,12) node [anchor=north west][inner sep=0.75pt]   [align=left] {$H<\isom(\U)$ confined by $\Gamma\act \U$ \\
with a compact confining set $P$};
% Text Node
\draw (515,42) node [anchor=north west][inner sep=0.75pt]    {$P$ is finite};
% Text Node
\draw (200.75,205) node [anchor=north west][inner sep=0.75pt]   [align=left] {Hopf decomp of $H\act (\pU,\mu_o)$ w.r.t. \\ conformal measure $\mu_o$ from $\Gamma\act \U$};
% Text Node
\draw (40,42) node [anchor=north west][inner sep=0.75pt]    {$P$ is compact};

\end{tikzpicture}
     \caption{A flow chart of main results, where the assumptions DIV, PEG and SCC  on the action $\Gamma\act \U$ are illustrated in Fig. \ref{fig:actionssurvol}}
    \label{fig:theorems}
    \label{fig:enter-label}
\end{figure}
\subsection{Ingredients in the proofs and organization of the paper}
We highlight some tools and ingredients in obtaining the results presented above. The reader may restrict their attention to the special case $H<\Gamma$ for simplicity. Our standing assumption is that the space $\U$ admits a convergence boundary $\pU$, and $\Gamma$ acts properly on $\U$. Readers who are not familiar with the background on convergence boundary might keep in mind the example where $\U$ is hyperbolic space with Gromov boundary $\pU$: for this special case we have illustrations of the main ideas throughout the text. 

We refer to Fig. \ref{fig:theorems} for an overview of the main theorems and their logical relations. Fig. \ref{fig:actionssurvol} illustrates the implications between assumptions on the actions. 
\\
\paragraph{\textbf{Results on Hopf decomposition}}
The proof of results on the Hopf decomposition of $H\act(\pU, \mu_o)$, as presented above in \textsection\ref{SSecErgodicity}, makes extensive use of the conformal measure theory $\{\mu_x:x\in \U\}$ developed in \cite{YANG22} on the convergence boundary, provided that $H\act\U$ is of divergence type. The case of horofunction boundary was independently obtained by Coulon \cite{Coulon22}. The key facts we use include the following: 
\begin{enumerate}
    \item 
    the Shadow Lemma \ref{ShadowLem} holds for $\mu_x$, and
    \item 
    $\mu_x$ is supported on conical points (Lemma \ref{ConicalPointsLem}), and is unique up to bounded multiplicative constants on the reduced horofunction boundary (Lemma \ref{Unique}). 
\end{enumerate}
A well-prepared reader would notice that these are standard consequences of the HTS dichotomy (\emph{e.g.} for CAT(-1) spaces stated in Theorem \ref{HTSCAT(-1)}).   In greater generality, this  has been established for groups with contracting elements as mentioned above in \cite{Coulon22,YANG22}. We recommend the reader to keep in mind the situation of CAT(-1) spaces in the sequel without losing the essentials.   

The proof of Theorem \ref{HalfGrowthDisThm} in the case that $\U$ is a Gromov hyperbolic space is obtained by recasting the combinatorial proof in free groups \cite{GKN} in appropriate geometric terms. The general case mimics the same outline, with the additional input of the recent result \cite{QYANG} on full measures supported on regularly contracting limit sets. If $H$ is a confined subgroup, we push a generic set of limit points into the horospheric limit set of $H$  to prove the conservativity of the action of $H$, Theorem \ref{ConfinedConsThm}. This strategy has appeared in the works \cite{Kat02,FM20} for normal subgroups of groups acting on hyperbolic spaces. 
\\
\paragraph{\textbf{Key technical results on confined subgroups}}
Toward the proof of Theorem \ref{ConfinedConsThm}, we prove the following key technical result stated in Lemma \ref{GoodConfiningSetP}, which can be viewed as an enhanced version of the Extension Lemma \ref{extend3}. 
\begin{lem}[=Lemma \ref{GoodConfiningSetP}]\label{GoodConfiningSetPIntro}
There exists a finite subset $F$ of contracting elements in $\Gamma$ with the following property:
for any $g\in \Gamma$, there exist $f\in F$ and $p\in P$ such that $gfpf^{-1}g^{-1}$ lies in $H$ and
$$
|d(o, gfpf^{-1}g^{-1}o) - 2d(o,go)|\le D
$$
where $D$ depends only on $F$ and $P$.
\end{lem}

To facilitate the reader with perspective from hyperbolic spaces,   we include a short section \textsection\ref{SecWarmup}  to demonstrate the main ideas and tools without invoking much preliminary material. Complete proofs of Theorem \ref{HalfGrowthDisThm} and Theorem \ref{ConfinedConsThm} are provided for this special case.

This lemma provides the geometric property of confined subgroups that allows us to prove statements in a manner similar to that of normal subgroups. In particular, adapting the argument in \cite{YANG22} for normal subgroups, we prove the following.
\begin{lem}[Shadow Principle in \textsection\ref{secshadow}]\label{ShadowPrincipleIntro}
Let $\{\mu_x\}_{x\in \U}$ be a $\e H$--dimensional $H$--quasi-equivariant quasi-conformal density supported on $\pU$.   Then there exists $r_0 > 0$ such that  
$$
\begin{array}{rl}
\|\mu_y\| e^{-\e H \cdot d(x, y)} \quad \prec_\lambda   \quad \mu_x(\Pi_{x}(y,r))\quad  \prec_{\lambda,\epsilon, r} \quad \|\mu_y\| e^{-\e H \cdot  d(x, y)}\\
\end{array}
$$
for any $x,y\in \Gamma o$ and $r \ge  r_0$.    
\end{lem}
The usual Shadow Lemma  asserts only the above inequalities on a smaller region of $x,y\in Ho\subseteq \Gamma o$, where $\|\mu_y\|=1$. Hence, the Shadow Principle provides useful information in the case that $\Gamma o$ is a much larger set. If $\Gamma\act\U$ is cocompact, the above result holds over the whole space $\U$. 
\\
\paragraph{\textbf{Cogrowth tightness of confined subgroups}}
The nonstrict part of inequality $\e H> \e\Gamma/2$ asserted by Theorem \ref{ConvTightThm} follows immediately from Theorems \ref{HalfGrowthDisThm} and \ref{ConfinedConsThm}. The strict part turns out to require a more subtle argument. The strategy follows the outline of the proof of strict inequality for normal subgroups of \emph{divergence type} in \cite{YANG22}(\textsection\ref{section:cogrowthTight}). The proof for the convergence type is easier, so we omit the discussion here. 

The main point in \cite{YANG22} is to obtain a uniform bound $M$ on the mass of $\mu_{y}$ over $y\in \Gamma o$:
\begin{equation}\label{muyMass}
\forall go\in \Gamma o,\quad \|\mu_{go}\|\le M    
\end{equation}
Once this is proved, the coefficient  $\|\mu_{y}\|$ in the Shadow Principle disappears, so we would obtain $\e H\ge \e \Gamma$ from a standard covering argument. That is, the number of shadows $ \Pi_{o}(go,r))$, $go\in A_\Gamma(o,n,\Delta)$ which contain a given point of $\partial X$ is bounded by a uniform constant $C$, and therefore
$$
e^{-\e H n} \sharp A_\Gamma(o,n,\Delta) \le C
$$
which yields  $\e H\ge \e \Gamma$ and thus $\e H= \e \Gamma$ follows for normal subgroups $H$ of divergence type. 

To make life easier, let us assume $\U$ is a CAT(-1) space, so the  set  (I) of conditions in Theorem \ref{HTSCAT(-1)} holds.   As $H\act \U$ is of divergence type, there exists a unique    PS measure class $\{\mu_x: x\in \U\}$ of dimension $\e\Gamma$, up to scaling. In particular, the PS measure  $\mu_{go}$ is the unique limit point of the following one 
$$
\mu_{go}^{s,y} =\frac{1}{\p_H(s, o, y)}\sum_{h\in H}\mathrm{e}^{-sd(go,hy)}\dirac{hy}
$$ for \emph{any} $y\in \U$,  where $\p_H(s, x,y)=\sum_{h\in H}\mathrm{e}^{-sd(x,hy)}$ is the Poincar\'e series  associated to $H$. Note that
\begin{equation}\label{PSEquations}
\forall s>\e H:\quad \|\mu_{go}^{s,go}\| =\frac{\p_H(s, go, go)}{\p_H(s, o, go)} =  \Bigg[ \frac{\p_{H}(s, go, o)}{\p_{H^g}(s, o, o)}\Bigg]^{-1}
\end{equation}
The proof would be finished at this point, if $H^g=H$ is a normal subgroup: the RHS is the inverse of the mass   $\|\mu_{go}^{s,o}\|$.   The uniqueness of the limit $\mu_{go}$ concludes that $\|\mu_{go}\|=1$, so the proof for normal subgroups is finished. 

Our effort indeed comes into this stage to handle a general confined subgroup. We have to take a different routine to show (\ref{muyMass}), through an argument by contradiction. Assuming that $2\e H=\e \Gamma$, we make a crucial use of Lemma \ref{GoodConfiningSetPIntro} to obtain the following estimates on growth function of each conjugate $gHg^{-1}$
$$
C \mathrm{e}^{n\e H}\le \sharp (gHg^{-1}\cap A_\Gamma(o,n,\Delta))\le C' \mathrm{e}^{n\e H}
$$
in Lemma \ref{EqualPSSeries}, which yields the coarse equality of the associated Poincar\'e series:
$$
\p_{H}(\e H, o, o) \asymp \p_{H^g}(\e H, o, o)
$$
The upper bound in the above inequality uses some ingredients in proving purely exponential growth in \cite{YANG10}. Substituting this equation into (\ref{PSEquations})  proves the boundedness of $\|\mu_{go}\|$ as above, resulting in the equality $\e H= \e \Gamma$. This contradicts our assumption, concluding the proof of the strict inequality.

The argument by contradiction leaves open whether equality $\e H= \e\Gamma$ holds for confined subgroups of divergence type, while it is known to hold for normal subgroups.  
\\
\paragraph{\textbf{Subgroups with nontrivial Hopf decomposition}}
The construction of subgroups satisfying Theorem \ref{MixedHopfThm} starts in Section \ref{secNontrivialHopf}. We first take the normal closure $G=\llangle f^n\rrangle$ of a contracting element $f$ for some sufficiently large $n$. It is well known that such a $G$ is a free group of infinite rank (\cite{DGO}). Removing one generator gives a subgroup $H$ of second kind, for which we show that it has the desired properties. The proof relies on the recent adaption of rotating family theory to projection complex (\cite{CMM21, BDDKPS}). 
The proof of nontrivial conservative component uses the inequality  $\e G<\e \Gamma$ by the Amenability Theorem \cite{CDST} for SCC action on {hyperbolic} spaces.  We mention that the equivalence of $\e G<\e \Gamma$ with non-amenability of $\Gamma/G$ is conjectured to hold for any SCC action with contracting elements. %In the last section \ref{secexample}, we show in Theorem \ref{ExampleMaxGrowth} that $H$ can be enlarged to get a subgroup with a full limit set and mixed Hopf decomposition. 
\\
\paragraph{\textbf{A guide to the sections of the paper}} In the preliminary \textsection \ref{prelim}, we introduce necessary materials on contracting elements, convergence boundary,  quasiconformal density on it, and a brief discussion on Hopf decomposition for conformal measures. The main results of the paper are then grouped into three different but closely related parts. The first part is devoted to the study of ergodic properties on boundary, establishing completely dissipative actions for subgroups with small growth (Theorem \ref{HalfGrowthDisThm}) in  \textsection \ref{secdiss}, and conservative actions for confined subgroups (Theorem \ref{ConfinedConsThm}) in \textsection \ref{seccons}. To demonstrate the main idea in our general case, we include  a short section \ref{SecWarmup} to explain their proof in hyperbolic setup.   The growth of confined subgroups forms the main content of the second part. Using a key lemma \ref{GoodConfiningSetP} obtained in the first part, we prove the shadow principle for confined subgroups in \textsection\ref{secshadow} and then complete the proof of strict inequality in Theorem \ref{ConvTightThm} in \textsection \ref{section:cogrowthTight}.   As a further application of the Shadow principle, Section \ref{secinequality} shows an inequality in Theorem \ref{CoulonInequality}   relating the growth and co-growth of confined subgroups. The last part  first explains a close relation between  quotient growth and Hopf decomposition and then shows the existence of nontrivial Hopf decomposition.  In \textsection\ref{secmaxgrowth}), the conservative action is  characterized by slower quotient growth in Theorem \ref{CharConsActioninHyp}. The last section \ref{secNontrivialHopf} constructs in abundance subgroups of second kind with non-trivial Hopf decomposition    (Theorem \ref{MixedHopfThm}).  %, and construct a variety of examples with nontrivial conservative component (cf. Theorem \ref{MixedHopfThm} in \textsection\ref{secexample}).  

\section{Preliminaries}\label{prelim}

Let $(\U, d)$ be a proper geodesic metric space. Let $\isom(\U,d)$ be the isometry group endowed with compact open topology. It is well known that a subgroup $\Gamma<\isom(\U,d)$ is discrete if and only if $\Gamma$  acts properly on $\U$ (\cite[Theorem 5.3.5]{Rat06}).

Let $\alpha :[s,t]\subseteq \mathbb R\to\U$ be a path parametrized by arc-length, from the initial point $\alpha^-:=\alpha(s)$ to the terminal point $\alpha^+:=\alpha(t)$. If $[s,t]=\mathbb R$, the restriction of $\alpha$ to $[a, +\infty)$ for $a\in\mathbb R$  is referred to  as a \textit{positive ray}, and its complement a \textit{negative ray}. By abuse of language, we often denote them by   $\alpha^+$ and $\alpha^-$ (in particular, when they represent boundary points to which the half rays converge as in Definition \ref{ConvBdryDefn}).

Given two parametrized points $x,y\in \alpha$, $[x,y]_\alpha$ denotes the parametrized
subpath of $\alpha$ going from $x$ to $y$, while 
 $[x, y]$ is a choice of a geodesic between $x, y\in \U$. 
 
A path $\alpha$ is called a \textit{$c$--quasi-geodesic} for $c\ge 1$ if for   any rectifiable subpath $\beta$,
$$\len(\beta)\le c \cdot d(\beta^-, \beta^+)+c$$
 where   $\ell(\beta)$ denotes the length of $\beta$. 
 
Denote by $\alpha\cdot \beta$ (or simply $\alpha\beta$) the concatenation of two paths $\alpha, \beta$  provided that $\alpha^+ =
\beta^-$.

Let $f, g$ be real-valued functions. Then $f \prec_{c_i} g$ means that
there is a constant $C >0$ depending on parameters $c_i$ such that
$f < Cg$. The symbol $\succ_{c_i}  $ is defined similarly, and  $\asymp_{c_i}$ means both $\prec_{c_i}  $ and $\succ_{c_i}  $  are true. The constant $c_i$ will be omitted if it is a universal constant.

%By convention, if $\gamma$ is a geodesic segment or ray, we assume that $\gamma$ is equipped with arclength parametrization: $\gamma: [0, \ell]\to \U$ is an isometric embedding, where $0\leq \ell\leq \infty$ denotes the length of $\gamma$. Moreover, if $0\le s\le t\le \ell$, we write $\gamma[s,t]$ to denote the segment restricting on the interval $[s,t]$.

\subsection{Contracting geodesics}
Let $Z$ be a closed subset of $\U$ and $x$ be a point in $\U$.  By $d(x, Z)$ we mean the set-distance 
between $x$ and $Z$, \emph{i.e.} 
\[
d(x, Z) : = \inf \big \{ d(x, y): y \in Z \big \}. 
\]
Let
\[ \pi_{Z}(x) : = \big \{ y\in Z: d(x, y) = d(x, Z) \big \} \]
be the set of closet point projections from $x$ to $Z$. Since $X$ is a proper metric space, 
$\pi_{Z}(x)$ is non empty. We refer to $\pi_{Z}(x) $ as the \emph{projection set} of $x$ to $Z$. Define $\proj_Z(x,y):=\|\pi_Z(x)\cup \pi_Z(y)\|$, where $\|\cdot\|$ denotes the diameter.

\begin{defn} \label{Def:Contracting}
We say a closed subset $Z \subseteq X$ is \emph{$C$--contracting} for a constant $C>0$ if,
for all pairs of points $x, y \in X$, we have
\[
d(x, y) \leq d(x, Z) \quad  \Longrightarrow  \quad \proj_Z(x,y) \leq  C.
\]
Any such $C$ is called a \emph{contracting constant} for $Z$. A collection of $C$--contracting subsets shall be referred to as a $C$--contracting system.

An element $h\in \isom(\U)$ is called \textit{contracting} if it acts co-compactly on a contracting bi-infinite quasi-geodesic. Equivalently, the map $n\in \mathbb Z\longmapsto h^no$ is a quasi-geodesic with a contracting image.  
\end{defn}

Unless explicitly stated, let us assume from now on that   $\Gamma<\isom(\U)$ is a discrete group, so $\Gamma \act \U$ is a proper action (\emph{i.e.} with discrete orbits and finite point stabilizers). 

A group is called \textit{elementary} if it is virtually $\mathbb Z$ or a finite group. In a discrete group, a contracting element must be of infinite order and is contained in a maximal elementary subgroup as described in the next lemma.

\begin{lem}\cite[Lemma 2.11]{YANG10}\label{elementarygroup}
For a contracting element $h\in \Gamma$,  we have
$$
E_\Gamma(h)=\{g\in \Gamma: \exists n\in \mathbb N_{> 0}, (\;gh^ng^{-1}=h^n)\; \lor\;  (gh^ng^{-1}=h^{-n})\}.
$$
\end{lem}
We shall suppress $\Gamma$ and write $E(h)=E_\Gamma(h)$ if $\Gamma$ is clear in context. 

Keeping in mind the basepoint $o\in\U$, the \textit{axis} of $h$  is defined as the following quasi-geodesic 
\begin{equation}\label{axisdefn}
\ax(h)=\{f o: f\in E(h)\}.
\end{equation} Notice that $\ax(h)=\ax(k)$ and $E(h)=E(k)$    for any contracting element   $k\in E(h)$.
 
An element $g\in \Gamma$ \textit{preserves the orientation} of a bi-infinite quasi-geodesic $\gamma$ if $\alpha$ and $g\alpha$ has finite Hausdorff distance for any half ray $\alpha$ of  $\gamma$. Let $E^+(h)$ be the subgroup of $E(h)$ with possibly index 2 which elements  preserve  the orientation of their axis. Then we have 
$$
E^+(h)=\{g\in \Gamma: \exists 0\ne n\in \mathbb Z, \;gh^ng^{-1}=h^n\}.
$$
and $E^+(h)$ contains all contracting elements in $E(h)$, and $E(h)\setminus E^+(h)$ consists of torsion elements.

\begin{defn}
Two  contracting elements $h_1, h_2$ in a discrete group $\Gamma$  are called \textit{independent} if the collection $\f=\{g\ax(h_i): g\in \Gamma;\ i=1, 2\}$ is a contracting system with bounded intersection: for any $r>0$, there exists $L=L(r)$ so that
$$\forall X\ne Y\in \f,\; \|N_r(X)\cap N_r(Y)\|\le L.$$ 
This is equivalent to the bounded projection: $\|\pi_X(Y)\|\le B$ for some  $B$ independent of $X\ne Y$. 

In a possibly nondiscrete group $\Gamma$, we say that $h_1, h_2\in \Gamma$ are \textit{weakly independent} if $\{h_1^n o: n\in\mathbb Z\}$ and $\{h_2^n o: n\in\mathbb Z\}$ have infinite Hausdorff distance. Note that in some papers, weak independence is referred to as independence. 
\end{defn}
\begin{rem}
Note that two conjugate contracting elements with disjoint fixed points are weakly independent, but not independent. In the current paper, we mainly work  with independent contracting elements, though many technical results hold for  weakly independent contracting ones.     
\end{rem}
  %This is slightly different from the independence  used by other researchers (cf. \cite{MT}), though they coincide in a proper action.  

\begin{defn}\label{barriers}
Fix $r>0$ and a set $F$ in $\Gamma$.  A geodesic $\gamma$ contains an \textit{$(r, f)$--barrier} for $f\in F$   if there exists    an element $g \in \Gamma$ so that 
\begin{equation}\label{barrierEQ}
\max\{d(g\cdot o, \gamma), \; d(g\cdot fo, \gamma)\}\le r.
\end{equation}
By abuse of language,    the point $ho$ or  the axis $h\ax(f)$ is called {$(r, F)$--barrier} on $\gamma$.
\end{defn}

\subsection{Extension Lemma}

{We fix a finite set $F \subseteq \Gamma$ of independent contracting elements and let $\f = \{g \ax(f) : f \in F, g \in \Gamma\}$.} The following notion of an admissible path allows     to construct   a quasi-geodesic  by concatenating geodesics via $\f$.
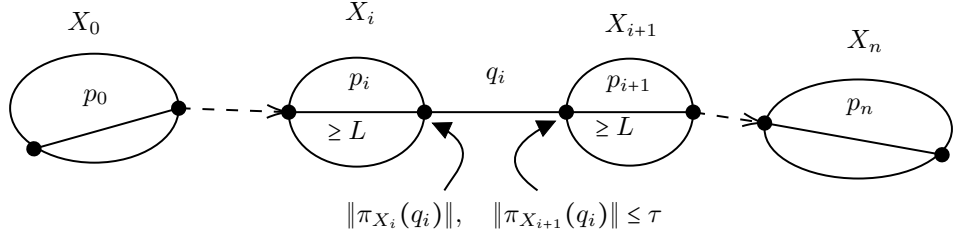
\begin{figure}
    \centering

\tikzset{every picture/.style={line width=0.75pt}} %set default line width to 0.75pt        

\begin{tikzpicture}[x=0.75pt,y=0.75pt,yscale=-1,xscale=1]

%Shape: Ellipse [id:dp981232074045038] 
\draw   (80.5,98.5) .. controls (80.5,83.31) and (99.53,71) .. (123,71) .. controls (146.47,71) and (165.5,83.31) .. (165.5,98.5) .. controls (165.5,113.69) and (146.47,126) .. (123,126) .. controls (99.53,126) and (80.5,113.69) .. (80.5,98.5) -- cycle ;
%Shape: Ellipse [id:dp5515090760160153] 
\draw   (221,100.5) .. controls (221,85.31) and (236.33,73) .. (255.25,73) .. controls (274.17,73) and (289.5,85.31) .. (289.5,100.5) .. controls (289.5,115.69) and (274.17,128) .. (255.25,128) .. controls (236.33,128) and (221,115.69) .. (221,100.5) -- cycle ;
%Shape: Ellipse [id:dp9569316519863855] 
\draw   (461,109) .. controls (461,95.19) and (482.15,84) .. (508.25,84) .. controls (534.35,84) and (555.5,95.19) .. (555.5,109) .. controls (555.5,122.81) and (534.35,134) .. (508.25,134) .. controls (482.15,134) and (461,122.81) .. (461,109) -- cycle ;
%Shape: Ellipse [id:dp08814137395275101] 
\draw   (361,100.5) .. controls (361,85.59) and (375.33,73.5) .. (393,73.5) .. controls (410.67,73.5) and (425,85.59) .. (425,100.5) .. controls (425,115.41) and (410.67,127.5) .. (393,127.5) .. controls (375.33,127.5) and (361,115.41) .. (361,100.5) -- cycle ;
%Straight Lines [id:da6078452439739415] 
\draw  [dash pattern={on 4.5pt off 4.5pt}]  (165.5,98.5) -- (219,100.43) ;
\draw [shift={(221,100.5)}, rotate = 182.06] [color={rgb, 255:red, 0; green, 0; blue, 0 }  ][line width=0.75]    (10.93,-3.29) .. controls (6.95,-1.4) and (3.31,-0.3) .. (0,0) .. controls (3.31,0.3) and (6.95,1.4) .. (10.93,3.29)   ;
%Straight Lines [id:da08444614363149383] 
\draw  [dash pattern={on 4.5pt off 4.5pt}]  (425,100.5) -- (459.02,105.7) ;
\draw [shift={(461,106)}, rotate = 188.69] [color={rgb, 255:red, 0; green, 0; blue, 0 }  ][line width=0.75]    (10.93,-3.29) .. controls (6.95,-1.4) and (3.31,-0.3) .. (0,0) .. controls (3.31,0.3) and (6.95,1.4) .. (10.93,3.29)   ;
%Straight Lines [id:da6342171637110467] 
\draw    (289.5,100.5) -- (361,100.5) ;
\draw [shift={(361,100.5)}, rotate = 0] [color={rgb, 255:red, 0; green, 0; blue, 0 }  ][fill={rgb, 255:red, 0; green, 0; blue, 0 }  ][line width=0.75]      (0, 0) circle [x radius= 3.35, y radius= 3.35]   ;
\draw [shift={(289.5,100.5)}, rotate = 0] [color={rgb, 255:red, 0; green, 0; blue, 0 }  ][fill={rgb, 255:red, 0; green, 0; blue, 0 }  ][line width=0.75]      (0, 0) circle [x radius= 3.35, y radius= 3.35]   ;
%Curve Lines [id:da5763834883978867] 
\draw    (299.5,140) .. controls (305.29,130.35) and (319,118.84) .. (296.13,107.26) ;
\draw [shift={(293.5,106)}, rotate = 24.36] [fill={rgb, 255:red, 0; green, 0; blue, 0 }  ][line width=0.08]  [draw opacity=0] (8.93,-4.29) -- (0,0) -- (8.93,4.29) -- cycle    ;
%Curve Lines [id:da018062139538306266] 
\draw    (346.5,140) .. controls (338.7,134.15) and (322.34,121.65) .. (353.04,106.19) ;
\draw [shift={(355.5,105)}, rotate = 154.8] [fill={rgb, 255:red, 0; green, 0; blue, 0 }  ][line width=0.08]  [draw opacity=0] (8.93,-4.29) -- (0,0) -- (8.93,4.29) -- cycle    ;
%Straight Lines [id:da4154675109008572] 
\draw    (92.5,119) -- (165.5,98.5) ;
\draw [shift={(165.5,98.5)}, rotate = 344.31] [color={rgb, 255:red, 0; green, 0; blue, 0 }  ][fill={rgb, 255:red, 0; green, 0; blue, 0 }  ][line width=0.75]      (0, 0) circle [x radius= 3.35, y radius= 3.35]   ;
\draw [shift={(92.5,119)}, rotate = 344.31] [color={rgb, 255:red, 0; green, 0; blue, 0 }  ][fill={rgb, 255:red, 0; green, 0; blue, 0 }  ][line width=0.75]      (0, 0) circle [x radius= 3.35, y radius= 3.35]   ;
%Straight Lines [id:da2554814009339297] 
\draw    (221,100.5) -- (289.5,100.5) ;
\draw [shift={(221,100.5)}, rotate = 0] [color={rgb, 255:red, 0; green, 0; blue, 0 }  ][fill={rgb, 255:red, 0; green, 0; blue, 0 }  ][line width=0.75]      (0, 0) circle [x radius= 3.35, y radius= 3.35]   ;
%Straight Lines [id:da9758945233701857] 
\draw    (361,100.5) -- (425,100.5) ;
\draw [shift={(425,100.5)}, rotate = 0] [color={rgb, 255:red, 0; green, 0; blue, 0 }  ][fill={rgb, 255:red, 0; green, 0; blue, 0 }  ][line width=0.75]      (0, 0) circle [x radius= 3.35, y radius= 3.35]   ;
%Straight Lines [id:da7998611301902778] 
\draw    (461,106) -- (550.5,122) ;
\draw [shift={(550.5,122)}, rotate = 10.14] [color={rgb, 255:red, 0; green, 0; blue, 0 }  ][fill={rgb, 255:red, 0; green, 0; blue, 0 }  ][line width=0.75]      (0, 0) circle [x radius= 3.35, y radius= 3.35]   ;
\draw [shift={(461,106)}, rotate = 10.14] [color={rgb, 255:red, 0; green, 0; blue, 0 }  ][fill={rgb, 255:red, 0; green, 0; blue, 0 }  ][line width=0.75]      (0, 0) circle [x radius= 3.35, y radius= 3.35]   ;

% Text Node
\draw (108,48.4) node [anchor=north west][inner sep=0.75pt]    {$X_{0}$};
% Text Node
\draw (248,43.4) node [anchor=north west][inner sep=0.75pt]    {$X_{i}$};
% Text Node
\draw (379,49.4) node [anchor=north west][inner sep=0.75pt]    {$X_{i+1}$};
% Text Node
\draw (501,57.4) node [anchor=north west][inner sep=0.75pt]    {$X_{n}$};
% Text Node
\draw (248,144.4) node [anchor=north west][inner sep=0.75pt]    {$\| \pi _{X_{i}}( q_{i}) \| ,$};
% Text Node
\draw (322,144.4) node [anchor=north west][inner sep=0.75pt]    {$\| \pi _{X_{i+1}}( q_{i}) \| \leq \tau $};
% Text Node
\draw (250,79.4) node [anchor=north west][inner sep=0.75pt]    {$p_{i}$};
% Text Node
\draw (116,87.4) node [anchor=north west][inner sep=0.75pt]    {$p_{0}$};
% Text Node
\draw (380,80.4) node [anchor=north west][inner sep=0.75pt]    {$p_{i+1}$};
% Text Node
\draw (319,76.4) node [anchor=north west][inner sep=0.75pt]    {$q_{i}$};
% Text Node
\draw (501,92.4) node [anchor=north west][inner sep=0.75pt]    {$p_{n}$};
% Text Node
\draw (239,103.4) node [anchor=north west][inner sep=0.75pt]    {$\geq L$};
% Text Node
\draw (374,101.4) node [anchor=north west][inner sep=0.75pt]    {$\geq L$};

\end{tikzpicture}
    \caption{Admissible path}
    \label{fig:admissiblepath}
\end{figure}
\begin{defn}[Admissible Path]\label{AdmDef} Given $L,\tau\geq0$, a path $\gamma$ is called $(L,\tau)$-\textit{admissible} in $\U$, if $\gamma$ is a concatenation of geodesics $p_0q_1p_1\cdots q_np_n$ $(n\in\mathbb{N})$, where the two endpoints of each $p_i$ lie in some $X_i\in \f$, and   the following   \textit{Long Local} and \textit{Bounded Projection} properties hold:
\begin{enumerate}
\item[(LL)] Each $p_i$  for $1\le i< n$ has length bigger than $L$, and  $p_0,p_n$ could be trivial;
\item[(BP)] For each $X_i$, we have $X_i\ne X_{i+1}$ and $\max\{\|\pi_{X_i}(q_i)\|,\|\pi_{X_i}(q_{i+1})\|\}\leq\tau$, where $q_0:=\gamma_-$ and $q_{n+1}:=\gamma_+$ by convention.
\end{enumerate} 
The collection $\{X_i: 1\le i\le n\}$ is referred to as contracting subsets associated with the admissible path.
\end{defn}
\begin{rem}\label{ConcatenationAdmPath}
\begin{enumerate}
    \item 
    The path $q_i$ could be allowed to be trivial, so by the (BP) condition, it suffices to check $X_i\ne X_{i+1}$. It will be useful to note that admissible paths could be concatenated as follows: Let $p_0q_1p_1\cdots q_np_n$ and $p_0'q_1'p_1'\cdots q_n'p_n'$ be $(L,\tau)$--admissible. If $p_n=p_0'$ has length bigger than $L$, then the concatenation $(p_0q_1p_1\cdots q_np_n)\cdot (q_1'p_1'\cdots q_n'p_n')$ has a natural $(L,\tau)$--admissible structure.  
    \item 
    In many situations, $\tau$ could be chosen as the bounded projection constant $B$ of $\f$. In fact, if $L$ is large relative to $\tau$,  an $(L,\tau)$--admissible path could be always truncated near contracting subsets $X_i$  so that it becomes an $(\hat L,B)$--admissible path. See \cite[Lemma 2.14]{YANG22}.
\end{enumerate}
\end{rem}

We frequently construct a path labeled by a word $(g_1, g_2,\cdots,g_n)$, which by convention means the following concatenation
$$
[o,g_1o]\cdot g_1[o,g_2o]\cdots (g_1\cdots g_{n-1})[o,g_no]
$$
where the basepoint $o$ is understood in context. With this convention, the paths labeled by $(g_1,g_2,g_3)$ and $(g_1g_2, g_3)$ respectively differ, depending on whether  $[o,g_1o]g_1[o,g_2o]$ is a geodesic or not. 

A sequence of points $x_i$ in a path $p$   is called \textit{linearly ordered} if $x_{i+1}\in [x_i, p^+]_p$ for each $i$. %$d(x_i, x_{i+1})\ge 1$ and .

\begin{defn}[Fellow travel]\label{Fellow}
Let   $\gamma = p_0 q_1 p_1 \cdots q_n p_n$ be an $(L, \tau)-$admissible
path. We say $\gamma$ has \textit{$r$--fellow travel} property for some $r>0$   if for any geodesic  
$\alpha$  with the same endpoints as $\gamma$,   there exists a sequence of linearly ordered points $z_i,
w_i$ ($0 \le i \le n$) on $\alpha$ such that  
$$d(z_i, p_{i}^-) \le r,\quad d(w_i, p_{i}^+) \le r.$$
In particular, $\|N_r(X_i)\cap \alpha\|\ge L$ for each $X_i\in \f(\gamma)$. 
\end{defn}
The following result  says that   a local long admissible path enjoys the fellow travel property.

\begin{prop}\label{admisProp}\cite{YANG6}
For any $\tau>0$, there exist $L,  r>0$ depending only on $\tau,C$ such that  any $(L, \tau)$--admissible path   has $r$--fellow travel property. In particular, it is a $c$--quasi-geodesic.
\end{prop}

The next lemma gives a way to build admissible paths.
\begin{lem}[Extension Lemma]\label{extend3}

{For any independent contracting elements $h_{1}, h_{2}, h_{3} \in \Gamma$,} there exist constants  $L, r, B>0$ depending only on $C$ with the following property.  

Choose any element $f_i\in \langle h_i\rangle$ for each $1\le i\le 3$  to form the set $F$ satisfying $\|Fo\|_{\min}\ge L$. Let $g,h\in \Gamma$ be any two elements.
\begin{enumerate}
\item
There exists an element $f \in F$ such that   the path  $$\gamma:=[o, go]\cdot(g[o, fo])\cdot(gf[o,ho])$$ is an $(L, \tau)$--admissible path relative to $\f$. 
\item
The point  $go$  is an $(r, f)$--barrier for any geodesic  $[\gamma^-,\gamma^+]$.	%In particular, the path $[o, go][go, gfo][gfo, gfho]$ is a $(1, 4\epsilon_0)$--quasi-geodesic.
\end{enumerate}
\end{lem}

\begin{rem}
%\begin{enumerate}
%\item
%In \cite{YANG10}, the geodesic   $\alpha=[o, ho]$ is chosen to be ending at $\Gamma o$ for some $h\in \Gamma$. However, the proof given there works for a   geodesic starting from $o$ possibly ending at any point in $\U$. 
%\item
Since admissible paths are local conditions, we can connect via $F$  any number of elements  $g\in G$ to satisfy (1) and (2). We refer the reader to \cite{YANG10} for a precise formulation.
%\end{enumerate}
\end{rem}

The following elementary fact will be invoked frequently.
\begin{lem}\label{InjectiveExtMap}
Assume that two words $(g_1, f_1, h_1)$ and $(g_2, f_2, h_2)$ label two $(L,\tau)$--admissible paths with the same endpoints, where $(C,\tau, L)$ satisfy Proposition \ref{admisProp}. For any $\Delta>1$, there exists $R=R(\Delta, C, \tau)$ with the following property.
If $|d(o,g_1o)-d(o,g_2o)|\le \Delta$ and $d(g_1o,g_2o)>R$, then $g_1=g_2$.   
\end{lem}

\subsection{Horofunction boundary}
We recall the notion of horofunction boundary.

Fix a basepoint $o\in \U$. For  each $y \in  \U$, we define a Lipschitz map $b_y:  \U\to \U$     by $$\forall x\in \U:\quad b_y(x)=d(x, y)-d(o,y).$$ This   family of $1$--Lipschitz functions sits in the set of continuous functions on $ \U$ vanishing at $o$.  Endowed  with the compact-open topology, the  Arzela-Ascoli Lemma implies that the closure  of $\{b_y: y\in  \U\}$  gives a compactification of $ \U$. The complement of $X$ in this compactification is called  the \textit{horofunction boundary} of $ \U$ and is denoted by $\hU$.

A \textit{Buseman cocycle} $B_\xi:  \U\times \U \to \mathbb R$ (independent of $o$) is given by $$\forall x_1, x_2\in  \U: \quad B_\xi(x_1, x_2)=b_\xi(x_1)-b_\xi(x_2).$$

%A horosphere denoted by $\mathcal{HS}(\xi)$ at $\xi$ is the set of points $x\in \U$ such that $B_\xi(x,y) \asymp 0$ for any two $x,y\in HS_\xi$. A horoball denoted by $\mathcal{HB}(\xi)$ at $\xi$ is the set of points $x\in \U$ such that $b_\xi(x) \prec b$ for some $b\in \mathbb R$.
  
The topological type of horofunction boundary is independent of  the choice of a basepoint. Every isometry $\phi$ of $\U$ induces a homeomorphism on $\bU$:  
$$
\forall y\in \U:\quad\phi(\xi)(y):=b_\xi(\phi^{-1}(y))-b_\xi(\phi^{-1}(o)).
$$
Depending on the context, we may use both $\xi$ and $b_\xi$ to denote a point in the horofunction boundary.
\\
\paragraph{\textbf{Finite difference relation}.}
Two horofunctions $b_\xi, b_\eta$ have   \textit{$K$--finite difference} for $K\ge 0$ if the $L^\infty$--norm of their difference is $K$--bounded: $$\|b_\xi-b_\eta\|_\infty\le K.$$ 
The   \textit{locus} of     $b_\xi$ consists of  horofunctions $b_\eta$ so that $b_\xi, b_\eta$ have   $K$--finite difference for some $K>0$.  The loci   $[b_\xi]$  of    horofunctions $b_\xi$ form a \textit{finite difference equivalence relation} $[\cdot]$ on $\hU$. The \textit{locus} $[\Lambda]$ of a subset $\Lambda\subseteq \hU$ is the union of loci of all points in $\Lambda$.

If $x_n\in \U\to \xi\in \pU$ and  $y_n\in \U\to\eta\in \pU$ are sequences with $\sup_{n\ge 1}d(x_n, y_n)<\infty$, then  $[\xi]=[\eta]$.

\subsection{Convergence boundary}

Let $(\U, d)$ be a proper metric space admitting an isometric action of a non-elementary countable group $\Gamma$ with a contracting element. Consider a metrizable compactification $\bU:=\pU\cup \U$, so that $\U$ is open and dense in $\bU$. We also assume that the action of $\isom(\U)$ extends by homeomorphism to  $\pU$. 

We   equip $\pU$    with a  $\isom(\U)$--invariant  partition $[\cdot]$:   $[\xi]=[\eta]$ implies $[g\xi]=[g\eta]$ for any $g\in \isom(\U)$.  We say that $\xi$ is \textit{minimal} if $[\xi]=\{\xi\}$, and a subset $U$ is \textit{$[\cdot]$--saturated} if $U=[U]$.

We say that $[\cdot]$ restricts to be a \textit{closed} partition on a $[\cdot]$--saturated subset $U\subseteq \pU$ if  $x_n\in U\to \xi\in \pU$ and $y_n\in U\to\eta\in \pU$ are two sequences with $[x_n]=[ y_n]$, then $[\xi]=[\eta]$. (The points $\xi, \eta$ are not necessarily in $U$.) If $U=\pU$, this is equivalent to saying that the relation $\{(\xi,\eta): [\xi]=[\eta]\}$ is a closed subset in $\pU\times \pU$, so the quotient space $[\pU]$ is Hausdorff. In general, $[\cdot]$ may not be closed over the whole $\pU$ (\emph{e.g.}, the horofunction boundary with finite difference relation), but is closed when restricted to certain interesting subsets; see for example Assumption (C) below.  

We say that $x_n$ \textit{tends} (resp. \textit{accumulates}) to $[\xi]$ if the limit point (resp. any accumulate point) is contained in the subset $[\xi]$. This implies that $[x_n]$ tends or accumulates to $[\xi]$ in the quotient space $[\pG]$. So, an infinite ray $\gamma$ \textit{terminates} at $[\xi] \in \pU$ if any sequence of points in $\gamma$ accumulates in $[\xi]$. 
\begin{figure}
    \centering

\tikzset{every picture/.style={line width=0.75pt}} %set default line width to 0.75pt        

\begin{tikzpicture}[x=0.75pt,y=0.75pt,yscale=-1,xscale=1]

%Straight Lines [id:da6801304211610646] 
\draw    (27.5,123) -- (217.5,123) ;
\draw [shift={(219.5,123)}, rotate = 180] [color={rgb, 255:red, 0; green, 0; blue, 0 }  ][line width=0.75]    (10.93,-3.29) .. controls (6.95,-1.4) and (3.31,-0.3) .. (0,0) .. controls (3.31,0.3) and (6.95,1.4) .. (10.93,3.29)   ;
\draw [shift={(27.5,123)}, rotate = 0] [color={rgb, 255:red, 0; green, 0; blue, 0 }  ][fill={rgb, 255:red, 0; green, 0; blue, 0 }  ][line width=0.75]      (0, 0) circle [x radius= 3.35, y radius= 3.35]   ;
%Curve Lines [id:da4277597460569684] 
\draw    (146.5,78) .. controls (135.94,102) and (140.13,107.57) .. (144.02,121.25) ;
\draw [shift={(144.5,123)}, rotate = 255.07] [color={rgb, 255:red, 0; green, 0; blue, 0 }  ][line width=0.75]    (10.93,-3.29) .. controls (6.95,-1.4) and (3.31,-0.3) .. (0,0) .. controls (3.31,0.3) and (6.95,1.4) .. (10.93,3.29)   ;
\draw [shift={(146.5,78)}, rotate = 113.75] [color={rgb, 255:red, 0; green, 0; blue, 0 }  ][fill={rgb, 255:red, 0; green, 0; blue, 0 }  ][line width=0.75]      (0, 0) circle [x radius= 3.35, y radius= 3.35]   ;
%Shape: Ellipse [id:dp7897587100437391] 
\draw   (393,102) .. controls (393,90.95) and (408.68,82) .. (428.01,82.01) .. controls (447.34,82.01) and (463,90.97) .. (463,102.01) .. controls (463,113.06) and (447.33,122.01) .. (428,122.01) .. controls (408.67,122) and (393,113.05) .. (393,102) -- cycle ;
%Shape: Ellipse [id:dp22425643702990228] 
\draw   (310,107.98) .. controls (310,96.93) and (325.67,87.98) .. (345,87.99) .. controls (364.33,87.99) and (380,96.95) .. (380,107.99) .. controls (380,119.04) and (364.33,127.99) .. (345,127.99) .. controls (325.67,127.98) and (310,119.03) .. (310,107.98) -- cycle ;
%Curve Lines [id:da8931704937311309] 
\draw    (257.5,110.97) .. controls (307.5,113.98) and (376.5,118.99) .. (411.5,132) ;
\draw [shift={(411.5,132)}, rotate = 20.39] [color={rgb, 255:red, 0; green, 0; blue, 0 }  ][fill={rgb, 255:red, 0; green, 0; blue, 0 }  ][line width=0.75]      (0, 0) circle [x radius= 3.35, y radius= 3.35]   ;
%Curve Lines [id:da7230316762704059] 
\draw    (257.5,110.97) .. controls (296.5,102.98) and (433.5,89) .. (478.5,78) ;
\draw [shift={(478.5,78)}, rotate = 346.26] [color={rgb, 255:red, 0; green, 0; blue, 0 }  ][fill={rgb, 255:red, 0; green, 0; blue, 0 }  ][line width=0.75]      (0, 0) circle [x radius= 3.35, y radius= 3.35]   ;
\draw [shift={(257.5,110.97)}, rotate = 348.42] [color={rgb, 255:red, 0; green, 0; blue, 0 }  ][fill={rgb, 255:red, 0; green, 0; blue, 0 }  ][line width=0.75]      (0, 0) circle [x radius= 3.35, y radius= 3.35]   ;
%Curve Lines [id:da06488231904786601] 
\draw  [dash pattern={on 4.5pt off 4.5pt}]  (478.5,78) .. controls (508.04,93.74) and (489.09,100.79) .. (522.92,106.73) ;
\draw [shift={(524.5,107)}, rotate = 189.46] [color={rgb, 255:red, 0; green, 0; blue, 0 }  ][line width=0.75]    (10.93,-3.29) .. controls (6.95,-1.4) and (3.31,-0.3) .. (0,0) .. controls (3.31,0.3) and (6.95,1.4) .. (10.93,3.29)   ;
\draw [shift={(478.5,78)}, rotate = 28.05] [color={rgb, 255:red, 0; green, 0; blue, 0 }  ][fill={rgb, 255:red, 0; green, 0; blue, 0 }  ][line width=0.75]      (0, 0) circle [x radius= 3.35, y radius= 3.35]   ;
%Curve Lines [id:da21072210820562565] 
\draw  [dash pattern={on 4.5pt off 4.5pt}]  (146.5,78) .. controls (176.19,93.82) and (150.03,95.96) .. (212.58,113.46) ;
\draw [shift={(214.5,114)}, rotate = 195.48] [color={rgb, 255:red, 0; green, 0; blue, 0 }  ][line width=0.75]    (10.93,-3.29) .. controls (6.95,-1.4) and (3.31,-0.3) .. (0,0) .. controls (3.31,0.3) and (6.95,1.4) .. (10.93,3.29)   ;
\draw [shift={(146.5,78)}, rotate = 28.05] [color={rgb, 255:red, 0; green, 0; blue, 0 }  ][fill={rgb, 255:red, 0; green, 0; blue, 0 }  ][line width=0.75]      (0, 0) circle [x radius= 3.35, y radius= 3.35]   ;
%Curve Lines [id:da17846536477484798] 
\draw    (202.5,267) .. controls (180.5,254) and (181.5,222) .. (203.5,208) ;
\draw [shift={(203.5,208)}, rotate = 327.53] [color={rgb, 255:red, 0; green, 0; blue, 0 }  ][fill={rgb, 255:red, 0; green, 0; blue, 0 }  ][line width=0.75]      (0, 0) circle [x radius= 3.35, y radius= 3.35]   ;
\draw [shift={(202.5,267)}, rotate = 210.58] [color={rgb, 255:red, 0; green, 0; blue, 0 }  ][fill={rgb, 255:red, 0; green, 0; blue, 0 }  ][line width=0.75]      (0, 0) circle [x radius= 3.35, y radius= 3.35]   ;
%Straight Lines [id:da1153290975468293] 
\draw    (86.12,240) -- (183.5,240) ;
\draw [shift={(185.5,240)}, rotate = 180] [color={rgb, 255:red, 0; green, 0; blue, 0 }  ][line width=0.75]    (10.93,-3.29) .. controls (6.95,-1.4) and (3.31,-0.3) .. (0,0) .. controls (3.31,0.3) and (6.95,1.4) .. (10.93,3.29)   ;
\draw [shift={(86.12,240)}, rotate = 0] [color={rgb, 255:red, 0; green, 0; blue, 0 }  ][fill={rgb, 255:red, 0; green, 0; blue, 0 }  ][line width=0.75]      (0, 0) circle [x radius= 3.35, y radius= 3.35]   ;
%Curve Lines [id:da6227887857681849] 
\draw  [dash pattern={on 4.5pt off 4.5pt}]  (230.5,209) .. controls (268.92,217.87) and (253,232.55) .. (294.56,233.95) ;
\draw [shift={(296.5,234)}, rotate = 181.3] [color={rgb, 255:red, 0; green, 0; blue, 0 }  ][line width=0.75]    (10.93,-3.29) .. controls (6.95,-1.4) and (3.31,-0.3) .. (0,0) .. controls (3.31,0.3) and (6.95,1.4) .. (10.93,3.29)   ;
%Curve Lines [id:da18331422860761903] 
\draw  [dash pattern={on 4.5pt off 4.5pt}]  (234.5,269) .. controls (258.14,262.11) and (253.64,241.63) .. (295.55,239.1) ;
\draw [shift={(297.5,239)}, rotate = 177.4] [color={rgb, 255:red, 0; green, 0; blue, 0 }  ][line width=0.75]    (10.93,-3.29) .. controls (6.95,-1.4) and (3.31,-0.3) .. (0,0) .. controls (3.31,0.3) and (6.95,1.4) .. (10.93,3.29)   ;

% Text Node
\draw (122,131.4) node [anchor=north west][inner sep=0.75pt]    {$\pi _{\gamma }( y_{n})$};
% Text Node
\draw (160,61.4) node [anchor=north west][inner sep=0.75pt]    {$y_{n}$};
% Text Node
\draw (222,112.4) node [anchor=north west][inner sep=0.75pt]    {$[ \xi ]$};
% Text Node
\draw (21,127.4) node [anchor=north west][inner sep=0.75pt]    {$o$};
% Text Node
\draw (337,133.39) node [anchor=north west][inner sep=0.75pt]  [rotate=-0.01]  {$\gamma _{1}$};
% Text Node
\draw (527,97.42) node [anchor=north west][inner sep=0.75pt]  [rotate=-0.01]  {$[ \xi ]$};
% Text Node
\draw (416,89.4) node [anchor=north west][inner sep=0.75pt]  [rotate=-0.01]  {$\gamma _{n}$};
% Text Node
\draw (408,138.39) node [anchor=north west][inner sep=0.75pt]  [rotate=-0.01]  {$y_{1} \in \Omega _{o}( \gamma _{1})$};
% Text Node
\draw (443,48.4) node [anchor=north west][inner sep=0.75pt]    {$y_{n} \in \Omega _{o}( \gamma _{n})$};
% Text Node
\draw (253.5,91.37) node [anchor=north west][inner sep=0.75pt]    {$o$};
% Text Node
\draw (68.57,230.4) node [anchor=north west][inner sep=0.75pt]    {$o$};
% Text Node
\draw (212,200.4) node [anchor=north west][inner sep=0.75pt]    {$x_{n}$};
% Text Node
\draw (213,260.4) node [anchor=north west][inner sep=0.75pt]    {$y_{n}$};
% Text Node
\draw (301,227.4) node [anchor=north west][inner sep=0.75pt]    {$[ \xi ] \subseteq \mathcal{C}$};
% Text Node
\draw (370,228.4) node [anchor=north west][inner sep=0.75pt]    {$\Longrightarrow $};
% Text Node
\draw (421,224.4) node [anchor=north west][inner sep=0.75pt]    {$d( o,[ x_{n} ,y_{n}])\rightarrow \infty $};
% Text Node
\draw (18,10) node [anchor=north west][inner sep=0.75pt]   [align=left] {(\textbf{A})};
% Text Node
\draw (251,10) node [anchor=north west][inner sep=0.75pt]   [align=left] {(\textbf{B})};
% Text Node
\draw (20,232) node [anchor=north west][inner sep=0.75pt]   [align=left] {(\textbf{C})};

\end{tikzpicture}
    \caption{Assumptions (A)(B)(C) in convergence boundary}
    \label{fig:convbdry}
\end{figure}
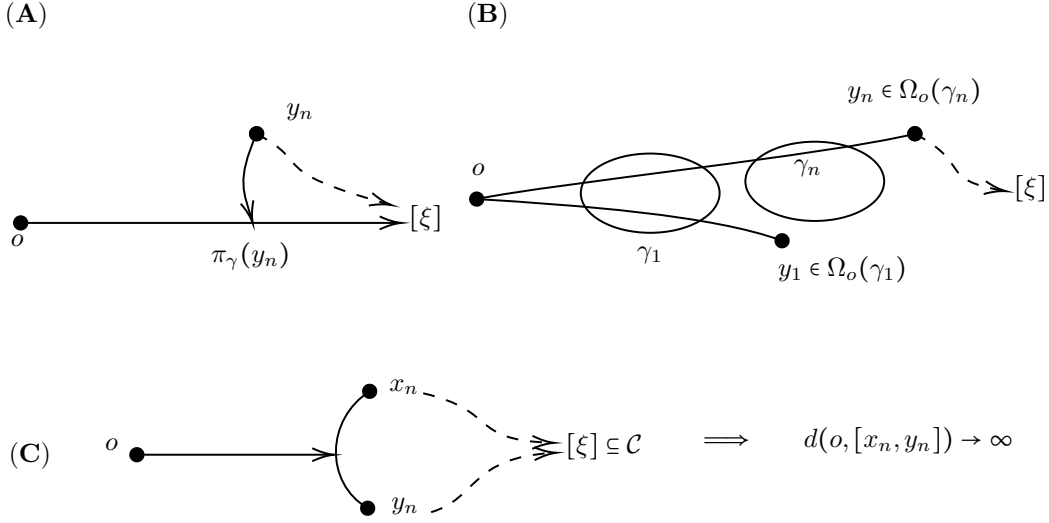

Recall that $\Omega_x(A):=\{y\in \U: \exists [x,y]\cap A\ne\emptyset\}$ is the cone of a subset $A\subseteq \U$ with light source at $x$. A sequence of subsets $A_n\subseteq \U$ is \textit{escaping} if $d(o,A_n)\to \infty$ for some (or any) $o\in \U$.

\begin{defn}\label{ConvBdryDefn}
    
We say that $(\bU,[\cdot])$ is a \textit{convergence compactification} of $\U$ if the following assumptions hold.
\begin{enumerate}
    \item[\textbf{(A)}]Any contracting geodesic ray $\gamma$ accumulates into a closed subset $[\xi]$ for some  $\xi\in \pU$; and any sequence $y_n\in \U$ with escaping projections $\pi_\gamma(y_n)$ tends to $[\xi]$. 

    \item[\textbf{(B)}]
    Let $\{X_n\subseteq \U :n\ge 1\}$ be an escaping sequence of $C$--contracting  quasi-geodesics for some $C>0$. Then for any given $o\in \U$, there exists a  subsequence of $\{Y_n:=\Omega_o(X_n): n\ge 1\}$ (still denoted by $Y_n$)  and $\xi\in \pU$ such that $Y_n$  accumulates to $[\xi]$: any convergent sequence of points $y_n\in Y_n$ tends to a point in $[\xi]$.
    \item[\textbf{(C)}]
    The set $\mathcal C$ of \textit{non-pinched} points  $\xi\in \pU$ is non-empty. We say $\xi$ is non-pinched if $x_n, y_n\in \U$ are two sequences of points converging to $[\xi]$, then $[x_n,y_n]$ is an escaping sequence of geodesic segments. Moreover,  for any $\xi\in \mathcal C$, the partition $[\cdot]$ restricted to $G[\xi]$ forms a closed relation.
\end{enumerate}  
\end{defn} 
\begin{rem}
Assumption (\textbf{C}) is a non-triviality condition: any compactification with the coarsest partition asserting $\pU$ as one $[\cdot]$--class satisfy (\textbf{A}) and (\textbf{B}). We require $\mathcal C$ to be non-empty to exclude such situations.  The ``moreover" statement is newly added relative to the one in \cite{YANG22}. If the restriction of the partition to $\mathcal C$ is maximal in the sense that every $[\cdot]$--class is singleton, then the ``moreover" assumption is always true.   
\end{rem} 
By Assumption (A), a contracting quasi-geodesic ray $\gamma$ determines a unique $[\cdot]$--class of a boundary point denoted by $[\gamma^+]$. Similarly, the positive and negative rays of a  bi-infinite contracting quasi-geodesic $\gamma:(-\infty,\infty)\to \U$ determine respectively two boundary $[\cdot]$--classes denoted by $[\gamma^+]$ and $[\gamma^-]$. 

\begin{examples}\label{ConvbdryExamples}
The first three convergence boundaries below are equipped with  a \textit{maximal} partition $[\cdot]$ (that is, $[\cdot]$--classes are singletons).
    \begin{enumerate}
        \item 
        Hyperbolic space $\U$ with Gromov boundary $\pU$, where  all boundary points are non-pinched.
        \item 
        CAT(0) space $\U$ with visual boundary $\pU$ (homeomorphic to horofunction boundary), where all boundary points are non-pinched.
        \item
        The Cayley graph $\U$ of a relatively hyperbolic group with Bowditch or Floyd boundary $\pU$, where  conical limit points are non-pinched.

        If $\U$ is infinitely ended, we could also take $\pU$ as the end boundary with the same statement.
        \item
        Teichm\"{u}ller space $\U$ with Thurston  boundary $\pU$, where $[\cdot]$ is given by Kaimanovich-Masur partition \cite{KaMasur} and uniquely ergodic   points are non-pinched.
        \item
        Any proper metric space $\U$ with horofunction boundary $\pU$, where $[\cdot]$ is given by finite difference partition and all boundary points are non-pinched.
        If $\U$ is the cubical CAT(0) space, the horofunction boundary is exactly the Roller boundary. 
        If $\U$ is the Teichm\"{u}ller space with Teichm\"{u}ller metric, the horofunction boundary is  the Gardiner-Masur  boundary (\cite{LS14,Walsh19}). 
    \end{enumerate}
\end{examples}
From these examples, one can see that the convergence boundary is not a canonical object associated to a given proper metric space. In some sense, the horofunction boundary provides a ``universal" convergence boundary with non-pinched points for any proper action.
\begin{thm}\label{HorobdryConvergence}\cite[Theorem 1.1]{YANG22}
The horofunction boundary is a convergence boundary with finite difference relation $[\cdot]$, where all boundary points are non-pinched. 
\end{thm} 
\begin{proof}
Only the ``moreover" statement in Assumption (C) requires clarification. Under the finite difference relation $[\cdot]$, if $[\xi]$ is $K$--bounded for some $K>0$, then any $g[\xi]$ is $K$--bounded as well. According to the topology of the horofunction compactification,  if $|b_{x_n}(\cdot)-b_{y_n}(\cdot)|\le K$ and $x_n\to \xi, y_n\to \eta$, then $|b_{\xi}(\cdot)-b_{\eta}(\cdot)|\le K$. This shows $[\xi]=[\eta]$ (however, $[\xi]$ may  not be $K$--bounded).      
\end{proof}

If $f\in \isom(\U)$ is a contracting element,  the positive and negative rays of a contracting quasi-geodesic $n\in \mathbb Z\mapsto f^n o$ determine two boundary points denoted by $[f^+]$ and $[f^-]$ respectively. Write $[f^\pm]=[f^+]\cup[f^-]$. We say that $f$ is  \textit{non-pinched} if $[f^\pm]\subseteq \mathcal C$ are non-pinched. On the horofunction boundary, any contracting element is non-pinched by Theorem \ref{HorobdryConvergence}. In contrast, there exist pinched contracting elements in the Bowditch boundary of a relatively hyperbolic group; for instance, those contracting elements in a parabolic group are pinched. The following two results fail if the non-pinched condition is dropped.

Recall that $\Gamma<\isom(\U)$ is a discrete group, which we assume in the next results. 

\begin{lem}\label{NonPinchedStabilizer}
If $f\in \Gamma$ is a non-pinched contracting element, then $E(f)$ coincides with the set stabilizer of the two fixed points $[f^\pm]$.    
\end{lem}
\begin{proof}
By Assumption A, the two half ray of the axis $\ax(f)$  accumulates either to $[f^-]$ or to $[f^+]$. The same holds for $g\ax(f)$, which accumulates  to $g[f^\pm]$.  By definition,  $E(f)$ is the set of elements $g\in G$ which preserves  $\ax(f)$ up to finite Hausdorff distance. Let $g\in G$ preserve $[f^\pm]$ but $g\notin E(f)$. Choose two sequence of points $x_n, y_n\in g\ax(f)$ so that $x_n\to [f^-], y_n\in [f^+]$, and let $z_n\in \pi_{\ax(f)}(x_n)$ and $w_n\in \pi_{\ax(f)}(y_n)$. We claim that both $z_n$ and $w_n$ are unbounded sequences. If not, assume for concreteness that $z_n$ is bounded. Choose $u_n\in \ax(f)$ tending to $[f^-]$ ($[f^-]$, if $w_n$ is bounded). The contracting property implies $[x_n, u_n]\cap B(z_n,2C)\ne\emptyset$, where $C$ is the contracting constant of $\ax(f)$. As $z_n$ is bounded and $x_n, u_n$ tend to $[f^-]$, this contradicts $[f^-]\subseteq\mathcal C$ by Definition \ref{ConvBdryDefn}(C).  
\end{proof}

\begin{lem}\cite[Lemma 3.12]{YANG22}\label{DisjointFixpts}
If $f, g\in \Gamma$ are two non-pinched independent contracting elements, then either $[f^\pm]=[g^\pm]$ or $[f^\pm]\cap[g^\pm]=\emptyset$.    
\end{lem}

\begin{lem}\cite[Lemma 3.10]{YANG22}\label{DoubleDense}
Let $g,f\in \Gamma$ be two independent  contracting elements. Then for all $n\gg 0$, $h_n:=g^nf^n$ is a contracting element so that $[h_n^+]$ tends to $[g^+]$ and $[h_n^-]$ tends to $[f^-]$.    
\end{lem}

The following lemma is often used to choose an independent contracting element.
\begin{lem}\label{IndepElemExists}
Let $H<\isom(\U)$ be a non-elementary discrete group. Assume that $f\in H$  is  a contracting element. Then there exists a contracting element $h\in H$ that is independent with $f$.     
\end{lem}
\begin{proof}
The set of translated axis $\f:=\{g\ax(f): g\in \Gamma\}$ has bounded intersection. Pick two distinct conjugates denoted by $f_1,f_2$ of $f$. By Lemma \ref{DoubleDense}, $h:=f_1^nf_2^n$ is a contracting element whose fixed points tend to $[f_1^+]$ and $[f_2^-]$ respectively for any large $n\gg 0$. Note that the axis $\ax(h)$ of $h$ has an overlap with  $\ax(f_1)$ and $\ax(f_2)$  of a large length $d(o,f_i^no)$ for $i=1,2$. Taking into account every pair of elements in $\f$ has bounded intersection, we conclude  that $\ax(h)$ has bounded intersection with every element in $\f$. That is, $h$ is independent with $f$  by definition.
\end{proof}

Let $U$ be a subset of $\U$. The limit set of $U$, denoted by $\Lambda(U)$, is the set of all accumulation points of $U$ in $\pU$. If $H<\Gamma$ is a subgroup,   $\Lambda(Ho)$ may depend on the choice of the basepoint $o\in \U$, but the $[\cdot]$--loci $[\Lambda(Ho)]$ does not  by  Definition \ref{ConvBdryDefn}(B). For this reason, we shall refer $[\Lambda(Ho)]$ as the limit set of a subgroup $H$. Moreover, the limit set of a discrete group $\Gamma$ enjoys the following desirable property as in the general theory of convergence group action.
\begin{lem}\cite[Lemma 3.9]{YANG22}\label{FixptsDense}
Assume that $\Gamma$ is non-elementary. Let $f\in \Gamma$ be a   contracting element. Then $[\pG]=[\overline{\Gamma\xi}]$ for any $\xi\in [f^\pm]$.    
\end{lem}

%In certain situation, we could obtain a unique minimal limit set.

%\begin{lem}\cite[Lemma 3.22]{YANG22}\label{UniqueLimitSet}
%Assume that $f\in \Gamma$ is a non-pinched contracting element with minimal fixed points.   Then $\overline{\Gamma\xi}$ is the unique minimal $\Gamma$--invariant subset. Moreover, it is the same as the closure of all minimal fixed points of non-pinched contracting elements. 

%Let $\Gamma$ be a non-elementary discrete subgroup. Let $\Lambda$ the closure of all minimal fixed points of non-pinched contracting elements.  Then $[\Lambda(\Gamma o)]=[\Lambda]$.
%\end{lem}

\subsection{Quasi-conformal density and Patterson's construction}
Let $\bU:=\pU\cup\U$ be a convergence compactification, with a nonempty set $\mathcal C$ of non-pinched points in Definition \ref{ConvBdryDefn}.  We need to restrict to a smaller subset of  $\mathcal C$, on which some weak continuity of Busemann cocycles can be guaranteed.

\begin{assumpD}\label{AssumpE}
There exists a subset $\mathcal C^{\mathrm{hor}}\subseteq \mathcal C$ with a family  of  Buseman quasi-cocycles
$$\big\{B_\xi: \quad \U \times \U \longrightarrow \mathbb R\big\}_{\xi \in
\mathcal C^{\mathrm{hor}}}$$ so that  for any $x, y \in \U$, we have 
\begin{align}\label{BusemanConvError}
\limsup_{z\to[\xi]} |B_\xi(x, y)-B_z(x, y)| \le \epsilon,   
\end{align}
where $\epsilon\ge 0$ may depend on $[\xi]$ but not on $x,y$. 

For a given  $\epsilon$, let $\mathcal C^{\mathrm{hor}}_\epsilon$  denote the set of points $\xi\in \mathcal C^{\mathrm{hor}}$ for which (\ref{BusemanConvError}) holds.  Thus, $\mathcal C^{\mathrm{hor}}=\bigcup_{\epsilon\ge 0}\mathcal C^{\mathrm{hor}}_\epsilon$.  
\end{assumpD}

\iffalse\ywy{Here are some examples clarifying the relation of $\mathcal C^{\mathrm{hor}}_\epsilon$ with the  
\begin{examples}
\begin{enumerate}
    \item For horofunction boundary, $\mathcal C^{\mathrm{hor}}_0$ might be a proper subset of $\hU$, as the difference of two horofunctions in one $[\xi]$--class may depend on  $[\xi]$. 
    \item 
    The Gromov boundary $\pU$ of a hyperbolic space $\U$ is contained in $\mathcal C^{\mathrm{hor}}_D$ for some large $D$.
    \item 
    The set of uniquely ergodic points $\ue$ in the Thurston boundary is contained in $\mathcal C^{\mathrm{hor}}_0$. 
\end{enumerate}  
\end{examples}}\fi
In practice, the assumption (\ref{BusemanConvError})  could allow to take the following concrete definition:
\begin{align*}
B_{\xi}(x,y)&=\limsup_{z_n\to \xi} \left[d(x,z_n)-d(y,z_n)\right]   
\end{align*}
The following facts shall be used implicitly.
\begin{equation}\label{BusemanCocylesP1}
\begin{array}{l}
B_{\xi}(x,y)\le d(x,y)\\
B_{\xi}(x,y)=B_{g\xi}(gx,gy),\; \forall g\in \isom(\U)
\end{array}
\end{equation}

\paragraph{\textbf{Horofunctions in convergence boundary}} 
We now  define Buseman quasi-cocycles at a $[\cdot]$--class $[\xi]$. Namely, given $\xi\in \mathcal C^{\mathrm{hor}}_\epsilon$, define  a   \textit{Busemann quasi-cocycle} at $\xi$ (resp. $[\xi]$)  as follows
$$\begin{aligned}
%B_{\xi}(x,y)=\limsup_{z_n\to \xi} \left[d(x,z_n)-d(y,z_n)\right]\\
B_{[\xi]}(x,y)=\limsup_{z_n\to [\xi]} \left[d(x,z_n)-d(y,z_n)\right]
\end{aligned}$$ 
where the convergence  $z_n\in \U\to \xi$ (resp. $z_n\to[\xi]$ ) takes place in $\bU=\U\cup\pU$. The equations in (\ref{BusemanCocylesP1}) hold for $B_{[\xi]}(x,y)$, and moreover, for all $\xi\in [\xi]$,
\begin{equation}\label{Horofunctionatclass}
\begin{array}{rl}
 |B_{\xi}(x,y) -B_{[\xi]}(x,y)|&\le \epsilon\\
|B_{[\xi]}(x,y)-B_{[\xi]}(x,z)-B_{[\xi]}(z,y)|&\le 2\epsilon    
\end{array}
\end{equation} 

\paragraph{\textbf{Horoballs at $[\cdot]$--classes}} 
%In the horofunction compactification, for given $x\in \U$ and $\xi\in \hU$,  a horoball $\mathcal{HB}(\xi,x)$ centered at $\xi$   is   the set of points $y\in \U$ satisfying $B_{\xi}(x,y)\le 0$. If    two horofunctions $b_\xi, b_\eta$  have  at most $K$--difference for some $K>0$, it is not clear whether $\mathcal{HB}(\xi,x)$ agrees with $\mathcal{HB}(\eta,x)$ up to a finite Hausdorff distance (depending only on  $K$). The answer is yes for hyperbolic spaces, but  seems to be negative in general.  

We give the following tentative definition of horoballs at a $[\cdot]$--class in a general convergence boundary. Given  a real number $L$,  we define the horoball of (algebraic) depth $L$. We take the convention that the center $\xi$ has depth $-\infty$. 

\begin{defn}\label{HoroballDefn}
Let $\xi\in \mathcal C^{\mathrm{hor}}_\epsilon$ for $\epsilon\ge 0$.  We define a horoball  centered at  $[\xi]$ of depth $L$ as follows
$$
\mathcal{HB}([\xi],L) =\{y\in \U: B_{[\xi]}(y,o)\le L\}
$$
We omit   $L$ if $L=0$ or it does not matter in context.  
%Also,  a horoball  with depth $L$ is as follows
%$$
%\mathcal{HB}([\xi],o, L) =\{y\in \U: B_{[\xi]}(o,y)\ge L\}$$ If the base point $o$ is  understood in context, we write  $\mathcal{HB}([\xi], L)$ for simplicity. 

As the limit supper of continuous functions,   $B_{[\xi]}(x,y)$ is  lower semi-continuous, {so   $\mathcal{HB}([\xi],L)$ is a closed subset}. By abuse of language, we say the level set $\{y\in\U: B_{[\xi]}(y,o)= L\}$ is a \textit{horosphere} or the \textit{boundary} of $\mathcal{HB}([\xi],L)$.
\end{defn}
\begin{lem}\label{HoroballProperty}
For $L_1>L_2$, we have \begin{enumerate}
    \item $\mathcal{HB}([\xi],L_1)\subseteq N_L(\mathcal{HB}([\xi],L_2))$ for $L=L_1-L_2$. 
    \item 
    $\mathcal{HB}([\xi],L_2)$ has distance at least $L-2\epsilon$ to the complement $\U\setminus \mathcal{HB}([\xi],L_1)$.  
\end{enumerate}
%These two definitions are equivalent, upon an appropriate choice of $x$ and $L$.  If $B_\xi(x,o)\ge 0$, we could choose $z\in \gamma$ so that $B_\xi(z,o)=0$ (the other case is treated by switching  $o$ and $x$). Thus, $\mathcal{HB}([\xi],x)=\mathcal{HB}([\xi],L)$ with $L=-d(x,z)$.  
\end{lem}
\begin{proof}
(1). Observe that,  for any $x\in \U$ and $\xi\in \pU$, there exists a geodesic ray $\gamma$ starting at $x$ so that $B_\xi(x,y)=d(x,y)$ for any $y\in\gamma$. ($\gamma$ is called  a gradient line in some literature). By an Ascoli-Arzela argument,  $\gamma$ is obtained as a limiting ray of $[x,z_n]$ when $z_n\to \xi$. Thus, any point $x\in \mathcal{HB}([\xi],L_1)$ has a distance $L$ to some $y\in \mathcal{HB}([\xi],L_2)$. 

(2). If $d(x,y)\le L-2\epsilon$ for some $x\in \mathcal{HB}([\xi],L_2)$ and $y\in \U$, we obtain $B_{[\xi]}(y,o)\le B_{[\xi]}(y,x)+B_{[\xi]}(x,o)+2\epsilon\le L+L_2\le L_1$ so  $y\in \mathcal{HB}([\xi],L_1)$. The conclusion follows. 
\end{proof}
By abuse of language, we also call the set $\mathcal{HB}_L([\xi],x)$ a horoball in any $\eta\in [\xi]$. In other words, all points in the same class $[\xi]$ share a common horoball.

%\begin{lem}\label{HoroballatConicalPts}
%Let $\xi\in \mathcal C^{\mathrm{hor}}_\epsilon$. Then  for any $x\in \U$, $\mathcal{HB}(\xi,x)$ has Hausdorff distance at most $100\epsilon$  to $\mathcal{HB}([\xi],x)$.    
%\end{lem}
%\begin{proof}
%Let $y\in \U$ so that $B_\xi(x,y)=0$. We need prove that for $\eta\in [\xi]$,  $d(y,\mathcal{HB}(\eta,x))\le 100\epsilon$.    
%\end{proof}

Let $ \mathcal M^+(\bU)$ be the set of finite positive Radon measures on $\bU:=\pU\cup\U$, on which  $\Gamma$ acts  by push-forward: for any Borel set $A$, $$g_\star\mu(A)=\mu(g^{-1}A).$$ 
\begin{defn}\label{ConformalDensityDefn}

Let $\omega \in [0, \infty[$. A  map 
\begin{align*}
    \mu:\quad \U \quad \longrightarrow &\quad \mathcal M^+(\bU)\\
    x \quad \longmapsto &\quad \mu_x
\end{align*} is a
\textit{$\omega$--dimensional, $\Gamma$--quasi-equivariant, quasi-conformal density}  if  for any $ g\in \Gamma$ and $x,y\in \U$,
\begin{align}\label{almostInv}
\mu_{gx}-\mathrm{a.e.}\; \xi\in \pU: & \quad
\frac{dg_\star\mu_{x}}{d\mu_{gx}}(\xi) \in [\frac{1}{\lambda}, \lambda],\\
\label{confDeriv}
\mu_y-\mathrm{a.e.}\;
\xi\in \pU: & \quad
\frac{1}{\lambda} e^{-\omega B_\xi (x, y)}  \le  \frac{d\mu_{x}}{d\mu_{y}}(\xi) \le \lambda e^{-\omega B_\xi (x, y)}
\end{align}
for a  universal constant $\lambda\ge 1$.  We normalize $\mu_o$ to be a probability: its mass $\|\mu_o\|=\mu_o(\bU)=1$.

If   $\{\mu_x:x\in\U\}$ is supported on non-pinched boundary points (\emph{i.e.}: $\mu_o(\mathcal C)=1$), we say it is a \textit{non-trivial} quasi-conformal density.
\end{defn}
\begin{rem}
A {non-trivial} quasi-conformal density is much weaker than saying that $\mu_o$ is supported on the conical points. Take for  instance the horofunction boundary $\pU$, where $\mathcal C=\pU$ is the whole boundary. Moreover, all Patterson-Sullivan measures are {non-trivial} quasi-conformal density in Examples \ref{ConformalDensityExamples}.

If $\lambda=1$   for (\ref{almostInv}), the map $\mu: \U \to \mathcal M^+(\bU)$ is $\Gamma$--equivariant: that is, 
$\mu_{gx} = g_{\star}\mu_x$ (equivalently, $\mu_{gx}(gA)=\mu_x(A)$).
If in both (\ref{almostInv}) and (\ref{confDeriv}), $\lambda=1$, we call $\mu$ a conformal density.     
\end{rem}
\paragraph{\textbf{Patterson's construction of quasi-conformal density}}
Fix a basepoint $o \in \U$. Consider the orbital points in   the ball of radius $R>0$:
\begin{equation}\label{BallEQ}
N_\Gamma(o, n)=\{v\in \Gamma o: d(o, v)\le n\}
\end{equation}   
The \textit{critical exponent} $\e A$ for a subset $A \subseteq \Gamma$   is independent of the choice of $o \in \U$:
\begin{equation}\label{criticalexpo}
\omega_A = \limsup\limits_{R \to \infty} \frac{\log \sharp(N_\Gamma(o, R)\cap A o)}{R},
\end{equation}
which is intimately related  to the (partial) Poincar\'e series 
\begin{equation}\label{PoincareEQ}
s\ge 0, x,y\in\U, \quad \p_A(s,x, y) = \sum\limits_{g \in A} e^{-sd(x, gy)}
\end{equation}
as $\p_{A}(s,x,y)$ diverges for $s<\e A$ and converges for $s>\e A$. The   action   $\Gamma\act \U$  is called of \textit{divergence type} (resp.
\textit{convergence type})   if $\p_{\Gamma}(s,x,y)$ is divergent (resp. convergent) at
$s=\e \Gamma$.  

Fix $\Delta\ge 1$. The family of  the annulus-like sets of radius $n$ centered at $o\in \U$
\begin{equation}\label{AnnulusEQ}
A_\Gamma(o, n, \Delta)=\{v\in \Gamma o: |d(o, v)-n|\le\Delta\}
\end{equation} 
covers $Go$ with  multiplicity at most $2\Delta$. It is useful to keep in mind that for $s> \e A, x,y\in\U,$
\begin{equation}\label{AnnulusPoincareEQ}
 \quad \p_A(s,o, o) \asymp_\Delta \sum\limits_{g \in A\cap A_\Gamma(o, n, \Delta)} e^{-sd(o, go)}
\end{equation}

Fix $y\in \U$. We start by constructing a family of measures $\{\mu_x^{s,y}\}_{x \in \U}$ supported on $\Gamma y$ for any given $s >\e \Gamma$. Assume that $\p_\Gamma(s, x,y)$ is divergent at $s=\e \Gamma$. Set
\begin{equation}\label{PattersonEQ}
\mu_{x}^{s, y} = \frac{1}{\p_\Gamma(s, o, y)} \sum\limits_{g \in \Gamma} e^{-sd(x, gy)} \cdot \dirac{gy},
\end{equation}
where $s >\e \Gamma$ and $x \in \U$. Note that $\mu^{s, y}_o$ is a probability measure
 supported on $\Gamma y$. 
If $\p_\Gamma(s, x,y)$ is convergent at $s=\e \Gamma$, the Poincar\'e series in (\ref{PattersonEQ}) needs to be replaced by a modified series as in \cite{Patt}.

Choose $s_i \to \e \Gamma$ such that for each $x\in \U$, $\mu_x^{s_i, y}$ is a convergent sequence in
$\mathcal M(\pG)$.  The family of limit
measures $\mu_x^y = \lim \mu_x^{s_i,y}$ ($x\in\U$) are called \textit{Patterson-Sullivan measures}. 

By construction, Patterson-Sullivan measures are by no means unique a priori, depending on the choice of $y\in \U$ and  $s_i\to \e \Gamma$. Note that $\mu_o^y(\pG) = 1$ is a normalized condition, where $o\in \U$ is a priori chosen basepoint. In what follows, we usually write $\mu_x=\mu_x^o$ for $x\in \U$ (\emph{i.e.}: $y=o$).

There are other sources of conformal densities, not necessarily coming from Patterson's construction.  Here we describe in some details the construction of conformal densities on Thurston boundary. See \cite{AFT07} for other examples constructed on ends of hyperbolic 3-manifolds. 
\begin{example}\label{ThurstonMeasure}
Consider the Teichm\"{u}ller space $\T_g$ of a closed oriented surface $\Sigma_g$ ($g\ge 2$). The space $\mf$ of measured foliations on $\Sigma_g$, which is homeomorphic to $\mathbb R^{6g-6}$, admits a $\mathrm{Mod}(\Sigma_g)$--invariant ergodic measure,  $\mu_{\mathrm{Th}}$, called Thurston measure by the work of Masur and Veech. Positive reals $\mathbb R_{>0}$ act on $\mf$ by scaling, and let   $$\pi:\mf\longmapsto \pmf=\mf/\mathbb R_{>0}$$  
be the natural quotient map.
 
$\pmf$ is also called the Thurston boundary, and it gives  a convergence boundary for $\T_g$  with equipped with the Teichm\"{u}ller metric. Here we take the Kaimanovich-Masur partition $[\cdot]$, whose precise description is not relevant here, but the set $\ue$ of uniquely ergodic points in $\pmf$ is  partitioned into singletons. It is a well-known fact that $\mu_{\mathrm{Th}}$ is supported on $\pi^{-1}(\ue)$.  
 
The $\mu_{\mathrm{Th}}$ induces a $(6g-6)$--dimensional, $\mathrm{Mod}(\Sigma_g)$--equivariant, conformal density on $\pmf$ as follows. Given $x\in \T_g$, consider the extremal length function $$\mathrm{Ext}_x: \mf\longmapsto \mathbb R_{\ge 0}$$ which is square homogeneous, $\mathrm{Ext}_x(t\xi)=t^2\mathrm{Ext}_x(\xi)$. Take the  ball-like set $B_\mathrm{Ext}(x)=\{\xi\in \mf: \mathrm{Ext}_x(\xi)\le 1\}.$ For any $A\subseteq \pmf$, set $\mu_x(A)=\mu_{\mathrm{Th}}(B_\mathrm{Ext}(x)\cap \pi^{-1}{A})$. Direct computation gives
$$\mu_y-\mathrm{a.e. }\; \xi\in \pmf:\;\frac{d\mu_x}{d\mu_y}(\xi)=\left[\frac{\sqrt{\mathrm{Ext}_x(\xi)}}{\sqrt{\mathrm{Ext}_y(\xi)}}\right]^{6g-6}
$$
The family $\{\mu_x\}$ forms a conformal density of dimension $6g-6$: the term in the right-hand bracket coincides with $\mathrm{e}^{-B_\xi(x,y)}$ for  $\xi\in \ue$, and $\mu_x$ charges the full measure to $\ue$.
Note that this family is the same as the conformal density obtained from the action $\mathrm{Mod}(\Sigma_g)\act \T_g$ through the Patterson construction above (see \cite{YANG22}). 
\end{example}

\begin{thm}\label{ConformalDensityExists}
Suppose that $G$ acts properly on a proper geodesic space  $\U$ compactified with horofunction boundary $\hU$. Then the family $\{\mu_x:x\in\U\}$  of Patterson-Sullivan measures is  a $\e G$-dimensional $G$-equivariant conformal density supported on $[\Lambda Go]$.     
\end{thm}

In the sequel, we write PS-measures as shorthand for
Patterson-Sullivan measures.
 
Let $\{\mu_x:x\in \U \}$ be a \emph{nontrivial} $\omega$--dimensional $\Gamma$--equivariant quasiconformal density on a convergence compactification $\U\cup\pU$: $\mu_o$ is supported on the set $\mathcal C$ of non-pinched points in Definition \ref{ConvBdryDefn}(C). 
 
\subsection{Shadow Principle}

Let  $F\subseteq \Gamma$ be a set of three (mutually) independent contracting elements  $f_i$ ($i=1,2,3$), which form a $C$-contracting system  
\begin{equation}\label{SystemFDef}
\f =\{g\cdot \ax(f_i):   g\in \Gamma \}
\end{equation}
where  the axis $\ax(f_i)$  in (\ref{axisdefn}) is $C$-contracting with $C$  depending on the choice of the basepoint $o\in \U$. We may often assume $d(o,fo)$ is large as possible, by taking sufficiently high power of $f\in F$. The contracting  constant $C$ is not effected.  

First of all, define the usual cone and shadow:  
$$\Omega_{x}(y, r) := \{z\in \U: \exists [x,z]\cap B(y,r)\ne\emptyset\}$$
and $\Pi_{x}(y, r) \subseteq \pU$ be the topological closure  in $\pU$ of $\Omega_{x}(y, r)$.

The partial shadows $\Pi_o^F(go, r)$ and cones $\Omega_o^F(go, r)$   defined in the following depend on the choice of a contracting system $\f$  in (\ref{SystemFDef}).

\begin{defn}[Partial cone and shadow]\label{ShadowDef}
For $x\in \U, y\in \Gamma o$, the \textit{$(r, F)$--cone} $\Omega_{x}^F(y, r)$ is the set of elements $z\in \U$ such that $y$ is a $(r, F)$--barrier for some geodesic $[x, z]$.  %the \textit{$(r, F)$--cone} $\Omega_{x}^F(y, r)$ is the set of elements $g\in G$ such that $y$ is a $(r, F)$--barrier for $[x, go]$.

The \textit{$(r, F)$--shadow} $\Pi_{x}^F(y, r) \subseteq \pU$ is the topological closure  in $\pU$ of the cone $\Omega_{x}^F(y, r)$.
\end{defn}

The follow terminology is from Roblin \cite{Roblin2}.
\begin{defn}\label{ShadowPrincipleDef}   
We say that $\{\mu_x:x\in \U \}$ satisfies \textit{Shadow Principle} over a subset $Z\subseteq \U$  if there is  some large constant $r_0>0$ so that the following holds 
\begin{align*}
\|\mu_y\| \mathrm{e}^{-\omega \cdot  d(x, y)}  \prec    \mu_x(\Pi_x(y,r))   \prec_r  \|\mu_y\|   \mathrm{e}^{-\omega \cdot  d(x, y)}   
\end{align*}
for any $x,y\in Z$ and $r>r_0$, where the implied constant depends on $Z$.
\end{defn}
  
The most fundamental example is provided by the Sullivan's Shadow Lemma, where $Z$ could be taken to be any $\Gamma$--cocompact subset (by $\Gamma$--equivariance, $\|\mu_x\|$ thus remains uniformly bounded over $x\in Z$). Here are some other examples.
\begin{examples}
\begin{enumerate}
    \item 
    Roblin realized that $Z$ could be enlarged to $Z:=N(\Gamma)o$, provided that the normalizer $N(\Gamma)$ is a discrete subgroup. 
    
    \item 
    Let $\U$ be a rank-1 symmetric space. If $\omega$ is the Hausdorff dimension of the visual boundary $\pU$, there exists  a unique (up to scaling) $\omega$--dimensional $\Gamma$--equivariant conformal density $\{\mu_x:x\in \U \}$  on $\pU$.  It satisfies the Shadow Principle over the whole space $Z=\U$, where $\|\mu_x\|=1$ for any $x\in \U$. 
    \item 
    We shall establish the Shadow Principle for confined subgroups in \textsection\ref{secshadow}.
    \item 
    We conjecture that the unique $(6g-6)$--dimensional conformal density on $\pmf$ also satisfies the Shadow Principle over $Z= \T_g$. See the relevance to Question \ref{questHorlimitset} as explained in Remark \ref{ExtensionMCG}.
\end{enumerate}    
\end{examples}

We now recall the following shadow lemma on the convergence boundary.
\begin{lem}[{\cite[Lemma 6.3]{YANG22}}]\label{ShadowLem}
Let $\{\mu_x\}_{x\in \U}$ be a \emph{nontrivial} $\omega$--dimensional $G$--equivariant quasi-conformal density  for some $\omega>0$ (\emph{i.e.} supported on the set $\mathcal C\subseteq \pU$ of non-pinched points). Then there exist $r_0,  L_0 > 0$ with the following property. 

Assume that $d(o,fo)>L_0$ for each $f\in F$.  For given $r \ge  r_0$, there exist $C_0=C_0(F),C_1=C_1(F, r)$ such that  
$$
\begin{array}{rl}
   C_0 \mathrm{e}^{-\omega \cdot  d(o, go)}  \le   \mu_o(\Pi_o^F(go,r))  \le \mu_o(\Pi_o(go,r))   \le C_1     \mathrm{e}^{-\omega \cdot  d(o, go)} 
\end{array}
$$
for any $go\in Go$.
\end{lem}
\begin{rem}
We emphasize that the $\omega$--dimensional $\Gamma$--conformal density exists only for $\omega\ge \e \Gamma$ (\cite[Prop. 6.6]{YANG22}). Even if the action $\Gamma \act \U$ is of divergence action (\emph{i.e.} a non-uniform lattice in $\isom(\mathbb H^n)$), a $\omega$--dimensional conformal density may exist for $\omega>\e \Gamma$. See \cite{AFT07} for relevant discussions.
\end{rem}

In what follows, when speaking of $\Pi_o^F(go,r)$, we assume $r, F$ satisfy the Shadow Lemma \ref{ShadowLem}. 
\subsection{Conical points}
We now give  the    definition  of a conical point. Recall that $\mathcal C\subseteq \pU$ is the  set of non-pinched boundary points in Definition \ref{ConvBdryDefn}.  Roughly speaking, conical points are non-pinched and shadowed by infinitely many contracting segments with fixed parameter.

%\begin{defn}[Conical points]\label{ConicalDef1}
%A point $\xi \in \mathcal C$ is called \textit{$L$--conical point} for $L>0$ relative to $\f$ if  for some $x\in \U$, there exists a sequence of $X_n\in \f$ such that $d_{X_n}(x, \xi)\ge L$. Denote by $\Lambda_L^\f(Go)$ the set of all such points.
%\end{defn}

\begin{defn}\label{ConicalDef2}
A point $\xi \in \mathcal C$ is called \textit{$(r, F)$--conical point}   if for some $x\in \Gamma o$, the point $\xi$ lies in infinitely many $(r, F)$--shadows $\Pi_x^{F}(y_n, r)$ for $y_n\in \Gamma o$.  We denote by  $\cG$ the set of $(r, F)$--conical points.

We also denote by $\ccG$ the set of \textit{conical points} $\xi\in \mathcal C$ for which there exists $g_no\in \Gamma o$ satisfying $\xi\in \Pi_o(g_no, r)$ for some large $r>0$. By definition, $\cG\subseteq \ccG$.
%If $\sup_{n\ge 1} \{d(y_n, y_{n+1})\}<\infty$, then   $\xi$ is called  \textit{uniformly conical point}. 
\end{defn}

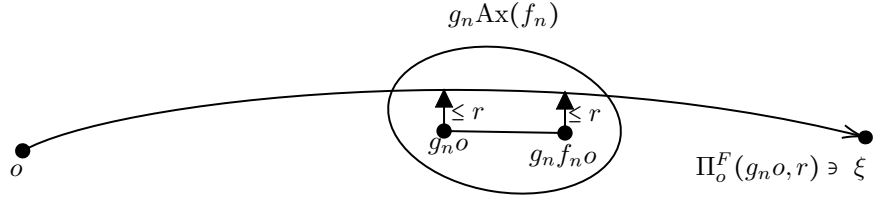
\begin{figure}
    \centering

\tikzset{every picture/.style={line width=0.75pt}} %set default line width to 0.75pt        

\begin{tikzpicture}[x=0.75pt,y=0.75pt,yscale=-1,xscale=1]

%Curve Lines [id:da4571525086961632] 
\draw    (100,103) .. controls (144.5,78) and (352.5,52) .. (524.5,96) ;
\draw [shift={(524.5,96)}, rotate = 194.35] [color={rgb, 255:red, 0; green, 0; blue, 0 }  ][line width=0.75]    (10.93,-3.29) .. controls (6.95,-1.4) and (3.31,-0.3) .. (0,0) .. controls (3.31,0.3) and (6.95,1.4) .. (10.93,3.29)   ;
\draw [shift={(100,103)}, rotate = 330.67] [color={rgb, 255:red, 0; green, 0; blue, 0 }  ][fill={rgb, 255:red, 0; green, 0; blue, 0 }  ][line width=0.75]      (0, 0) circle [x radius= 3.35, y radius= 3.35]   ;
%Shape: Ellipse [id:dp6145654361966988] 
\draw   (284.6,78.45) .. controls (287.68,58.59) and (316.32,46.54) .. (348.57,51.54) .. controls (380.83,56.54) and (404.48,76.69) .. (401.4,96.55) .. controls (398.32,116.41) and (369.68,128.46) .. (337.43,123.46) .. controls (305.17,118.46) and (281.52,98.31) .. (284.6,78.45) -- cycle ;
%Straight Lines [id:da8300165501090848] 
\draw    (312.5,93) -- (373.5,94) ;
\draw [shift={(373.5,94)}, rotate = 0.94] [color={rgb, 255:red, 0; green, 0; blue, 0 }  ][fill={rgb, 255:red, 0; green, 0; blue, 0 }  ][line width=0.75]      (0, 0) circle [x radius= 3.35, y radius= 3.35]   ;
\draw [shift={(312.5,93)}, rotate = 0.94] [color={rgb, 255:red, 0; green, 0; blue, 0 }  ][fill={rgb, 255:red, 0; green, 0; blue, 0 }  ][line width=0.75]      (0, 0) circle [x radius= 3.35, y radius= 3.35]   ;
%Straight Lines [id:da9963947687245229] 
\draw    (312.5,93) -- (312.5,75) ;
\draw [shift={(312.5,72)}, rotate = 90] [fill={rgb, 255:red, 0; green, 0; blue, 0 }  ][line width=0.08]  [draw opacity=0] (8.93,-4.29) -- (0,0) -- (8.93,4.29) -- cycle    ;
%Straight Lines [id:da47310964430537683] 
\draw    (373.5,94) -- (373.5,76) ;
\draw [shift={(373.5,73)}, rotate = 90] [fill={rgb, 255:red, 0; green, 0; blue, 0 }  ][line width=0.08]  [draw opacity=0] (8.93,-4.29) -- (0,0) -- (8.93,4.29) -- cycle    ;
%Shape: Free Drawing [id:dp21311757957990007] 
\draw  [line width=5.25] [line join = round][line cap = round] (524.8,96.22) .. controls (524.8,96.22) and (524.8,96.22) .. (524.8,96.22) ;

% Text Node
\draw (303.28,94.91) node [anchor=north west][inner sep=0.75pt]  [rotate=-1.93]  {$g_{n} o$};
% Text Node
\draw (354.29,97.42) node [anchor=north west][inner sep=0.75pt]  [rotate=-1.93]  {$g_{n} f_{n} o$};
% Text Node
\draw (313,26.4) node [anchor=north west][inner sep=0.75pt]    {$g_{n}\mathrm{Ax}( f_{n})$};
% Text Node
\draw (92,108.4) node [anchor=north west][inner sep=0.75pt]    {$o$};
% Text Node
\draw (438,102.4) node [anchor=north west][inner sep=0.75pt]    {$\Pi _{o}^{F}( g_{n} o,r) \ni \ \xi $};
% Text Node
\draw (314.5,78.4) node [anchor=north west][inner sep=0.75pt]    {$\leq r$};
% Text Node
\draw (373.5,79.4) node [anchor=north west][inner sep=0.75pt]    {$\leq r$};

\end{tikzpicture}
    \caption{Conical points}
    \label{fig:conicalpts}
\end{figure}
The following is the Borel-Cantelli Lemma stated in terms of conical points. 
\begin{lem}\label{ConvLimitNull}
Let $\mu_{0}$ be a probability measure on $\partial X$ that satisfies the shadow lemma as in Lemma \ref{ShadowLem}, and let $A\subseteq \Gamma$ be a subset so that the  partial Poincar\'e series $\p_A(\e \Gamma,x, y)<\infty$. Then the following set 
$$
\Lambda := \limsup_{g\in A} \Pi_o^F(go, r)
$$
is a $\mu_o$--null set, for any $r\gg 0$.
\end{lem}
\begin{proof}
By definition, we can write $\Lambda =\cap_{n\ge 1}\Lambda_n$ where
$$
\Lambda_n =\bigcup_{g\in A: d(o,go)\ge n} \Pi_o^F(go, r)
$$
If $\mu_1(\Lambda)>0$ then $\mu_1(\Lambda_n)>\mu_1(\Lambda)/2$ for all $n\gg 0$. On the other hand, the shadow lemma implies 
$$
\mu_o(\Lambda_n)\le \sum_{g\in A: d(o,go)\ge n} \mu_o(\Pi_o^F(go, r))\le \sum_{g\in A: d(o,go)\ge n} \mathrm{e}^{-\e \Gamma d(o,go)}
$$
so the tails of the series $\p_A(\e \Gamma,x, y)$ has a uniform positive lower bound, contradicting to the assumption. The proof is complete.
\end{proof}

We list the following useful properties from \cite[Lemma 4.6, Theorem 1.10, Lemma 5.5]{YANG22} about conical points.
\begin{lem}\label{ConicalPointsLem}
The following holds for any $\xi\in \cG$:
    \begin{enumerate}
        \item 
        $\xi$ is {visual}:   for any basepoint $o\in \U$ there exists a geodesic ray starting at $o$ ending at $[\xi]$. 
\item 
If $G\act \U$ is of divergence type, $\mu_o$ charges full measure on $[\cG]$ and every $[\xi]$--class is $\mu_o$--null.
        \item 
        Let $y_n, z_n\in \U$ be two sequences tending to $[\xi]$ in $\bU=\U\cup\pU$. Then for any $x\in \U$,
         $$\limsup_{n \to \infty} |B_{y_n}(x)-B_{z_n}(x)|\le 20C.$$
         In particular, $[\cG]$ is a subset of $\mathcal C^{\mathrm{hor}}_{20C}$ (defined in \ref{AssumpE}).
    \end{enumerate}
\end{lem}

It is well known that in hyperbolic spaces, a horoball  centered at   a point  in the Gromov boundary has the unique limit point at the center. This could be proved for uniquely ergodic points   in Thurston boundary of Teichm\"{u}ller spaces. Analogous to these facts, we have the following.

\begin{lem}\label{HoroballUniqueLimit}
Let $\xi\in \cG$. Consider a horoball $\mathcal {HB}([\xi])$ centered at $[\xi]$ defined as Definition \ref{HoroballDefn}. Then any escaping sequence $x_n\in \mathcal {HB}([\xi])$ has the accumulation points in $[\xi]$.
\end{lem}
\begin{proof} 
According to definition, there exists a sequence of elements $g_n\in\Gamma$ so that $\xi\in \Pi_r^F(g_no)$. By \cite[Lemma 4.4]{YANG22}, there exists a geodesic ray $\gamma$ starting from $o$ terminating at $[\xi]$ so that $$\|\gamma\cap N_r(g_n\ax(f_n))\|\ge L$$
where $f_n\in F$ and $L:=\min\{d(o,fo):f\in F\}-2r$.

Fix any $X:=g_n\ax(f_n)$, and consider the exit point $y$ of $\gamma$ at $N_C(X)$.
Let $z_m\in \gamma\to [\xi]$ as $m\to\infty$. Given $x_n\in \mathcal {HB}([\xi])$,  the following holds  for $m\gg 0$ by Lemma \ref{ConicalPointsLem}(3): 
\begin{equation}\label{HoroballEQ1}
B_{z_m}(x_n)=d(x_n, z_m)-d(o,z_m)\le 21C.    
\end{equation} 
 We claim that the projection $y_n\in X$ of all but finitely many $x_n$  to $X$  lies in a $10C$--neighborhood of $y$. Indeed, if not, $d(y_n,y)>10C$ holds for infinitely many $n$.   The contracting property of $X$ implies that $[z_m,x_n]$ intersects $N_C(X)$ for all $n, m\gg 0$, so $d(y,[x_n,x_m])\le 2C$.  Combined with (\ref{HoroballEQ1}), we obtain $d(y,x_n)\le d(y,o)+100C$. This is a contradiction, as $x_n$ is a unbounded sequence in a fixed ball  around $y$.

If $L>100C$ is assumed,  the above claim implies the corresponding projections of $o$ and $x_n$ to $X$ have distance at least $90C$, so the contracting property shows $[o,x_n]\cap N_C(X)\ne\emptyset$. That is, $x_n\in \Omega_o(N_C(X))$ for infinitely many $x_n$. As $X=g_n\ax(f_n)$ is chosen arbitrarily along $\gamma$,   the assumption (B) in Definition \ref{ConvBdryDefn} implies that $x_n$ and $g_no$ have the same limit, so $x_n\to [\xi]$.   
\end{proof}
%\begin{rem}

%\end{rem}

\paragraph{\textbf{Uniqueness of quasi-conformal measures}} Suppose that $\Gamma\act \U$  is of divergence type, where $\U$ is compactified by the horofunction boundary $\hU$. The quasi-conformal measures on  $\hU$ are in general neither ergodic nor unique.  To remedy this, we shall push the measures to the reduced horofunction boundary $[\hU]$, modding out the finite difference relation. However, $[\hU]$ is a pathological topological object (\emph{i.e.} may be non-Hausdorff). The remedy is to consider the reduced Myrberg limit set, which provides better topological property. Define
\begin{equation}\label{MyrbergDefn}
\mG=\bigcap \cG    
\end{equation} 
where the intersection is taken over  all possible choice of three independent contracting elements $F$ in $\Gamma$. By \cite[Lemma 2.27 \& Lemma 5.6]{YANG22}, the quotient map 
\begin{align*}
\mG\quad \longrightarrow&\quad [\mG]\\
\xi\quad \longmapsto &\quad [\xi]
\end{align*} is a closed map with compact fibers, thus $[\mG]$ is a completely metrizable and second countable topological space.

Let $\{\mu_x\}_{x\in \U}$ be a $\e \Gamma$--dimensional $\Gamma$--equivariant conformal density    on   $\hU$.
This is pushed forward to  a $\e \Gamma$--dimensional $\Gamma$--quasi-equivariant quasi-conformal density,  denoted by $\{[\mu_x]\}_{x\in \U}$, on $[\mG]$.

\begin{lem}[{\cite[Lemma 8.4]{YANG22}}]\label{Unique}
Let $\{[\mu_x]\}_{x\in \U}$ and $\{[\mu_x']\}_{x\in \U}$ be two $\e \Gamma$--dimensional $\Gamma$--quasi-equivariant quasi-conformal densities    on the quotient  $[\mG]$. Then the
Radon-Nikodym derivative $d[\mu_o]/d[\mu_o']$ is bounded from above and below by a constant depending only on    $\lambda$ in Definition \ref{ConformalDensityDefn}.

\end{lem}

\subsection{Regularly contracting limit points}\label{SSecRegConPts}
In this subsection, we first introduce a class of statistically convex-cocompact actions in \cite{YANG10}  as a generalization of convex-cocompact actions, encompassing Examples in \ref{SCCexamples}. This notion was independently introduced by Schapira-Tapie \cite{ST21} as \textit{strongly positively recurrent} manifold in a dynamical context.

Given constants $0\leq M_1\leq M_2$, let $\mathcal{O}_{M_1,M_2}$ be the set of elements $g\in \Gamma$ such that there exists some geodesic $\gamma$ between $N_{M_2}(o)$ and $N_{M_2}(go)$ with the property that the interior of $\gamma$ lies outside $N_{M_1}(\Gamma o)$.

\begin{defn}[SCC Action]\label{SCCDefn}
If there exist positive constants $M_1,M_2>0$ such that $\e {\mathcal{O}_{M_1,M_2}}<\e\Gamma<\infty$, then the proper action of $\Gamma$ on $\U$ is called \textit{statistically convex-cocompact} (SCC).
\end{defn}
\begin{rem}
The motivation for defining the set $\mathcal{O}_{M_1,M_2}$ comes from the action of  the fundamental group of a finite volume negatively curved Hadamard manifold on its universal cover. In that situation, it is rather easy to see that for appropriate constants $M_1, M_2>0$, the set $\mathcal{O}_{M_1,M_2}$ coincides with  the union of the orbits of cusp subgroups up to a  finite Hausdorff distance. The assumption in SCC actions is called the \textit{parabolic gap condition} by Dal'bo, Otal and Peign\'{e} in \cite{DOP}. The growth rate $\e {\mathcal{O}_{M_1,M_2}}$ is called \textit{complementary growth exponent} in \cite{ACTao} and \textit{entropy at infinity} in \cite{ST21}. 
\end{rem}

SCC actions have purely exponential growth, and thus are of divergence type. Therefore, we have the unique  conformal density on a convergence boundary by Lemma \ref{Unique}. 
\begin{lem}
Suppose that $\Gamma\act \U$ is a non-elementary SCC action with contracting elements. Then $\Gamma$ has purely exponential growth: for any $n\ge 1$, $\sharp N_\Gamma(o,n)\asymp \mathrm{e}^{\e \Gamma n}$.    
\end{lem}

For the remainder of this subsection, assume that $\Gamma\act \U$ is an SCC action with a contracting element. Fix a set $F$ of contracting elements in $\Gamma$. 

We now recall another specific class of conical limit points, introduced in \cite{QYANG}, called \textit{regularly contracting} limit points, which have been shown to be generic for PS measures there. %It is conjectured to hold only if the action is assumed to have purely exponential growth. 
This notion is modeled on the following purely metric notion of regularly contracting geodesics introduced earlier in \cite{GQR22}.   For any ratio $\theta\in [0,1]$, a \textit{$\theta$--interval} of a geodesic segment $\gamma$ means a connected subsegment of $\gamma$ with length $\theta \ell(\gamma)$. 

\begin{defn}
Fix constants $\theta, r, C, L > 0$.
We say that a geodesic $\gamma$  is \textit{$(r, C, L)$--contracting at $\theta$–frequency}     if  every $\theta$--interval of $\gamma$ contains a segment of length $L$ that is $r$--close to a  $C$–contracting geodesic. 

A geodesic ray $\gamma$  is \textit{$(r, C, L)$--contracting at $\theta$–frequency} if any sufficiently long initial segment of $\gamma$ (\emph{i.e.} $\gamma[0,t]$ for  $t\gg 0$) is $(r, C, L)$--contracting at $\theta$–frequency.  

Furthermore, $\gamma$ is \textit{frequently $(r, C, L)$--contracting} if it is $(r, C, L)$--contracting at $\theta$–frequency for any $\theta\in (0,1)$.
\end{defn}
%\begin{rem}
%Assume the sequence of the $C$--contracting geodesic segment $p_n$ is escaping in the above definition. According to  Definition \ref{ConvBdryDefn}(B),    $\gamma$ accumulates to a unique $[\cdot]$--class denoted by $[\gamma^+]$ in any convergence boundary $\pU$, and the sequence of $p_n$ does so: any sequence $x_n\in p_n$ has accumulation points in $[\gamma^+]$.   
%\end{rem}

The above notion is not used in this paper, but motivates  the following analogous notion   involving a proper action $\Gamma\act \U$.

\begin{defn}
Fix $\theta\in (0,1]$ and $r>0$ and $f\in \Gamma$. We say that a geodesic $\gamma$ contains \textit{$(r, f)$--barriers at $\theta$--frequency} if 
for  every $\theta$--segment of $\gamma$ has $(r, f)$--barriers. 

An element $g\in G$ has \textit{$(r, f)$--barriers at $\theta$--frequency} if there exists a geodesic $\gamma$ between $B(o, M)$ and $B(go, M)$ such that $\gamma$ has  $(r, f)$--barriers at $\theta$--frequency.
\end{defn}

%Fix any $\theta\in (0,1], r>M$ and $f\in \Gamma$ in the next  lemma. 

Let $\mathcal W(\theta,r,F)$ denote the set of elements in $\Gamma$ having \textbf{no} $(r, f)$--barriers at $\theta$--frequency for some $f\in F$. That is, an element $g$ of $\Gamma$ belongs to $\mathcal{W}(\theta, r, F)$ if and only if any geodesic between $B(o, M)$ and $B(go, M)$ contains a  $\theta$--interval that  has no  $(r, f)$--barriers.

\begin{lem}\label{frequentbarriers}\cite[Lemma 4.7]{QYANG}
Fix $\theta\in (0,1], r>M$. Then  $\mathcal W(\theta,r,F)$ is growth tight.
\end{lem}

Set $V:=\Gamma\setminus \mathcal W(\theta,r,F)$. Let  $\mathcal{FC}(\theta,r, F)$  denote the set of  limit points $\xi\in \pG $ that is contained in infinitely many shadows at elements in $V$.  More precisely, $\mathcal{FC}({\theta,r, F})$ is the limit superior of the following sequence, as $n\to \infty$,  $$\bigcup\limits_{v \in   A_\Gamma(o, n,\Delta) \cap {V}}  \Pi_{o}^F(v,r)$$ where   $A_\Gamma(o,n,\Delta)$ is defined in (\ref{AnnulusEQ}). Hence,  $$\Lambda_n:=  \bigcup\limits_{k\ge n}\Bigg(\bigcup\limits_{v \in   A_\Gamma(o, k,\Delta)\cap {V} }  \Pi_{o}^F(v,r)\Bigg) \searrow \mathcal{FC}(\theta,r, F) \quad \textrm{as $n \rightarrow \infty$}.$$ 

Recall $F^n=\{f^n: f\in F\}$ for an integer $n\ge 1$. Define the  set of \textit{regularly contracting} points as follows $$\rcG :=\bigcap_{n\in \mathbb N}\bigcap_{\theta\in (0,1]\cap \mathbb Q} \mathcal{FC}(\theta,r, F^n)$$
We could take a further countable intersection over all possible $F$ of three independent contracting elements. 
In general, this is a proper subset of Mryberg set in \cite{YANG22}.

It is proved in \cite[Lemma 4.12]{QYANG} that the set of $(r, C,L)$--regularly contracting rays lies in  $\mathcal{FC}(\theta,r, F)$ for certain $F$. 
The following result thus implies  \cite[Theorem A]{QYANG}, saying that $(r, C,L)$--regularly contracting rays is $\mu_o$--full. 
\begin{prop}\cite{QYANG}\label{FreqContFull}
Assume that $\Gamma\act \U$ is an SCC action with contracting elements. Then the regularly contracting limit set $[\rcG]$ is a $\mu_o$--full subset of $[\cG]$.    
\end{prop}
\begin{proof}
Recall that $[\cG]$ is the limit super of $\{\Pi_o^F(go,r): g\in \Gamma \}$, and it is $\mu_o$--full  by Lemma \ref{ConicalPointsLem}. The set $W=\mathcal W(\theta,r,F)$ is growth tight by Lemma \ref{frequentbarriers}, so the limit super $\Lambda$ of $\{\Pi_o^F(go,r): g\in W\}$ is $\mu_o$--null by Lemma \ref{ConvLimitNull}. Noting   $V=\Gamma\setminus W$ as in the above definition,  $[\rcG]$ is contained in the complementary subset of $\Lambda$ in $[\cG]$. Hence, $\mu_o([\rcG])=1$ is proved.    
\end{proof}
\begin{rem}\label{ConjPEGRem}
If $\U$ is a negatively curved Riemannian manifold with finite BMS measure on geodesic flow, then the PS measure is supported on the frequently contracting limit points. The same is true for certain CAT(0) groups with rank-1  elements and mapping class groups in \cite[Theorems 5.1 \& 7.1]{GQR22}. In these settings, finiteness of the BMS measure is equivalent to having purely exponential growth (\cite{Roblin}).  

In the current coarse setting, we expect that the SCC action assumption could be replaced in Proposition \ref{FreqContFull} with purely exponential growth. %Moreover, a conjecture of Sullivan \cite[p. 436]{Sul} asserts that the ergodicity of geodesic flows (with possibly infinite BMS measure) suffices to obtain the genericity of frequently contracting limit points.
\end{rem}

\subsection{Hopf decomposition}
In this subsection, we consider a (possibly non-discrete) countable subgroup $H<\isom(\U)$. In particular, $H\act \U$ is not necessarily a proper action with discrete orbits. 

Assume $H$  admits a measure class preserving action on $(\pU, m)$.  We say that the action  $H\act (\pU,m)$ is \textit{(infinitely) conservative}, if for any $A\subseteq \pU$ with $m(A)>0$, there are (infinitely many) $h\in H\setminus \{1\}$ so that $hA\cap A\ne\emptyset.$ It is called \textit{dissipative} if there is a  measurable wandering set $A$: $gA\cap A=\emptyset$ for each $1\ne h\in H$. If in addition, $\cup_{h\in H}hA=\pU$, then it is \textit{completely dissipative}. 

Let $\mathbf C\subseteq \pU$ be the union of purely atomic ergodic components, and $\mathbf D=\pU\setminus \mathbf C$. Denote $\mathbf D_{1}$ (resp. $\mathbf D_{>1}$) the subset of $\mathbf D$ with trivial (resp. nontrivial) stabilizers,  $\mathbf D_{\infty}$ the subset of $\mathbf D$ with infinite   stabilizers.  The partition   $\mathbf{Cons}=\mathbf C\cup \mathbf D_{>1}$ and $\mathbf{Diss}=\mathbf D_1$ (mod 0)  gives the \textit{Hopf decomposition} $H\act (\pU, m)$ as the disjoint union of conservative and completely dissipative components. 
If $H\act \U$ is proper on a hyperbolic space, $\mathcal D_\infty$ consists of parabolic points with infinite stabilizer. If $H$ is torsion-free, parabolic points are countable.

By \cite[Theorem 29]{Kai10}, the infinite conservative part $\mathbf{Cons}_\infty=\mathbf C\cup\mathbf D_\infty$ coincides (mod 0) with the following set
$$
\begin{aligned}
\mathbf{Cons}_\infty&=\{\xi\in \pU: \exists t>0 \sharp\{dg_\star m(\xi)/dm(\xi)>t\}=\infty\}\\
&=\{\xi\in \pU: \sum_{g\in H}dg_\star m(\xi)/dm(\xi)=\infty \}  
\end{aligned}
$$
We now consider the $\Gamma$--equivariant $\e\Gamma$--dimensional conformal density $\{\mu_x:x\in \U\}$ on $\pU$. We emphasize that $H$ is not necessarily contained in $\Gamma$, but is assumed to preserve the measure class of $m:=\mu_o$, so that the following holds for $g_\star \mu_o=\mu_{go}$:
$$
\frac{dg_\star \mu_o}{d\mu_o}(\xi) \asymp \mathrm{e}^{-\e\Gamma B_\xi(go,o)}
$$
Recall that $\mu_o$ is supported on $[\pG]$.

We do not assume $H$ to be a discrete group, so $Ho$ may not be a discrete subset in $\U$. We still denote by $\Lambda(Ho)$ the accumulation points of the orbit $H o$ in the convergence boundary $\pU$. It depends on the basepoint $o$, but $[\Lambda(Ho)]$ does not, because of Assumption (B) in Definition \ref{ConvBdryDefn}. 

Let $\mathcal C^{\mathrm{hor}}$ be defined in \ref{AssumpE} as the set of points $\xi\in \mathcal C$ at which Busemann cocycles satisfy some weak form of continuity. 
According to the above discussion, $\mathbf{Cons}_\infty$ is exactly the so-called big horospheric limit point defined as follows.

\begin{defn}\label{HoroLimitPtsDef}
A point $\xi\in [\Lambda(Ho)]\cap \mathcal C^{\mathrm{hor}}$ is called \textit{big horospheric limit point} if there exists $h_n\in H$ such that $B_{[\xi]}(h_no,o)\ge L$ for some $L\in\mathbb R$. If, in addition,  $B_{[\xi]}(h_no,o)\ge L$ holds for any given $L\in\mathbb R$, then $\xi$ is called \textit{small horospheric limit point}.       Denote by $\HG$ (resp. $\hG$) the set of the big (resp. small) horospheric limit points.
\end{defn}
\begin{rem}
By definition, a big or small horospheric limit point is a property of a $[\cdot]$--class.  Thus, $\HG=[\HG]$ is $[\cdot]$--saturated. In terms of horoballs (\ref{HoroballDefn}), $\xi$  is a big (resp. small) horospheric limit point if some (resp. any) horoball centered at $[\xi]$ contains infinitely many $h_no$ with $h_no\to \xi$.  
\end{rem}
 
%\begin{rem}
%As suggested in \cite{Kai10},  the big horospheric limit set is   more naturally than the small one in the sense that  the former is identified with the infinite conservative part. However,  
%\end{rem}

Let $\pU_{>1}$ (resp. $\pU_{1}$) denote the set of boundary points with nontrivial (resp. trivial) stabilizers in $H$.
Summarizing the above discussion, we have the Hopf decomposition for conformal measures.
\begin{lem}\label{HopfdecompLem}
Let $H$ act properly on $\U$ and preserve the measure class of  $\mu_o$. Then  $\HG$ is the infinite conservative part,  $\mathbf{Cons}=\HG\cup \pU_{>1}$ and $\mathbf{Diss}=\pU\setminus \mathbf{Cons}= \pU_{1}\setminus \HG$.    
\end{lem}

%At last we explain an example of conservative action which explains some terminologies in our content. This is however ``exotic" example, which we are not interested in the current paper.
It is  wide  open whether the  big horospheric limit set differs from the small one only in a negligible set (cf. Question \ref{questHorlimitset}). This has been confirmed for Kleinian groups \cite{Sul81} for spherical measures, free subgroups \cite{GKN} for visual measures and  normal subgroups of divergent type actions \cite{FM20} for general conformal measures without atoms. Here is an example of atomic conformal measures in which the large and small horospheric sets differ by a positive measure set. 
\begin{example}\label{requirenoatoms}
Let $\Gamma$ be a non-uniform lattice in $\isom(\mathbb H^m)$ with $\omega_\Gamma=m-1\ge 1$. We can put a conformal measure on the boundary at infinity, supported on the set $\mathbf{P}$ of countably many parabolic fixed points. Indeed, fix 
any $\omega>(m-1)$ and $\mu_o(\xi)=c>0$ for given $\xi\in \mathbf{P}$. The value of $c$ will be adjusted below. The other points $\eta=g\xi$ in $\Gamma \xi$ is then determined by $$\mu_o(\eta)=c\cdot \mathrm{e}^{-\omega B_\xi(g^{-1}o,o)}$$ where $\eta=g\xi$ which does not depend on $g\in \Gamma$ (for $B_\xi(\cdot,\cdot)$ is  invariant under the stabilizer  $\Gamma_\xi$). Fix $g_\eta\in \Gamma$ so that $g_\eta \xi=\eta$. Observe that $\mu_o(\Gamma \xi)$ can realize any value in $(0,1]$, by  adjusting $c$ with the following fact for given $\omega>\e\Gamma $,
$$
\sum_{\eta\in \Gamma\xi} \mathrm{e}^{-\omega B_\eta(o,g_\eta o)} <\infty
$$
which, in turn, follows from purely exponential growth of double cosets of the parabolic subgroup $\Gamma_\xi$ (\emph{e.g.} \cite{HYZ}) $$\sharp \{g\mathcal{HB}(\xi): d(\mathcal{HB}(\xi),g\mathcal{HB}(\xi))\le n\}\asymp \mathrm{e}^{\e\Gamma n}$$
As $\sharp \mathbf{P}/\Gamma$ is finite, we are able to achieve that $\mu_o(\mathbf{P})=1$.

The set of parabolic points is the infinite conservative part. If the number of cusps is at least 2, the action is not ergodic. Moreover,  parabolic points are the big horospheric limit set (mod 0), but  not  small horospheric limit points. So $\mu_o(\HG)=1$ and $\mu_o(\hG)=0$. 
\end{example}

We shall prove that $\mu_o(\HG\setminus \hG)=0$ for subgroups in a hyperbolic group, where $\mu_o$ is the Patterson-Sullivan measure for $\Gamma$.      

At last, we collect the notations of various limit sets studied in this paper: 
\begin{itemize}
    \item 
    $\pG$ is the limit set of $\Gamma o$, while $[\pG]$ is the $[\cdot]$--classes over $\pG$.
    \item $\cG$ is the set of $(r,F)$--conical points, and $\ccG$ denotes the set of usual conical points.
    \item 
    $\Lambda^\mathrm{Hor} (\Gamma o)$ (resp. $\Lambda^\mathrm{hor} (\Gamma o)$) is the big/small horospheric limit set. 
    \item 
    $\mG$ is the Myrberg limit set, and $\rcG$ is the regularly contracting limit set. 
    
\end{itemize}
They are related as
$$\rcG\subseteq \mG\subseteq \cG\subseteq \ccG\subseteq \Lambda^\mathrm{hor} (\Gamma o)\subseteq \Lambda^\mathrm{Hor} (\Gamma o)\subseteq \pG$$
where each inclusion is proper in general.

\part{Hopf decomposition of confined subgroups}
\section{Prelude: proof of selected results in the hyperbolic setting}\label{SecWarmup}

The goal of this section is two-fold: to illustrate some of our key tools via well-known facts in hyperbolic spaces; and present in this setting the proofs for (essentially all) theorems in \textsection\ref{secdiss} and \textsection\ref{seccons} and a key tool used in subsequent sections. This section could be skipped without affecting the remaining ones.  

Throughout this section, we assume that $\U$ is a Gromov hyperbolic space and $H$, $\Gamma$, $G$ are subgroups of $\isom(\U)$ with $H$ and $\Gamma$ being discrete. Let $\{\mu_x:x\in\U\}$ be Patterson-Sullivan measures  on the Gromov boundary $\pU$ constructed from $\Gamma\act \U$. If $\Gamma\act \U$ is of divergence type, then $\mu_o$ is atomless, and charges on the conical limit set (Lemma \ref{ConicalPointsLem}), and Shadow Lemma \ref{ShadowLem} holds. We assume the reader is familiar with these facts (\emph{e.g.} which could be found in \cite{MYJ20}). 

Since $\U$ is Gromov hyperbolic, there exists $\delta>0$ such that any geodesic triangle admits a \textit{$\delta$--center}, that is, a point within a $\delta$--distance to each side. 

\subsection{Preliminary}
The notion of an admissible path (Def. \ref{AdmDef}) is a generalization of the $L$--local quasi-geodesic paths:
$$
p_0q_1p_1\cdots q_np_n$$
where any two consecutive paths give a quasi-geodesic by (\textbf{BP}), of length at least $\ell(p_i)>L$ (\textbf{LL}). It is well known that, in hyperbolic spaces, if $L\gg 0$, a local quasi-geodesic  is a global quasi-geodesic.

Assume that $\Gamma\act \U$ is a proper non-elementary action, so there exist at least three loxodromic elements with pairwise disjoint fixed points. The following result, which is a special case of the Extension Lemma \ref{extend3}, is a consequence of the fact on quasi-geodesics mentioned above.
\begin{lem}\label{extend3inhyp}

There exists a set $F$ of three loxodromic elements of $\Gamma$ and constants $L, c>0$ that depend only on $F$ with the following property.  For any $g,h\in \Gamma$, there exists an element $f \in F$ such that   the path  $$\gamma:=[o, go]\cdot(g[o, fo])\cdot(gf[o,ho])$$ is a $c$--quasigeodesic. 
\end{lem}
This result is well known to experts in the field and, to the best of our knowledge, appeared first in \cite[Lemma 3]{AL}. It was subsequently reproved or implicitly used in many works, which we do not attempt to track down here.     

\subsection{Completely dissipative actions}

In this subsection, we further assume that $\U$ is proper and $\Gamma$ acts on $\U$ cocompactly. The goal is to prove the conclusion of Theorem \ref{HalfgrowthimplyDiss} in this setting.
\begin{lem}\label{CloserInHoroball}
Let $\xi\in \HG$ be a big horospheric point. Then for any $\epsilon>0$, there exist infinitely many orbit points $h_no\in Ho$ such that $d(o,z_n)>d(o, h_no)(1/2-\epsilon)$, where $z_n$ is a $\delta$--center of $\Delta(o,\xi,h_no)$ for some $\delta>0$.    
\end{lem}
\begin{proof}
By the definition of a big horospheric point, there is an infinite sequence of points $h_no \in Ho\to \xi$ such that $h_no$ lies in some horoball $\mathcal{HB}(\xi, L)$ for all $n\ge 1$ and for some $L$ depending on $\xi$. Recall the definition of a horoball:  $\mathcal{HB}(\xi, L)=\{z\in\U: B_{\xi}(z,o)\le L\}$, so we obtain  $$B_z(h_no,o)=d(z,h_no)-d(z,o)\le L+10\delta$$ for $z\in [o,\xi]$ with $d(o,z)$ sufficiently large. By hyperbolicity, if $z_n\in [o,\xi]$ is a $\delta$--center of $\Delta(o,z,h_no)$ for some $\delta>0$, then $d(z,z_n)+d(z_n,h_no)\le 2\delta+d(z,h_no)$.  Combining these inequalities gives $$
\begin{aligned}
d(z,z_n)+d(z_n,h_no)&\le 12\delta+L+d(z,o)\\
&\le 12\delta+L+d(o,z_n)+d(z_n,z)    
\end{aligned}$$ yielding $d(z_n,h_no)\le 12\delta+L+d(o,z_n)$. Thus,  $d(o,h_no)\le d(o,z_n)+d(z_n,h_no)\le 12\delta+L+2d(o,z_n)$. Note however that $L$ depends on $\xi$. Given any $\epsilon>0$, we can still obtain $(1/2-\epsilon)d(o,h_no)\le d(o,z_n)$ after dropping finitely many $h_no$. The proof is complete.
\end{proof}
%\begin{proof}
%By hyperbolicity, up to a finite Hausdorff distance depending on $\delta$, a horoball  $\mathcal{HB}(\xi,o)$ is the increasing union of metric balls  centered at a point $z$ on the geodesic ray $[o,\xi]$ with radius $d(o,z)\to \infty$. As $z_n$ is a $\delta$--center for $\Delta(o,\xi,g_no)$, we see that $z_n$ is a $10\delta$--center for $\Delta(o,z,g_no)$ for $z\in [o,\xi]$ with $d(o,z)\gg 2d(o,z_n)$. The definition of $\delta$--center implies  $g_no$ lies in   a metric ball centered at $z_n$ with radius $d(o,z_n)+D$, where $D$ depends only on $\delta$. This completes the proof.   
%\end{proof}

%If $G\act \U$ is co-compact, then we can choose $g_n\in G$ such that $d(g_no, z_n)\le M$ where $M=diam(\U/G)<\infty$. The implies the following.
%\begin{cor}\label{HorInConical}
%The big horospherical limit set of $H$ is contained in the conical limit set of $G$.    
%\end{cor}

\begin{thm}[Theorem \ref{HalfGrowthDisThm} in hyp. setup]\label{smallgrowthhorhaszeromeasure}
If $\e H<\e \Gamma/2$,   then $\mu_o(\HG)=0$. 
\end{thm}
\begin{proof}
Fix a small $\epsilon>0$, so $\e H<\e \Gamma(1/2-\epsilon)$. Let $\hat Z$ be the set of points $z\in \U$ that satisfy $d(o,z)>d(o,ho)(1/2-\epsilon)$, where $z$ is a $\delta$ center of the triangle $\Delta(o, ho,\xi)$ for some $h\in H$ and some $\xi\in \HG$. 

Choose a sufficiently large constant $r>18\delta$ such that the Shadow Lemma \ref{ShadowLem} applies. If $Z\subseteq \hat Z$ is a maximal $r$--separated net in $\hat  Z$, then by Lemma \ref{CloserInHoroball}, $\HG$ is contained in the limit superior  $\limsup_{z\in Z} \Pi_o(z,2r)$ of shadows. As the action $\Gamma\act \U$ is cocompact, we may assume $Z\subseteq \Gamma o$ upon increasing $r$ again. 

Note that, for each $h \in H$, the corresponding $z$ in the first paragraph is lying on the $\delta$--neighborhood of $[o, ho]$. This neighborhood contains at most $Cd(o, ho)$ points in the $r$--separated set $Z$ where $C>0$ is a universal constant.

By the Shadow Lemma \ref{ShadowLem} for the $\Gamma$--conformal density, we compute 
$$
\begin{aligned}
\sum_{z\in  Z} \mu_o(\Pi_o(z,2r)) & \prec \sum_{h\in H}d(o, ho) \cdot 
 \mathrm{e}^{-\e \Gamma  d(o,ho)(1/2-\epsilon)} \\
&\prec  \sum_{h\in H} \mathrm{e}^{-\omega   d(o,ho)} <\infty 
\end{aligned}
$$
which is finite by definition of $\e H$ and $\e H<\e \Gamma(1/2-\epsilon)$. The Borel-Cantelli lemma implies that $\limsup_{z\in Z} \Pi_o(z,2r))$  is $\mu_o$--null.  Hence, $\mu_o(\HG)=0$, and the theorem is proved. 
\end{proof}

\subsection{Conservative actions for confined subgroups}
This subsection is   an abridged version of Sections \textsection\ref{secconfine} and \textsection\ref{seccons} in hyperbolic setup. In this restricted setting, we will give two proofs that the horospheric limit set of a confined subgroup $H$ has full Patterson-Sullivan measure. One of the proofs will work whenever $H$ has a compact confining set, whereas the other requires a finite confining set. However, the latter proof is more readily adapted to the setting of general actions on metric spaces with contracting elements, which we will consider in Section 6. Therefore, we find it instructive to include both proofs.

We first state a key tool that allows us to use a finite confining set geometrically in hyperbolic spaces.   

\begin{lem}[$\ll$ Lemma \ref{EllipticRadical}]\label{EllipticRadicalinHyp}
Let $P\subseteq \isom(\U)$ be a finite set, so that no element in $P$ fixes pointwise the limit set $\Lambda(\Gamma o)$. Then there exists a finite set $F\subseteq \Gamma$ of loxodromic elements and a constant $D>0$ with the following property. 

For any $g, h\in G$, one can find $f\in F$ so that for any $p\in P$, the word $(g, f, p, f^{-1}, h)$ labels a quasi-geodesic:  
$$
|d(o, gfpf^{-1}ho) - d(o,go)-d(o,ho)|\le D
$$ 
\end{lem}
\begin{proof}
Let $p\in P$. By assumption, as all fixed points of loxodromic elements are dense in $\Lambda(\Gamma o)$, $p$ moves the two fixed points of some $f$. If $q\in P$ moves $g^\pm$ for another $g$, then $p,q$ must move the two fixed points of $f^ng^n$, which tends to $f^+$ and $g^-$ as $n\to\infty$.   As $P$ is a finite set, a common $f$ could be chosen for each $p\in P$: $pf^\pm\cap f^\pm\ne\emptyset$. Equivalently, $p\ax(f)$ has bounded projection with $\ax(f)$ for each $f\in P\setminus 1$. 

With a bit more effort, we produce $F=\{f_1,f_2,f_3\}$ where $p\in P$ moves the two endpoints of each $f\in F$, and the set of axis $\{p\ax(f): p\in P, f\in F\}$ has bounded projection.  This implies that each $g\in \Gamma$ has bounded projection to at least two of $\ax(f_1), \ax(f_2), \ax(f_3)$. Consequently, for any $g,h\in \Gamma$, there exists a common $f\in F$  so that 
$$
\max\{\pi_{\ax(f)}([o,go]),\pi_{\ax(f)}([o,ho])\}\le \tau
$$
Set $L=\min\{d(o,fo): f\in F\}$.
The word $(g, f,p, f^{-1}, h)$ labels an $L$--local $c$--quasi-geodesic path, denoted by $\gamma$, where $c$ depends on $\tau$. This concludes the proof.    
\end{proof}

The following is an immediate consequence of Lemma \ref{EllipticRadicalinHyp} applied to $(g, g^{-1})$, and  $p\in P$ is chosen for $gf$ according to the definition of confined subgroups.

\begin{lem}[$\ll$ Lemma \ref{GoodConfiningSetP}]\label{GoodConfiningSetPinHyp}
Under the assumption of Lemma \ref{EllipticRadicalinHyp},
for any $g\in G$, there exist $f\in F$ and $p\in P$ such that $gfpf^{-1}g^{-1}$ lies in $H$ and
$$
|d(o, gfpf^{-1}g^{-1}o) - 2d(o,go)|\le D
$$ 
\end{lem}

Recall that a point $\xi\in \ccG$ is a \textit{conical point} if there exist a sequence of $g_n \in \Gamma$ so that $g_no$ lies in a $r$--neighborhood of a geodesic ray $\gamma=[o,\xi]$. 
 
The next two results prove  Theorem \ref{ConfinedConsThm} in hyperbolic setup under the condition (i) and (ii) accordingly. 
The following does not use Lemmas \ref{GoodConfiningSetPinHyp} and \ref{EllipticRadicalinHyp}.
\begin{thm}[Theorem \ref{ConfinedConsThm}(i) in hyp. setup]\label{ConInHorLimitSetBaby}
Assume that $H<\isom(\U)$ is a torsion-free discrete subgroup confined by $\Gamma$ with a compact confining subset $P$.   Then the big horospheric limit set $\HG$  contains all but countably many points of $\ccG$. 
\end{thm}
\begin{proof}
We denote by $\Lambda_0 \in \ccG$ the countable union of fixed points of all loxodromic and parabolic elements $h \in H$.  We shall prove that any conical point  $\xi\in \ccG\setminus \Lambda_0$ is contained in $\HG$.   By definition, there exists a sequence of $g_n \in \Gamma$ so that $g_no$ lies in a $r$--neighborhood of a geodesic ray $\gamma=[o,\xi]$.

Set $D:=\max\{d(o,po): p\in P\}<\infty$. The confined subgroup $H$ implies the existence of $p_n\in P$ so that $h_n:=g_n p_n g_n^{-1}\in H$. Now, let $z\in [o,\xi]$ so that $d(g_no, z)\le r$. Thus, $d(h_no,z) \le r+D+d(o,g_no)\le 2r+D+d(o,z)$. Direct computation shows that $h_n o \in Ho$ lies in the horoball $\mathcal {HB}(\xi,o, 2r+D+\epsilon)$ (or see Lemma \ref{SameHoroball} for this general fact).  

A horoball $\mathcal {HB}(\xi)$ only accumulates at the center $\xi$. As the orbit $Ho\subset X$ is  discrete, it suffices to prove that   $\{h_no: n\ge 1\}$ is an infinite subset. Arguing by contradiction,  assume now that $\{h_no: n\ge 1\}$ is a finite set. By taking a subsequence,  we may assume that $h:=h_n=h_m$ for any $n, m\ge 1$.

By assumption,  $H$ is torsion-free, so $p_n$ must be of infinite order. According to the classification of isometries, $p_n$ is either hyperbolic or parabolic, so is $h_n$.  

As $d(hg_no,g_no)=d(o,p_no)<D$, the convergence $g_no\to \xi$ implies that $hg_no\to \xi$ and then $h$ fixes $\xi$.  So $\xi\in \Lambda_0$ gives a contradiction, which completes the proof. 
\end{proof}

We now demonstrate how to use Lemma \ref{GoodConfiningSetPinHyp} to deal with finite confining subset $P$.
\begin{thm}[Theorem \ref{ConfinedConsThm}(ii) in hyp. setup]\label{ConInHorLimitSetBaby2}
Assume that $H<\isom(\U)$ is a discrete subgroup confined by $\Gamma$ with a finite confining subset $P$. Assume that no element in $P$ fixes pointwise $\Lambda(\Gamma o)$.   Then the big horospheric limit set $\HG$  contains $\ccG$.     
\end{thm}
\begin{proof}
Given $\xi\in \ccG$, there exists  a sequence of $g_n \in \Gamma$ such that   $d(g_no, [o,\xi])\le r$. By Lemma \ref{GoodConfiningSetPinHyp}, we can choose $f_n\in F$ and $p_n\in P$ so that $h_n:=g_n f_n p_n f_n^{-1} g_n^{-1}\in H$ labels a quasi-geodesic path.  

If $d(g_no, g_mo)\gg 0$ for any $n\ne m$, then $h_no\ne h_mo$ follows by Morse Lemma. Thus $\{h_no: n\ge 1\}$ is an infinite  subset. Setting $$D=\max_{f\in F}\{d(o,fo)\}+\max_{p\in P}\{d(o,po)\}$$ we argue exactly as in the proof of Theorem \ref{ConInHorLimitSet} and obtain  that  $h_n o \in Ho$ lies in the horoball $\mathcal {HB}(\xi, 2r+2D+\epsilon)$. Hence,  $\{h_no: n\ge 1\}$  converges to $\xi$, so $\xi$ is a big horospheric limit point. 
\end{proof}
\iffalse\subsection{Strict inequality in the hyperbolic setting}
In this subsection, we outline a proof of the strict inequality in Theorem \ref{ConvTightThm}. 
Throughout this section we assume $\Gamma$ acts geometrically on a proper geodesic Gromov hyperbolic space $X$ and $H<Isom(X)$ is confined by $\Gamma$ with finite confining set $P$. 
We will require significant input from Patterson-Sullivan theory. For $x,y\in X$ and $r>0$ define the shadow $\Pi_{x}(y,r)$ to be the set of points $\zeta \in \partial X$ such that some geodesic ray starting at  $x$ and converging to $\zeta$ passes the ball of radius $r$ around $y$.
The classical shadow lemma for Patterson-Sullivan measures, originally proved by Sullivan for real hyperbolic spaces, asserts that an $\nu_H$ is any $\delta(H)$ dimensional $H$-quasiconformal density then for any $x\in X$ there is a $C_x>0$ with
$C_x e^{-d(gx,x)} \leq \nu_{H}(\Pi_{x}(gx,r)\leq C_x e^{-d(gx,x)}$. 
We will strengthen this to a shadow principle allowing $g$ to lie not merely in $H$ but also in $\Gamma$:
for any $x\in X$ there is a $C_x>0$ with
$$C_x ||\mu_{gx}|| e^{-d(gx,x)} \leq \nu_{H}(\Pi_{x}(gx,r)\leq C_x  ||\mu_{gx}|| e^{-d(gx,x)}.$$ 
\fi

\section{Completely dissipative actions}\label{secdiss}
Our goal of this section is to prove Theorem \ref{HalfGrowthDisThm} under the following setup. 
\begin{itemize}
    \item 
    The auxiliary proper action $\Gamma\act \U$ is assumed to be SCC, in particular it is of divergence type.
    \item 
    Let $\pU$ be a convergence boundary for $\U$ and $\{\mu_x:x\in \U\}$ be the unique quasi-conformal, $\Gamma$--equivariant density of dimension $\e \Gamma$ on $\pU$.
    \item 
    Let $H<\isom(\U)$ be a discrete subgroup, which is confined by $\Gamma$ with a  compact confining subset $P$ in $\isom(\U)$. 
\end{itemize}

We emphasize that $H$ is not necessarily contained in $\Gamma$, but preserves the measure class of $\mu_o$. This  is motivated by the following example. If $\U$ is the rank-1 symmetric space equipped with the Lebesgue measure $\mu_o$ on the visual boundary, any subgroup $H<\isom(\U)$ preserves the measure class of $\mu_0$. More generally, we have.  
\begin{lem}\label{PreserveMeasureClass}
Suppose that $\U$ is a hyperbolic space on which $\Gamma <\isom(\U)$ acts properly and co-compactly. Let $\{\mu_x:x\in \U\}$ be the unique Patterson-Sullivan measure class of dimension $\e\Gamma$ on $\pU$. Then $\isom(\U)$ preserves the measure class of $\mu_o$.    
\end{lem}
\begin{proof}
It is known that $\mu_o$ coincides with the Hausdorff measure of $\pU$ with respect to the visual metric. Namely, if $\rho_\epsilon$ is the visual metric for parameter $\epsilon$, then $\mu_o(A)\asymp \mathcal H^d_{\rho_\epsilon}(A)$ where $d=\e\Gamma/\epsilon$ is the Hausdorff dimension.  As an isometry on hyperbolic spaces induces a bi-Lipschitz map on boundary, it preserves the Hausdorff measure class. The proof is complete.%   Let $g\in \isom(\U)$. Then there exists $h\in \Gamma$ such that $t=h^{-1}g$ moves the basepoint $o$ with a distance independent of $g$ and $h$. Note that $t$ induces an bi-Lipschitz map on $\pU$, so $t_\star \mu_o$ is absolutely continuous with respect to $\mu_o$. Note that $g_\star\mu_o(A)=\mu_o(g^{-1}A)=\mu_{ho}(t^{-1}A)$ and $=\mu_{ho}$
\end{proof}
%Let us recall the following characterization of the conservative part for the action $G\act (\pU, \mu_o)$, which is due to Kaimanovich in this general setup.
%\begin{prop}
%Let $G<\isom(\U)$ be a non-elementary torsion-free group acting properly on $\U$. Then the conservative part  of $G$ on $(\pU,\mu_o)$ coincides with the big horospherical limit set of $G$. 
%\end{prop}
%\begin{proof}
If $H$ contains no torsion, the set $\mathcal D_{<\infty}$ of points with nontrivial stabilizer is empty. Thus the conservative component is exactly the infinite conservative part, which by Lemma \ref{HopfdecompLem} coincides with the big horospherical limit set.      Hence, we shall prove $\mu_o(\HG)=0$, provided that $\e H <\e \Gamma/2$. 

Our argument is essentially a geometric interpretation of the corresponding one in \cite{GKN}, where the same conclusion is proven for free groups. The proof in \cite{GKN} is more combinatorial: it crucially uses the so-called Nielsen-Schreier system given by a spanning tree in the Schreier graph.

%We assume that $\Gamma\act \U$ is SCC, so that by Proposition \ref{FreqContFull}, the conformal density on $\Lambda(\Gamma o)$ is supported on the set of regularly contracting limit points. 

Recall that $\mathcal C_\epsilon^{\mathrm{hor}}$ is a subset in $\mathcal C$ (\ref{AssumpE}) on which the Busemann cocycles converge up to an additive error $\epsilon$ in (\ref{BusemanConvError}).

\begin{thm}\label{HalfgrowthimplyDiss}
If $\e H<\e \Gamma/2$, then $\mu_o(\HG)=0$. In particular, if $H$ is torsion-free, the action of $H$ on $(\pU,\mu_o)$ is completely dissipative.  
\end{thm}
\begin{proof}
By Proposition \ref{FreqContFull}, $\mu_o$ has full measure on $\rcG$, and $\rcG$ is a subset of $\mathcal C^{\mathrm{hor}}_{20C}$. By taking the intersection $\rcG\cap \HG$, we may assume in addition that every point  $\xi\in \HG$ is a frequently contracting point. This is only the place in the proof where we use the SCC action.

For any sufficiently small $\theta\in (0,1]$, instead of Lemma \ref{CloserInHoroball}, we now prove 
\begin{claim}
Let $\xi\in \HG$. There exist two sequences of elements $h_n, t_n\in G$  such that   
\begin{align}
\label{tlengthEQ} d(o,t_no)>(1/2-\theta) d(o,h_no)\\
\label{tcenter1EQ}d(t_no, [o, \xi])\le r\\
\label{tcenter2EQ}d(t_no, [o, h_no])\le C+\|Fo\|
\end{align}
\end{claim}
\begin{proof}[Proof of the claim]
As $\xi$ is a big horospheric point, there exists $h_no\in \mathcal{HB}([\xi],L_0)$ tending to $\xi$, where $L_0$ depends on $\xi$. That is, $B_{[\xi]}(h_no, o)\le L_0$ where $\xi\in \mathcal C^{\mathrm{hor}}_{20C}$. Setting  $M=L_0+20C$, we obtain from (\ref{BusemanConvError}):  
\begin{equation}\label{InBallEq}
\forall z\in [o,\xi]: \quad d(h_no,z)\le d(o,z)+M    
\end{equation}

Fix a big $L\gg M$ to be decided below, which also depends on $\xi$. 

Choose a point $z\in [o,\xi]$ so that  $d(o,z)=d(o,h_no)-L$ for $n\gg 0$. 
Consider the  initial segment $\alpha:=[o,z]$ of $[o,\xi]$ 
and one $\theta$--interval $[x, y]\subseteq \alpha$ at the middle satisfying $d(o,x)=\ell(\alpha)/2$ and $d(y,\alpha_+)=(1/2+\theta)\ell(\alpha)$. By definition of $
\xi \in \rcG$, any $\theta$--interval of $\alpha$ contains $(r, f)$--barrier $t\in G$, so $$
\max\{d(to, [x,y]), d(tfo, [x,y])\}\le r
$$
where $f\in F$ is a contracting element. Up to enlarge $r$, we assume $[to,tfo]\subseteq N_r([x,y]),$ so (\ref{tcenter1EQ}) is satisfied.

\begin{figure}
    \centering
\tikzset{every picture/.style={line width=0.75pt}} %set default line width to 0.75pt        

\begin{tikzpicture}[x=0.75pt,y=0.75pt,yscale=-1,xscale=1]
%uncomment if require: \path (0,362); %set diagram left start at 0, and has height of 362

%Curve Lines [id:da060142391847852306] 
\draw    (49.5,39) .. controls (123.5,95) and (211.5,95) .. (326.5,222) ;
\draw [shift={(326.5,222)}, rotate = 47.84] [color={rgb, 255:red, 0; green, 0; blue, 0 }  ][fill={rgb, 255:red, 0; green, 0; blue, 0 }  ][line width=0.75]      (0, 0) circle [x radius= 3.35, y radius= 3.35]   ;
%Curve Lines [id:da2851207193644869] 
\draw    (491,38) .. controls (373.5,59.5) and (271.5,79) .. (326.5,222) ;
\draw [shift={(326.5,222)}, rotate = 68.96] [color={rgb, 255:red, 0; green, 0; blue, 0 }  ][fill={rgb, 255:red, 0; green, 0; blue, 0 }  ][line width=0.75]      (0, 0) circle [x radius= 3.35, y radius= 3.35]   ;
%Straight Lines [id:da2753032850023345] 
\draw    (49.5,39) -- (237.5,40) -- (387.5,38) -- (591.5,38) ;
\draw [shift={(594.5,38)}, rotate = 180] [fill={rgb, 255:red, 0; green, 0; blue, 0 }  ][line width=0.08]  [draw opacity=0] (10.72,-5.15) -- (0,0) -- (10.72,5.15) -- (7.12,0) -- cycle    ;
\draw [shift={(143.5,39.5)}, rotate = 0.3] [color={rgb, 255:red, 0; green, 0; blue, 0 }  ][fill={rgb, 255:red, 0; green, 0; blue, 0 }  ][line width=0.75]      (0, 0) circle [x radius= 3.35, y radius= 3.35]   ;
\draw [shift={(312.5,39)}, rotate = 359.24] [color={rgb, 255:red, 0; green, 0; blue, 0 }  ][fill={rgb, 255:red, 0; green, 0; blue, 0 }  ][line width=0.75]      (0, 0) circle [x radius= 3.35, y radius= 3.35]   ;
\draw [shift={(491,38)}, rotate = 0] [color={rgb, 255:red, 0; green, 0; blue, 0 }  ][fill={rgb, 255:red, 0; green, 0; blue, 0 }  ][line width=0.75]      (0, 0) circle [x radius= 3.35, y radius= 3.35]   ;
\draw [shift={(49.5,39)}, rotate = 0.3] [color={rgb, 255:red, 0; green, 0; blue, 0 }  ][fill={rgb, 255:red, 0; green, 0; blue, 0 }  ][line width=0.75]      (0, 0) circle [x radius= 3.35, y radius= 3.35]   ;
%Straight Lines [id:da3817677764023575] 
\draw    (170.5,60) -- (276.5,61) ;
\draw [shift={(276.5,61)}, rotate = 0.54] [color={rgb, 255:red, 0; green, 0; blue, 0 }  ][fill={rgb, 255:red, 0; green, 0; blue, 0 }  ][line width=0.75]      (0, 0) circle [x radius= 3.35, y radius= 3.35]   ;
\draw [shift={(170.5,60)}, rotate = 0.54] [color={rgb, 255:red, 0; green, 0; blue, 0 }  ][fill={rgb, 255:red, 0; green, 0; blue, 0 }  ][line width=0.75]      (0, 0) circle [x radius= 3.35, y radius= 3.35]   ;
%Curve Lines [id:da8388799973951007] 
\draw    (223.5,60.5) .. controls (248,109.01) and (291.71,78.29) .. (319.8,106.22) ;
\draw [shift={(321.5,108)}, rotate = 227.91] [fill={rgb, 255:red, 0; green, 0; blue, 0 }  ][line width=0.08]  [draw opacity=0] (10.72,-5.15) -- (0,0) -- (10.72,5.15) -- (7.12,0) -- cycle    ;
\draw [shift={(223.5,60.5)}, rotate = 63.2] [color={rgb, 255:red, 0; green, 0; blue, 0 }  ][fill={rgb, 255:red, 0; green, 0; blue, 0 }  ][line width=0.75]      (0, 0) circle [x radius= 3.35, y radius= 3.35]   ;
%Shape: Brace [id:dp715660225913292] 
\draw  [line width=0.75]  (313.5,37) .. controls (313.47,32.33) and (311.13,30.01) .. (306.46,30.04) -- (233.53,30.47) .. controls (226.86,30.51) and (223.52,28.2) .. (223.49,23.53) .. controls (223.52,28.2) and (220.2,30.55) .. (213.53,30.59)(216.53,30.58) -- (150.46,30.97) .. controls (145.79,31) and (143.47,33.34) .. (143.5,38.01) ;

% Text Node
\draw (125,19.4) node [anchor=north west][inner sep=0.75pt]    {$x$};
% Text Node
\draw (331,16.4) node [anchor=north west][inner sep=0.75pt]    {$y$};
% Text Node
\draw (486,14.4) node [anchor=north west][inner sep=0.75pt]    {$z$};
% Text Node
\draw (283,207.4) node [anchor=north west][inner sep=0.75pt]    {$h_{n} o$};
% Text Node
\draw (216,64.4) node [anchor=north west][inner sep=0.75pt]    {$u$};
% Text Node
\draw (330,110.4) node [anchor=north west][inner sep=0.75pt]    {$v$};
% Text Node
\draw (150,55.4) node [anchor=north west][inner sep=0.75pt]    {$to$};
% Text Node
\draw (292,59.4) node [anchor=north west][inner sep=0.75pt]    {$tfo$};
% Text Node
\draw (42,45.4) node [anchor=north west][inner sep=0.75pt]    {$o$};
% Text Node
\draw (580,44.4) node [anchor=north west][inner sep=0.75pt]    {$\xi $};
% Text Node
\draw (201,4.4) node [anchor=north west][inner sep=0.75pt]    {$\theta \ell ( \alpha )$};
% Text Node
\draw (378,15.4) node [anchor=north west][inner sep=0.75pt]    {$\alpha =[ o,z]$};

\end{tikzpicture}

    \caption{Schematic configuration in the proof of Theorem \ref{HalfgrowthimplyDiss}}
    \label{fig:triangle}
\end{figure}
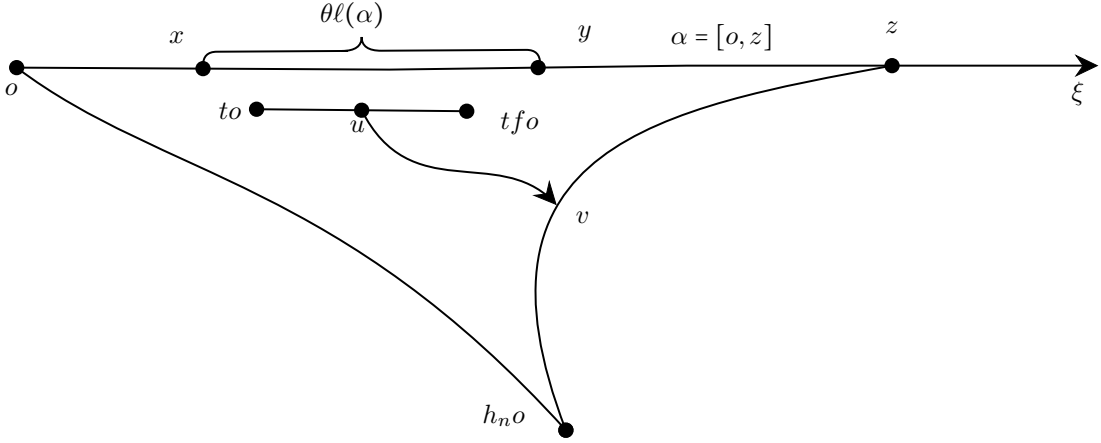
Look at the triangle with vertices $(o,h_no,z)$ for $z\in [o,\xi]$. See Fig. (\ref{fig:triangle}).
By the contracting property, $t\ax(f)$ intersects non-trivially either $N_C([o,h_no])$ or $N_C([z,h_no])$. Indeed, if this is false, the  projection of $[o,h_no]$ and $[z,h_no]$ to $t\ax(f)$ together has diameter at most $2C$. Meanwhile, the corresponding initial and terminal segments of $\alpha$ entering and leaving $N_C(t\ax(f))$  both project to $t\ax(f)$ with diameter at most $C$ as well. Hence, we obtain $\|\alpha\cap N_C(t\ax(f))\|\le 4C$. On the other hand,  noting the above $\|\alpha\cap N_r(t\ax(f))\|\ge d(o,fo)$, we would get a contradiction if $d(o,fo)\gg 4C+2r$.   

Moreover, as  $[to,tfo]\subseteq N_r(\alpha)\cap t\ax(f)$, the above argument further shows that $[to,tfo]$ intersects non-trivially either $N_C([o,h_no])$ or $N_C([z,h_no])$. 

We now claim   that $[to,tfo] \cap N_C([z,h_no])=\emptyset$. Indeed, if not, we have  $d(u, v)\le C$  for some $u\in [to,tfo]$ and $v\in [z,h_no]$.  Thus, 
\[
\begin{aligned}
d(v,h_no)&= d(z, h_no)-d(z,v)\\
&\le d(z,h_no)-d(z,u)+C\\
&\le d(z,o)-d(z,u)+C+M    
\end{aligned}
\]
where the last inequality uses (\ref{InBallEq}).
On the other hand, as $d(o, h_no)=d(o,z)+L$, we have 
\[
\begin{aligned}
d(v,h_no)&\ge d(o,h_no)-d(o,v)\\
&\ge d(o,z)-d(o,u)+L-C
\end{aligned}
\]If $L>2C+M$ is assumed, this would give a contradiction, so the above claim is proved. 

By the claim, let us choose $u\in [to,tfo] \cap N_C([o,h_no])$.  Thus, we have  $d(to, [o,h_no])\le C+d(o,fo)$, so (\ref{tcenter2EQ}) is satisfied.

Recall that $\ell(\alpha)=d(o,h_no)-L$ and $d(to,tfo)\le d(x,y)+2r$. For $u\in [to,tfo]$, we have    
$$d(u,to)\le d(x,y)+2r\le \theta \alpha+2r\le \theta(d(o,h_no)-L)+2r$$ and  $$d(o,u)\ge d(o,x)-r\ge \ell(\alpha)/2-r\ge (d(o,h_no)-L)/2-r$$ which together give 
$$
d(o,to) \ge d(o,u)-d(u,to)\ge (d(o,h_no)-L)(1/2-\theta)-3r.
$$
As $d(o,h_no)\to \infty$, we may drop finitely many $h_no$ to remove $L$ (which depends on $\xi$ though), so that $d(o,to) \ge  d(o,h_no)(1/2-2\theta).$  
This shows that  $t_n:=t$ is the desired element with (\ref{tlengthEQ},\ref{tcenter1EQ},\ref{tcenter2EQ}).   
\end{proof}
Now we choose $\theta, \epsilon>0$ so that $\e H+\epsilon<\e \Gamma(1/2-\theta)$. Let $Z$ be the set of  points $t\in \Gamma$  with the above property in the above Claim. Note that $\rcG\cap \HG$ is contained in the limit superior of $\{\Pi_o^F(to, r): t\in Z\}$. Note again that $h$ is over used by $t$ at most $C_0d(o,ho)$ times for a universal constant $C_0$, as $to$ in (\ref{tcenter2EQ}) lies in a fixed neighborhood of $[o,ho]$.  Shadow lemma hence allows  to compute 
$$
\sum_{t\in Z} \mu_o(\Pi_o^F(to, r)) \le  \sum_{h\in H} d(o,ho) \mathrm{e}^{-\e \Gamma (1/2-\theta)d(o,ho)}\prec \sum_{h\in H} d(o,ho) \mathrm{e}^{-(\e H +\epsilon)d(o,ho)} <\infty.
$$
Hence, $$\mu_o(\rcG\cap \HG)=0$$   by Borel-Cantelli Lemma. Thus, $\mu_o(\rcG)=1$ implies  $\mu_o(\HG)=0$, completing  the proof.
\end{proof}
\section{Preliminaries on confined subgroups}\label{secconfine} 
From this section to \textsection \ref{secinequality}, we study the boundary actions of confined subgroups, growth and co-growth inequalities. In this section we first recall the notion of confined subgroups and review some known facts for later reference. Then we show an enhanced version of the Extension Lemma, which will be crucial in the later sections. 
\subsection{Confined subgroups} 
\begin{defn}
Let $ G$ be a locally compact, metrizable, topological group and $\Gamma< G$.
We say that  $H< G$ is  \textit{confined} by $\Gamma$ if  there exists a compact subset $P$   in $ G$ such that  $P$ intersects \textit{non-trivially} with any conjugate $g^{-1}Hg$ for $g\in \Gamma$:  $$g^{-1}Hg\cap (P\setminus \{1\})\ne\emptyset.$$ 
We refer   $P$  as a \textit{confining subset} for $H$. If $\Gamma= G$, we say that $H$ is a confined subgroup in $G$.
\end{defn}
In this paper, we are mainly interested in the case where $G=\isom(\U)$ or $G=\Gamma$.

Note that any nontrivial normal subgroup is confined. Here are a few remarks in order.
\begin{rem}
If $P$ is the confining subset, so is  $P^n=\{p^n: p\in P\}$ for $n\ge 1$.    
If $P$ is finite, we may further assume that $P$ is \textit{minimal}: for any $p\in P$, there exists $g\in G$ so that $gpg^{-1}\in H$.

If $H$ is a confined subgroup in $G$, then any subgroup containing $H$ is also confined. Note that being confined is inherited to finite index subgroups. Indeed,  if $K$ is a finite index subgroup of $H$, then any element $h\in H\setminus K$ admits a sufficiently large power in $K$, so $K$ is a confined subgroup in $G$.
\end{rem}

Let $H$ be a discrete subgroup of $\isom(\U)$. We define the injectivity radius of the action at a point $x\in \U$ as follows:
$$
\textrm{Inj}_H(x)=\min\{d(x,hx): h\in H\setminus H_x \}
$$
where $H_x:=\{h\in H: hx=x\}$ is finite by the proper action. By definition, $\textrm{Inj}_H(hx)=\textrm{Inj}_H(x)$ for $h\in H$, so the injectivity radius function descends to the quotient $\U/H$. When $H$ is torsion free, this coincides with the usual notion of injectivity radius, that is the maximal radius of the  disk centered at $x$ in $\U$   injecting into $\U/H$.  It is clear that $\textrm{Inj}_G(Hx)\ge \textrm{Inj}_G(Gx)$ for any $H<G$.

\begin{lem}\label{CharConfinedSubgroups}
Assume that $\Gamma<\isom(\U)$ acts cocompactly (and not necessarily properly) on a $\Gamma$--invariant subset $Z\subseteq \U$. Denote $\pi: \U\to \U/H$. Then the subset    $\pi(Z)$ of $\U/H$ has  injectivity radius  bounded from above if and only if $H$ is confined by $\Gamma$.    
\end{lem}
\begin{proof}
(1). Suppose that $\textrm{Inj}_H:\U/H\to \mathbb R$ is  bounded from above by $D>0$ over the set $\pi(Z)$. That is, for any $x\in Z$,  there exists $h\in H\setminus H_x$ such that $d(x,hx)\le D$.  Fixing a basepoint $o\in Z$, as $\Gamma$ acts cocompactly on $X$, there exists  a constant $M>0$ such that $Z\subseteq N_M(\Gamma o)$. Given  $f\in \Gamma$, let $x\in Z$ so that $d(fo, x)\le M$, so we have $d(o,f^{-1}hfo)\le 2D+M$.  The set $P:=\{g\in \isom(\U): d(o,go)\le 2D+M\}$ is compact, so $H$ is confined by $\Gamma$. 

(2). Suppose that $H$ is confined by $\Gamma$ with a confining subset $P$, we need to bound $\textrm{Inj}_G(Z/H)$.   
For any $x\in Z$, the cocompact action of $\Gamma$ on $Z$ implies the existence of $g\in \Gamma$ such that $d(go,x)\le M$. As $H$ is confined by $\Gamma$, there exist $h\in H$ and  $p\in P$   such that $g^{-1}hg=p$. Thus, we have $d(x, hx)\le d(go, hgo)+2M\le d(o,po)+2M$. Setting $D=2M+\max\{d(o,po): p\in P\}<\infty$ completes the other direction of the proof: $\textrm{Inj}_H(Hx)\le D$ over  $x\in Z$.
\end{proof}

A direct corollary is the following.
\begin{lem}
Assume that a group $\Gamma<\isom(\U)$ acts cocompactly on $\U$. Then for a discrete subgroup $H<G:=\isom(\U)$, the following are equivalent:
\begin{enumerate}
    \item 
    $H$ is confined  by $\Gamma$;
    \item 
    $H$ is confined  in $G$;
    \item
    $\U/H$ has  injectivity radius bounded from above.
\end{enumerate}       
\end{lem}

\subsection{Elliptic radical}
According to the definition, the product of any subgroup and a non-trivial finite normal subgroup is confined.  To eliminate such pathological examples, we consider the notion of elliptic radical which, roughly speaking, consists of elements that fix the boundary pointwise. %Aiming for further potential applications, we formulate the notions and results in greater generality.     %In our setup, this turns out  to be the only possible source of pathological examples.
 
Consider a convergence boundary $\pU$ of $\U$, where the action of $\isom(\U)$ extends to $\pU$ by homeomorphisms and preserves the partition $[\cdot]$ (see Definition \ref{ConvBdryDefn}). Assume that $\Gamma$ is a non-elementary discrete subgroup of $\isom(\U)$ with non-pinched contracting elements. 

\begin{defn}\label{EllipticRadicalDefn}
Let $G<\isom(\U)$ be a group. The \textit{elliptic radical} $E_G(\Gamma)$  consists of elements $g\in G$ that induces an identity map on the limit set $[\pG]$: 
$$E_G(\Gamma):=\{g\in G: g[\xi]=[\xi], \forall [\xi]\subseteq [\pG] \}.$$
Note that $g\in E_G(\Gamma)$ may not fix pointwise $\pG$. If $G=\isom(\U)$ we write $E(\Gamma)=E_G(\Gamma)$.   
\end{defn}
\begin{rem}
Recall that, given a contracting element $f\in \Gamma$, $E_G(f)$ is the maximal elementary subgroup in $G$ that contains $f$. By Lemma \ref{NonPinchedStabilizer}, if $f$ is non-pinched, then $E_G(f)$ is the set stabilizer of $[f^\pm]$ in $G$. 
%As   $E_G(\Gamma)$ is a normal subgroup in $G$, we deduce that if a confined subgroup $H$ with finite confining subset intersects trivially $E(G)$, then the (minimal) confining subset does so. 

If $G=\Gamma$, $E_G(\Gamma)$ is the intersection of the maximal elementary subgroups $E_G(f)$ (in Lemma \ref{elementarygroup}) for all contracting elements $f\in G$. This is the unique maximal finite normal subgroup of $G$.    
\end{rem}

\subsection{An enhanced extension lemma}

The following result will be a crucial tool in the next sections, which could be thought of as an enhanced version of Lemma \ref{extend3}.
\begin{lem}\label{EllipticRadical}
Let $P\subseteq \isom(\U)$ be a finite set  disjoint with $ E(\Gamma)$. Then there exists a finite set $F\subseteq \Gamma$ of contracting elements and constants $\tau, D>0$ with the following property.  Set $L=\min\{d(o,fo): f\in F\}$.

For any $g, h\in \Gamma$, one can find $f\in F$ so that for any $p\in P$, the word $(g, f, p, f^{-1}, h)$ labels an $(L, \tau)$--admissible path with associated contracting subsets $g\ax(f)$ and $gfp\ax(f)$. Moreover,  
$$
|d(o, gfpf^{-1}ho) - d(o,go)-d(o,ho)|\le D
$$ 
\end{lem}

\begin{proof}
First of all, fixing a non-pinched contracting element $f\in \Gamma$, we claim that for any $p\in P$, there exists  $g\in \Gamma$ so that   $pg[f^\pm]\ne g[f^\pm]$.  

Indeed, if not, then $[pg\eta]=[g\eta]$ for any $g\in \Gamma$ and for  any $\eta\in [f^\pm]$. By Lemma \ref{FixptsDense},  the set of $\Gamma$--translates of its fixed points $[f^\pm]$ is dense in $\pG$: $[\pG]=[\overline{\Gamma\eta}]$. To be more precise, for any $\xi\in\pG$, there exists $g_n\in \Gamma$ such that $g_n\eta\to \xi'$ for some $\xi'\in[\xi]$.  By assumption, $[pg_n\eta]=[g_n\eta]$.  According to Definition \ref{ConvBdryDefn}(C), $[\cdot]$  on $\Gamma[g^\pm]$ forms a closed relation, this implies $[p\xi']=[\xi']$:  that is, $p$ fixes every $[\cdot]$--class in $[\pG]$, so $p\in E(\Gamma)$. This gives a contradiction. Hence, we see that $p\in P$ does not stabilize $g[f^\pm]$ for any $g\in \Gamma$.   

By Lemma \ref{DoubleDense},  if $f,g$ are non-pinched contracting elements in $\Gamma$, then $h_n:=f^ng^n$ for any $n\gg 0$ is a contracting element so that the attracting fixed point $[h^+]$ tends to $[f^+]$ and the repelling fixed point $[h^+]$ tends to $[g^-]$ as $n\to \infty$.  By the previous paragraphs, for any $p,q\notin E(\Gamma)$, we can find two non-pinched contracting elements $f, g\in \Gamma$, so that $p$ does not fix $[f^+]$ and $q$ does not fix $[g^-]$. As $[f_+]$ and $[pf_+]$ are disjoint closed subsets in $\pU$ by Definition \ref{ConvBdryDefn}(A), choose disjoint open neighbourhoods such that  Lemma \ref{DoubleDense} shows  that $p, q$ do not stabilize the fixed points $[h_n^{\pm}]$ of $h_n$    for any $n\gg 0$.  We do not conclude here that $p, q\in E_\Gamma(h_n)$, as $E_\Gamma(h_n)$ is the subgroup in $\Gamma$ (not in $\isom(\U)$) stabilizing $[h_n^{\pm}]$.

As $P$ is a finite set, a repeated application of the previous paragraph produces an infinite set $\tilde F\subseteq \Gamma$ of independent contracting elements  such that $p[f^\pm]\cap [f^\pm]=\emptyset$ for any $p\in P$ and $f\in \tilde F$. 

Now to conclude the proof, it suffices to invoke the same arguments in the proof  of Lemma \ref{extend3}. Namely, for any finite set $F\subseteq \tilde F$, there is a bounded intersection function $\tau: \mathbb R\to \mathbb R$ so that for any $f\ne f'\in F$ and any $R>0$, 
$$
\|N_R(\ax(f))\cap N_R(\ax(f'))\|\le \tau(R).
$$
The function $\tau$ depends only on the configuration of  the axis $\ax(f)=E_\Gamma(f)\cdot o$ for $f\in F$. That is to say, $\tau $ remains the same if $F'$ is chosen as a different set of representatives  $f'\in E_\Gamma(f)$ for $f\in F$. Recall that $E_\Gamma(f)<\Gamma$ is the maximal elementary subgroup in $\Gamma$ defined in Lemma \ref{elementarygroup}.

By the contracting property, for any $g\in \Gamma$ and any given $R\gg 0$, the diameter $\|\pi_{\ax(f)}([o,go])\|$ is comparable with $\|N_R(\ax(f))\cap [o,go]\|$. Hence, with   at most one exception $f_0\in F$,  
\begin{equation}\label{pboundedEQ}
\forall f\in F\setminus \{f_0\}:\quad \pi_{\ax(f)}([o,go])\le \tau, 
\end{equation}
and, since $p\ax(f)\ne \ax(f)$ for any $f\in F$ and a finite set of $p\in P$,
\begin{equation}\label{pboundedEQ2}
\forall f\in F:\quad \pi_{\ax(f)}([o,po])\le \tau, 
\end{equation}
where $\tau=\tau(R)$ is a constant  depending on the axis $\ax(f)$ for $f\in F$ as above. 
Consequently, for any $g,h\in \Gamma$, there exists a common $f\in F$ satisfying (\ref{pboundedEQ})  so that 
\begin{equation}\label{gboundedEQ}
\max\{\pi_{\ax(f)}([o,go]),\pi_{\ax(f)}([o,ho])\}\le \tau
\end{equation}
Setting $L=\min\{d(o,fo): f\in F\}$, the above equations (\ref{pboundedEQ})(\ref{pboundedEQ2})(\ref{gboundedEQ}) verify that for any $p\in P$, the word $(g, f,p, f^{-1}, h)$ labels an $(L,\tau)$--admissible path, denoted by $\gamma$, associated with the contracting sets $g\ax(f)$ and $gfp\ax(f)$. Taking a high power of $f\in F$ if necessary, the constant $L$ can be  large enough to satisfy Lemma \ref{extend3}, so any geodesic $\alpha$ with same endpoints as    $\gamma$ $r$--fellow travels $\gamma$, so 
$go, gfo$ and $gfpo, gfpf^{-1}o$ have at most a distance $r$ to $\alpha$. Letting $D=8r+L+\max\{d(o,po):p\in P\}$ concludes the proof of the lemma. 
\end{proof}

\subsection{First consequences on confined subgroups}
Until the end of this subsection, consider a  subgroup  $H<G$, which is  confined by $\Gamma$ with a finite confining subset $P$. Assume that $E_G(\Gamma)\cap P\setminus \{1\}\ne \emptyset$. The main situations we keep in mind are  $G=\isom(\U)$ or $G=\Gamma$.

The following is an immediate consequence of Lemma \ref{EllipticRadical} applied to $(g, g^{-1})$, and  $p\in P$ is chosen for $gf$ according to the definition of confined subgroups.

\begin{lem}\label{GoodConfiningSetP}
There exists a finite subset $F$ of contracting elements in $\Gamma$ with the following property:
for any $g\in \Gamma$, there exist $f\in F$ and $p\in P$ such that $gfpf^{-1}g^{-1}$ lies in $H$ and
$$
|d(o, gfpf^{-1}g^{-1}o) - 2d(o,go)|\le D
$$
where $D$ depends only on $F$ and $P$.
\end{lem}

This allows us to derive the following desirable property.
\begin{lem}\label{UniformContractingElem}
There exists a finite set $F$ of pairwise independent contracting elements in $\Gamma$ with the following property. Fix a conjugate $gHg^{-1}$ for $g\in \Gamma$. For any $f_1\ne f_2\in F$ there exists $p_1, p_2\in P$ such that $f_ip_if_i^{-1}\in gHg^{-1}$ and their product $c(f_1,f_2):=(f_1p_1f_1^{-1})\cdot (f_2p_2f_2^{-1})$ is a contracting element. Moreover, any two such $c(f_1,f_2)$ for distinct pairs $(f_1,f_2)\in F\times F$ are independent.  
\end{lem}
\begin{proof}
For $i=1,2$, applying Lemma \ref{GoodConfiningSetP} to $g^{-1}$ with $f_i$ yields the choice of $p_i$ in $P$ so that $h_i:=f_ip_if_i^{-1}$ lies in $gHg^{-1}$. It remains to note that $h:=h_1\cdot h_2$ is contracting. To this end, as $\ax(f_1)$ and $\ax(f_2)$ have bounded intersection, setting $L:=\min\{d(o,fo): f\in F\}$, we have that the power $h^n$ labels an $(L,\tau)$--admissible path as follows
$$
\gamma:=\cup_{n\in \mathbb Z} h^n ([o,f_1o]f_1[o, p_1o]f_{1}p_{1}[o,f_1^{-1}o]\cdot h_1[o,f_2o]f_2[o, p_2o]f_{2}p_{2}[o,f_2^{-1}o]) 
$$
where the associated contracting subsets are given by the corresponding translated axis of $f_1$ and $f_2$.
Choosing $L$ sufficiently large, the path $\gamma$ is contracting, so $h$ is contracting. 

Finally, if $h=c(f_1,f_2)$ and $h'=c(f_1',f_2')$ for $(f_1,f_2)\ne (f_1',f_2')$, then we need to show that $h$ and $h'$ are independent. That amounts to proving that their axes $\gamma$ and $\gamma'$ have infinite Hausdorff distance. Indeed, if not, then $\gamma$ lies in a finite neighborhood of  $\gamma'$, so they fellow travel a common bi-infinite geodesic $\alpha$ (which could be obtained by applying Cantor argument with Ascoli-Arzela Lemma). For definiteness, say  $f_1'\ne f_1\ne f_2'$.  The fellow travel property by Proposition \ref{admisProp} then forces a large intersection of some axis of $f_1$ with the axis of $f_1'$ or the axis of $f_2'$. This would lead to a contraction with bounded intersection of $F$.  
\end{proof}

We get the following corollary on confined subgroups.
%\begin{cor}\label{ConfinedNonElementary}
%If a subgroup $H<G$ confined by $\Gamma$ with a finite confining subset that intersects trivially $E_G(\Gamma)$, then $H$ must be a non-elementary subgroup with a contracting element.    
%\end{cor}

\begin{lem}\label{ConfinedNonElementary}
Assume that $H<G$ is a subgroup confined by $\Gamma$ with a finite confining subset $P$ that intersects trivially $E_G(\Gamma)$.  Then  $H$ must be a non-elementary subgroup with a contracting element. Moreover, $\Lambda(\Gamma o)\subseteq [\Lambda(H o)]$. 
\end{lem}
\begin{proof}
Let $g_no\to \xi\in \Lambda (\Gamma o)$ for $g_n\in \Gamma$. As above, we then choose $f_n\in F$ and $p_n\in P$ so that $h_n:=g_n f_n p_n f_n^{-1} g_n^{-1}\in H$ labels an $(L,\tau)$--admissible path. Consider the sequence of contracting quasi-geodesic $X_n=g_n\ax(f_n)$,  so it follows that $h_no\cap N_r(X_n)\ne\emptyset$ where $r$ is given by Proposition \ref{admisProp}. Hence, we proved that $g_no, h_no\in \Omega(N_r(X_n))$, so the Assumption (B) in Definition \ref{ConvBdryDefn} implies that $h_no$ also tends to $[\xi]$. The proof is now complete.    
\end{proof}

\section{Conservativity of boundary actions of confined subgroups}\label{seccons} 
Our goal is to prove Theorem \ref{ConfinedConsThm}. In this section we retain the same setting as \textsection \ref{secdiss}.

\subsection{Some preliminary geometric lemmas}
We first prepare some geometric lemmas. Recall that $\mathcal C_\epsilon^{\mathrm{hor}}$ is a subset in $\mathcal C$ (\ref{AssumpE}) on which the Busemann cocycles converge up to an additive error $\epsilon$ in (\ref{BusemanConvError}). 
\begin{lem}\label{SameHoroball}  
Let $\gamma=[o,\xi]$ be a geodesic ray ending at $[\xi]$ for some $\xi\in \mathcal C_\epsilon^{\mathrm{hor}}$. If $y\in \U$ is a point that satisfies $d(y,z)\le d(o,z)+D$ for some $z\in \gamma$ and a real number $D\in \mathbb R$, then $y\in \mathcal {HB}(\xi, \epsilon+D)$. 
\end{lem}
\begin{proof}
According to Definition \ref{HoroballDefn}, the horoball $\mathcal {HB}(\xi, L)$ consists of points $x\in \U$ such that $B_{[\xi]}(x,o)\le L$. Noting that $d(\gamma(t),z)+d(z,o)=t$ for $z\in \gamma$ and $t\ge d(o,z)$, so   $d(\gamma(t),y)-t \le  [d(\gamma(t),z)+d(z,y)]-d(\gamma(t),z)-d(z,o)\le d(z,y)-d(z,o) \le D.$ Taking $t\to\infty$ shows $B_{[\xi]}(y,o)\le \epsilon+D$ by (\ref{BusemanConvError}).  That is, $y\in \mathcal {HB}(\xi, \epsilon+D)$.   
\end{proof}
%\ilya{I don't understand the first inequality in the above line. I don't think it's just triangle inequality since $d(\gamma(t),z)+d(z,o)\geq t$. Maybe we need some convexity here, but this only holds for CAT(0) and quasi-holds for hyperbolic spaces.}\ywy{I do not get your point, as $z$ is on the geodesic $\gamma$, we have $d(\gamma(t),z)+d(z,o)=d(\gamma(t),\gamma(0))= t$ for $t\gg 0$. }
%\ilya{Yes you are right:)}

We now arrive at a crucial observation in the proof of Theorem \ref{ConInHorLimitSet}. 
\begin{lem}
Given $\tau, C>0$ there exists $L=L(C,\tau)$ with the following property.
Assume that $\ax(f)$ is $C$--contracting.
Consider a distinct axis  $t\ax(f)\ne \ax(f)$ for some $t\in G$. Let   $g_1o\in \ax(f), g_2o\in  t\ax(f)$ so that $[g_1o, g_2o]$ intersects $\ax(f)$ in a diameter at least $L$, and $g_2o$ lies in a $\tau$--neighborhood of the projection of $\ax(f)$ to $t\ax(f)$. Then $h:=g_1^{-1}g_2$ is a contracting element.   
\end{lem}
\begin{proof}
The $C$--contracting property of $\ax(f)$ implies that any geodesic outside $N_C(\ax(f))$ has $C$--bounded projection, so $[o,ho]$ has $2C$--bounded projection to $\ax(f)$.  
According to the assumption, $[o,ho]$ has $\tau$--bounded projection to $h\ax(f)$, so this verifies that the path $$\gamma=\cup_{n\in \mathbb Z} h^n[o,ho]$$ is an $(L,\tau+2C)$--admissible path  with  associated contracting subsets $\{h^n\ax(f), h^nt\ax(f): n\in \mathbb Z\}$. If $L$ is sufficiently large, then $\gamma$ is a contracting quasi-geodesic, concluding the proof that $g$ is contracting.   
\end{proof}

As a corollary, we have.
\begin{lem}\label{TightContractingTransition}
Assume that $\xi\in \cG$ is an $(r,F)$--conical point. Let  $[o,\xi]$ be a geodesic ray ending at $[\xi]$ with two distinct  $(r,f)$--barriers $X\ne Y$. Let $g_1o\in X, g_2o\in Y$ be $C$--close to the corresponding entry points of $[o,\xi]$ in $N_C(X)$ and $N_C(Y)$. Then $g_1^{-1}g_2$ is a contracting element.       
\end{lem}

Recall that a geodesic metric space is geodesically complete, if any geodesic segment extends to a (possibly non-unique) bi-infinite geodesic. A smooth Hadamard manifold is geodesically complete.
\begin{lem}\label{UniqueFixedPt4Parabolic}
Assume that $\U$ is a proper, geodesically complete, CAT(0) space. Let $p\in \isom(\U)$ be a parabolic isometry. If $p$ fixes a conical point $\xi$ in $\cG$ for a proper action of $\Gamma$ on $\U$, then $p$ has the unique fixed point.    
\end{lem}
\begin{proof}
By \cite{FNS06}, under the assumption on $\U$, the fixed point set  $\mathrm{Fix}(p)$ of a parabolic element has diameter at most $\pi/2$ in the Tits metric on the visual boundary $\pU$. A conical point $\xi$ is visible from any other point $\xi\ne\eta\in \pU$; that is, there is a bi-infinite geodesic between $\xi$ and $\eta$. Thus, the angular metric between  $\xi$ and $\eta$ is $\pi$, so being the induced length metric, the Tits metric from $\xi$ to $\eta$ is  at least $\pi$. Hence, $\mathrm{Fix}(p)$ is singleton, completing the proof.     
\end{proof}

\subsection{Proof of Theorem \ref{ConfinedConsThm}}

Assume that $H$ preserves the measure class of $\mu_o$. By Lemma \ref{ConicalPointsLem}, $\mu_o$ is supported on the set of conical points $[\cG]$. By Lemma \ref{ConicalPointsLem}, $[\cG]$ is a subset of $\mathcal C_{20C}^{\mathrm{hor}}$ (defined in \ref{AssumpE}). 

The conservativity of the action of $H$ on $(\pU, \mu_o)$ as stated in Theorem \ref{ConfinedConsThm},  follows from a combination of the next two theorems \ref{ConInHorLimitSet} and \ref{ConInHorLimitSet2}, which applies assuming the first condition (i) and second condition (ii) respectively. Note that there are situations (torsion-free $H$ with a finite confining set) where both theorems apply, but we stress that the arguments are of different flavors. The first result applies under the assumption of a compact confining subset $P$ without torsion; the second result crucially uses the Lemma \ref{GoodConfiningSetP} based on the finiteness of $P$ with torsion allowed.

\begin{thm}\label{ConInHorLimitSet}
Assume that $H<\isom(\U)$ is a torsion-free discrete confined subgroup. If   $\U$ is   neither hyperbolic nor geodesically complete \textrm{CAT(0)}, assume, in addition, that $P$ is finite. Then the big horospheric limit set $\HG$  contains a $\mu_o$--full subset of $[\cG]$.    
\end{thm}
\begin{rem}
The proof shows that the conclusion of Theorem \ref{ConInHorLimitSet} remains valid for any atomless measure $\nu$ supported on $[\cG]$.  If $\U$ is hyperbolic, $\cG$ could be replaced with $\ccG$. In particular, harmonic measures on horofunction boundary which can arise as the hitting measure of a random walk on $\Gamma$ are among such examples. 
\end{rem}
\begin{proof}
Let $\Lambda_0$ be the countable union of fixed points $[h^\pm]$ of all contracting elements $h \in H$. If $\U$ is hyperbolic or CAT(0), we adjoin into $\Lambda_0$ the  fixed points of all parabolic elements in $H$, which are countably many by Lemma \ref{UniqueFixedPt4Parabolic}. Note that  $\mu_o(\Lambda_0)=0$ for $\mu_0$ has no atom. By Lemma \ref{ConicalPointsLem}, we will prove the $\mu_o$--full set of points $\xi\in [\cG]\setminus \Lambda_0$ is contained in $[\HG]$.   By definition of an $(r,F)$--conical point, there exists  a sequence of $(r, f_n)$--barriers $g_n \in \Gamma$ on a geodesic ray $\gamma=[o,\xi]$,
where $f_n\in F$ is taken over the finite set $F\subseteq \Gamma$. By definition, we have $d(g_no, [o,\xi]), d(g_n f_no, [o,\xi])\le r$. Passing to a subsequence, we may assume that $f:=f_n$ for all $n\ge 1$. Moreover, we can assume that $X_n$ are all distinct; otherwise $\xi$ lies in the limit set of the same $X_n$ for  $n\gg 0$, so is  fixed by a contracting element $g_nfg_n^{-1}$. This implies $\xi\in \Lambda_0$,    contradicting the assumption.   

%As $\xi\in \rcG$,  $[o, g_no]$ is  contracting at any frequency $\theta\in (0,1]$, for all but  finitely many $n$. In particular, fixing $\theta=1/5$, let us choose $[1/5,2/5]$ and $[3/5,4/5]$--intervals of $[o,g_no]$, which   contains $(r,F)$--barriers $X_n$ and $Y_n$ respectively. It is necessary that $X_n\ne Y_n$, otherwise  the $[2/5,3/5]$--interval of $[o,g_no]$ is contained in $N_C(X_n)$; such set of elements are in $A$ excluded at the very beginning.

Let $P\subseteq \isom(\U,d)$ be a  compact confining subset for $H$, so $D:=\max\{d(o,po): p\in P\}<\infty$. 

Choose $p_n\in P$ so that $h_n:=g_n p_n g_n^{-1}\in H$. Now, let $z\in [o,\xi]$ so that $d(g_no, z)\le r$. Thus, $d(h_no,z) \le r+D+d(o,g_no)\le 2r+D+d(o,z)$. According to  Lemma \ref{SameHoroball}, we see that  $h_n o \in Ho$ lies in the horoball $\mathcal {HB}(\xi,o, 2r+D+\epsilon)$.  

We shall prove that   $\{h_no: n\ge 1\}$ is an infinite (discrete) subset, which thus converges to $\xi$ by Lemma \ref{HoroballUniqueLimit}. Arguing by contradiction,  assume now that $\{h_no: n\ge 1\}$ is a finite set. By taking a subsequence,   the proper action of $H$ allows us to assume that $h:=h_n=h_m$ for any $n, m\ge 1$.

\textbf{Case 1}. {$P$ is finite.} We may then assume  that $p_n=p$ for all $n\ge 1$. We thus obtain $g_m^{-1}g_n h=h g_m^{-1}g_n$. By Lemma \ref{TightContractingTransition}, $g_m^{-1}g_n$ is a contracting element for any two $n\ne m\ge 1$. If $h$ is of infinite order, $h\in E(g_m^{-1}g_n)$ gives a contradiction, as  any infinite order element in $E(g_m^{-1}g_n)$ is a contracting element. 

\textbf{Case 2}.  {$P$ may be an infinite compact set, but $\U$ is assumed to be hyperbolic or geodesically complete CAT(0).}  Note  that  $p_n$ might be distinct in general.  As $H$ is torsion-free, $p_n$ must be of infinite order. According to classification of isometries, $p_n$ is either hyperbolic or parabolic.    

As $d(hg_no,g_no)=d(o,p_no)<D$, the convergence $g_no\to [\xi]$ implies that $hg_no\to [\xi]$ and then $h$ fixes $[\xi]$.  So $\xi\in \Lambda_0$ gives a contradiction. The proof is complete. 
%\textbf{Case 2.} $p_n$ is a contracting element.  By the choice of $A$, the second half of $[o, g_no]$ contains an   $(r,f)$--barrier $t_n$ in any $\theta$--percentage of it. That is to say, we have $d(tfo, [o,g_no]), d(t_no, [o,g_no])\le r$ and $d(o,t_no)>d(o, g_no)/2$.  
%The contracting property implies that the barrier is close to $[o, h_no]$. Consequently, $h_no\ne h_mo$ for $n\gg m$. Thus, $\{g_no: n\ge 1\}$ is infinite.   The proof is complete.
\end{proof}

In the following statement, we emphasize that $H<\isom(\U)$ is not necessarily a discrete subgroup.  The proof relies on Lemma \ref{GoodConfiningSetP} applied with $G=\isom(\U)$ here, which does not use the proper action $H\act\U$ as well.
Compared with Theorem \ref{ConInHorLimitSet}, $H$ may contain torsion elements, but $P\cap E(\Gamma)$ is assumed to be trivial.  
\begin{thm}\label{ConInHorLimitSet2}
Assume that $H<\isom(\U)$ is a subgroup confined by $\Gamma$ with a finite confining subset $P$ that intersects trivially $E(\Gamma)$.  Then  $\HG$  contains $[\cG]$ as a subset. Moreover, $\Lambda(\Gamma o)\subseteq [\Lambda(H o)]$. 
\end{thm}
 
\begin{proof}
By definition of an $(r,F)$--conical point, there exists  a sequence of $(r, F)$--barriers $g_n \in \Gamma$ on a geodesic ray $\gamma=[o,\xi]$. By definition, we have $d(g_no, [o,\xi])\le r$.

By Lemma \ref{GoodConfiningSetP}, we can choose $f_n\in F$ and $p_n\in P$ so that $h_n:=g_n f_n p_n f_n^{-1} g_n^{-1}\in H$ labels an $(L,\tau)$--admissible path.  If $d(g_no, g_mo)\gg 0$ for any $n\ne m$, then $h_no\ne h_mo$ by Lemma \ref{InjectiveExtMap}. Thus $\{h_no: n\ge 1\}$ is an infinite  subset. Setting $$D=\max_{f\in F}\{d(o,fo)\}+\max_{p\in P}\{d(o,po)\}$$ we argue exactly as in the proof of Theorem \ref{ConInHorLimitSet} and obtain  that  $h_n o \in Ho$ lies in the horoball $\mathcal {HB}([\xi], 2r+2D+\epsilon)$ (the main issue there was proving the infiniteness of $\{h_no:n\ge 1\}$). Hence,  $\{h_no: n\ge 1\}$  converges to $[\xi]$ by Lemma \ref{HoroballUniqueLimit}, so $\xi$ is a big horospheric limit point. 
\end{proof}
 
\begin{cor}\label{FullLimitSet}
In the setting  of Theorems \ref{ConInHorLimitSet} or  \ref{ConInHorLimitSet2}, if $H$ is a subgroup of $\Gamma$, then  $[\pG] = [\Lambda (Ho)]$.    
\end{cor}
\begin{proof}
Under the assumptions of Theorem \ref{ConInHorLimitSet}, we proved that a $\mu_o$--full subset $\Lambda$ of $[\cG]$ is contained in the horospheric limit set $\HG$. Taking a countable intersection $\Gamma:=\cap_{g\in\Gamma} g\Lambda$ allows to assume that $\Lambda$ is $\Gamma$--invariant. If $\U$ is hyperbolic or CAT(0),  it is well-known that the limit set $\Lambda(\Gamma o)$ is a $\Gamma$--invariant minimal subset. Hence, the topological closure of $\Lambda$ recovers $\Lambda(\Gamma o)$, thus verifying $\Lambda(\Gamma o)\subseteq \Lambda(Ho)$, so the proof is completed in this case.  

Otherwise, $P$ is a finite set by assumption. As $H<\Gamma$ is assumed, the conclusion follows immediately from  Lemma \ref{ConfinedNonElementary}.   
\end{proof}

To conclude this subsection, we record a stronger statement, provided that $H$ is a normal subgroup. The proof strategy is due to \cite{FM20}. 

\begin{thm}
Suppose that $H$ is an infinite normal subgroup of $\Gamma$. Then $\hG$ contains $\mG$ as a subset.   
\end{thm}
\begin{proof}
Given $\xi\in \mG$, let $\gamma$ be a geodesic ray starting at $o$ and ending at $[\xi]$. Let us  take a contracting element $f\in H$, as  an infinite normal subgroup contains infinitely many ones. By definition of $\mG$ in (\ref{MyrbergDefn}),  $\gamma$ contains infinitely many  $(r, f^n)$--barriers $g_n$ for any $n\ge 1$.  By normality of $H$, we have $g_nf^ng_n^{-1}\in H$ forms a sequence of elements, which enters any given horoball based on $[\xi]$. With Lemma \ref{HoroballUniqueLimit}, this implies that $\mG\subseteq [\hG]$.
\end{proof}

\part{Growth   inequalities for confined subgroups}
\section{Shadow Principle for confined subgroups}\label{secshadow}
We shall establish in this section a Shadow Principle for confined subgroups, which is crucial for the next two sections \textsection \ref{section:cogrowthTight} and \ref{secinequality}. The basic setup is as follows. 
\begin{itemize}
    \item 
    The auxiliary proper action $\Gamma\act \U$ is assumed to be  of divergence type with contracting elements.
    \item 
    Let $\pU$ be a convergence boundary for $\U$ and $\{\mu_x:x\in \U\}$ be a  quasi-conformal, $\Gamma$--equivariant density of dimension $\e \Gamma$ on $\pU$.
    \item 
    Let $H$ be a discrete subgroup of a group $G<\isom(\U)$,  confined by $\Gamma$, with a finite confining subset $P$ that intersects trivially $E(\Gamma)\setminus \{1\}$.  
\end{itemize}
Recall that $E(\Gamma)$ is the $[\cdot]$--class stabilizer of $[\pG]$ in $\isom(\U)$ (Def. \ref{EllipticRadicalDefn}). %We may assume $P$ is minimal: for any $p\in P$, there exists $g\in  G$ so that $gpg^{-1}\in H$. Consequently, $P\cap E_G(\Gamma)\subseteq \{1\}$

%The condition $H\cap E(\Gamma)\ne \{1\}$  could not be removed, as the direct product of any subgroup with a finite normal subgroup fails to have this property.

In this section, let $F\subseteq \Gamma$ be a set of three independent contracting elements, and the constants $L, \tau, r$ given by Lemma \ref{EllipticRadical}.

\subsection{Shadow Principle for conformal density supported on boundary}

Below, the constant $\lambda$ is given   in Definition \ref{ConformalDensityDefn}, and $ \epsilon$ is the convergence error of the horofunction in \ref{AssumpE}. Note that  $\epsilon=0$ for the horofunction boundary. By Corollary \ref{FullLimitSet}, we have $[\Lambda (Ho)=[\Lambda (\Gamma o)]$. 

\begin{lem}[Shadow Principle]\label{ShadowPrinciple}
Let $\{\mu_x\}_{x\in \U}$ be a $\e H$--dimensional $H$--quasi-equivariant quasi-conformal density supported on $\pU$.   Then there exists $r_0 > 0$ such that  
$$
\begin{array}{rl}
\|\mu_y\| e^{-\e H \cdot d(x, y)} \quad \prec_\lambda   \quad \mu_x(\Pi_{x}^F(y,r))\quad  \prec_{\lambda,\epsilon, r} \quad \|\mu_y\| e^{-\e H \cdot  d(x, y)}\\
\end{array}
$$
for any $x,y\in \Gamma o$ and $r \ge  r_0$.
\end{lem}

\begin{proof}
We start by proving the lower bound, which is the key part of the proof.
Write explicitly $x=g_1o$ and $y=g_2o$ for $g_1,g_2\in \Gamma$.  Let $g=g_1^{-1}g_2$. 

\textbf{1. Lower Bound.}  Fix a Borel subset $U:=[\Lambda(Ho)]$ such that $\mu_o(U)>0$. For given $\xi \in U$,  there exists a sequence of points $h_n o\in \U$ such that $h_no \to \xi$. Since $F$  is a finite set,   up to taking a subsequence of $h_n o$, there exist $f\in F$ and $r>0$ given by Lemma \ref{EllipticRadical} such that for any $p\in P$, $g$ is an $(r, f)$--barrier for any geodesic $[o, gfpf^{-1}h_no]$. This implies $gfpf^{-1}h_no\in \Omega_o^F(go,  r)$, which tends to $gfpf^{-1}\xi$. In addition, we can choose $p_f\in P$ according to definition of the confined subgroup $H$  so that $g_2fp_f(g_2f)^{-1}\in H$. By definition of the shadow, 
\begin{equation}\label{TransferEQ}
gfp_ff^{-1}\xi\in \Pi_o^F(go,  r)
\end{equation} 
Note that  $f\in F$   depends on $\xi\in U$, and $F$ consists of three elements.

Consequently,  the set $U$ can be decomposed as a disjoint  union of three sets $U_1, U_2, U_3$: for each $U_i$, there exist $f_i\in F, p_i\in P$ such that $gf_ip_if_i^{-1}U_i \subseteq  \Pi_o^F(go, r)$; equivalently,
\begin{equation}\label{MoveSetEQ}
g_2f_ip_if_i^{-1}U_i\subseteq \Pi_{g_1o}^F(g_2o,  r).
\end{equation}

Denote $L=\max\{2d(o,fo)+d(o,po):f\in F, p\in P\}<\infty$. As $d(o,f_ip_if_i^{-1}o)\le L,$ there exists    $\theta=\theta(L,\e H)>0$ by quasi-conformality (\ref{confDeriv}) such that 
\begin{equation}\label{BddDiffEQ}
\mu_{g_2o} (g_2f_ip_if_i^{-1}U_i)\ge \theta \cdot \mu_{g_2f_ip_if_i^{-1}o} (g_2f_ip_if_i^{-1}U_i).
\end{equation}
\begin{figure}
    \centering

\tikzset{every picture/.style={line width=0.75pt}} %set default line width to 0.75pt        

\begin{tikzpicture}[x=0.75pt,y=0.75pt,yscale=-1,xscale=1]

%Curve Lines [id:da3594502423086745] 
\draw    (122.5,64) .. controls (279.71,110.77) and (465.63,87.24) .. (554.17,93.9) ;
\draw [shift={(555.5,94)}, rotate = 184.55] [color={rgb, 255:red, 0; green, 0; blue, 0 }  ][line width=0.75]    (10.93,-3.29) .. controls (6.95,-1.4) and (3.31,-0.3) .. (0,0) .. controls (3.31,0.3) and (6.95,1.4) .. (10.93,3.29)   ;
\draw [shift={(122.5,64)}, rotate = 16.57] [color={rgb, 255:red, 0; green, 0; blue, 0 }  ][fill={rgb, 255:red, 0; green, 0; blue, 0 }  ][line width=0.75]      (0, 0) circle [x radius= 3.35, y radius= 3.35]   ;
%Shape: Ellipse [id:dp6011555735072158] 
\draw   (240.6,98.45) .. controls (243.68,78.59) and (272.32,66.54) .. (304.57,71.54) .. controls (336.83,76.54) and (360.48,96.69) .. (357.4,116.55) .. controls (354.32,136.41) and (325.68,148.46) .. (293.43,143.46) .. controls (261.17,138.46) and (237.52,118.31) .. (240.6,98.45) -- cycle ;
%Straight Lines [id:da3725401883695221] 
\draw    (268.5,113) -- (329.5,114) ;
\draw [shift={(329.5,114)}, rotate = 0.94] [color={rgb, 255:red, 0; green, 0; blue, 0 }  ][fill={rgb, 255:red, 0; green, 0; blue, 0 }  ][line width=0.75]      (0, 0) circle [x radius= 3.35, y radius= 3.35]   ;
\draw [shift={(268.5,113)}, rotate = 0.94] [color={rgb, 255:red, 0; green, 0; blue, 0 }  ][fill={rgb, 255:red, 0; green, 0; blue, 0 }  ][line width=0.75]      (0, 0) circle [x radius= 3.35, y radius= 3.35]   ;
%Straight Lines [id:da34064240081292274] 
\draw    (268.5,113) -- (266.76,92.99) ;
\draw [shift={(266.5,90)}, rotate = 85.03] [fill={rgb, 255:red, 0; green, 0; blue, 0 }  ][line width=0.08]  [draw opacity=0] (8.93,-4.29) -- (0,0) -- (8.93,4.29) -- cycle    ;
%Straight Lines [id:da8099045489207108] 
\draw    (330.5,113) -- (330.5,95) ;
\draw [shift={(330.5,92)}, rotate = 90] [fill={rgb, 255:red, 0; green, 0; blue, 0 }  ][line width=0.08]  [draw opacity=0] (8.93,-4.29) -- (0,0) -- (8.93,4.29) -- cycle    ;
%Straight Lines [id:da052635960001782056] 
\draw    (122.5,64) -- (249.09,190.59) ;
\draw [shift={(250.5,192)}, rotate = 225] [color={rgb, 255:red, 0; green, 0; blue, 0 }  ][line width=0.75]    (10.93,-3.29) .. controls (6.95,-1.4) and (3.31,-0.3) .. (0,0) .. controls (3.31,0.3) and (6.95,1.4) .. (10.93,3.29)   ;
%Straight Lines [id:da44129060446208546] 
\draw    (329.5,114) -- (378.5,115.92) ;
\draw [shift={(380.5,116)}, rotate = 182.25] [color={rgb, 255:red, 0; green, 0; blue, 0 }  ][line width=0.75]    (10.93,-3.29) .. controls (6.95,-1.4) and (3.31,-0.3) .. (0,0) .. controls (3.31,0.3) and (6.95,1.4) .. (10.93,3.29)   ;
%Straight Lines [id:da1341116489879124] 
\draw    (380.5,116) -- (441.5,117) ;
\draw [shift={(441.5,117)}, rotate = 0.94] [color={rgb, 255:red, 0; green, 0; blue, 0 }  ][fill={rgb, 255:red, 0; green, 0; blue, 0 }  ][line width=0.75]      (0, 0) circle [x radius= 3.35, y radius= 3.35]   ;
\draw [shift={(380.5,116)}, rotate = 0.94] [color={rgb, 255:red, 0; green, 0; blue, 0 }  ][fill={rgb, 255:red, 0; green, 0; blue, 0 }  ][line width=0.75]      (0, 0) circle [x radius= 3.35, y radius= 3.35]   ;
%Straight Lines [id:da8071275378353184] 
\draw    (441.5,117) -- (510.53,104.36) ;
\draw [shift={(512.5,104)}, rotate = 169.62] [color={rgb, 255:red, 0; green, 0; blue, 0 }  ][line width=0.75]    (10.93,-3.29) .. controls (6.95,-1.4) and (3.31,-0.3) .. (0,0) .. controls (3.31,0.3) and (6.95,1.4) .. (10.93,3.29)   ;
%Curve Lines [id:da16751380092554813] 
\draw    (357.5,164) .. controls (377.45,153.55) and (370.33,135.88) .. (364.42,118.71) ;
\draw [shift={(363.5,116)}, rotate = 71.57] [fill={rgb, 255:red, 0; green, 0; blue, 0 }  ][line width=0.08]  [draw opacity=0] (8.93,-4.29) -- (0,0) -- (8.93,4.29) -- cycle    ;

% Text Node
\draw (261.28,115.91) node [anchor=north west][inner sep=0.75pt]  [rotate=-1.93]  {$go$};
% Text Node
\draw (312.29,117.42) node [anchor=north west][inner sep=0.75pt]  [rotate=-1.93]  {$gf o$};
% Text Node
\draw (134,42.4) node [anchor=north west][inner sep=0.75pt]    {$o$};
% Text Node
\draw (402,62.4) node [anchor=north west][inner sep=0.75pt]    {$\Pi _{o}^{F}( g o,r) \ni \ gfpf^{-1} h_{n} \xi $};
% Text Node
\draw (270.5,98.4) node [anchor=north west][inner sep=0.75pt]    {$\leq r$};
% Text Node
\draw (329.5,99.4) node [anchor=north west][inner sep=0.75pt]    {$\leq r$};
% Text Node
\draw (254.5,189.4) node [anchor=north west][inner sep=0.75pt]    {$h_{n} o\rightarrow \xi \in U$};
% Text Node
\draw (468.5,115.4) node [anchor=north west][inner sep=0.75pt]    {$gfpf^{-1} h_{n} o$};
% Text Node
\draw (395,99.4) node [anchor=north west][inner sep=0.75pt]    {$f^{-1}$};
% Text Node
\draw (347,164.4) node [anchor=north west][inner sep=0.75pt]    {$p\in P$};

\end{tikzpicture}
    \caption{Upper bound in Shadow Principle: transporting a large set $U$ into the shadow by Lemma \ref{EllipticRadical}.}
    \label{fig:shadowprinciple}
\end{figure}
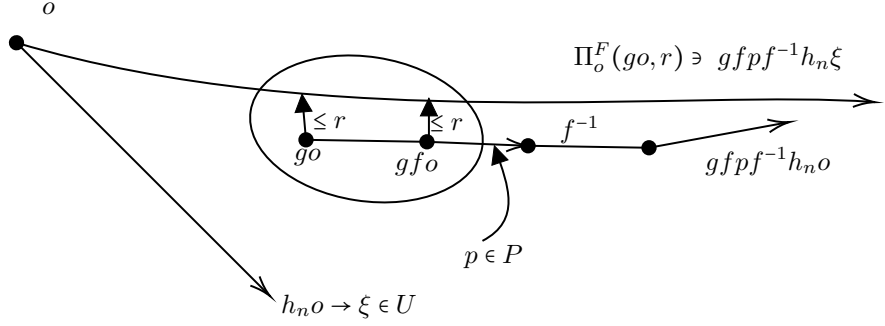
We apply the $H$--quasi-equivariance (\ref{almostInv})  with $g_2f_ip_i^{-1} f_i^{-1}g_2^{-1}\in H$ to the right-hand side:  
\begin{equation}\label{GEquivEQ}
\mu_{g_2f_ip_if_i^{-1}o} (g_2f_ip_if_i^{-1}U_i) \ge \lambda^{-1} \cdot  \mu_{g_2o} (g_2U_i).  
\end{equation} 
Combining together (\ref{MoveSetEQ}), (\ref{BddDiffEQ}) and (\ref{GEquivEQ}), we have
\[
\begin{aligned}
\mu_{g_2o}(\Pi_{g_1o}^F(g_2o, r))&\ge \displaystyle\sum_{1\le i\le 3} \mu_{g_2o} (g_2U_i)/3 \\
&\ge \lambda^{-1}\theta\cdot\mu_{g_2o}(U)/3    
\end{aligned} 
\]
where the last line uses the $H$--invariance of $U=[\Lambda(Ho)]=[\Lambda(\Gamma o)]$ by Lemma \ref{ConfinedNonElementary}.

Denote $M:=  \lambda^{-2}\cdot \theta/3$. We now conclude the proof for the lower bound via  quasi-conformality (\ref{confDeriv}): 
$$\begin{array}{rl}
\mu_{g_1o}(\Pi_{g_1o}^F(g_2o,   r)) &\ge \lambda^{-1}  \cdot e^{-\omega\cdot d(g_1o, g_2o)} \cdot \mu_{g_2o}( \Pi_{g_1o}^F(g_2o, r)) \\
\\
&\ge M \cdot \mu_{g_2o}(U) \cdot e^{-\omega\cdot d(g_1o, g_2o)}
\end{array}
$$

\textbf{2. Upper Bound.} Fix $r\ge r_0$.
Given $\xi \in   \Pi_{g_1o}^F(g_2o,r)$, there is a sequence of   $z_n\in \U$ tending to $\xi$ such that $\gamma_n \cap B(g_2o, r)
\neq \emptyset$ for $\gamma_n:=[g_1^{-1}o, z_n]$.  Since Buseman cocyles extend continuously to the horofunction boundary, we obtain  $$|B_\xi(g_1o, g_2o)
- d(g_1o, g_2o))| \le 2r.$$
The upper bound is given as follows:
$$
\begin{array}{rl}
\mu_{g_1o}(\Pi_{g_1o}^F(g_2o,r)) 
&\le     \displaystyle \int_{\Pi_{g_1o}^F(g_2o,r)}e^{-\omega B_\xi(g_1o, g_2o)} d \mu_{g_2o}(\xi)  \\
\\
& \le  \lambda e^{2\omega r} \|\mu_{g_2o}\| \cdot e^{-\omega d(g_1o, g_2o)}
\end{array}
$$
The proof of lemma is complete.
\end{proof}

%\ywy{Great. I wrote mine, see Lemma \ref{ShadowPrinciple2}. Probably it is the same idea:) I read your argument and agreed with yours. Let us merge them together.}
%{\color{blue}I'll try another computation... Again, the setting is: $G$ acts on $Y$ properly with a contracting element, and $H$ is a torsion-free confined subgroup.
\subsection{Shadow Principle for conformal density supported inside}
Fix a basepoint $o\in \U$ and $\omega>\e H$. Given $x\in \U$, define  $\mu_x=\frac{1}{\p_{H}(\omega,o,o)} \sum_{h\in H} \mathrm{e}^{-\omega d(x,ho)}$  supported on $Ho$, we have 
$$
\forall z\in Ho:\quad \frac{d\mu_x}{d\mu_y}(z) = \mathrm{e}^{-\omega [d(x,z)-d(y,z)]}
$$

%The key trick is, if $f \in g Hg^{-1}$ is a contracting element, then $fg H = gH$. Hence, we can shift elements in $gH$ toward the direction of $f$ (in the beginning!) and the membership in $gH$ does not change. 

\begin{lem}\label{lem:shadowBaby}
There are   constants $\theta, D > 0$ such that, given any $\Delta>0$ and, $g_1, g_2 \in G$ we have
\[
\# H \cap A_\Gamma(g_2o, n, \Delta) \le \theta \#  H \cap A_\Gamma(g_2o, n, \Delta + D) \cap \Omega_{g_1o}^F(g_2o,D).
\]
\end{lem}

\begin{proof}
\iffalse
Using Lemma \ref{UniformContractingElem}, we can pick a finite set $F'$ of independent contracting elements of $G$ such that: \begin{enumerate}
    \item for each $g_{1} \in G$ there exists a finite subset $F = \{f_{1}, \ldots, f_{10}\}$ contained in $g_{1}^{-1}Hg_{1} \subseteq F'$;
    \item for each $g, g' \in G$ there exists at least 8 elements $f \in F$ such that $([go, o], f[o, g'o])$ is an admissible path.
\end{enumerate}
($M$ depends on the max of the norms of $\|f_{i}\|$'s and their contracting powers over $F$, which does not depend on the choice of $g_{1}$)\fi

Let $h \in H$ and $g:=g_1^{-1}g_2$. According to Lemma \ref{EllipticRadical}, for given $g$ and $g_2^{-1}h$, we pick  $f_{h} \in F$     such that  for any $p\in P$, the word $(g, f_h, p, f_h^{-1}, g_2^{-1}h)$ labels  an $(L,\tau)$--admissible path, where $go$ is an $(r,F)$--barrier for $f_h\in F$. This implies $g f_h p f_h^{-1} g_2^{-1}h\in \Omega_{o}^F(go,r)$, or equivalently $\mathbf{h}:=g_2 f_h p f_h^{-1} g_2^{-1}h\in \Omega_{g_1o}^F(g_2o,r)$. %; moreover, 
%\begin{equation}\label{eqn:shadowBaby2}
%|d(o, gf_hpf_h^{-1} g^{-1}h o) -  d(o, go) -d(go, h o)|\le D
%\end{equation}
%where $D$ depends only on $F$ and $P$.
As $p\in P$ can be arbitrarily chosen, the definition of confined subgroup $H$ allows us to choose    $p_h\in P$ for the element $g_2^{-1}h\in G$  so that $h^{-1}g_2f_h\cdot p_h\cdot (f_h^{-1}g_2^{-1}h)$ lies in $H$ and so $\mathbf{h}=g_2f_hp_hf_h^{-1}g_2^{-1}h$ lies in $H$. 

Set $D=\max\{d(o,fo)+d(o,po): f\in F, p\in P\}$.
If $d(ho, g_2o)\in [n-\Delta, n+\Delta]$ we have $d(\mathbf{h}o, g_2o)\in [n-\Delta-D, n+\Delta+D]$.
We claim that this assignment 
$$
\begin{aligned}
\pi:H \cap A_\Gamma(g_2o, n, \Delta)&\longrightarrow H \cap A_\Gamma(g_2o, n, \Delta + D) \cap \Omega_{g_1o}^F(g_2o,r) \\
 h &\longmapsto \mathbf{h}
\end{aligned}
$$ is uniformly finite to one. Indeed, assume $\mathbf{h}=\mathbf{h'}$ for some $h\ne h'$.  Applying multiplication by $g_1^{-1}$ on the left  for $\mathbf{h}=\mathbf{h'}$ gives    $gf_hp_hf_h^{-1}g_2^{-1}h=gf_{h'}p_{h'}f_{h'}^{-1}g_2^{-1}h'$. As mentioned above, these two words label two $(L,\tau)$--admissible paths with the same endpoints. Noting that $g^{-1}h, g^{-1}h'\in A_\Gamma(o,n,\Delta)$, Lemma \ref{InjectiveExtMap} implies $d(g^{-1}ho,g^{-1}h'o)=d(ho,h'o)\le R$. Consequently, at most $N:=\sharp \{h\in H: d(o,ho)\le R\}$ elements have the same image under the map $\pi$. Setting $\theta=1/N$ thus completes the proof. 
\end{proof}

The following is the Shadow Principle we want for all $s > w_{H}$.

\begin{lem}\label{lem:shadowLemCogrowth}
Fix any $\omega> \e H$.
Let $\{\mu_x\}_{x\in \U}$ be a $\omega$--dimensional $H$--quasi-equivariant quasi-conformal density supported on $Ho\subseteq \U$.   Then there exists $r_0 > 0$ such that  
$$
\begin{array}{rl}
\|\mu_y\| e^{-\omega \cdot d(x, y)} \quad \prec_\lambda   \quad \mu_x(\Omega_{x}^F(y,r))\quad  \prec_{\lambda,\epsilon, r} \quad \|\mu_y\| e^{-\omega \cdot  d(x, y)}\\
\end{array}
$$
for any $x,y\in Go$ and $r \ge  r_0$. 
\end{lem}

\begin{proof}
If $z\in \Omega_x^F(y,r)$ we have $d(x,z)+2r\ge d(x,y)+d(y,z)$, so  $\mu_x(\Omega_{x}^F(y,r))  \le \mathrm{e}^{2r\omega} \|\mu_y\| e^{-\omega \cdot  d(x, y)}$. Thus, it remains to prove the lower bound. As the triangle inequality implies  $\mu_x(\Omega_x^F(y,r))\ge e^{-\omega d(x, y)}\mu_y(\Omega_x^F(y,r))$, it suffices to show $\mu_y(\Pi_x^F(y,r))\ge \Theta \|\mu_y\|$ for some $\Theta>0$.

By definition, 
$$
\p_{H}(\omega,o,o)\|\mu_y\|=\sum_{h\in H} \mathrm{e}^{-\omega d(y,ho)} <\infty
$$
Assume that $x=g_1o, y=g_2o$, and set $g=g_1^{-1}g_2$ for simplicity. Note that  $\Omega_x^F(y,r)=g_1\Omega_o^F(go,r)$. We remark, however, that $\mu_y(\Omega_x^F(y,r))=\mu_{go}(\Omega_o^F(go,r))$ may not hold, as $\mu_x$ is only $H$--conformal density but $g_1$ may not be in $H$.

By Lemma \ref{lem:shadowBaby}, we have \[\begin{aligned}
\sum_{h \in H} e^{-s d(g_2o, ho)} &\le C_{1}(s) \sum_{n =0}^{\infty} e^{-s n} \#  H \cap A(g_2o, n, \Delta) \\
&\le \theta C_{1}(s) \sum_{n =0}^{\infty} e^{-s n} \cdot\# \left( H \cap A(g_2o, n, \Delta + D) \cap \Omega_{g_1o}^F(g_2o,r) \right) \\
&\le \theta C_{1}(s) C_{2}(s) \sum_{h \in  H \cap \Omega_{g_1o}^F(g_2o,r)} e^{-s d(g_2o, ho) }.
\end{aligned}
\]
Here, $C_{1}(s) = \Delta \cdot e^{-s \Delta}, C_{2}(s) = (\Delta + D) \cdot e^{-s (\Delta + D)}$. Setting $\Theta=\theta C_{1}(s) C_{2}(s)$, this is  equivalent to the following
$$
\p_{H}(\omega,o,o)\mu_y(\Omega_x^F(y,r)) =\sum_{h\in \Omega_x^F(y,r)} \mathrm{e}^{-\omega d(y,ho)} \ge \Theta\sum_{h\in \Omega_x^F(y,r)} \mathrm{e}^{-\omega d(y,ho)} =  \p_{H}(\omega,o,o)\|\mu_y\|
$$ 
Hence, the shadow lemma is proved.
\end{proof}

\section{Cogrowth tightness of confined subgroups}\label{section:cogrowthTight}
%\ywy{It seems some more powerful ingredients could be proved in Lemma \ref{GoodConfiningSetP}. This allows to prove cogrowth tightness for any confined subgroup with contracting elements in Theorem \ref{mainthm} in a way very similar to the case of normal subgroups (but still no claim on the equality in divergence case). Moreover, we can recover Coulon's inequality via desired shadow principle. I will write the details soon. This seems too good to be true. But let us check it carefully!}

We continue the setup of \textsection\ref{secshadow}, but with the first item replaced with a stronger assumption 
\begin{itemize}
    \item The proper action $\Gamma\act \U$ has purely exponential growth.
\end{itemize}

%\ywy{I made some more assumptions so that torsion-free can be removed. The only assumption is that confined subgroup intersects trivially the elliptic radical. More derails shall be coming soon (hopefully). Lemma \ref{GoodConfiningSetP} is strengthened, but the proof has not been updated. }

%\begin{proof}
%First of all, it contains a contracting element. 
%If $H$ is virtually cyclic, then $H$ contains a finite index subgroup  that is a confined cyclic subgroup. By Corollary \ref{FullLimitSet}, its limit set coincides with $\Lambda G$. This is a contradiction, as $H$ generated by a contracting element, whose limit set consists of two fixed points.      
%\end{proof}

%We start with some preliminary results.

We first give a uniform upper bound for all conjugates of a confined subgroup $H$.
\begin{lem}\label{ExpUpperBound}
There exist constants $C,\Delta>0$ so that the following holds 
$$
\sharp (gHg^{-1}\cap A_\Gamma(o,n,\Delta))\le C \mathrm{e}^{n\e H}
$$
for any $g\in G$.
\end{lem}
\begin{proof}
By Lemma \ref{UniformContractingElem}, there exists a finite set of contracting elements $\tilde F$ such that   any $gHg^{-1}$ contains  a set, denoted by $F_g$, of three pairwise independent contracting elements  from $\tilde F$. Applying Lemma \ref{extend3} with $F_g$, the conclusion follows by the same argument as in \cite[Proposition 5.2]{YANG10}, where $C$ does not depend on $g$ since there are only finitely many choices of $F_g\subseteq \tilde F$ independent of $g$. We include the sketch below and refer to \cite[Proposition 5.2]{YANG10} for full details.    

Set $S_n:=gHg^{-1}\cap A_\Gamma(o,n,\Delta)$.  The idea is to prove the following for any $n, m\ge 1$
$$
\sharp S_n \sharp S_m\le \theta \sum_{-k\le j\le k}\sharp S(n+m+j)
$$
where $\theta, k$ do not  depend on $g$. This is proved by  considering the map 
$$
\begin{aligned}
\pi:\quad S_n \times S_m &\quad \longrightarrow \quad gHg^{-1}\\
(a,b)&\quad \longmapsto\quad  afb
\end{aligned}
$$
where $f\in F_g$ is chosen by Lemma \ref{extend3}. In particular, $afb$ labels admissible path so $|d(o,afb)-d(o,ao)-d(o,bo)\le k$ where $k$ depends on $d(o,fo)$. We proved there (also see Lemma \ref{InjectiveExtMap}) that $\pi$ fails to be injective, only if $d(ao,bo)$ is greater than a constant depending only on $F_g$. This constant determines the value of $\theta$, so the proof of the above inequality is proved.
\end{proof}

%{\color{blue}Inhyeok: Lemma 8.2 and 8.3 seem good to me}

\begin{lem}\label{ExpLowerBound}
There exist constants $C,\Delta>0$ with the following property. For any $g\in \Gamma$ and any $n\ge 1$, we have 
$$
\sharp (gHg^{-1}\cap A_\Gamma(o,n,\Delta)) \ge C \mathrm{e}^{n\e \Gamma /2}
$$ 
\end{lem}
\begin{proof}
Let us consider the conjugate $g_0Hg_0^{-1}$ for given $g_0\in \Gamma$. 
Define a map as follows
$$
\begin{aligned}
\pi_{g_0}:\quad  A_\Gamma(g_0^{-1}o,n,\Delta)&\quad  \longrightarrow \quad  g_{0} H g_{0}^{-1}  \\
 g&\quad \longmapsto \quad  g_0gfp(g_0gf)^{-1}
\end{aligned}
$$ 
where $f$  and $p$  are chosen for $g_0g$  according to Lemma \ref{GoodConfiningSetP}.  Hence, $|d(o,\pi_{g_0}(g)o)-2d(o,g_0go)|\le D$ so we obtain $\pi_{g_0}(g)\in A_\Gamma(o,2n,\Delta')$ where $\Delta':=2\Delta+D$.

We shall prove that $\pi_{g_0}$ is uniformly finite to one. That is to say, there exists a constant $N>1$ independent of $g_0$ and $g\in H$ so that $\pi_{g_0}^{-1}(g)$ contains at most $N$ elements. 

Indeed, if $\pi_{g_0}(g)=\pi_{g_0}(g')$ then $g_0gfp(g_0gf)^{-1}=g_0g'fp(g_0g'f)^{-1}$. According to Lemma \ref{GoodConfiningSetP}, the  words $(g_0g,f,p,f^{-1},(g_0g)^{-1})$ and $(g_0g',f,p,f^{-1},(g_0g')^{-1})$ label respectively two  $(L,\tau)$--admissible paths with the same endpoints. Thus, any geodesic $\alpha$ with the same endpoints $r$--fellow travels them, so  $g_0go$ and $g_0g'o$ have a distance  at most $r$ to $\alpha$, where $r$ is given by Proposition \ref{admisProp}.  For $g,g'\in A_\Gamma(g_0^{-1}o,n,\Delta)$, we have $|d(o,g_0go)-d(o,g_0g'o)|\le 2\Delta$. Hence, the $r$--closeness then implies that $d(g_0go,g_0g'o)\le 4r+2\Delta$. Thus, if $d(go, g'o)$ is  larger than a constant $R=R(r, \Delta)$ given by Lemma \ref{InjectiveExtMap},  we would obtain a contradiction.  The finite number $N:=\sharp \{go: d(o,go)\le R\}$ is hence the desired upper bound on the preimage of $\pi_{g_0}$, which is clearly uniform independent of $g_0\in \Gamma$. 

To conclude the proof, since $\Gamma\act \U$ is assumed to   have purely exponential growth, we have $\sharp A_\Gamma(o,n,\Delta) \ge C_0 \mathrm{e}^{n\e\Gamma}$ for some $C_0, \Delta$. Setting $C=C_0/N$, we have $\sharp (gHg^{-1}\cap A_\Gamma(o,n,\Delta')) \ge C \mathrm{e}^{n\e \Gamma /2}$. The lemma is proved. 
\end{proof}

An immediate corollary is as follows. If $\Gamma\act \U$ is SCC, this  could be also derived from   Theorems \ref{ConfinedConsThm} and \ref{HalfGrowthDisThm}, for more general confined subgroups with compact confining subsets. See Corollary \ref{Nonstrict}.
\begin{cor}\label{ConfinedDivergeAtHalf}
The Poincar\'e series associated to $H$ diverges at $s=\e \Gamma/2:$
$$\p_H(\e \Gamma/2, o,o)=\infty$$
In particular, $\e H \ge \e \Gamma/2$.
\end{cor}

Here is another corollary from the proof.
For any $g\in G$, there exists $p_g\in P$ such that $gp_gg^{-1}\in H$. This defines a map as follows:
$$
\pi_1: g\in G\longmapsto gp_gg^{-1}
$$

\begin{lem}\label{ConjugateDivSeries}
There exists   a subset $K\subseteq \Gamma$, on which the above map $\pi_1$ is injective, so that $\sum_{g\in K} \mathrm e^{-\e \Gamma d(o,go)}=\infty$. 
\end{lem}
\begin{proof}
The above defined map $\pi_1$ is exactly the one $\pi_{g_0}$ for $g_0=1$ in the proof of Lemma \ref{ExpLowerBound}. As $\pi_{g_0}$ is uniformly finite to one and $\Gamma\act \U$ is of divergent type, we can then find $K$ with the desired divergence property.    
\end{proof}
 
%The next results assumes in addition that $H$ contains contracting elements. It is not clear whether this assumption could be removed: whether there exists a confined subgroup without contracting elements (but with torsion-free confining subsets). Such examples do not exist, if $G$ acts properly on Gromov hyperbolic spaces. Indeed, such confined subgroups contains only  parabolic and elliptic elements, so the limit set has only one point. However, this is impossible by Corollary \ref{FullLimitSet}.
\begin{lem}\label{EqualPSSeries}
Suppose that  $\e H=\e \Gamma/2$. Then for any $s>\e H$ and $g\in \Gamma$, we have 
$$
\sum_{h\in H} \mathrm{e}^{-sd(o, ghg^{-1}o)} \asymp \sum_{h\in H} \mathrm{e}^{-sd(o, ho)}
$$
where the implicit constant does not depend on $g$.
\end{lem}
\begin{proof}
We can write for any $g\in \Gamma$ and any $s>\e H$,
$$
\sum_{h\in H} \mathrm{e}^{-sd(o, ghg^{-1}o)} \asymp \sum_{n\ge 1} \sharp (gHg^{-1}\cap A_\Gamma(o,n,\Delta)) \mathrm{e}^{-s n}
$$
where the implicit constant depends on the width $\Delta$ of the annulus, but is independent of $g$.

As $H$ contains contracting elements by Lemma \ref{ConfinedNonElementary}, we have the upper bound by Lemma \ref{ExpUpperBound}
$$
C \mathrm{e}^{n\e H}\le \sharp (gHg^{-1}\cap A_\Gamma(o,n,\Delta))\le C' \mathrm{e}^{n\e H}
$$
where the lower bound is given by Lemma \ref{ExpLowerBound}. The conclusion follows by substituting the growth estimates   into the above series. 
%\textcolor{magenta}{Why is the upper bound independent of $g$? Wenyuan's \cite[Theorem B]{YANG10} doesn't give a uniform $C'$ over all subgroups, does it? }\ywy{Thanks for noticing this!  It is indeed not known from there that $C'$ is uniform. After   discussion with Tianyi, the uniform $C'$ could be proved via the argument in \cite[Prop 5.2]{YANG10} with Inhyeok's observation of contracting elements $h_ih_j$. I will write the details to fill the gap. Hope it is ok.}
\end{proof}
\begin{comment}
{\color{blue} In fact, I think the proof of Lemma 8.1 can tell that $H$ contains a contracting element. We just take 10 (the number is quite arbitrary) independent contracting elements $f_{1}$, $\ldots$, $f_{10} \in F$ and pick $p_{1}, \ldots, p_{10} \in P$ such that $h_{i} :=f_{i} p_{i} f_{i}^{-1} \in H$ and $f_{i} p_{i }f_{i}^{-1}$ labels an admissible path for each $i$. I think there exists $i, j \in \{1, \ldots, 10\}$ such that $f_{i}^{-1} f_{j}$ and $f_{j}^{-1} f_{i}$ are both admissible paths. Then $h_{i}h_{j}$ is a contracting element.
}
\ywy{Great to observe that $h_{i}h_{j}$ is contracting! This does not contradict my above confusing description. I intended to say that I have NO such examples in mind! It seems you proved that any confined subgroup with torsion-free confining subset contains contracting element. } {\color{blue} Aha Ok! Thanks for confirming}\ywy{I wrote Lemma \ref{UniformContractingElem} following your idea. Please check it.} {\color{blue}Thanks, this looks good!}
\end{comment}

\subsection{Confined subgroups of divergence type}

We say that a subset intersects trivially a subgroup if their intersection is contained in the trivial subgroup.

\begin{prop}\label{TightDivCase}
If  a discrete group $H<\isom(\U)$ of divergence type  is  confined by $\Gamma$ with a  finite confining subset $P$ that intersects trivially with $E(\Gamma)$, then $\e H>\e \Gamma/2$. 
\end{prop}

The proof   presented here is similar to the proof of \cite[Theorem 9.1(Case 1)]{YANG22}, where $H$ is assumed to be normal in $\Gamma$. We emphasize that $H$ is not necessarily to be contained in $\Gamma$. We need some preparation before starting the proof.

Fix $go\in \Gamma o$. Let $ \{\mu_x^{go}\}_{x\in\U} $ be a PS measure on the horofunction boundary $\bX$, which is an accumulation point of $\{\mu_{x}^{s,o}\}$ in (\ref{PattersonEQ}) supported on $Hgo$ for a sequence $s\searrow \e H$. This is a $\e H$--dimensional $H$--equivariant conformal density.   
 
As $H\act \U$ is of divergence type,  Lemma \ref{Unique} implies that push forwarding $ \{\mu_x^{go}\}_{x\in\U} $ for different $go\in \Gamma o$ almost gives the unique quasi-conformal density on the reduced horofunction boundary. That is to say, there are absolutely continuous with respect to each other,  with uniformly bounded  derivatives.  Now, the Patterson's construction  gives 
$$
\forall x\in \U:\quad \|\mu_x^{s, go}\| = \lim_{s\to \e H+} \frac{\p_H(s, x, go)}{\p_H(s, o, go)},\quad \|\mu_x^{s, o}\| = \lim_{s\to \e H+} \frac{\p_H(s, x, o)}{\p_H(s, o, o)}.
$$
Recall that $H^g=g^{-1}Hg$. From the definition of Poincar\'e series (\ref{PoincareEQ}), we have 
$$
\forall s>\e H:\quad \frac{\p_H(s, go, go)}{\p_H(s, o, go)} =  \Bigg[ \frac{\p_{H}(s, go, o)}{\p_{H^g}(s, o, o)}\Bigg]^{-1}
$$

On account of Lemma \ref{Unique}, we fix in   the following statement, a PS measure $ \{\mu_x\}_{x\in\U} $  on the horofunction boundary $\bX$. This is a limit point of $\{\mu_{x}^{s,o}\}$ in (\ref{PattersonEQ}) supported on $Ho$ for a sequence $s\searrow \e H$.
\begin{lem}\label{EqualMass}
If $\e H=\e \Gamma/2$, then there exists       a constant $M>0$ such that    
\begin{equation}\label{BoundedMassEQ}
\forall x\in \Gamma o: \quad \|\mu_x\|\le M    
\end{equation}
\end{lem}

\begin{proof}

By Lemma \ref{EqualPSSeries}, if $\e H=\e \Gamma/2$, then $\p_{H^g}(s, o, o)\asymp \p_{H}(s, o, o)$ for $s>\e H$. Hence, we have 
$$
\forall s>\e H:\quad \frac{\p_H(s, go, go)}{\p_H(s, o, go)} \asymp  \Bigg[ \frac{\p_{H}(s, go, o)}{\p_{H}(s, o, o)}\Bigg]^{-1}
$$
where the implicit constant does not depend on $g$.

For given $go\in \Gamma o$, let us take the same sequence of $s$ (depending on $go$) so that   $\{\mu_{x}^{s, go}\}_{x\in\U}$  and $\{\mu_{x}^{s, o}\}_{x\in\U}$  converge to an $H$--conformal density $ \{\mu_x^{go}\}_{x\in\U} $ and $H$--conformal density $ \{\mu_x^{o}\}_{x\in\U} $  respectively. The above relation then gives 
\begin{equation}\label{MassEqualityEQ}\|\mu_{go}^{go} \|\asymp \|\mu_{go}^o \|^{-1}
\end{equation}

%\begin{claim}
% $\{\tilde \mu_x:=\frac{1}{\|\nu_{g^{-1}o}^o\|}g_\star \nu_{g^{-1}x}^o: x\in \U\}$ is a $\e H$--dimensional $H$--quasiconformal density.   
%\end{claim}
%\begin{proof}[Proof of the claim]
%First of all,  we verify the $H$--equivariance: for any $h\in H$,
%$$\tilde \mu_{hx}(hA)=\nu_{g^{-1}hx}(g^{-1}hA)=  \nu_{g^{-1}hx}^o  (g^{-1}hA) =\nu_{g^{-1}x}^o(g^{-1}A) =\tilde \mu_{x}(A)$$
%where the third equality follows by the $H^g$--equivariance of $\nu_\star^o$ with $g^{-1}h^{-1}g$. The conformality is direct:
%$$
%\frac{d\tilde \mu_x}{d\tilde \mu_y} (\xi) =\frac{d \nu_{g^{-1}x}}{d \nu_{g^{-1}y}} (g^{-1}\xi) = \mathrm{e}^{-\e H B_\xi(x,y)}
%$$
%Note that $\|\tilde \mu_o\|=1$. The proof is complete.
%\end{proof}
Push forward the limiting measures $ \{ \mu_x^o\}_{x\in\U} $ and $ \{\mu_x^{go}\}_{x\in\U} $  to  the  quasi-conformal densities on the quotient $[\partial_m H]$, which we keep by the same notation.  The constant $\lambda$ in Definition \ref{ConformalDensityDefn} is universal for any $ \{\mu_x^{go}\}_{x\in\U} $ with $go\in \Gamma o$, as the difference of two horofunctions in the same locus of a Myrberg point is universal. Thus, Lemma \ref{Unique}  gives a constant $M_0=M(\lambda)$ independent of  $go\in \Gamma o$: $$M^{-1} \|\mu_{go}^{go}\| \le \| \mu_{go}^o\|\le M \|\mu_{go}^{go}\|$$
so   the relation (\ref{MassEqualityEQ}) implies $\|\mu_{go}^o\|^2\le M$.  

Recall that $\{\mu_x\}$ and $\{\mu_x^o\}$ may be different limit points of $\{\mu_{x}^{s, o}\}_{x\in\U}$, where $\{\mu_x\}$ is the PS measure fixed at the beginning of the proof. However, applying Lemma \ref{Unique} again gives  $\|\mu_{x}\|\asymp_\lambda \|\mu_{x}^o\|$ for any $x\in \U$ and thus for $x=go$.  Combinning with the above bound $\|\mu_{go}^o\|^2\le M$, we obtain the desired upper bound on $\|\mu_{go}\|$  in (\ref{BoundedMassEQ}) depending only on $\lambda, M_0$. The proof   is completed.
\end{proof} 

We now recall the following result, which says that a boundary point lies in a uniformly finite many  shadows from a fixed annulus.   
\begin{lem}\label{OverlapAnnulus}\cite[Lemma. 6.8]{YANG22}
For given $\Delta>0$, there exists $N=N(\Delta)$ such that  for every $n\ge 1$,  any $\xi\in \pU$ is contained in at most $N$   shadows $\Pi_{o}^F(v, r)$ where $v\in A_\Gamma(o,n,\Delta)$.
\end{lem}

We are ready to complete the proof of Proposition \ref{TightDivCase}.
\begin{proof}[Proof of Proposition \ref{TightDivCase}]
We assume $\e H=\e \Gamma/2$ by way of contradiction, so Lemma \ref{EqualMass} could apply. With the Shadow Principal     \ref{ShadowPrinciple}, we obtain 
$$
\mu_o(\Pi_o^F(go, r)) \asymp \mathrm{e}^{-\e H d(o, go)}$$
for any $go\in \Gamma o$. The same argument as in \cite[Prop. 6.6]{YANG22} proves $\e H\ge \e \Gamma$. Indeed,  by Lemma \ref{OverlapAnnulus},  that every   point $\xi\in \hU$ are contained in at most $N_0$ shadows for some uniform $N_0>0$.  Hence,
$$
\sum_{v\in A_\Gamma(o, n,\Delta)} \mu_o(\Pi_{o}^F(v, r))\le N_0 \mu_o(\hU)\le N_0
$$ 
which by the Shadow Lemma \ref{ShadowLem} gives $\sharp A_\Gamma(o,n,\Delta)\prec \mathrm{e}^{n\e H}$, and thus $\e \Gamma \le \e H$ gives  a contradiction.    
\end{proof}

\subsection{Completion of proof of Theorem \ref{ConvTightThm}}
First of all, the non-strict inequality $\e H\geq  \frac{\e G}{2}$  follows from Corollary \ref{ConfinedDivergeAtHalf}.  The difficulty is the strict part of this inequality.

If $H\act \U$ is of divergence type, then the conclusion follows by Proposition \ref{TightDivCase}. If $H\act \U$ is of convergence type, then 
$$
\sum_{h\in H} \mathrm{e}^{-\e G d(o,ho)/2}<\infty$$
and assume for the contradiction that    $\e H= \frac{\e G}{2}$. Let  $K$ be given by Lemma \ref{ConjugateDivSeries}, so we obtain  that
$$
\sum_{h\in H} \mathrm{e}^{-\e G d(o,ho)/2} \ge \sum_{g\in K} \mathrm{e}^{-\e G d(o,gp_gg^{-1}o)/2} \succ \sum_{g\in K} \mathrm{e}^{-\e G d(o,go)}
$$ 
This contradicts  the divergence action of $G\act \U$. This completes the proof of Theorem \ref{ConvTightThm}.

\section{Growth and co-growth Inequality}\label{secinequality}
Under the setup of \textsection\ref{secshadow}, we now prove Theorem \ref{CoulonInequality} following closely Coulon's argument in \cite{Coulon22}. 

Again, $H$ is only assumed to be confined by $\Gamma$, but not necessarily contained in $\Gamma$. Denote $\Gamma/H=\{Hg: g\in \Gamma\}$ the collection of right $H$--cosets with representatives in $\Gamma$. Note that  $H$ and $Hg$ may  be not necessarily  contained entirely in $\Gamma$. Consider the Hilbert space $\mathcal H=\ell^2(\Gamma/H)$ with over the coset space $\Gamma/H$.

For each $t>0$, define 
$$
\begin{aligned}
\varphi_t:\quad \Gamma/H\quad\longrightarrow&\quad\mathbb R\\
  Hg\quad\longmapsto   &\quad \mathrm{e}^{-td(Ho, Hgo)}  
\end{aligned}
$$
Recall that $\e{\Gamma/H}$ is the convergence radius of the series
$$
\sum_{g\in \Gamma}\mathrm{e}^{-td(Ho, Hgo)}
$$
Thus, if $t>\e{\Gamma/H}/2$ then $\varphi_t\in \mathcal H$; if $t<\e{\Gamma/H}/2$ then $\varphi_t\notin \mathcal H$.
For each $s>0$, define 
$$
\begin{aligned}
\phi_s:\quad \Gamma/H\quad\longrightarrow &\quad \mathbb R\\
  Hg\quad\longmapsto  &\quad \sum_{h\in H} \mathrm{e}^{-sd(o, hgo)}  
\end{aligned}
$$
We now prove the following. 
\begin{lem}
For any $s>\e G$, we have $\phi_s\in \mathcal H$.    
\end{lem}
\begin{proof}
Consider the PS-measure $\mu_x^s$ for $s>\e H$ supported on $Ho$:
$$
\mu^s_x=  \frac{1}{\p_H(s, o, o)} \sum\limits_{h \in H} e^{-sd(x, ho)} \cdot \dirac{ho}.
$$
We compute the norm 
\begin{equation}\label{phiNormEq}
\begin{aligned}
\|\phi_s\|^2&=\sum_{\underset{g_1,g_2\in \Gamma}{Hg_1=Hg_2}} \mathrm{e}^{-s(d(o,g_1o)+d(o,g_2o))}\\
&=\sum_{g\in \Gamma}  \mathrm{e}^{-sd(o,go)} \left(\sum_{h\in H} \mathrm{e}^{-sd(go, ho)}\right)\\
& = \p_G(s, o, o) \sum_{g\in \Gamma}  \mathrm{e}^{-sd(o,go)}   \|\mu_{go}^s\|   
\end{aligned}
\end{equation}
   
{By Shadow Principle \ref{lem:shadowLemCogrowth}, we have 
$$
\|\mu_{go}^s\| \mathrm{e}^{-sd(o,go)}    \prec  \mu_{o}^s(\Omega_o^F(go, r))
$$}

Given $n\ge 1$,  any element in the following set  $$\Omega(o,n)=\{h\in G: d(o,ho)>n\}$$ is  contained at most  $N$ members from the family of cones $\{\Omega_o^F(go, r):g\in A_\Gamma(o, n, \Delta)\}$, where  $N$ is a uniform number depending on $r,\Delta$.  

It follows that 
$$
\sum_{g\in A(o, n, \Delta)} \|\mu_{go}^s\| \mathrm{e}^{-sd(o,go)}    \prec \sum_{g\in A(o, n, \Delta)}  \mu_{o}^s(\Omega_o^F(go, r)) \prec \mu_o^s( \Omega(o,n))
$$
From (\ref{phiNormEq}), we thus obtain 
$$
\begin{aligned}
\|\phi_s\|^2 &\prec \sum_{n\ge 1} \left( \sum_{h\in\Omega(o,n)} \mathrm{e}^{-sd(o, ho)}\right)   \\
& \prec   \sum_{h\in\Omega(o,n)} d(o,ho) \mathrm{e}^{-sd(o, ho)}.
\end{aligned}
$$
The last term is the derivative of the Poincar\'e series $\p_H(s,o,o)$, so is finite for $s>\e H$. 
\end{proof}

Let $s>\e H$ and $t>\e {\Gamma/H}/2$. Recall that $\p_\Gamma(s+t,o,o)=\sum_{g\in \Gamma} \mathrm{e}^{-(s+t)d(o,go)}$.  Note that $\Gamma$ may be properly contained in the union of $Hg$ over $g\in \Gamma$. Summing up the elements of $\Gamma$ in the same $Hg$ and noting $d(Ho, Hgo)\le d(o,Hgo)$, we have  $$\p_\Gamma(s+t,o,o)\le \sum_{Hg\in \Gamma/H} \left(\mathrm{e}^{-sd(Ho,Hgo)} \sum_{h\in H} \mathrm{e}^{-td(o,hgo)}\right)$$ 
The Cauchy-Schwartz inequality gives the finiteness on the scalar product of $\phi_s$ and $\varphi_t$:
$$
\begin{aligned}
(\phi_s, \varphi_r) &=\sum_{Hg\in \Gamma/H}  \left(\mathrm{e}^{-td(Ho, Hgo)} \sum_{h\in G} \mathrm{e}^{-td(o, hgo)}\right)\\
&\le \|\phi_s\|^2\|\varphi_t\|^2 <\infty  
\end{aligned}
$$
This implies that $s+t\ge \e H$ and thus $\e H+\e {\Gamma/H}/2 \ge \e \Gamma$.
Theorem \ref{CoulonInequality} is proved.

\part{More about Hopf decomposition, and a relation with quotient growth} 

\section{Characterization of maximal quotient growth}\label{secmaxgrowth}

In this section, we relate the conservative/dissipative boundary actions to the growth of quotient spaces. The main result is a dichotomy of quotient growth for any subgroup in a hyperbolic group, where the slower growth is equivalent to the conservative action on the Gromov boundary.  

We start with a general development on the Dirichlet domain and its relation with small horospheric limit set.

\subsection{Dirichlet domain and small horospheric limit set}
The construction of Dirichlet domain tessellates the real hyperbolic spaces into  convex polyhedrons (also known as \textit{Voronoi tessellation}), so provides an important tool to study  Kleinian groups with Poincar\'e polyhedron theorem. In general, the Dirichlet domain can fail to be convex even in other symmetric spaces. The construction of Dirichlet domain is general and relies purely on the metric; in particular, it can be discussed in the setting of coarse metric geometry. Notably, the Nielsen spanning tree is a Dirichlet domain in disguise in the work on free groups of \cite{GKN}.
We first give some variant of the Dirichlet domain in our setting, which relates to the small horospheric limit set (Definition \ref{HoroLimitPtsDef}). This is a key tool for analyzing the quotient growth. 

Assume that $H$ acts properly on a proper geodesic space $\U$ with a contracting element.
Fix a basepoint $o\in \U$. Without loss of generality, we may assume that the stabilizer of $o$ in $H$ is trivial: as we can always embed isometrically $\U$ into a larger one (say, attaching a cone at a point with nontrivial stabilizer to make the finite group action free on the base).

Given a (possibly negative) real number $R$, let $\mathbf D_R(o)$ be the set of points $x\in \U$ that are $R$--closer to $o$ than any point in $Ho$. Namely, $x\in \mathbf D_R(o)$ if and only if $d(x,o)\le d(x, Ho)+R$. Equivalently, $\mathbf D_R(o)=\cap_{1\ne h\in H} \overleftarrow{\mathbf H}_R(o, ho)$ is a countable intersection of  \textit{$R$--half spaces} defined as follows $$\overleftarrow{\mathbf H}_R(x, y):=\{z\in\U: d(z,x)\le d(z,y)+R\}$$  This forms a locally finite family of closed sets (via the same proof of \cite[Theorem 6.6.13]{Rat06}), so  $\mathbf D_R(o)$ is a closed subset. It is obvious that $\mathbf D_R(o)\subseteq \mathbf D_{R'}(o)$ for $R\le R'$, and $\mathbf D_{-R}(o)\subseteq \mathbf D_{0}(o)\subseteq \mathbf D_{R}(o)$ for $R>0$. See Fig. \ref{fig:Dirichletdomain} for an illustration of these notions. 

We remark that the introduction of $\mathbf D_R(o)$ with negative $R$ is an essential novelty here, which is crucial in the Claim 2 in proof of Theorem \ref{CharConsAction}.
\\
\paragraph{\textbf{Dirichlet domain}} For $R=0$,  $\mathbf D(o):=\mathbf D_0(o)$ is the so-called Dirichlet domain centered at $o$ for the action $H\act \U$. It is a fundamental domain in the following sense: $H\cdot \mathbf D(o)=\U$ and $h\cdot \textrm{Int}(\mathbf D(o))\cap \textrm{Int}(\mathbf D(o))=\emptyset$ for any $h\ne 1$. Denote by $\pi: \U\to \U/H$ the  quotient map, whose quotient topology is induced by the metric $\bar d$.  By \cite[Theorem 6.6.7]{Rat06}, $\U/H$ is homeomorphic to $\pi(\mathbf D(o))$, where the latter is equipped with quotient topology via the restriction $\pi: \mathbf D(o)\to \pi(\mathbf D(o))$. See \cite[Ch. 6]{Rat06} for relevant discussion.

In a real hyperbolic space, the notion of a bisector  is useful in analyzing the Dirichlet domain, as it intersects the convex polyhedron $\mathbf{D}(o)$ in faces. We here adopt a more general version of bisectors. For $R\ge 0$, set $\mathrm{Bis}(x,y;R):=\{z\in\U: |d(z,x)-d(z,y)|\le R\}$, and then $\overleftarrow{\mathbf H}_R(x, y)=\overleftarrow{\mathbf H}_0(x, y)\cup \mathrm{Bis}(x,y;R)$.  
\begin{examples}
Here are two examples to clarify the notion of bisectors.
\begin{enumerate}
    \item In trees,  it is readily checked that         $\overleftarrow{\mathbf H}_{R}(x, y)$ is not contained in  $N_T(\overleftarrow{\mathbf H}_{0}(x, y))$ for any $T>0$. Indeed, the limit set of the former set properly contains that of the latter set  in the end boundary.  
    \item 
    In a simply connected negatively pinched Riemannian manifold, it can be shown that $\mathrm{Bis}(x,y;R)$ has finite Hausdorff distance (depending on $R$) to $\mathrm{Bis}(x,y;0)$ (using crucially the lower bound of curvature, \emph{i.e.} the fatness of the comparison triangle). It follows that $\mathrm{Bis}(x,y;R)$ is quasiconvex and its limit set remains the same  for any $R\ge 0$.
\end{enumerate}    
    
\end{examples}
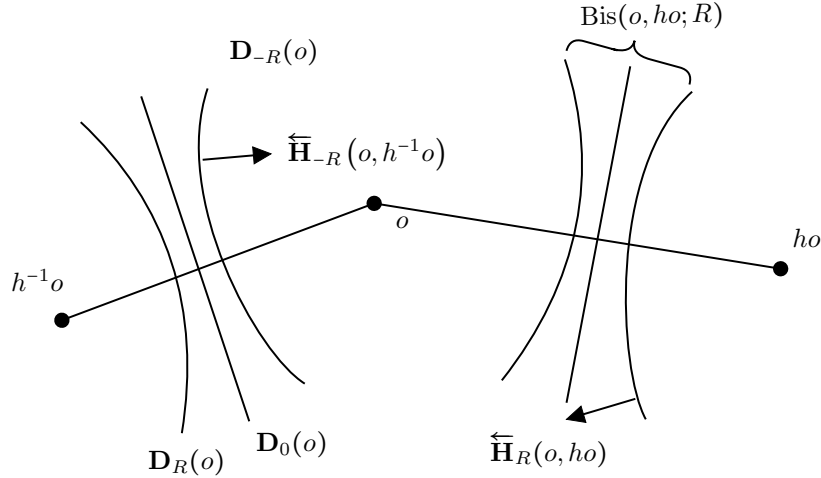
\begin{figure}
    \centering

\tikzset{every picture/.style={line width=0.75pt}} %set default line width to 0.75pt        

\begin{tikzpicture}[x=0.75pt,y=0.75pt,yscale=-1,xscale=1]

%Curve Lines [id:da30509699147202185] 
\draw    (188.5,53) .. controls (170,108) and (213.5,187) .. (237.5,202) ;
%Curve Lines [id:da8449677098106094] 
\draw    (124.5,70) .. controls (168.5,111) and (185.5,165) .. (175.5,227) ;
%Straight Lines [id:da8989522692881093] 
\draw    (155,57) -- (209.5,221) ;
%Curve Lines [id:da34071817932395554] 
\draw    (432.86,54.63) .. controls (389.95,93.69) and (395.82,195.22) .. (409.5,220) ;
%Curve Lines [id:da9704838674402216] 
\draw    (368.68,38.36) .. controls (387.18,95.58) and (375.77,151.03) .. (336.88,200.34) ;
%Straight Lines [id:da9966049568778899] 
\draw    (401.65,41.83) -- (369.51,211.64) ;
%Straight Lines [id:da4508515432939135] 
\draw    (115,170) -- (272.5,111) ;
\draw [shift={(115,170)}, rotate = 339.46] [color={rgb, 255:red, 0; green, 0; blue, 0 }  ][fill={rgb, 255:red, 0; green, 0; blue, 0 }  ][line width=0.75]      (0, 0) circle [x radius= 3.35, y radius= 3.35]   ;
%Straight Lines [id:da09244224447890281] 
\draw    (272.5,111) -- (477.5,144) ;
\draw [shift={(477.5,144)}, rotate = 9.14] [color={rgb, 255:red, 0; green, 0; blue, 0 }  ][fill={rgb, 255:red, 0; green, 0; blue, 0 }  ][line width=0.75]      (0, 0) circle [x radius= 3.35, y radius= 3.35]   ;
\draw [shift={(272.5,111)}, rotate = 9.14] [color={rgb, 255:red, 0; green, 0; blue, 0 }  ][fill={rgb, 255:red, 0; green, 0; blue, 0 }  ][line width=0.75]      (0, 0) circle [x radius= 3.35, y radius= 3.35]   ;
%Straight Lines [id:da8052772449665013] 
\draw    (186,89) -- (217.51,86.26) ;
\draw [shift={(220.5,86)}, rotate = 175.03] [fill={rgb, 255:red, 0; green, 0; blue, 0 }  ][line width=0.08]  [draw opacity=0] (8.93,-4.29) -- (0,0) -- (8.93,4.29) -- cycle    ;
%Straight Lines [id:da45636876461902975] 
\draw    (404.22,209.65) -- (372.37,219.14) ;
\draw [shift={(369.5,220)}, rotate = 343.4] [fill={rgb, 255:red, 0; green, 0; blue, 0 }  ][line width=0.08]  [draw opacity=0] (8.93,-4.29) -- (0,0) -- (8.93,4.29) -- cycle    ;
%Shape: Brace [id:dp8325616525139163] 
\draw  [line width=0.75]  (431.5,51) .. controls (432.73,46.5) and (431.1,43.63) .. (426.6,42.4) -- (411.65,38.3) .. controls (405.22,36.54) and (402.63,33.41) .. (403.86,28.91) .. controls (402.63,33.41) and (398.79,34.78) .. (392.36,33.01)(395.26,33.81) -- (378.1,29.1) .. controls (373.6,27.87) and (370.73,29.5) .. (369.5,34) ;

% Text Node
\draw (282,116.4) node [anchor=north west][inner sep=0.75pt]    {$o$};
% Text Node
\draw (88,142.4) node [anchor=north west][inner sep=0.75pt]    {$h^{-1} o$};
% Text Node
\draw (482,122.4) node [anchor=north west][inner sep=0.75pt]    {$ho$};
% Text Node
\draw (157,233.4) node [anchor=north west][inner sep=0.75pt]    {$\mathbf{D}_{R}( o)$};
% Text Node
\draw (198,26.4) node [anchor=north west][inner sep=0.75pt]    {$\mathbf{D}_{-R}( o)$};
% Text Node
\draw (211.5,224.4) node [anchor=north west][inner sep=0.75pt]    {$\mathbf{D}_{0}( o)$};
% Text Node
\draw (227,73.4) node [anchor=north west][inner sep=0.75pt]    {$\overleftarrow{\mathbf{H}}_{-R}\left( o,h^{-1} o\right)$};
% Text Node
\draw (329,224.4) node [anchor=north west][inner sep=0.75pt]    {$\overleftarrow{\mathbf{H}}_{R}( o,ho)$};
% Text Node
\draw (374.86,9.25) node [anchor=north west][inner sep=0.75pt]  [rotate=-358.87]  {$\mathrm{Bis}( o,ho;R)$};

\end{tikzpicture}
    \caption{Illustrating $R$--half spaces for possible $R$ and their resulted Dirichlet domains}
    \label{fig:Dirichletdomain}
\end{figure}
Let $S^{+R}$ denote the closed $R$--neighborhood of a subset $S$ in $\U$.

\begin{lem}\label{StarShape}
The following holds for any real number $R\in \mathbb R$.
\begin{enumerate}
    \item $\mathbf D_R(o)$ is star-shaped at $o$: any geodesic $[o,x]$ for $x\in \mathbf D_R(o)$ is contained in $\mathbf D_R(o)$.
    \item If $R\ge 0$, then $\mathbf D^{+R}(o))\subseteq \mathbf D_{2R}(o)$.
\end{enumerate}    
\end{lem}
\begin{proof}
(1). For any $x\in \mathbf D_R(o)$, we wish to prove that any geodesic $[x,o]$ is contained in $\mathbf D_R(o)$. Indeed, if $y\in (x,o)$ is not in $\mathbf D_R(o)$, there exists $h\in H$ so that $d(y,ho)<d(y,o)+R$. This implies $d(x,ho)\le d(x,y)+d(y,ho)<d(x,o)+R$, a contradiction with $x\in \mathbf D_R(o)$.   

(2). If $x\in \mathbf D^{+R}(o)$, then for some $y\in \mathbf D(o)$, we have $d(x,y)\le R$ and $d(y,o)\le d(y, Ho)$. So $d(x,o)\le d(y,Ho)+R\le d(x, Ho)+2R$, \emph{i.e.}: $x\in \mathbf D_{2R}(o)$.
\end{proof}

%\begin{lem}
%Let $A, B$ be two quasi-convex subsets in a hyperbolic space $\U$. Then there exists a constant $\delta$ depending on the hyperbolicity constant and quasi-convex constants so that $N_R(A)\cap N_R(B)$ lies in the $(R+\delta)$--neighborhood of $N_\delta(A)\cap N_\delta(B)$ for any $R\ge 0$.    
%\end{lem}
%\begin{proof}\end{proof}

%For given $R_1\le R_2$, we have  $\overleftarrow{\mathbf H}_{R_1}(o, ho)\subseteq \overleftarrow{\mathbf H}_{R_2}(o, ho)$ for any $h\in H$, so  $\mathbf D_{R_1}(o)\subseteq \mathbf D_{R_2}(o)$. 

We now fix a convergence boundary $\pU$ for $\U$, denote by $\Lambda(S)$ the set of accumulations points of $S$ in $\pU$.    By  Definition \ref{ConvBdryDefn}(B), we have $[\Lambda(S^{+R})]=[\Lambda(S)]$.

The main result of this subsection is the following.

\begin{thm}\label{NonHor=DirichletDomain}
For any large $R>0$, \begin{enumerate}
    \item $[\cG]\setminus \HG = H\cdot \left([\Lambda \mathbf D_{R}(o)]\cap [\cG]\right).$
    \item 
    $\hG\cap [\Lambda \mathbf D_{R}(o)]=\emptyset$.
\end{enumerate}
\end{thm}

It is worth noting the following much simplified statement in hyperbolic or CAT(0) spaces.  We remark that the proof in this case could be quite short and straightforward (cf. \cite[Cor. 2.14]{MT98}). We recommend the readers to figure out the proof themselves, instead of reading the following one which  deals mainly with general metric spaces.  
\begin{thm}\label{NonHor=DirichletDomainHyp}
Let $\U$ be Gromov hyperbolic with Gromov boundary $\pU$ or CAT(0) with visual boundary $\pU$. Then for any large $R>0$.
\begin{enumerate}
    \item $\pU\setminus \HG = H\cdot \Lambda \mathbf D_{R}(o).$
    \item 
    $\hG\cap \Lambda \mathbf D_{R}(o)=\emptyset$.
\end{enumerate}
\end{thm}
The proof is achieved by  a series of elementary lemmas. 

Recall that a point $\xi$ is \textit{visual} if a geodesic ray $\gamma$ originating at $o$ terminates at $[\xi]$, \emph{i.e.} all the accumulation points of $\gamma$ are contained in $[\xi]$. We emphasize that $\xi$ may not be an accumulation point of $\gamma$, so it is necessary to make statements on $[\cdot]$--classes rather than boundary points.   Any $(r,F)$--conical point is visual by Lemma \ref{ConicalPointsLem}.  
\begin{rem}\label{GromovVisualFine}
The assumption $\xi\in \cG$  in the next three results is  used to guarantee the following. If $x_n\to \xi$, then any limiting geodesic ray of (a subsequence of) $[o,x_n]$ accumulates into $[\xi]$. If $\U$ is  hyperbolic or CAT(0), this fact holds for any $\xi \in \pU$ with $\pU$ being   Gromov boundary or  visual boundary. Therefore, $\cG$ could be replaced with $\pU$ in these setups.     
\end{rem}

Recall the definition of horoballs in (\ref{HoroballDefn}). The assertion (1) in Theorem  \ref{NonHor=DirichletDomain} is proved by the following.
\begin{lem}\label{NonHorInDirichletDomain}
There exists a constant    $R=R(F)>0$ with the following property. Let $\xi\in [\cG] \setminus \HG$ be  not a \textbf{big} horospheric limit point.  Then there exist some $h\in H$ and a geodesic $[ho,\xi]$ which is contained in   $\mathbf D_{R}(ho)$.   

\end{lem}
\begin{proof}
According to definition, if $\xi\notin \HG$,  some  horoball centered at $[\xi]$ contains only finitely many elements   from $Ho$, so  $L=\min\{B_{[\xi]}(ho, o):h\in H\}$ is a finite value. Moreover,  the horoball $\mathcal{HB}([\xi], L)$ contains some $ho\in Ho$, but  $\mathcal{B}_{[\xi]}(ho,h'o)\le 0$ for any $h'o\in Ho$. Let   $C$ be the common contracting constant for elements in $F$. Note that $[\cG]$ is a subset of $\mathcal C^{\mathrm{hor}}_{20C}$ by Lemma \ref{ConicalPointsLem}. For any sequence of $x_n\to \xi\in \cG$,  we obtain from (\ref{BusemanConvError}) for $n\gg 0$: $$B_{x_n}(h_no,h_n'o)=d(x_n,ho)-d(x_n,h'o)\le 21C$$  That is, setting $R=21C$,  $x_n\in \overleftarrow{\mathbf{H}}_R(ho,h'o)$ for any $h'\in H$, so  we obtain $x_n\in \mathbf D_R(o)$ by definition. See Fig. \ref{fig:horoLimitDirichlet}. Now, take  a geodesic ray  $\gamma$ starting at $ho$ and ending at $[\xi]$ by Lemma \ref{ConicalPointsLem} (cf. Remark \ref{GromovVisualFine}).  If $x_n$ is taken on $\gamma$ tending to $[\xi]$, the star-sharpedness implies $\gamma\subseteq \mathbf D_R(o)$ by Lemma \ref{StarShape}. Thus $\xi\in [\Lambda \mathbf D_R(ho)]$.
\end{proof}
\begin{figure}
    \centering

\tikzset{every picture/.style={line width=0.75pt}} %set default line width to 0.75pt        

\begin{tikzpicture}[x=0.75pt,y=0.75pt,yscale=-1,xscale=1]

%Curve Lines [id:da7911104608819344] 
\draw  [dash pattern={on 4.5pt off 4.5pt}]  (397.5,49) .. controls (439.5,54) and (481.5,92) .. (494.5,133) ;
%Shape: Circle [id:dp16079065328506204] 
\draw   (373.86,110.86) .. controls (375.78,83.86) and (399.23,63.53) .. (426.23,65.45) .. controls (453.23,67.37) and (473.56,90.82) .. (471.64,117.82) .. controls (469.72,144.82) and (446.27,165.15) .. (419.27,163.23) .. controls (392.27,161.31) and (371.94,137.86) .. (373.86,110.86) -- cycle ;
%Shape: Free Drawing [id:dp19520277882677362] 
\draw  [line width=4.5] [line join = round][line cap = round] (263.5,163) .. controls (263.5,163) and (263.5,163) .. (263.5,163) ;
%Curve Lines [id:da6338611371991463] 
\draw [line width=3] [line join = round][line cap = round]    ;
%Shape: Free Drawing [id:dp33097130001138275] 
\draw  [line width=4.5] [line join = round][line cap = round] (418.5,180) .. controls (418.5,180) and (418.5,180) .. (418.5,180) ;
%Shape: Free Drawing [id:dp2518396931867599] 
\draw  [line width=4.5] [line join = round][line cap = round] (393.5,153) .. controls (393.5,153) and (393.5,153) .. (393.5,153) ;
%Straight Lines [id:da9955305237459713] 
\draw    (394.5,153.34) -- (451.29,78.59) ;
\draw [shift={(452.5,77)}, rotate = 127.23] [color={rgb, 255:red, 0; green, 0; blue, 0 }  ][line width=0.75]    (10.93,-3.29) .. controls (6.95,-1.4) and (3.31,-0.3) .. (0,0) .. controls (3.31,0.3) and (6.95,1.4) .. (10.93,3.29)   ;
%Curve Lines [id:da8473460921380005] 
\draw  [dash pattern={on 4.5pt off 4.5pt}]  (152.5,148) .. controls (156.5,106) and (187.5,75) .. (213.5,71) ;
%Straight Lines [id:da8390105644109942] 
\draw    (263.5,161) -- (173.19,104.07) ;
\draw [shift={(171.5,103)}, rotate = 32.23] [color={rgb, 255:red, 0; green, 0; blue, 0 }  ][line width=0.75]    (10.93,-3.29) .. controls (6.95,-1.4) and (3.31,-0.3) .. (0,0) .. controls (3.31,0.3) and (6.95,1.4) .. (10.93,3.29)   ;
%Shape: Free Drawing [id:dp6553647453044606] 
\draw  [line width=4.5] [line join = round][line cap = round] (170.5,103) .. controls (170.5,103) and (170.5,103) .. (170.5,103) ;
%Curve Lines [id:da8286040186371351] 
\draw    (152.5,148) .. controls (205.5,153) and (228.5,166) .. (248.5,209) ;
%Curve Lines [id:da6241695838850823] 
\draw    (213.5,71) .. controls (228.5,104) and (243.5,131) .. (294.5,127) ;
%Curve Lines [id:da8494190462058677] 
\draw  [dash pattern={on 4.5pt off 4.5pt}]  (300.5,168) .. controls (353.47,162.06) and (307.44,179.64) .. (356.97,167.38) ;
\draw [shift={(358.5,167)}, rotate = 165.96] [color={rgb, 255:red, 0; green, 0; blue, 0 }  ][line width=0.75]    (10.93,-3.29) .. controls (6.95,-1.4) and (3.31,-0.3) .. (0,0) .. controls (3.31,0.3) and (6.95,1.4) .. (10.93,3.29)   ;
%Shape: Free Drawing [id:dp5457294051049757] 
\draw  [line width=4.5] [line join = round][line cap = round] (433.5,103) .. controls (433.5,103) and (433.5,103) .. (433.5,103) ;
%Shape: Brace [id:dp616358881762531] 
\draw  [line width=0.75]  (200.5,60) .. controls (196.81,57.14) and (193.54,57.55) .. (190.68,61.24) -- (168.6,89.73) .. controls (164.51,95) and (160.63,96.2) .. (156.94,93.35) .. controls (160.63,96.2) and (160.43,100.27) .. (156.35,105.54)(158.18,103.17) -- (137.25,130.17) .. controls (134.39,133.86) and (134.8,137.14) .. (138.49,140) ;

% Text Node
\draw (460,57.4) node [anchor=north west][inner sep=0.75pt]    {$\xi $};
% Text Node
\draw (306,75.4) node [anchor=north west][inner sep=0.75pt]    {$\mathcal{HB}( \xi ,L)$};
% Text Node
\draw (275,143.4) node [anchor=north west][inner sep=0.75pt]    {$o$};
% Text Node
\draw (364,144.4) node [anchor=north west][inner sep=0.75pt]    {$ho$};
% Text Node
\draw (429,170.4) node [anchor=north west][inner sep=0.75pt]    {$h'o\in Ho$};
% Text Node
\draw (179,90.4) node [anchor=north west][inner sep=0.75pt]    {$h^{-1} \xi $};
% Text Node
\draw (188,135.4) node [anchor=north west][inner sep=0.75pt]    {$\mathbf{D}_{R}( o)$};
% Text Node
\draw (320,168.4) node [anchor=north west][inner sep=0.75pt]    {$h$};
% Text Node
\draw (439,99.4) node [anchor=north west][inner sep=0.75pt]    {$x_{n}$};
% Text Node
\draw (101,73.4) node [anchor=north west][inner sep=0.75pt]    {$\Lambda \mathbf{D}_{R}( o)$};

\end{tikzpicture}
    \caption{Proof of Lemma \ref{NonHorInDirichletDomain}}
    \label{fig:horoLimitDirichlet}
\end{figure}

%Define $\overleftarrow{\mathbf H}^{+R}(o, go)=N_R(\overleftarrow{\mathbf H}(o, go))$ the closed $R$--neighborhood of the half space $\overleftarrow{\mathbf H}(o, go)$. Note that  $\mathbf D^{+R}(o)$ is contained in $\cap_{g\in G} \overleftarrow{\mathbf H}^{+R}(o, go)$, but the converse direction may not be true, unless the ambient space is Gromov hyperbolic (with larger $R$ in $\mathbf D^{+R}(o)$).

%The Gromov boundary  of a  hyperbolic space $\U$ could be as a convergence boundary equipped with a trivial partition $[\cdot]$.  The thin-triangle allows  the following improvement to Lemma \ref{NonHorInDirichletDomain}.
%\begin{lem}\label{NonHorInDirichletDomainHyp}
%Under the assumption of Lemma \ref{NonHorInDirichletDomain}, assume that $\U$ is a hyperbolic space with the Gromov boundary $\pU$.  Then for some $h\in H$,  any geodesic $[ho,\xi]$ lies in the  set $\mathbf D^{+R}(ho)$.   In particular, if $\U$ is CAT($-1$), then we could take $R=0$.   
%\end{lem}
%\begin{proof}
%We proceed with the same proof, departing to finding a point in $D_0(o)$ $R$--close to $[o,\xi]$ by using hyperbolicity. Recall that for any $x\in \gamma$ sufficiently far from $o$, we have $d(x,ho)-d(x,h'o)\le R$ for any $h'\in G$. The $\delta$--thin triangle $\Delta(x,ho,h'o)$ implies  the first half of $[ho, h'o]$ is $R'$--close to $\gamma$, where $R'$ depends on $\delta$ and $R$.  Consequently, $\gamma$ lies in the $R'$--neighborhood of $D_0(o)$. [THIS CONTAINS GAP] Renaming $R=R'$ completes the proof. 
%\end{proof}

The following preparatory result is required in proving the assertion (2)  in Theorem  \ref{NonHor=DirichletDomain}. This follows easily in hyperbolic or CAT(0) spaces, as explained in Remark \ref{GromovVisualFine}.
\begin{lem}\label{DirichletIsVisual}
Let $\xi\in [\Lambda \mathbf D_R(o)]\cap [\cG]$ for some $R>0$. Then  $\mathbf D_R(o)$ contains a geodesic ray ending at $[\xi]$ starting at $o$.    
\end{lem}

\begin{proof}
%\ywy{For any $x\in \mathcal {HB}([\xi],  o)$, the distance from  $\mathcal {HB}([\xi],  x)$  to the complement of $\mathcal {HB}([\xi],  o)$ agrees with $\mathcal B_{[\xi]}(x,o)$ up to bounded error depending on $F$ only.}
By \cite[Lemma 4.4]{YANG22},  such a geodesic ray  exists for any $\xi\in \cG$. We briefly recall the proof, and then explain $\gamma$ can be taken into $\mathbf D_R(o)$ by using the star-shaped $\mathbf D_R(o)$. First,  a conical point  $\xi\in \cG$ by definition lies in infinitely many partial shadows $\Pi_o^F(g_no,r)$ for $g_n\in \Gamma$. By the proof of \cite[Lemma 4.4]{YANG22}, we can find a sequence of $y_n\in \Pi_o^F(g_no,r)$ tending to $\xi$, so that $[o,y_n]$ intersects  $N_C(g_n\ax(f))$ for  a common $f\in F$ in a diameter comparable with $d(o,fo)$ (as $\sharp F=3$). By Ascoli-Arzela Lemma, the limiting geodesic ray of (a subsequence of) $[o,y_n]$  tends to $[\xi]$ by  Definition \ref{ConvBdryDefn}(B).

Now take $x_n\in \mathbf D_R(o)$ tending to $[\xi]$. If $x_n\in \Pi_o^F(g_no,r)$,  then $[o,x_n]\subseteq \mathbf D_R(o)$ converges locally uniformly to a geodesic ray $\gamma\subseteq \mathbf D_R(o)$ by Lemma \ref{StarShape}.   In the general case, if $d(o,fo)$ is sufficiently large,  the $C$--contracting property of $g_n\ax(f)$ implies that $[o,x_n]$ intersects $N_C(g_n\ax(f))$ in a diameter comparable with $d(o,fo)$. Hence, a subsequence of $[o,x_n]$ converges locally uniformly to a geodesic ray $\alpha$ ending at $[\xi]$.  As we wanted,  $\alpha$ is contained in the star-shaped $D_R(o)$. \end{proof}

We now prove the assertion (2) in Theorem \ref{NonHor=DirichletDomain}
\begin{lem}\label{DirichletDomainOutSmallHor}
Let $\xi\in [\cG]\cap [\Lambda \mathbf D_R(o)]$ for some $R>0$.  Then $\xi\notin \hG$.
\end{lem}
\begin{proof}
By Lemma \ref{DirichletIsVisual}, take  a geodesic ray $\gamma\subseteq \mathbf D_R(o)$  starting at $o$ and terminating at $[\xi]$.

By way of contradiction, assume that $\xi\in \hG$ is a small horospheric limit point. Hence, any   horoball $\mathcal {HB}([\xi])$ centered at $[\xi]$ contains  a sequence of $h_no\in Go$ tending to $\xi$. That is, $B_\xi(h_no,o)\to -\infty$ as $n\to\infty$. Note that the Busemann cocycle associated to $\gamma$ differs from $B_{[\xi]}(h_no, o)$ up to a uniform error. This implies that $d(x,h_no)-d(x,o)\to -\infty$ as $x\in \gamma$ tends to $[\xi]$. Consequently, the definition of $\mathbf D_R(o)$ shows that $x\notin \mathbf D_R(o)$ for any fixed $R$. This contradicts the choice of $\gamma$ in $\mathbf D_R(o)$.
\end{proof}

Finally, the proof of Theorem \ref{NonHor=DirichletDomainHyp} follows from Lemma \ref{NonHorInDirichletDomain} and Lemma \ref{DirichletDomainOutSmallHor}, where $[\cG]$ could be replaced with $\pU$ by Remark \ref{GromovVisualFine}. 
\subsection{Big and small horospheric limit set}
In this and next subsections,  assume that 
\begin{itemize}
    \item 
    The proper action $\Gamma\act \U$ is cocompact on a proper hyperbolic  space $(\U,d)$, which is compactified with Gromov boundary $\pU$ endowed with maximal partition (\emph{i.e.} each $[\cdot]$-class is singleton, so $[]$ is omitted in what follows).
    \item 
    Let $\{\mu_x:x\in \U\}$ be the family of PS measures on $\Lambda(\Gamma o)=\pU$ constructed from $\Gamma\act \U$. Then $\mu_o$ is a doubling measure without atoms; indeed it is the Hausdorff measure with respect to visual metric at the correct dimension (see \cite{Coor}). Any subgroup $H<\isom(\U)$ preserves the measure class of $\mu_o$ by Lemma \ref{PreserveMeasureClass}.
\end{itemize}  Let a countable group $H$ (possibly not contained in $\Gamma$) act properly on $\U$.  We shall prove the following by elaborating on the proof of Sullivan  \cite{Sul81} in Kleinian groups (see also \cite[Lemma 5.7]{MT98}).
\begin{thm}\label{BigSmallEqualInHyp}
$\mu_o(\HG \setminus \hG)=0$.   
\end{thm}
%Assume first that $\U$ is CAT(-1) space.

The key part is to prove that the set of Garnett points is $\mu_o$--null. Recall that a point $\xi\in \HG$ is called \textit{Garnett point} if there exists a unique horoball $\mathcal {HB}(\xi,L)$ (Definition \ref{HoroballDefn}) at $\xi$ so that
\begin{itemize}
    \item 
    The interior of $\mathcal {HB}(\xi,L)$ contains no point from $Ho$: $B_{\xi}(ho,o)\ge L$ for any $ho\in Ho$;
    \item 
    any larger horoball contains infinitely many points of $Ho$: $B_{\xi}(ho,o)\in [L,L+\epsilon]$ for infinitely many $ho\in Ho$ and any $\epsilon<0$; 
\end{itemize}
where $B_{\xi}(ho,o)=\limsup_{z\to \xi} [d(z,ho)-d(z,o)]$.
Denote by $\mathbf {Gar}$ the set of Garnett points for $H$.

We first recall a useful observation of Tukia \cite[Appendix]{Tu97}.
\begin{lem}\label{Tukia}
The set $A=\HG \setminus (\hG\cup \mathbf {Gar})$ is $\mu_o$--null.     
\end{lem} 
\begin{proof}
We include a short proof for completeness. Indeed, let $\xi \in \HG \setminus \hG$  be a non-Garnett point, so there exists a horoball $\mathcal {HB}(\xi)$ whose boundary contains a non-empty finite set of elements in $Ho$ but whose interior contains no point in $Ho$. Consequently, $A$ is a disjoint union of measurable subsets indexed over finite subsets in $Ho$, which are invariant under $H$, so we could choose a measurable fundamental domain for $H\act A$. By Lemma \ref{HopfdecompLem}, $\HG$ is a subset of the conservative part $\mathbf{Cons}$, so $\mu_o(A)=0$.    
\end{proof}

To complete the proof of Theorem \ref{BigSmallEqualInHyp}, it suffices to prove that $\mathbf {Gar}$ is a $\mu_o$--null set. 
\begin{lem}\label{Garnett}
$\mu_o(\mathbf {Gar})=0$.    
\end{lem}

\begin{proof}
By way of contradiction, assume that $\mu_o(\mathbf {Gar})>0$.

We define a height function $\mathfrak{h}: \mathbf {Gar}\to \mathbb R_{\ge 0}$. Set $\mathfrak{h}(\xi)=d(o,\mathcal {H}(\xi))$, where $\mathcal {H}(\xi)$ is  the unique horoball centered at $\xi$ whose interior does not contain any point of $Ho$ and any horoball centered at $\xi$ larger than that contains infinitely many points of $ho$.  We verify that $\mathfrak{h}$ is a measurable function. Indeed, if $\xi\in \Lambda \mathrm{Bis}(x,y;\epsilon)$, then $|B_\xi(x,y)|\le \epsilon$; if $|B_\xi(x,y)|\le \epsilon$, then $\xi\in \Lambda \mathrm{Bis}(x,y;\epsilon+1/n)$ for any $n>1$. This implies that the set $\xi\in \pU$ with $|B_\xi(x,y)|\le \epsilon$ is a countable intersection of closed subsets, so is Borel measurable. By definition, $\mathbf {Gar}$ is  the limit supremum  of such measurable sets. Moreover, the set of $\xi\in \mathbf {Gar}$ with $\mathfrak h(\xi)\in (r_1,r_2)$ is measurable. 

As $\mu_o$ is a doubling measure,   the Lebesgue density theorem holds: Let $\mathbf {G}$ be the $\mu_o$--full set of density points $\xi\in \mathbf {Gar}$, so 
\begin{enumerate}
    \item 
    $\mathfrak h: \mathbf {G}\to \mathbb R$ is approximately continuous: $\xi_n\in\mathbf {G} \to \xi\in \mathbf {G}$ implies $\mathfrak h(\xi_n)\to \mathfrak h(\xi)$;
    \item 
    for any sequence of metric balls $B_n\subseteq \pU$ centered at $\xi \in \mathbf{G}$ with radius tending to 0, we have 
$$
\frac{\mu_o(\mathbf {Gar}\cap B_n)}{\mu_o(B_n)}\to 1.
$$
\end{enumerate}    
As in \cite{Sul81}, we are going to find a ``forbidden" region $C_n\subseteq B_n$, consisting of non-Garnett points, with a definite  $\mu_o$--measure. This  contradicts the item (2).  The co-compact action is crucial to apply Shadow Lemma \ref{ShadowLem}: for any $x\in \U$ and any fixed large constant $r$, we have \begin{equation}\label{FullShadowPrincipleEQ}
\mu_o(\Pi_o(x,r))\asymp_r \mathrm{e}^{-\e \Gamma d(o,x)}.    
\end{equation}
\begin{figure}
    \centering

\tikzset{every picture/.style={line width=0.75pt}} %set default line width to 0.75pt        

\begin{tikzpicture}[x=0.75pt,y=0.75pt,yscale=-1,xscale=1]

%Straight Lines [id:da36432365396443767] 
\draw    (69.5,267) -- (176.64,40.81) ;
\draw [shift={(177.5,39)}, rotate = 115.35] [color={rgb, 255:red, 0; green, 0; blue, 0 }  ][line width=0.75]    (10.93,-3.29) .. controls (6.95,-1.4) and (3.31,-0.3) .. (0,0) .. controls (3.31,0.3) and (6.95,1.4) .. (10.93,3.29)   ;
\draw [shift={(69.5,267)}, rotate = 295.35] [color={rgb, 255:red, 0; green, 0; blue, 0 }  ][fill={rgb, 255:red, 0; green, 0; blue, 0 }  ][line width=0.75]      (0, 0) circle [x radius= 3.35, y radius= 3.35]   ;
%Straight Lines [id:da3254925541027416] 
\draw    (134.5,135) -- (232.5,127) ;
\draw [shift={(232.5,127)}, rotate = 355.33] [color={rgb, 255:red, 0; green, 0; blue, 0 }  ][fill={rgb, 255:red, 0; green, 0; blue, 0 }  ][line width=0.75]      (0, 0) circle [x radius= 3.35, y radius= 3.35]   ;
\draw [shift={(134.5,135)}, rotate = 355.33] [color={rgb, 255:red, 0; green, 0; blue, 0 }  ][fill={rgb, 255:red, 0; green, 0; blue, 0 }  ][line width=0.75]      (0, 0) circle [x radius= 3.35, y radius= 3.35]   ;
%Shape: Brace [id:dp1324177149166672] 
\draw  [line width=0.75]  (193.5,128) .. controls (193.16,123.35) and (190.66,121.19) .. (186.01,121.53) -- (176.14,122.24) .. controls (169.49,122.73) and (166,120.64) .. (165.66,115.99) .. controls (166,120.64) and (162.84,123.21) .. (156.19,123.69)(159.19,123.48) -- (144.97,124.51) .. controls (140.32,124.85) and (138.16,127.35) .. (138.5,132) ;
%Straight Lines [id:da4744561599912127] 
\draw  [dash pattern={on 4.5pt off 4.5pt}]  (195.5,132) -- (350.56,172.49) ;
\draw [shift={(352.5,173)}, rotate = 194.64] [color={rgb, 255:red, 0; green, 0; blue, 0 }  ][line width=0.75]    (10.93,-3.29) .. controls (6.95,-1.4) and (3.31,-0.3) .. (0,0) .. controls (3.31,0.3) and (6.95,1.4) .. (10.93,3.29)   ;
%Shape: Free Drawing [id:dp7900028303901214] 
\draw  [line width=6] [line join = round][line cap = round] (178.8,38.22) .. controls (178.8,38.22) and (178.8,38.22) .. (178.8,38.22) ;
%Shape: Free Drawing [id:dp48617475209625227] 
\draw  [line width=6] [line join = round][line cap = round] (342.8,144.22) .. controls (342.8,144.22) and (342.8,144.22) .. (342.8,144.22) ;
%Straight Lines [id:da26816951443304604] 
\draw  [dash pattern={on 4.5pt off 4.5pt}]  (134.5,135) -- (78.43,28.77) ;
\draw [shift={(77.5,27)}, rotate = 62.18] [color={rgb, 255:red, 0; green, 0; blue, 0 }  ][line width=0.75]    (10.93,-3.29) .. controls (6.95,-1.4) and (3.31,-0.3) .. (0,0) .. controls (3.31,0.3) and (6.95,1.4) .. (10.93,3.29)   ;
%Straight Lines [id:da7835194276378477] 
\draw  [dash pattern={on 4.5pt off 4.5pt}]  (134.5,135) -- (349.57,192.48) ;
\draw [shift={(351.5,193)}, rotate = 194.96] [color={rgb, 255:red, 0; green, 0; blue, 0 }  ][line width=0.75]    (10.93,-3.29) .. controls (6.95,-1.4) and (3.31,-0.3) .. (0,0) .. controls (3.31,0.3) and (6.95,1.4) .. (10.93,3.29)   ;
%Straight Lines [id:da25257930434911025] 
\draw  [dash pattern={on 4.5pt off 4.5pt}]  (195.5,132) -- (327.82,46.09) ;
\draw [shift={(329.5,45)}, rotate = 147.01] [color={rgb, 255:red, 0; green, 0; blue, 0 }  ][line width=0.75]    (10.93,-3.29) .. controls (6.95,-1.4) and (3.31,-0.3) .. (0,0) .. controls (3.31,0.3) and (6.95,1.4) .. (10.93,3.29)   ;
\draw [shift={(195.5,132)}, rotate = 327.01] [color={rgb, 255:red, 0; green, 0; blue, 0 }  ][fill={rgb, 255:red, 0; green, 0; blue, 0 }  ][line width=0.75]      (0, 0) circle [x radius= 3.35, y radius= 3.35]   ;
%Straight Lines [id:da19273668534417032] 
\draw    (195.5,132) -- (339.51,143.84) ;
\draw [shift={(341.5,144)}, rotate = 184.7] [color={rgb, 255:red, 0; green, 0; blue, 0 }  ][line width=0.75]    (10.93,-3.29) .. controls (6.95,-1.4) and (3.31,-0.3) .. (0,0) .. controls (3.31,0.3) and (6.95,1.4) .. (10.93,3.29)   ;
%Shape: Brace [id:dp8070943017098879] 
\draw  [line width=0.75]  (363.5,172) .. controls (368.17,172.03) and (370.52,169.72) .. (370.55,165.05) -- (370.99,105.24) .. controls (371.04,98.57) and (373.39,95.26) .. (378.06,95.3) .. controls (373.39,95.26) and (371.09,91.91) .. (371.14,85.24)(371.11,88.24) -- (371.45,43.05) .. controls (371.48,38.38) and (369.17,36.03) .. (364.5,36) ;
%Shape: Free Drawing [id:dp13104993805091536] 
\draw  [line width=6] [line join = round][line cap = round] (93.8,216.22) .. controls (93.8,216.22) and (93.8,216.22) .. (93.8,216.22) ;
%Curve Lines [id:da41211446126970586] 
\draw    (94.5,215) .. controls (4.5,174) and (32.5,-2) .. (177.5,39) ;
%Curve Lines [id:da8910952619462764] 
\draw    (94.5,215) .. controls (232.5,266) and (274.5,91) .. (177.5,39) ;
%Shape: Brace [id:dp23754178619019495] 
\draw  [line width=0.75]  (73.5,270) .. controls (77.69,272.05) and (80.82,270.99) .. (82.87,266.8) -- (89.52,253.21) .. controls (92.45,247.23) and (96.02,245.27) .. (100.21,247.32) .. controls (96.02,245.27) and (95.39,241.25) .. (98.32,235.26)(97,237.95) -- (101.7,228.36) .. controls (103.75,224.17) and (102.69,221.05) .. (98.5,219) ;
%Shape: Brace [id:dp3000949072665049] 
\draw  [line width=0.75]  (101.5,214) .. controls (105.72,215.99) and (108.83,214.88) .. (110.82,210.66) -- (119.94,191.36) .. controls (122.79,185.33) and (126.32,183.31) .. (130.54,185.3) .. controls (126.32,183.31) and (125.63,179.3) .. (128.48,173.27)(127.2,175.98) -- (138.84,151.32) .. controls (140.83,147.1) and (139.72,143.99) .. (135.5,142) ;

% Text Node
\draw (240,113.4) node [anchor=north west][inner sep=0.75pt]    {$h_{n} o\in \mathcal{HB}( \zeta ,L)$};
% Text Node
\draw (55,246.4) node [anchor=north west][inner sep=0.75pt]    {$o$};
% Text Node
\draw (111.5,121.9) node [anchor=north west][inner sep=0.75pt]    {$p_{n}$};
% Text Node
\draw (192,136.4) node [anchor=north west][inner sep=0.75pt]    {$q_{n}$};
% Text Node
\draw (162.39,100.04) node [anchor=north west][inner sep=0.75pt]  [rotate=-2.84]  {$D$};
% Text Node
\draw (179,11.4) node [anchor=north west][inner sep=0.75pt]    {$\xi $};
% Text Node
\draw (347,141.4) node [anchor=north west][inner sep=0.75pt]    {$\zeta $};
% Text Node
\draw (382,87.4) node [anchor=north west][inner sep=0.75pt]    {$C_{n}$};
% Text Node
\draw (277,219.4) node [anchor=north west][inner sep=0.75pt]    {$B_{n} :=\Pi _{o}( p_{n} ,r)$};
% Text Node
\draw (277,244.4) node [anchor=north west][inner sep=0.75pt]    {$C_{n} :=\Pi _{o}( q_{n} ,r)$};
% Text Node
\draw (74,209.4) node [anchor=north west][inner sep=0.75pt]    {$x$};
% Text Node
\draw (50,143.4) node [anchor=north west][inner sep=0.75pt]    {$\mathcal{HB}( \xi ,L)$};
% Text Node
\draw (104,241.4) node [anchor=north west][inner sep=0.75pt]    {$-L_{\xi }$};
% Text Node
\draw (133,173.4) node [anchor=north west][inner sep=0.75pt]    {$L_{n} /2$};

\end{tikzpicture}
    \caption{The proof of Lemma \ref{Garnett} where $\U$ is a tree and $r=0$. The forbidden region $C_n$ contains no Garnett point: a horoball at $\zeta\in C_n$ containing $x$ must contain $h_no$}
    \label{fig:forbiddenregion}
\end{figure}
For $\xi \in \mathbf{Gar}$, there exists a unique horoball $\mathcal{HB}(\xi,L_\xi)$ for some real number $L_\xi$, so that $B_\xi(o,ho)>L_\xi$ for all $h\in H$, but $L_\xi+1>B_\xi(o,h_no)>L_\xi$ for  infinitely many distinct $h_no\in Ho$. Let $x\in \mathcal{HB}(\xi,L_\xi)$ be a projection point of $o$ onto $\mathcal{HB}$, so $\mathfrak{h}(\xi)=d(o,x)$ by definition. As horoballs in hyperbolic spaces are quasi-convex, we have $|L_\xi-\mathfrak{h}(\xi)|\le 100\delta$, and $d(x,[o,h_no])\le 100\delta$. 

Denote $L_n:=d(x,h_no)$. 
 Choose $p_n, q_n\in [x,h_no]$ with  $d(x, p_n)=L_n/2$ and $d(x, q_n)=L_n/2+D$ for a large constant $D\gg r$.  Set $B_n:=\Pi_o(p_n, r)$ and $C_n:=\Pi_o(q_n, r)$. See Fig. \ref{fig:forbiddenregion} for approximate tree configuration of these points. 
 
As $d(x,[o,h_no])\le 100\delta$ and $|d(x,p_n)-d(x,q_n)|\le D$, we see that $|d(o,p_n)-d(o,q_n)|\le D+200\delta$.  Shadow estimates (\ref{FullShadowPrincipleEQ}) implies the existence of a positive  constant $c$ depending only on $D$: $${\mu_o(C_n)}\ge c {\mu_o(B_n)}$$  

We claim  that for all $n\gg 0$, $C_n$ consists of non-Garnett points up to a $\mu_o$-null set. If not, take a density point $\zeta\in  C_n\cap \mathbf{G}$. Let $\mathcal{HB}(\zeta,L_\zeta)$ be the unique horoball associated with $\zeta$, so $|L_\zeta-\mathfrak{h}(\zeta)|\le 100\delta$ follows as above. Let $y\in \mathcal{HB}(\zeta,L_\zeta)$ so that $d(o,y)=\mathfrak{h}(\zeta)$.  Our goal is to prove $\mathcal{HB}(\zeta,L_\zeta)$ must contain $h_no$,  contradicting the definition of a Garnett point.

Indeed, the approximate continuity implies $|\mathfrak{h}(\xi)-\mathfrak{h}(\zeta)|=|d(o,x)-d(o,y)|\le 1$ for $\zeta\in C_n\to \xi$. As $ \zeta\in \Pi_o(q_n, r)$, we have $d(q_n,[o,\zeta))\le r$. If $D\gg r$, the thin-triangle inequality shows that $d(x,[o,\zeta])\le 10\delta$ and thus $d(x,y)\le 100\delta$.    Now, by the choice of $d(x, p_n)=L_n/2$ and $d(h_no,q_n)= L_n/2-D$, we have $d(y,q_n)\ge d(x,q_n)-d(x,y)\ge L_n/2+D-100\delta$, implying $$d(h_no,q_n)\le L_n/2-D \le  d(y,q_n)-2D+100\delta$$
Noting $d(q_n,z)\le r$ for some $z\in [o,\zeta)$, we choose $D\gg r$ so that $d(h_no,z) \le  d(y,z)-100\delta$. Lemma \ref{SameHoroball} implies $h_no\in \mathcal{HB}(\zeta,L_\zeta)$. Hence, we conclude that $C_n\cap \mathcal{G}=\emptyset$.  Recall that ${\mu_o(C_n)}\ge c {\mu_o(B_n)}$,  this contradicts the property that $\xi$ is a density point as described in (2). The proof of the lemma is complete.    
\end{proof}

We conclude this subsection with a remark on possible extensions of Theorem \ref{BigSmallEqualInHyp} to mapping class groups.
\begin{rem}\label{ExtensionMCG}
Consider the (non-cocompact) action of the mapping class group $\Gamma=\mathbf{Mod}(\Sigma_g)$ on Teichm\"{u}ller space $\U$ for $g\ge 2$. 
The proof of Lemma \ref{Garnett} could be adapted to its subgroup $H<\Gamma$, provided that (\ref{FullShadowPrincipleEQ}) holds for all $x\in \U$. (Hyperbolicity could be circumvented by using partial shadows). By the Shadow Lemma \ref{ShadowLem}, it is known that (\ref{FullShadowPrincipleEQ}) holds for a cocompact subset of points. By work of \cite{ABEM}, the conformal density $\mu_x$ gives the same mass to Thurston boundary $\pmf$, and the volume growth in Teichm\"{u}ller space does not depend on the basepoint. This provides positive evidence that (\ref{FullShadowPrincipleEQ}) holds for all $x\in \U$. 
\end{rem}

\subsection{Characterizing maximal quotient growth}

Let $B(o,n)\subseteq \U$ and $\bar B(\pi(o),n)\subseteq \U/H$ denote respectively the   $d$--ball and $\bar d$--ball at $o$ and $\pi(o)$ of radius $n\ge 0$. Note that $\pi(B(o,n))=\bar B(\pi(o),n)$:  $Hgo \in B(\pi(o),n)$ if and only if $d(Ho, Hgo)\le n$. We say that $\U/H$ has \textit{slower growth} than $\U$ if given any maximal $R$--separated subset $Z$ (resp. $W$) of $\U$ (resp. $\U/H$) for some $R>0$, we have 
$$
\frac{\sharp W\cap B(\pi o,n)}{\sharp Z\cap  B(o,n)}\to 0,\;\text{as } n\to\infty.$$
This definition is independent of the choice of $Z, W$ and $R$, as different choice yields comparable ball growth function up to multiplicative constants. 
By the Shadow Lemma, the notion of slower growth in orbit points could be re-formulated using volume. For example, if $\U$ is a simply connected rank-1 manifold with a co-compact action, $\sharp Z\cap  B(o,n)$ is comparable to the volume of $B(o,n)$. In a non-cocompact action, an important potential example could be given by mapping class groups, if the full shadow principle  holds as stated in Remark \ref{ExtensionMCG}.  

The following result in Kleinian groups was first obtained in \cite[Theorem VI]{Sul81} and later in free groups in \cite[Theorem 4.1]{GKN} and \cite[Theorem 7]{BO10}. Recall from the previous subsection, $\Gamma\act\U$ is assumed to be geometric action on a hyperbolic space, and $\mu_o$ is the corresponding Hausdorff measure on the Gromov boundary $\pU$.

%\ywy{I can only prove one direction, and the other direction requires a more careful examination of configuration of Dirichlet domain which I'm not able to do it in this general setup.}
\begin{thm}\label{CharConsAction}
Assume that $H$ acts properly on a proper hyperbolic space $\U$.  Then  the small/big horospheric limit set   is $\mu_o$--full, if and only if  $\U/H$ grows slower  than $\U$. Moreover,   $\mu_o(\hG)<1$ is equivalent to the following
$$
\sharp \{Hgo: \bar d(Ho,Hgo)\le n, g\in \Gamma\} \asymp \mathrm{e}^{\e \Gamma n}.
$$
\end{thm}
As an immediate corollary, we obtain a characterization of conservative actions.
\begin{cor}
Assume that $H$ is torsion-free. Then    the action $H\act (\pU,\mu_o)$ is conservative if and only if  the action $H\act \U$ grows slower than $\Gamma\act \U$.   
\end{cor}

\begin{proof}[Proof of Theorem \ref{CharConsAction}]
By Theorems \ref{BigSmallEqualInHyp} and \ref{NonHor=DirichletDomain}, $\mu_o(\hG)=\mu_o(\HG)=1$ if and only if  $\mu_o(\Lambda \mathbf D_R(o))=0$ for any sufficiently large $R>0$. 

Let $\pi: \U\to \U/H$ be the projection. Fix a basepoint $o\in \U$ and consider the orbit $A=\Gamma o$.  Then $N_R(A)=\U$ for some $R>0$ as the action  $\Gamma\act \U$ is  co-compact by assumption. The image  $\bar A=\pi(A)$ is equipped with the metric $\bar d(Hx,Hy)=d(Hx,Hy)$ for $x,y\in A$, and write  
$$
A_{\Gamma/H}(o,n,\Delta) =\{Hgo: |d(go, Ho)-n|\le \Delta\}
$$  
According to (\ref{QuotientGrowthRateDefn}), we have 
$$\e{\Gamma/H}=\limsup_{n\to\infty} \frac{\log A_{\Gamma/H}(o,n,\Delta)}{n}$$  
Recall $A_\Gamma(o,n,\Delta)=\{go\in \Gamma o: |d(o,go)-n|\le \Delta \}$. For each $Hgo\in A_{\Gamma/H}(o,n,\Delta)$, choose $go\in \Gamma o$ so that $d(go, o)=d(Hgo,Ho)$. That is, $go$ is a closest point in $Hgo$ to $o$, so $go$ is an element   in $\mathbf D(o)$. Thus, every element in $A_{\Gamma/H}(o,n,\Delta)$ admits a (possibly non-unique) lift in the following intersection:
$$B_n:=A_\Gamma(o,n, \Delta)\cap \mathbf D(o)$$
Thus, $\sharp  A_{\Gamma/H}(o,n,\Delta)\le \sharp B_n$. 
%\ywy{I do not know whether the other direction is possible, as the $H$--fundamental domain $\mathbf D(o)$  may contain a finitely number of $go\in  \Gamma$ (on the boundary) equidistant to $o$. In rank-1 symmetric space, the boundary of $\mathbf D(o)$ lies on bisector with null Leb measure at the infinity. This however fails for trees, where the bisector could give an open set at infinity.}
%Let $W$ be a maximal $R$--separated set in $\U/H$.

%Recall that $\mathbf{D}(o)$ is the fundamental domain for $H\act \U$ so that $Hx$ for any point $x\in \U$ contains at least one but finitely many points to $o$ in $\mathbf{D}(o)$. Hence, $d(o, Hx)\le n+R$ for $x\in A=\Gamma o$ is equivalent to $x\in \mathbf D_R(o)\cap N_\Gamma(o,n)$, so   $\{Hx: d(o, Hx)\le n, x\in A\}$ has the same cardinality  as  $A_n\cap \mathbf{D}_R(o)$, up to $\Gamma$--action (only pick one from $\mathbf{D}_R(o)$ in the $H$--orbit).   

\begin{claim}
If  $\mu_o(\Lambda \mathbf D_R(o))=0$, then $\sharp B_n= o(\mathrm{e}^{\e \Gamma n})$ and   $\sharp A_{\Gamma/H}(o,n,\Delta)=o(\mathrm{e}^{\e \Gamma n})$    
\end{claim}
\begin{proof}[Proof of the Claim]
Define  $$\Lambda_n:=\bigcup_{go\in B_n} \Pi_o^F(go,r).$$ 
Let $\Lambda$  be the limit supremum of $\Lambda_n$. By construction, $\Lambda \mathbf D_{ R}(o)$ coincides with $\Lambda$. Indeed,  any point $\xi\in \Lambda \mathbf D_{ R}(o)$, there exists a geodesic ray in $\mathbf D_{R}(o)$ ending at $\xi$. Then we can choose $g_n\in B_n$ so that $d(g_no, \gamma)\le R$. This implies $\Lambda \mathbf D_{ R}(o)\subseteq \Lambda$. The converse inclusion follows from the definition of $B_n$. Hence, $\mu_o(\Lambda_n)\to 0$ as $n\to \infty$.

As the family of shadows $\Pi_o^F(x,r)$ for $x\in B_n$ has  uniformly bounded multiplicity, we have by the Shadow Lemma
$$
\mu_o(\Lambda_n)\asymp  \sum_{x\in B_n}  \mu_o(\Pi_o^F(x,r)) \asymp  \mathrm{e}^{-\e \Gamma n}\sharp B_n
$$
where the implicit constant depends on $r$. Hence, the claim follows.
\end{proof}

%The star-shaped domain $\mathbf D_R(o)$ implies the union  forms a decreasing sequence, whose limit gives $[\Lambda \mathbf D_R(o)]$ up to $[\cdot]$--closure.  
\begin{claim}
If   $\mu_o(\Lambda \mathbf D_R(o))>0$, then $\sharp A_{\Gamma/H}(o,n,\Delta) \asymp \mathrm{e}^{\e \Gamma n}$.    
\end{claim}
\begin{proof}[Proof of the Claim]
Recall that $\mathbf{D}(o)$ is a fundamental domain for $H\act \U$, so any point $x\in \U$ admits at least one translate $hx$ into $\mathbf{D}(o)$. If $h_1x\ne  h_2x\in \mathbf{D}(o)$, we  have $d(h_1x,o)=d(h_2x,o)$ and then $x$ lies on the bisector  $\mathrm{Bis}(h_1^{-1}o,h_2^{-1}o;0)$.  

Consider the facet-like set  $$F(h):=\mathrm{Bis}(o,ho;R)\cap \mathbf{D}_R(o)$$ and   the set $\Lambda(h)$ of accumulation points of $F(h)$  in $\pU$. 
By Theorem \ref{NonHor=DirichletDomain}, the intersection $\Lambda(h)\cap \HG$ has no small horospheric limit point, so is $\mu_o$--null by Theorem \ref{BigSmallEqualInHyp}. 

Note that $\mathbf{D}_{-R}(o)$  for $R>0$ is contained in the interior of $\mathbf{D}(o)$, where    $\pi: \U\to \U/H$ restricts as an injective map. See Fig. \ref{fig:Dirichletdomain}. By definition, $\Lambda \mathbf D_R(o)$ is the union of $\Lambda \mathbf D_{-R}(o)$ with  countably many facet-like sets $F(h)$, which are $\mu_o$-null, so we obtain $\mu_o(\Lambda \mathbf D_{-R}(o))=\mu_o(\Lambda \mathbf D_R(o))>0$.  

By injectivity of $\pi: \mathbf D_{-R}(o) \to \U/H$ restricted as  map, $\hat B_n:=A_\Gamma(o,n, \Delta)\cap \mathbf D_{-R}(o)$ injects into $\U/H$.    Consider similarly  $$\hat \Lambda_n:=\bigcup_{go\in \hat B_n} \Pi_o^F(go,r).$$ 
Let $\tilde \Lambda$  be the limit  of $\tilde \Lambda_m:= \cup_{m\ge n}\hat \Lambda_m$.  As the above claim, the star-shapedness of $\mathbf D_{-R}(o)$ shows that $\tilde \Lambda=\Lambda \mathbf D_{-R}(o)$, and $\tilde \Lambda_m$ is  covered with a uniform multiplicity by the family of $\Pi_o^F(go,r)$ for ${go\in \hat B_n}$.

Consequently,  $\mu_o(\hat B_n)>c$ for some uniform $c>0$ independent of $n$, and 
$$
\mu_o(\hat B_n) \asymp \sum_{x\in \hat B_n}  \mu_o(\Pi_o(x,r)) \asymp  \mathrm{e}^{-\e \Gamma n} \sharp \hat B_n
$$
We thus obtain
$\sharp \hat B_n \asymp \mathrm{e}^{\e \Gamma n}$, and then the conclusion follows by $\sharp \hat B_n\le \sharp A_\Gamma(o,n, \Delta)$.
\end{proof}
The proof of the theorem is completed by the above two claims. 
\end{proof}
The first claim actually does not make use of the assumption that $\Gamma\act \U$ is cocompact and $\mu_o$ is doubling. These two assumptions are only required in Lemma \ref{Garnett}. So we have. 
\begin{cor}\label{Zerohorlimitset}
Assume that $\Gamma \act \U$ is a proper action on a hyperbolic space and $\mu_x$ be the $\e \Gamma$--dimensional $\Gamma$--equivariant quasi-conformal density on $\pU$. Let $H$ be a subgroup in $\Gamma$. If  $\mu_o(\hG)=0$, then $\sharp A_{\Gamma/H}(o,n,\Delta)=o(\mathrm{e}^{\e \Gamma n})$.  
\end{cor}

\section{Subgroups with nontrivial conservative and dissipative components}\label{secNontrivialHopf}

To complement the part 1,  we construct in  this section examples with non-trivial Hopf decomposition. Through this section, suppose that 
\begin{itemize}
    \item 
    $\Gamma$ admits a proper SCC action on $\U$ with contracting elements.
    \item 
    Let $\pU$ be a convergence boundary for $\U$ and $\{\mu_x:x\in \U\}$ be a  quasi-conformal, $\Gamma$--equivariant density of dimension $\e \Gamma$ on $\pU$.
\end{itemize}  We are going to  construct a free subgroup of infinite rank with nontrivial conservative and dissipative components. The construction is motivated by examples in free groups given in \cite{GKN}.

\subsection{Preparation}
Pick up any non-pinched contracting element $f\in\Gamma$ and form the contracting system $\f =\{g\ax(f):g\in \Gamma\}$, whose stabilizers are accordingly $gE(f)g^{-1}$.  
\\
\paragraph{\textbf{Projection complex}} For a sufficiently large $K>0$,   the projection complex $\p_K (\f)$ defined by  Bestvina-Bromberg-Fujiwara \cite{BBF} is a quasi-tree of infinite diameter (\cite{BBF}):
\begin{itemize}
    \item The vertex set consists of all elements in $\f$.
    \item Two vertices $u\ne v\in \f$ are connected by one edge if for any $w\in \f$ with $u\ne w\ne v$, we have $\proj_w(u,v)>K$. Here the projection is understood between  $w,u,v\in \f$  as axes in $\U$.
\end{itemize}
\begin{rem}
Actually, the function ``$\proj_w$" is a slight modification of $\proj_w(u,v)=\|\pi_w(u)\cup\pi_w(v)\|$, up to an additive amount. See \cite[Def. 3.1]{BBF} and \cite[Sec. 4]{BBFS} for relevant discussions. We shall use the later one, which  does not affect the estimates in practice.     
\end{rem}
Any two vertices $u\ne v\in \f$ are connected by a \textit{standard path} $\gamma$ consisting of consecutive vertices $$w_0:=u, w_1,\cdots, w_n:=v$$ where $\proj_{w_i}(u,v)>K$. See \cite[Theorem 3.3]{BBF}.  This does not certifies the connectedness of $\p_K(\f)$, but also plays a crucial role in the whole theory. Among the useful facts, we mention
\begin{enumerate}
    \item Standard paths are $2$--quasi-geodesics (\cite[Corollary 3.7]{BBFS}).
    \item 
    A triangle formed by three standard paths $\alpha,\beta,\gamma$ is almost a tripod:  $\gamma$ is contained in the union $\alpha\cup\beta$, with an exception of at most two consecutive vertices (\cite[Lemma 3.6]{BBFS}). 
\end{enumerate}

By the rotating family theory, Dahmani-Guirardel-Osin \cite{DGO} proved that 
\begin{thm}
There exists a (not possibly every) large $n>0$ so that the normal closure $G:=\llangle f^n\rrangle $ is a free group of infinite rank.    
\end{thm} 
We recall some ingredients in the proof using projection complex recently given in \cite{BDDKPS} inspired by Clay-Mangahas-Margalit \cite{CMM21}.

Fix a basepoint $v_0:=\ax(f)$ at $\p_K(\f)$, so $v_0$   is   fixed by $E(f)<\Gamma$. By definition, $G$ is generated by all conjugates of $\langle f^n\rangle$ in $\Gamma$. If $n$ is chosen so $\langle f^n\rangle$ is normal in $E(f)$, then $G$ acts  on $\p_K (\f)$ with the stabilizers exactly from those conjugates of $\langle f^n\rangle$. In particular, the stabilizer of $v_0$ is now $\langle f^n\rangle$.

If such  $n\gg 0$ is large enough,  the action $G\act \p_K(\f)$ is proved in \cite[Thm 3.2]{BDDKPS} to be an \textit{$L$--spinning family} in sense of \cite{CMM21}: 
\begin{equation}\label{spinningEQ}
\proj_{v_0}(w,f^nw)>L, \;\forall v_0\ne w\in \f    
\end{equation} 
We follow closely the account of \cite{BDDKPS} to explain.
\\
\paragraph{\textbf{Construction of free base}} We do it inductively and start with $G_0=\langle f^n\rangle$ and $S_0=\{v_0\}$. Let $N_1$ be the $1$--neighborhood of $W_0:=S_0$ and $S_1$ be the set of vertices up to $G_0$--conjugacy in $N_1\setminus W_0$. Set  $G_1=\langle G_v: v\in S_1\cup S_0\rangle$ and $W_1=G_1 N_1$. Then $S_0\cup S_1$ is a fundamental set for the action $G_1\act W_1^0$, \emph{i.e.} containing exactly one point from each orbit. Here $W_1^0$ denotes the vertex set of $W_1$. 

For $i\ge 2$, let $W_i=G_i\cdot N_{i-1}$ and define $N_{i+1}$ to be the  $1$--neighborhood of $W_i$. Let $S_{i+1}$ be the set of vertices up to $G_i$--conjugacy in $N_{i+1}\setminus W_i$, so that each vertex in $S_{i+1}$ is adjacent to some in $S_i$. This can be done, as $W_i^0=G_i\cdot (\cup_{0\le j< i} S_j)$. Then $G_{i+1}:=\langle G_v: v\in \cup_{0\le j\le i} S_j\rangle$  is generated by the stabilizers of vertices in $N_{i+1}$. 
\\
\paragraph{\textbf{Canoeing path}}
Recall that $K$ is a large constant chosen for $\p_K(\f)$. Any two vertices $u,v$ in $W_i$ can be connected by an \textit{$L$--canoeing path} $\gamma$ for some $L=L(K)>0$ (see \cite[Def. 4.3 \& Lemma 4.9]{BDDKPS}): 
\begin{enumerate}
    \item $\gamma$ is a concatenation  $\gamma_1\cdot \gamma_2\cdot\cdots \cdot\gamma_n$ where $\gamma_i$ is a geodesic or the union of two geodesics. 
    \item 
    The common endpoint $v$ of  $\gamma_i$ and $\gamma_{i+1}$ has \textit{$L$--large angle}: if $v_-\in \gamma_i, v_+\in \gamma_{i+1}$ are previous and next vertices, then $\proj_{v}(v_-,v_+)>L$, where $v, v_-,v_+\in \f$ are understood as axes in $\U$. 
\end{enumerate}By \cite[Prop. 4.4]{BDDKPS}, $L$--large angle points are contained in the standard path from $u$ to $v$. 

In \cite{BDDKPS}, through the canoe path, it is inductively proven that $G_{i+1}$ is generated by $G_i$ with a free group generated by $S_i$. Exhausting $\p_K(\f)$ by $W_i$, $G$ is obtained as the direct limit of $G_i$, which is a free group generated by stabilizers of vertices in $S:=\cup_{0\le j<\infty} S_j$.  Recall that $G_0=\langle f^n\rangle$ and $S_0=\{v_0\}$, and consider the splitting   $$G=G_0 \star H$$ where $H:=\langle G_v: v\in S\setminus S_0  \rangle$ is a free group. We also need the associated Bass-Serre tree $T$. More precisely, as the quotient graph $T/G$ is an interval, we partition $T^0=V_1\cup V_2$ with the stabilizers of vertices in $V_1$ being conjugate to $G_0$ and the stabilizers of vertices in $V_2$ to $H$. Moreover,
\begin{itemize}
    \item 
    the set of $G_0$--vertices (resp. $H$--vertices) in $V_1$ (resp. $V_2$) are labeled by  left $G_0$--cosets (resp.  $H$--cosets) in $G$;
    \item 
    the set of edges adjacent to $u\in V_1$ (resp. $u\in V_2$) are bijective to the elements in $G_0$ (resp. $H$).
\end{itemize} 

\subsection{Subgroups with nontrivial conservative part}  
We now state the main theorem in this section.
\begin{thm}\label{PositiveHorLimitset}
The subgroup $H$ has proper limit set $[\Lambda(Ho)]\subsetneq [\Lambda(\Gamma o)]$. Moreover, If $\e G<\e \Gamma$, then $\mu_o(\HG)\ge \mu_o(\hG)>0$.  
\end{thm}
In particular, since $H$ has a proper limit set, it has maximal quotient growth by \cite{HYZ}. By Corollary \ref{CorSlower}, which will be proven in Section \ref{secmaxgrowth}, Theorem \ref{PositiveHorLimitset} implies that the action of $H$ has non-trivial Hopf decomposition. 
\begin{rem}
In the construction, we have that $\Gamma/G$ is non-elementary acylindrically hyperbolic, in particular it is non-amenable. The amenability conjecture states that for a normal subgroup $G$, $\e G=\e \Gamma$ if and only if the quotient group $\Gamma/G$ is amenable. If the amenability conjecture is verified for any SCC action, then the assumption that $\e G<\e \Gamma$ could be removed.
\end{rem}
The remainder of this subsection is devoted to the proof of Theorem \ref{PositiveHorLimitset}, which relies on the following decomposition of the limit set of $G$. 

The end boundary $\partial T$ of  the Bass-Serre tree $T$ associated to $G_0\star H$ consists of all geodesic rays from a fixed point. As $T$ is locally infinite, it is non-compact. Note that  $[\Lambda(G_0o)]=[f^-]\cup[f^+]$ for $G_0=\langle f^n\rangle$. In what follows, we understand $[\pG]$ as the quotient space by identifying each $[\cdot]$--class as one point. Recall that $\mathcal C$ is the set of non-pinched points in $\pU$ (Definition \ref{ConvBdryDefn}(C)), and the points outside  $\mathcal C$ are \textit{pinched}.
\begin{prop}\label{limitsetdecomp} 
We have the following covering \begin{equation}\label{LimitSetCover}
[\Lambda(Go)] = \left[\bigcup_{g\in G} g[\Lambda(Ho)]\right] \bigcup  \left[\bigcup_{g\in G} g[f^\pm]\right] \bigcup \iota(\partial T)    
\end{equation}
where  a map $\iota: \partial T\to [\pG]$ is defined below in Lemma \ref{BSTreeEnds}. 
Moreover, this covering becomes a disjoint union if every term on RHS and LHS is replaced by its intersection with the set of non-pinched points $\mathcal C$.
\end{prop}
In a variety of examples, it is known that $\pU=\mathcal C$ consists of non-pinched points, so the union in (\ref{LimitSetCover}) becomes a disjoint union. We record this for further reference.  
\begin{cor}\label{limitsetdecompcor}
Suppose that $(\U, \pU)$ is one of the items (1),(2),(3) and (5) in Examples \ref{ConvbdryExamples}. Then the following holds: \begin{equation}\label{LimitSetDisjCover}
[\Lambda(Go)] = \left[\bigsqcup_{g\in G} g[\Lambda(Ho)]\right] \bigsqcup  \left[\bigsqcup_{g\in G} g[f^\pm]\right] \bigcup \iota(\partial T)    
\end{equation}
where  a map $\iota: \partial T\to [\pG]$ is defined below in Lemma \ref{BSTreeEnds}. 
\end{cor}
As shown in (\ref{LimitSetDisjCover}), the first points of any contracting elements other than $f$ must be contained in some $G$--translate of $\Lambda(Ho)$.

In the next three steps, we shall describe in detail how a geodesic on the Bass-Serre tree $T$ produces a canoeing path in the projection complex $\p_K(\f)$, which is then lifted to an admissible path in the space $\U$. This is the crux of the proof of Proposition \ref{limitsetdecomp}.

Let $U$ be a non-empty reduced  word over the alphabet $(H\cup G_0)\setminus 1$. We decompose $U$ into nonempty subwords $U_i$ separated by letters in $G_0=\langle f^n \rangle$: 
\begin{align}\label{AlternatingWord}
U=(h_1, f^{n_1}, h_2,f^{n_2},\cdots, h_k, f^{n_k})    
\end{align} where  $h_i\in H$ and $n|n_i\ne 0$ unless the last one $i=k$. For each $h_i$, we have a minimal $j=j(U_i)$ so that  $h_i$  represents an element in $G_{j}\setminus G_0\subseteq H$. Then the element $g=h_1f^{n_1}\cdots h_kf^{n_k}$ is represented by $U$. If $g$ is not in $H$, then $k\ge 2$ and $n_1\ne 0$.  

Fix an oriented edge $e=[\mathbf v_0,\mathbf v_1]$ in $T$ with endpoint stabilizers $G_0$ and $H$ respectively. Denote   $\overleftarrow e=[\mathbf v_1,\mathbf v_0]$ the reversed edge. We consider  the basepoint $\mathbf v_0$ corresponding to the coset $G_0$ in $T$, and  the basepoint  $v_0=\ax(f)$ in $\p_K(\f)$ stabilized by $G_0$.
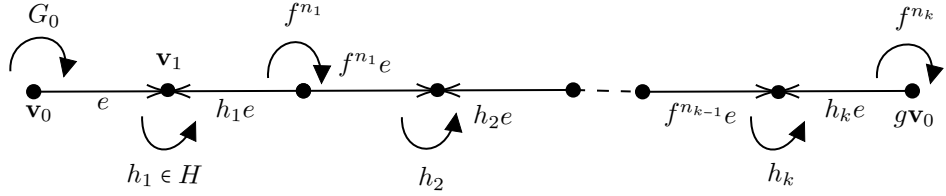
\begin{figure}
    \centering

\tikzset{every picture/.style={line width=0.75pt}} %set default line width to 0.75pt        

\begin{tikzpicture}[x=0.75pt,y=0.75pt,yscale=-1,xscale=1]

%Straight Lines [id:da73681241277161] 
\draw    (73.41,80.04) -- (139.44,79.82) ;
\draw [shift={(141.44,79.82)}, rotate = 179.82] [color={rgb, 255:red, 0; green, 0; blue, 0 }  ][line width=0.75]    (10.93,-3.29) .. controls (6.95,-1.4) and (3.31,-0.3) .. (0,0) .. controls (3.31,0.3) and (6.95,1.4) .. (10.93,3.29)   ;
\draw [shift={(73.41,80.04)}, rotate = 359.82] [color={rgb, 255:red, 0; green, 0; blue, 0 }  ][fill={rgb, 255:red, 0; green, 0; blue, 0 }  ][line width=0.75]      (0, 0) circle [x radius= 3.35, y radius= 3.35]   ;
%Straight Lines [id:da2972980234183604] 
\draw    (143.44,79.81) -- (209.47,79.6) ;
\draw [shift={(209.47,79.6)}, rotate = 359.82] [color={rgb, 255:red, 0; green, 0; blue, 0 }  ][fill={rgb, 255:red, 0; green, 0; blue, 0 }  ][line width=0.75]      (0, 0) circle [x radius= 3.35, y radius= 3.35]   ;
\draw [shift={(141.44,79.82)}, rotate = 359.82] [color={rgb, 255:red, 0; green, 0; blue, 0 }  ][line width=0.75]    (10.93,-3.29) .. controls (6.95,-1.4) and (3.31,-0.3) .. (0,0) .. controls (3.31,0.3) and (6.95,1.4) .. (10.93,3.29)   ;
%Curve Lines [id:da38337964132749036] 
\draw    (61.67,69.73) .. controls (60.3,48.45) and (92.91,45.77) .. (89.17,69.72) ;
\draw [shift={(88.61,72.43)}, rotate = 284.03] [fill={rgb, 255:red, 0; green, 0; blue, 0 }  ][line width=0.08]  [draw opacity=0] (8.93,-4.29) -- (0,0) -- (8.93,4.29) -- cycle    ;
%Curve Lines [id:da29122552846427685] 
\draw    (128.1,92.47) .. controls (127.6,111.57) and (143.92,115.59) .. (153.7,97.83) ;
\draw [shift={(155.04,95.18)}, rotate = 114.7] [fill={rgb, 255:red, 0; green, 0; blue, 0 }  ][line width=0.08]  [draw opacity=0] (8.93,-4.29) -- (0,0) -- (8.93,4.29) -- cycle    ;
%Curve Lines [id:da8240169875859356] 
\draw    (191.63,73.13) .. controls (190.25,51.75) and (217.77,46.32) .. (218.57,73.24) ;
\draw [shift={(218.57,75.84)}, rotate = 271.79] [fill={rgb, 255:red, 0; green, 0; blue, 0 }  ][line width=0.08]  [draw opacity=0] (8.93,-4.29) -- (0,0) -- (8.93,4.29) -- cycle    ;
%Straight Lines [id:da8494612424683525] 
\draw    (209.47,79.6) -- (275.5,79.39) ;
\draw [shift={(277.5,79.38)}, rotate = 179.82] [color={rgb, 255:red, 0; green, 0; blue, 0 }  ][line width=0.75]    (10.93,-3.29) .. controls (6.95,-1.4) and (3.31,-0.3) .. (0,0) .. controls (3.31,0.3) and (6.95,1.4) .. (10.93,3.29)   ;
\draw [shift={(209.47,79.6)}, rotate = 359.82] [color={rgb, 255:red, 0; green, 0; blue, 0 }  ][fill={rgb, 255:red, 0; green, 0; blue, 0 }  ][line width=0.75]      (0, 0) circle [x radius= 3.35, y radius= 3.35]   ;
%Straight Lines [id:da3204131750340864] 
\draw    (279.5,79.37) -- (345.52,79.16) ;
\draw [shift={(345.52,79.16)}, rotate = 359.82] [color={rgb, 255:red, 0; green, 0; blue, 0 }  ][fill={rgb, 255:red, 0; green, 0; blue, 0 }  ][line width=0.75]      (0, 0) circle [x radius= 3.35, y radius= 3.35]   ;
\draw [shift={(277.5,79.38)}, rotate = 359.82] [color={rgb, 255:red, 0; green, 0; blue, 0 }  ][line width=0.75]    (10.93,-3.29) .. controls (6.95,-1.4) and (3.31,-0.3) .. (0,0) .. controls (3.31,0.3) and (6.95,1.4) .. (10.93,3.29)   ;
%Curve Lines [id:da5512710229934352] 
\draw    (259.13,92.9) .. controls (258.63,112.1) and (280.63,115.72) .. (285.93,93.87) ;
\draw [shift={(286.5,91)}, rotate = 99.09] [fill={rgb, 255:red, 0; green, 0; blue, 0 }  ][line width=0.08]  [draw opacity=0] (8.93,-4.29) -- (0,0) -- (8.93,4.29) -- cycle    ;
%Straight Lines [id:da7541318089828473] 
\draw  [dash pattern={on 4.5pt off 4.5pt}]  (345.52,79.16) -- (374.51,79.92) ;
%Straight Lines [id:da05704204472508123] 
\draw    (380.51,80.08) -- (446.54,79.86) ;
\draw [shift={(448.54,79.86)}, rotate = 179.82] [color={rgb, 255:red, 0; green, 0; blue, 0 }  ][line width=0.75]    (10.93,-3.29) .. controls (6.95,-1.4) and (3.31,-0.3) .. (0,0) .. controls (3.31,0.3) and (6.95,1.4) .. (10.93,3.29)   ;
\draw [shift={(380.51,80.08)}, rotate = 359.82] [color={rgb, 255:red, 0; green, 0; blue, 0 }  ][fill={rgb, 255:red, 0; green, 0; blue, 0 }  ][line width=0.75]      (0, 0) circle [x radius= 3.35, y radius= 3.35]   ;
%Straight Lines [id:da06967995964109552] 
\draw    (450.54,79.85) -- (516.57,79.64) ;
\draw [shift={(516.57,79.64)}, rotate = 359.82] [color={rgb, 255:red, 0; green, 0; blue, 0 }  ][fill={rgb, 255:red, 0; green, 0; blue, 0 }  ][line width=0.75]      (0, 0) circle [x radius= 3.35, y radius= 3.35]   ;
\draw [shift={(448.54,79.86)}, rotate = 359.82] [color={rgb, 255:red, 0; green, 0; blue, 0 }  ][line width=0.75]    (10.93,-3.29) .. controls (6.95,-1.4) and (3.31,-0.3) .. (0,0) .. controls (3.31,0.3) and (6.95,1.4) .. (10.93,3.29)   ;
%Curve Lines [id:da9077658112889653] 
\draw    (435.21,92.51) .. controls (434.71,111.61) and (451.02,115.63) .. (460.81,97.87) ;
\draw [shift={(462.14,95.22)}, rotate = 114.7] [fill={rgb, 255:red, 0; green, 0; blue, 0 }  ][line width=0.08]  [draw opacity=0] (8.93,-4.29) -- (0,0) -- (8.93,4.29) -- cycle    ;
%Curve Lines [id:da6748637216550473] 
\draw    (498.73,73.17) .. controls (497.36,51.9) and (529.97,49.21) .. (526.23,73.16) ;
\draw [shift={(525.67,75.87)}, rotate = 284.03] [fill={rgb, 255:red, 0; green, 0; blue, 0 }  ][line width=0.08]  [draw opacity=0] (8.93,-4.29) -- (0,0) -- (8.93,4.29) -- cycle    ;
%Shape: Free Drawing [id:dp6465305973888171] 
\draw  [line width=5.25] [line join = round][line cap = round] (140.8,79.22) .. controls (140.8,79.22) and (140.8,79.22) .. (140.8,79.22) ;
%Shape: Free Drawing [id:dp20044558363220877] 
\draw  [line width=5.25] [line join = round][line cap = round] (276.8,79.22) .. controls (276.8,79.22) and (276.8,79.22) .. (276.8,79.22) ;
%Shape: Free Drawing [id:dp6823713108156264] 
\draw  [line width=5.25] [line join = round][line cap = round] (448.8,80.22) .. controls (448.8,80.22) and (448.8,80.22) .. (448.8,80.22) ;

% Text Node
\draw (103.89,81.24) node [anchor=north west][inner sep=0.75pt]  [rotate=-1.5]  {$e$};
% Text Node
\draw (119.02,114.64) node [anchor=north west][inner sep=0.75pt]  [rotate=-1.5]  {$h_{1} \in H$};
% Text Node
\draw (163.89,81.81) node [anchor=north west][inner sep=0.75pt]  [rotate=-1.5]  {$h_{1} e$};
% Text Node
\draw (69.31,34.39) node [anchor=north west][inner sep=0.75pt]  [rotate=-1.5]  {$G_{0}$};
% Text Node
\draw (198.32,33.69) node [anchor=north west][inner sep=0.75pt]  [rotate=-1.5]  {$f^{n_{1}}$};
% Text Node
\draw (225.76,58.55) node [anchor=north west][inner sep=0.75pt]  [rotate=-1.5]  {$f^{n_{1}} e$};
% Text Node
\draw (266.02,116.49) node [anchor=north west][inner sep=0.75pt]  [rotate=-1.5]  {$h_{2}$};
% Text Node
\draw (293.85,85.21) node [anchor=north west][inner sep=0.75pt]  [rotate=-1.5]  {$h_{2} e$};
% Text Node
\draw (389.99,84.27) node [anchor=north west][inner sep=0.75pt]  [rotate=-1.5]  {$f^{n_{k-1}} e$};
% Text Node
\draw (442.14,114.1) node [anchor=north west][inner sep=0.75pt]  [rotate=-1.5]  {$h_{k}$};
% Text Node
\draw (471,81.85) node [anchor=north west][inner sep=0.75pt]  [rotate=-1.5]  {$h_{k} e$};
% Text Node
\draw (506.43,35.73) node [anchor=north west][inner sep=0.75pt]  [rotate=-1.5]  {$f^{n_{k}}$};
% Text Node
\draw (67.77,85.29) node [anchor=north west][inner sep=0.75pt]  [rotate=-1.5]  {$\mathbf{v}_{0}$};
% Text Node
\draw (133.61,59) node [anchor=north west][inner sep=0.75pt]  [rotate=-1.5]  {$\mathbf{v}_{1}$};
% Text Node
\draw (505.88,87.76) node [anchor=north west][inner sep=0.75pt]  [rotate=-1.5]  {$g\mathbf{v}_{0}$};

\end{tikzpicture}
    \caption{Unfolding the word $U$ as a geodesic $\alpha$ in Bass-Serre tree $T$. The reader should be cautioned that the labels indicate the action of stabilizers on vertices understood up to appropriate conjugation.}
    \label{fig:geodesicBStree}
\end{figure}
\begin{enumerate}
    \item [Step 1] The word $U$  gets unfolded  as  a geodesic $\alpha$ in $T$ starting with the edge $ e$, at which $h_1\in H$ rotates the edge $e$ at $\mathbf v_1$ to $h_1  e$. Rotating edges consecutively in this way, we write down $$\alpha= e\cdot (h_1 \overleftarrow  e)\cdot (h_1f^{n_1}  e)\cdot (h_1f^{n_1}h_2  \overleftarrow e)\cdot\cdots\cdot (h_1f^{n_1}\cdots h_k  \overleftarrow e)$$
where the last edge  terminates at the endpoint fixed by corresponding conjugate of $f^{n_k}$. See Fig. \ref{fig:geodesicBStree}.
\item [Step 2]
Recall $W_{j(i)}=G_{j(i)}\cdot N_{j(i)-1}$ and $h_i\in G_{j(i)}$.
As mentioned above, there exists an $L$--canoeing path $\gamma_i$ in $W_{j(i)}$ from $v_0$ to $h_iv_0$. These  $h_i^{-1}\gamma_{i}$ and $\gamma_{i+1}$ are adjacent at $v_0$, around which  $f^{n_i}$ rotates $\gamma_{i+1}$. By (\ref{spinningEQ}),  $d_{v_0}(v,f^{n_i}v)>L$ for any $v\ne v_0$, so $h_iv_0$ is a $L$--large angle point of the concatenation $\gamma_{i}\cdot(h_if^{n_i}\gamma_{i+1})$.  We then obtain an $L$--canoeing path $\gamma$ from $v_0$ to $gv_0$ as the concatenation of $\gamma_i$ $(1\le i\le n)$ up to appropriate translation, 
$$
\gamma =\gamma_1 \cdot (h_1f^{n_1} \gamma_2)\cdot (h_1f^{n_1}h_2 f^{n_2}\gamma_3)\cdot \cdots\cdot (h_1f^{n_1}\cdots h_{k-1}f^{n_{k-1}}\gamma_k) 
$$
where the common endpoints of two consecutive paths are $L$--large angle points fixed by the corresponding conjugates of $f^{n_i}$. See Fig. \ref{fig:canoeingpath}.
\begin{figure}
    \centering

\tikzset{every picture/.style={line width=0.75pt}} %set default line width to 0.75pt        

\begin{tikzpicture}[x=0.75pt,y=0.75pt,yscale=-1,xscale=1]

%Curve Lines [id:da17865239038326242] 
\draw    (100,114) .. controls (132.5,79) and (170.5,89) .. (200,114) ;
\draw [shift={(200,114)}, rotate = 40.28] [color={rgb, 255:red, 0; green, 0; blue, 0 }  ][fill={rgb, 255:red, 0; green, 0; blue, 0 }  ][line width=0.75]      (0, 0) circle [x radius= 3.35, y radius= 3.35]   ;
\draw [shift={(100,114)}, rotate = 312.88] [color={rgb, 255:red, 0; green, 0; blue, 0 }  ][fill={rgb, 255:red, 0; green, 0; blue, 0 }  ][line width=0.75]      (0, 0) circle [x radius= 3.35, y radius= 3.35]   ;
%Curve Lines [id:da25120337970335815] 
\draw    (200,114) .. controls (229.5,171) and (277.5,151) .. (300,114) ;
\draw [shift={(300,114)}, rotate = 301.3] [color={rgb, 255:red, 0; green, 0; blue, 0 }  ][fill={rgb, 255:red, 0; green, 0; blue, 0 }  ][line width=0.75]      (0, 0) circle [x radius= 3.35, y radius= 3.35]   ;
%Curve Lines [id:da6078048482943865] 
\draw    (391,96) .. controls (423.5,61) and (461.5,71) .. (491,96) ;
\draw [shift={(491,96)}, rotate = 40.28] [color={rgb, 255:red, 0; green, 0; blue, 0 }  ][fill={rgb, 255:red, 0; green, 0; blue, 0 }  ][line width=0.75]      (0, 0) circle [x radius= 3.35, y radius= 3.35]   ;
\draw [shift={(391,96)}, rotate = 312.88] [color={rgb, 255:red, 0; green, 0; blue, 0 }  ][fill={rgb, 255:red, 0; green, 0; blue, 0 }  ][line width=0.75]      (0, 0) circle [x radius= 3.35, y radius= 3.35]   ;
%Curve Lines [id:da10177549340472836] 
\draw    (197.5,102) .. controls (205.22,87.53) and (229.7,94.47) .. (216.13,116.54) ;
\draw [shift={(214.5,119)}, rotate = 305.31] [fill={rgb, 255:red, 0; green, 0; blue, 0 }  ][line width=0.08]  [draw opacity=0] (8.93,-4.29) -- (0,0) -- (8.93,4.29) -- cycle    ;
%Curve Lines [id:da8722185753320331] 
\draw    (376.5,97) .. controls (368.78,84.45) and (375.03,66.32) .. (393.45,83.04) ;
\draw [shift={(395.5,85)}, rotate = 225] [fill={rgb, 255:red, 0; green, 0; blue, 0 }  ][line width=0.08]  [draw opacity=0] (8.93,-4.29) -- (0,0) -- (8.93,4.29) -- cycle    ;
%Straight Lines [id:da9513247725481224] 
\draw  [dash pattern={on 4.5pt off 4.5pt}]  (300,114) -- (389.04,96.39) ;
\draw [shift={(391,96)}, rotate = 168.81] [color={rgb, 255:red, 0; green, 0; blue, 0 }  ][line width=0.75]    (10.93,-3.29) .. controls (6.95,-1.4) and (3.31,-0.3) .. (0,0) .. controls (3.31,0.3) and (6.95,1.4) .. (10.93,3.29)   ;
%Curve Lines [id:da7016221526896691] 
\draw    (390,136) .. controls (402.31,120.24) and (374.36,136.5) .. (391.68,104.5) ;
\draw [shift={(392.5,103)}, rotate = 119.2] [color={rgb, 255:red, 0; green, 0; blue, 0 }  ][line width=0.75]    (10.93,-3.29) .. controls (6.95,-1.4) and (3.31,-0.3) .. (0,0) .. controls (3.31,0.3) and (6.95,1.4) .. (10.93,3.29)   ;

% Text Node
\draw (137,71.4) node [anchor=north west][inner sep=0.75pt]    {$\gamma _{1}$};
% Text Node
\draw (236,150.4) node [anchor=north west][inner sep=0.75pt]    {$\gamma _{2}$};
% Text Node
\draw (206.32,77.69) node [anchor=north west][inner sep=0.75pt]  [rotate=-1.5]  {$f^{n_{1}}$};
% Text Node
\draw (92.77,119.29) node [anchor=north west][inner sep=0.75pt]  [rotate=-1.5]  {$v_{0}$};
% Text Node
\draw (170,113.4) node [anchor=north west][inner sep=0.75pt]    {$h_{1} v_{0}$};
% Text Node
\draw (265,92.4) node [anchor=north west][inner sep=0.75pt]    {$h_{1} f^{n_{1}} h_{2} v_{0}$};
% Text Node
\draw (425,76.4) node [anchor=north west][inner sep=0.75pt]    {$\gamma _{k}$};
% Text Node
\draw (484.77,70.29) node [anchor=north west][inner sep=0.75pt]  [rotate=-1.5]  {$gv_{0}$};
% Text Node
\draw (334,135.4) node [anchor=north west][inner sep=0.75pt]    {$h_{1} f^{n_{1}} \cdots h_{k-1} v_{0}$};
% Text Node
\draw (363,52.4) node [anchor=north west][inner sep=0.75pt]    {$f^{n_{k-1}}$};

\end{tikzpicture}
    \caption{Canoeing path $\gamma$ in $\p_K(\f)$: join canoeing paths $\gamma_i$ at large angle points $h_1f^{n_1}\cdots h_if^{n_i}v_0$. Those large points are contained in the standard path $\beta$ (possibly $\beta\ne\gamma$) from $v_0$ to $gv_0$.}
    \label{fig:canoeingpath}
\end{figure}
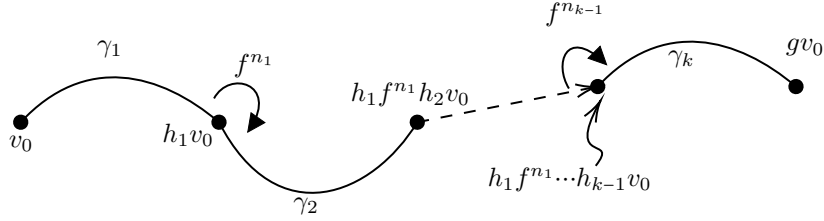
\item [Step 3]
By \cite[Prop. 4.4]{BDDKPS}, those $L$--large angle points lie  on the standard path $\beta$ from $v_0$ to $gv_0$.
We lift the standard path  $\beta$ to get an $(\hat L,\tau)$--admissible path $\tilde \beta$ in $\U$  as follows, with contracting sets given by the axis $w_i:=h_1f^{n_1}\cdots h_i\ax(f)$ where $w_0=\ax(f), w_{k+1}=g\ax(f)$:

Choose two points $o\in w_0$ and $go\in w_{k+1}$ in $\U$. Connect $o,go$ by concatenating geodesics in consecutive axis $w_i$ and then geodesics from $\pi_{w_{i-1}}(w_{i})$ to $\pi_{w_i}(w_{i-1})$. See a schematic illustration (\ref{fig:liftstandardpath}) of a  lifted path. 

By \cite[Lemma 2.22]{YANG22}, this gives an $(\hat L,\tau)$--admissible path from $o$ to $go$ with associated contracting axes $w_i$'s. The constant $\hat L$ is comparable with $L$ and $\tau$ depends only on $\f$. See \cite{YANG22} for details.
\end{enumerate}

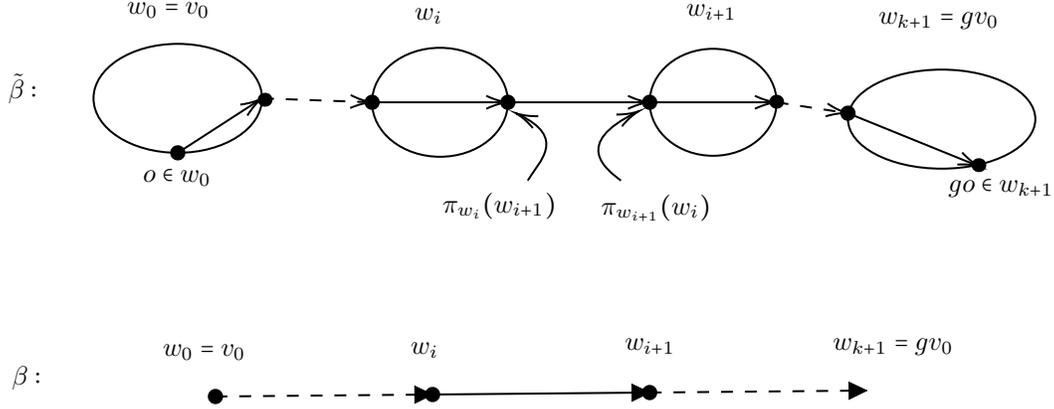
\begin{figure}
    \centering

\tikzset{every picture/.style={line width=0.75pt}} %set default line width to 0.75pt        

\begin{tikzpicture}[x=0.75pt,y=0.75pt,yscale=-1,xscale=1]

%Shape: Ellipse [id:dp8133183803081103] 
\draw   (80.5,98.5) .. controls (80.5,83.31) and (99.53,71) .. (123,71) .. controls (146.47,71) and (165.5,83.31) .. (165.5,98.5) .. controls (165.5,113.69) and (146.47,126) .. (123,126) .. controls (99.53,126) and (80.5,113.69) .. (80.5,98.5) -- cycle ;
%Shape: Ellipse [id:dp9538581916076896] 
\draw   (221,100.5) .. controls (221,85.31) and (236.33,73) .. (255.25,73) .. controls (274.17,73) and (289.5,85.31) .. (289.5,100.5) .. controls (289.5,115.69) and (274.17,128) .. (255.25,128) .. controls (236.33,128) and (221,115.69) .. (221,100.5) -- cycle ;
%Shape: Ellipse [id:dp0570658643817028] 
\draw   (461,109) .. controls (461,95.19) and (482.15,84) .. (508.25,84) .. controls (534.35,84) and (555.5,95.19) .. (555.5,109) .. controls (555.5,122.81) and (534.35,134) .. (508.25,134) .. controls (482.15,134) and (461,122.81) .. (461,109) -- cycle ;
%Shape: Ellipse [id:dp3305440760253098] 
\draw   (361,100.5) .. controls (361,85.59) and (375.33,73.5) .. (393,73.5) .. controls (410.67,73.5) and (425,85.59) .. (425,100.5) .. controls (425,115.41) and (410.67,127.5) .. (393,127.5) .. controls (375.33,127.5) and (361,115.41) .. (361,100.5) -- cycle ;
%Straight Lines [id:da06862882318968766] 
\draw  [dash pattern={on 4.5pt off 4.5pt}]  (165.5,98.5) -- (219,100.43) ;
\draw [shift={(221,100.5)}, rotate = 182.06] [color={rgb, 255:red, 0; green, 0; blue, 0 }  ][line width=0.75]    (10.93,-3.29) .. controls (6.95,-1.4) and (3.31,-0.3) .. (0,0) .. controls (3.31,0.3) and (6.95,1.4) .. (10.93,3.29)   ;
%Straight Lines [id:da3885240733967734] 
\draw  [dash pattern={on 4.5pt off 4.5pt}]  (425,100.5) -- (459.02,105.7) ;
\draw [shift={(461,106)}, rotate = 188.69] [color={rgb, 255:red, 0; green, 0; blue, 0 }  ][line width=0.75]    (10.93,-3.29) .. controls (6.95,-1.4) and (3.31,-0.3) .. (0,0) .. controls (3.31,0.3) and (6.95,1.4) .. (10.93,3.29)   ;
%Straight Lines [id:da45426404787104757] 
\draw    (289.5,100.5) -- (359,100.5) ;
\draw [shift={(361,100.5)}, rotate = 180] [color={rgb, 255:red, 0; green, 0; blue, 0 }  ][line width=0.75]    (10.93,-3.29) .. controls (6.95,-1.4) and (3.31,-0.3) .. (0,0) .. controls (3.31,0.3) and (6.95,1.4) .. (10.93,3.29)   ;
\draw [shift={(289.5,100.5)}, rotate = 0] [color={rgb, 255:red, 0; green, 0; blue, 0 }  ][fill={rgb, 255:red, 0; green, 0; blue, 0 }  ][line width=0.75]      (0, 0) circle [x radius= 3.35, y radius= 3.35]   ;
%Curve Lines [id:da5741489556522821] 
\draw    (299.5,140) .. controls (305.38,130.2) and (319.42,118.48) .. (295.04,106.72) ;
\draw [shift={(293.5,106)}, rotate = 24.36] [color={rgb, 255:red, 0; green, 0; blue, 0 }  ][line width=0.75]    (10.93,-3.29) .. controls (6.95,-1.4) and (3.31,-0.3) .. (0,0) .. controls (3.31,0.3) and (6.95,1.4) .. (10.93,3.29)   ;
%Curve Lines [id:da9053990847778868] 
\draw    (346.5,140) .. controls (338.62,134.09) and (322.01,121.39) .. (354,105.72) ;
\draw [shift={(355.5,105)}, rotate = 154.8] [color={rgb, 255:red, 0; green, 0; blue, 0 }  ][line width=0.75]    (10.93,-3.29) .. controls (6.95,-1.4) and (3.31,-0.3) .. (0,0) .. controls (3.31,0.3) and (6.95,1.4) .. (10.93,3.29)   ;
%Straight Lines [id:da5681871136384862] 
\draw    (221,100.5) -- (287.5,100.5) ;
\draw [shift={(289.5,100.5)}, rotate = 180] [color={rgb, 255:red, 0; green, 0; blue, 0 }  ][line width=0.75]    (10.93,-3.29) .. controls (6.95,-1.4) and (3.31,-0.3) .. (0,0) .. controls (3.31,0.3) and (6.95,1.4) .. (10.93,3.29)   ;
\draw [shift={(221,100.5)}, rotate = 0] [color={rgb, 255:red, 0; green, 0; blue, 0 }  ][fill={rgb, 255:red, 0; green, 0; blue, 0 }  ][line width=0.75]      (0, 0) circle [x radius= 3.35, y radius= 3.35]   ;
%Straight Lines [id:da21630167414658685] 
\draw    (361,100.5) -- (423,100.5) ;
\draw [shift={(425,100.5)}, rotate = 180] [color={rgb, 255:red, 0; green, 0; blue, 0 }  ][line width=0.75]    (10.93,-3.29) .. controls (6.95,-1.4) and (3.31,-0.3) .. (0,0) .. controls (3.31,0.3) and (6.95,1.4) .. (10.93,3.29)   ;
\draw [shift={(361,100.5)}, rotate = 0] [color={rgb, 255:red, 0; green, 0; blue, 0 }  ][fill={rgb, 255:red, 0; green, 0; blue, 0 }  ][line width=0.75]      (0, 0) circle [x radius= 3.35, y radius= 3.35]   ;
%Straight Lines [id:da9618236987765327] 
\draw  [dash pattern={on 4.5pt off 4.5pt}]  (142,249) -- (248.5,248.03) ;
\draw [shift={(251.5,248)}, rotate = 179.48] [fill={rgb, 255:red, 0; green, 0; blue, 0 }  ][line width=0.08]  [draw opacity=0] (8.93,-4.29) -- (0,0) -- (8.93,4.29) -- cycle    ;
\draw [shift={(142,249)}, rotate = 359.48] [color={rgb, 255:red, 0; green, 0; blue, 0 }  ][fill={rgb, 255:red, 0; green, 0; blue, 0 }  ][line width=0.75]      (0, 0) circle [x radius= 3.35, y radius= 3.35]   ;
%Straight Lines [id:da39541325955792983] 
\draw    (251.5,248) -- (358,247.03) ;
\draw [shift={(361,247)}, rotate = 179.48] [fill={rgb, 255:red, 0; green, 0; blue, 0 }  ][line width=0.08]  [draw opacity=0] (8.93,-4.29) -- (0,0) -- (8.93,4.29) -- cycle    ;
\draw [shift={(251.5,248)}, rotate = 359.48] [color={rgb, 255:red, 0; green, 0; blue, 0 }  ][fill={rgb, 255:red, 0; green, 0; blue, 0 }  ][line width=0.75]      (0, 0) circle [x radius= 3.35, y radius= 3.35]   ;
%Straight Lines [id:da5893028725052216] 
\draw  [dash pattern={on 4.5pt off 4.5pt}]  (361,247) -- (467.5,246.03) ;
\draw [shift={(470.5,246)}, rotate = 179.48] [fill={rgb, 255:red, 0; green, 0; blue, 0 }  ][line width=0.08]  [draw opacity=0] (8.93,-4.29) -- (0,0) -- (8.93,4.29) -- cycle    ;
\draw [shift={(361,247)}, rotate = 359.48] [color={rgb, 255:red, 0; green, 0; blue, 0 }  ][fill={rgb, 255:red, 0; green, 0; blue, 0 }  ][line width=0.75]      (0, 0) circle [x radius= 3.35, y radius= 3.35]   ;
%Straight Lines [id:da5510369540343465] 
\draw    (123,126) -- (163.82,99.59) ;
\draw [shift={(165.5,98.5)}, rotate = 147.09] [color={rgb, 255:red, 0; green, 0; blue, 0 }  ][line width=0.75]    (10.93,-3.29) .. controls (6.95,-1.4) and (3.31,-0.3) .. (0,0) .. controls (3.31,0.3) and (6.95,1.4) .. (10.93,3.29)   ;
\draw [shift={(123,126)}, rotate = 327.09] [color={rgb, 255:red, 0; green, 0; blue, 0 }  ][fill={rgb, 255:red, 0; green, 0; blue, 0 }  ][line width=0.75]      (0, 0) circle [x radius= 3.35, y radius= 3.35]   ;
%Straight Lines [id:da6732056349546951] 
\draw    (461,106) -- (525.65,132.25) ;
\draw [shift={(527.5,133)}, rotate = 202.1] [color={rgb, 255:red, 0; green, 0; blue, 0 }  ][line width=0.75]    (10.93,-3.29) .. controls (6.95,-1.4) and (3.31,-0.3) .. (0,0) .. controls (3.31,0.3) and (6.95,1.4) .. (10.93,3.29)   ;
\draw [shift={(461,106)}, rotate = 22.1] [color={rgb, 255:red, 0; green, 0; blue, 0 }  ][fill={rgb, 255:red, 0; green, 0; blue, 0 }  ][line width=0.75]      (0, 0) circle [x radius= 3.35, y radius= 3.35]   ;
%Shape: Free Drawing [id:dp7056560271799237] 
\draw  [line width=5.25] [line join = round][line cap = round] (166.8,99.22) .. controls (166.8,99.22) and (166.8,99.22) .. (166.8,99.22) ;
%Shape: Free Drawing [id:dp5227058445043054] 
\draw  [line width=5.25] [line join = round][line cap = round] (424.8,100.22) .. controls (424.8,100.22) and (424.8,100.22) .. (424.8,100.22) ;
%Shape: Free Drawing [id:dp9503890572435096] 
\draw  [line width=5.25] [line join = round][line cap = round] (526.8,132.22) .. controls (526.8,132.22) and (526.8,132.22) .. (526.8,132.22) ;

% Text Node
\draw (96,48.4) node [anchor=north west][inner sep=0.75pt]    {$w_{0} =v_{0}$};
% Text Node
\draw (241,51.4) node [anchor=north west][inner sep=0.75pt]    {$w_{i}$};
% Text Node
\draw (378,49.4) node [anchor=north west][inner sep=0.75pt]    {$w_{i+1}$};
% Text Node
\draw (475,53.4) node [anchor=north west][inner sep=0.75pt]    {$w_{k+1} =gv_{0}$};
% Text Node
\draw (255,144.4) node [anchor=north west][inner sep=0.75pt]    {$\pi _{w_{i}}( w_{i+1})$};
% Text Node
\draw (335,145.4) node [anchor=north west][inner sep=0.75pt]    {$\pi _{w_{i+1}}( w_{i})$};
% Text Node
\draw (239,220.4) node [anchor=north west][inner sep=0.75pt]    {$w_{i}$};
% Text Node
\draw (347,218.4) node [anchor=north west][inner sep=0.75pt]    {$w_{i+1}$};
% Text Node
\draw (452,218.4) node [anchor=north west][inner sep=0.75pt]    {$w_{k+1} =gv_{0}$};
% Text Node
\draw (114,221.4) node [anchor=north west][inner sep=0.75pt]    {$w_{0} =v_{0}$};
% Text Node
\draw (104,132.4) node [anchor=north west][inner sep=0.75pt]    {$o\in w_{0}$};
% Text Node
\draw (510.25,137.4) node [anchor=north west][inner sep=0.75pt]    {$go\in w_{k+1}$};
% Text Node
\draw (38,232.4) node [anchor=north west][inner sep=0.75pt]    {$\beta :$};
% Text Node
\draw (36,84.4) node [anchor=north west][inner sep=0.75pt]    {$\tilde{\beta } :$};

\end{tikzpicture}
    \caption{Lift a standard path $\beta$ (at bottom) in $\p_K(\f)$ to admissible path (at top) in $\U$}
    \label{fig:liftstandardpath}
\end{figure}

%This allows to see the splitting   $G=G_0 \star H$, where $H:=\langle G_v: v\in S\setminus S_0  \rangle$ is a free group. 
We first establish the desired map mentioned in Proposition \ref{limitsetdecomp}.
\begin{lem}\label{BSTreeEnds}
There exists a map  $\iota: \partial T\to[\pG]$ so that the image $\iota(\partial T)$ is contained in the conical limit set $[\Lambda^{\mathrm{con}}(Go)]$. Moreover, if $\iota(p)=\iota(q)$ for $p\ne q$, then $\iota(p)$ is outside the set  $\mathcal C$ of non-pinched points.    
\end{lem}
\begin{proof}
By construction, $T^0=\{gH, gG_0: g\in G\}$, so any geodesic ray $\alpha$ in $T$ from $v_0=G_0$ has the list of all vertices  $A_n$, which are alternating left $H$ and $G_0$ cosets. This determines a unique infinite word of form (\ref{AlternatingWord}). By Step (3), the lifted $(\tilde L,\tau)$--admissible path $\tilde \beta$ in $\U$ intersects infinitely many translates $X_n$ of $N_r(\ax(f))$ under $G$, as $A_n$ contains infinitely many $G_0$--cosets. By Definition \ref{ConvBdryDefn}(B), $X_n$ accumulates into a $[\cdot]$--class $[\xi]$. If we choose a sequence $g_n\in A_n$, we see $d(g_no, [o,\xi])\le r$ so $\xi$ lies in $[\Lambda^{\mathrm{con}}(Go)]$. Setting $\iota(\alpha^+)=[\xi]$ defines  the desired map  $\iota: \partial T\to [\Lambda^{\mathrm{con}}(Go)]$.   

To justify the claim of injectivity, consider a distinct geodesic $\alpha\ne \beta$ originating at $v_0$ with the set of vertices $B_n$, so that $[\xi]=\iota(\alpha^+)=\iota(\beta^+)$. We shall show that $[\xi]$ is pinched. Let $\gamma$ be the bi-infinite geodesic in the tree $T$ from $\alpha^+$ to $\beta^+$, \emph{i.e.}: $\gamma=(\alpha\cup\beta)\setminus [v_0,w_0)$ as a set, where $w_0$ is the  vertex that  $\beta$ departs from  $\alpha$. As shown in the first paragraph, there are two sequences $g_no\in A_no$ and $g'_no\in B_no$ tending to $[\xi]$.

We read off from $\gamma$ a bi-infinite alternating word $U$ in the form (\ref{AlternatingWord}).  As described in Step (3), we obtain from $\gamma_n$  an $(\tilde L,\tau)$--admissible path $\hat \gamma_n$ from  $g_no$  to $g_n'o$, where the contracting subsets correspond to the $G_0$--vertices on $\gamma$. By Proposition \ref{admisProp}, $[g_n, g_n'o]$ has at most  $r$--distance to the entry and exit points of the contracting subsets. See Fig. \ref{fig:admissiblepath}.  
Recall that $T$ is a bipartite graph with $G_0$--vertices and $H$--vertices. Look at   the first $G_0$--vertex $w$ on $\beta$ after $w_0$, and its exit point $x$ of the contracting subset. By the above discussion, we have $d(x,[g_no,g_n'o])\le r$, that is $[g_no,g'_no]$ is non-escaping. Hence, $[\xi]\notin \mathcal C$ is pinched. 
\end{proof}
%There exists $\tau>0$ such that   $\proj_{g\ax(f)}(o,ho)\le \tau$ for any $h\in H$ and $g\in \Gamma$. If $\tau$ is large enough relative to $K$, then $\proj_{g\ax(f)}(o,ho)>\tau$ implies $\proj_{g\ax(f)}(\ax(f),h\ax(f))>K$. By definition,  $g\ax(f)$ lies in the standard path from $v_0=\ax(f)$ to $hv_0$.   
We continue to examine the possible intersection in (\ref{LimitSetCover}).
\begin{lem}\label{DisjointHLimitSet}
The following hold.
\begin{enumerate}
    \item 
    For any $g\in G\setminus H$, the intersection $[\Lambda(Ho)]\cap g[\Lambda(Ho)]$ consist of only pinched points.
    \item 
    $[\Lambda(Ho)]$ intersects $\iota(\partial T)$ only in pinched points.
    \item 
    $[\Lambda(Ho)]\cap g[f^\pm]=\emptyset$ for any $1\ne g\in G$.
\end{enumerate} 
\end{lem}
\begin{proof}
\textbf{(1).} Without loss of generality, we may represent $g$ as  a nonempty alternating   word $U=\mathfrak h_1f^{n_1}\cdots  \mathfrak h_kf^{n_k}$ as in (\ref{AlternatingWord}), ending with letters in $G_0$. Let $h_no\in Ho$ and $g_no\in gHo$ tend to the same  $[\xi]$. We will show $[\xi]\notin \mathcal C$ ia pinched: that is, $\gamma_n:=[h_no, g_no]$  intersects a compact set for all $n\ge 1$.  

To this end, consider the nonempty alternating   word $W$ representing $h_n^{-1}g_n$. As $h_n\in H$ and $g_n\in gH$, the word $W$ has the form $\tilde h_n f^{n_1}\cdots  \mathfrak h_kf^{n_k}$, where $\tilde h_n =h_n^{-1}\mathfrak h_1$ gives one letter in $H$. We obtain from $W$, as described in the Step (3) above, an $(L,\tau)$--admissible path $\beta$, which has the same endpoints $o, h_n^{-1}g_no$ as $h_n^{-1}\gamma_n$. Note that $\mathfrak h_1f^{n_1} o$ is the exit vertex of the admissible path $h_n\beta$. Up to translation, we obtain that $\gamma_n$ has $r$--distance to the fixed point $\mathfrak h_1f^{n_1} o$ by Proposition \ref{admisProp}.  This shows that $[\xi]$ lies  outside of the non-pinched points $\mathcal C$. 

\textbf{(2).} The same reasoning shows that $\Lambda(Ho)$ intersects $\iota(\partial T)$ only in pinched points. Indeed, if $[\xi]$ lies in $\iota(\partial T)$, then it is an accumulation point of $g_no$, where $g_n$ lies in infinitely many $G_0$--cosets (see the proof of Lemma \ref{BSTreeEnds}). The same argument as above implies that $\gamma_n$ is non-escaping. 

\textbf{(2).} The same argument proves that $\Lambda(Ho)\cap g[f^\pm]$ must be pinched. However, since $[f^\pm]$ is assumed to be non-pinched points, we obtain $\Lambda(Ho)\cap g[f^\pm]=\emptyset$. 
\end{proof}

It is now fairly easy to derive from Lemma \ref{BSTreeEnds} and Lemma \ref{DisjointHLimitSet}:
\begin{proof}[Proof of Proposition \ref{limitsetdecomp}]
Note that $G$ is the union of all left $G_0$--cosets and $H$--cosets. Consider  a limit point $\xi\in [\Lambda(Go)]$ in the following complement 
$$[\Lambda(Go)] \setminus  \left[\bigcup_{g\in G} g[\Lambda(Ho)]\right] \bigcup  \left[\bigcup_{g\in G} g[\Lambda(G_0o)]\right] $$
so that $g_no \to[\xi]$, for a sequence of $g_n\in G$ belonging to infinitely many distinct $G_0$--cosets and/or $H$--cosets denoted by $ A_n$. By the above construction of Bass-Serre tree, $A_n$  corresponds to vertices on $T$, so after passing to subsequence,  we obtain a limiting geodesic ray $\alpha$ in $T$ starting at $\mathbf{v_0}$: that is, the intersection of $\alpha$ with $[\mathbf{v_0},g_n\mathbf{v_0}]$ becomes unbounded as $n\to\infty$. By the proof of Lemma \ref{BSTreeEnds}, this implies $\xi$ is contained in $\iota(\partial T)$.

The ``moreover" statement follows from Lemma \ref{DisjointHLimitSet}. The proof is now complete.  
\end{proof}

\begin{proof}[Proof of Theorem \ref{PositiveHorLimitset}]
First of all, as $G<\Gamma$ is an infinite normal subgroup,   $\mu_o(\Lambda^{\mathrm{Hor}}(Go))=\mu_o(\Lambda^{\mathrm{hor}}(Go))=1$ by Theorem \ref{ConfinedConsThm}. 

By Proposition \ref{limitsetdecomp}, $[\Lambda(Go)]$ is covered by a countable union $G\cdot [\Lambda (Ho)]$ with the remaining ``exceptional" set  $E$:
$$
E:=\left[\bigcup_{g\in G} g[\Lambda(G_0o)]\right] \cup \iota(\partial T).
$$  As $\iota(\partial T)$ is contained in the conical points, so is shadowed by the elements in $Go$. By the assumption that $\e G<\e \Gamma$, we have $\iota(\partial T)$  is $\mu_o$--null by Lemma \ref{ConvLimitNull}. This then proves $\mu_o(E)=0$.

This concludes the proof modulo of the following claim.
\begin{claim}
A  small/big horospheric limit points of $G$ is either in   $E$ or a  small/big horospheric limit point in  a $G$--translate of $\Lambda(Ho)$. That is,
$$
[\Lambda^{\mathrm{Hor}}(Go)]\subseteq E\bigcup \left[\bigcup_{g\in G} g[\Lambda^{\mathrm{Hor}}(Ho)]\right]
$$
\end{claim}
\begin{proof}[Proof of the claim]Indeed, let $\xi\in \Lambda^{\mathrm{Hor}}(Go)\setminus E$.   Then there exists a sequence $g_no\in Go$ tending to $[\xi]$ in some (or any)  horoball $\mathcal{HB}([\xi])$ based at $[\xi]$. As $\xi\notin E$, $g_n\in G$ lies in finitely many $H$--cosets, and after passing to subsequence, we may assume $g_nH=g_0H$ for $n\gg 1$. By Lemma \ref{DisjointHLimitSet}, $g_0\Lambda (Ho)\cap \Lambda (Ho)=\emptyset$ for $g_0\notin H$, so we conclude that $g_0H=H$ and thus $\xi\in \Lambda^{\mathrm{Hor}}(Ho)$.  
%Note that $g_n\in g_0H$ is represented by $W_n=(U, U_n)$, where the letter $U_n$ represents $h_n=g_0^{-1}g_n\in H$.   As in Step (2), this produces canoeing path $\gamma_n$ with large angle point corresponding to the letter $f^{n_i}$ in $U$ for $1\le i\le k$  in $\p_K(\f)$.  We then obtain a lifted admissible path from $o$ to $g_no$ with associated axis $g_0\ax(f)$ according to the last letter $f^{n_k}$ in $U$. This implies $N_r(g_0\ax(f))\cap [o,g_no]$ has diameter at least $L$ by Proposition \ref{admisProp}. Consider the triangle $\Delta$ with vertices $o, g_no, \xi$. As $g_0$ is fixed and then $d(o,h_no)\to\infty$, we see that $g_0\ax(f)\cap [o,\xi]\ne\emptyset$. This implies $\xi$ is a big/small horospheric limit point.  As $g_no$ lies in the horoball $\mathcal{HB}(\xi)$, then $h_no$ lies in the horoball  $\mathcal{HB}(\xi)$ as well.      
\end{proof}
As $E$ is $\mu_o$--null, we thus obtain from Proposition \ref{limitsetdecomp} that $\mu_o(\HG)\ge \mu_o(\hG)>0$.
\end{proof}

We conclude this section with a corollary which might be of independent interest. Define a topology on $M:=T^0\cup \partial T$ as follows (cf. \cite{CSW}). A sequence of points $x_n\to x$ if and only if the geodesic $[x_n,x]$ eventually passes no element in any given finite set of edges (which is required adjacent to $x$ if $x\in T^0$). This makes $M$ homeomorphic to a Cantor space, in particular, it is compact. If $T$ is the Bass-Serre tree of $\Gamma=H\star K$, then $M$ is exactly the Bowditch boundary of a hyperbolic group $\Gamma$  relative to $\{H, K\}$. Moreover, if both $H, K$ are one-ended, it is also the end boundary of $\Gamma$.

\begin{cor}
In the setting of Corollary \ref{limitsetdecompcor}, there exists a continuous surjective $G$--map from $\Lambda(Go)$ to   $M$, where $M$ can be identified with the Bowditch boundary of $G$ relative to factors $G_0$ and $H$.     
\end{cor}
The reader may have noticed that item (4) in Examples \ref{ConvbdryExamples} is not covered in the corollaries. It is not clear whether the decomposition holds on the Thurston boundary without taking the intersection with $\ue$. 
 
% For alignments use AmS-LaTeX constructions not \eqnarray.

% For acknowledgements use \ack or \acks immediately before the references

\bibliographystyle{amsplain}
 \bibliography{bibliography}

\end{document}